\documentclass[12pt]{amsart}
\author{Roman Bezrukavnikov and Ivan Losev}
\title{Etingof's conjecture for quantized quiver varieties}
\usepackage{amssymb,amsfonts,amsmath,longtable}
\newcommand{\K}{\mathbb{C}}
\newcommand{\F}{\operatorname{F}}

\newcommand{\Fi}{\mathbb{F}}

\newcommand{\g}{\mathfrak{g}}
\newcommand{\Spec}{\operatorname{Spec}}
\newcommand{\Z}{\mathbb{Z}}
\newcommand{\A}{\mathcal{A}}
\newcommand{\Af}{\mathfrak{A}}
\newcommand{\gr}{\operatorname{gr}}

\newcommand{\Str}{\mathcal{O}}
\newcommand{\h}{\mathfrak{h}}

\newcommand\M{\mathcal{M}}
\newcommand\Loc{\operatorname{Loc}}
\newcommand\Ext{\operatorname{Ext}}

\newcommand\GL{\operatorname{GL}}
\newcommand\gl{\mathfrak{gl}}

\newcommand\red{/\!\!/\!\!/}
\newcommand\I{\mathcal{I}}

\newcommand\J{\mathcal{J}}

\newcommand\End{\operatorname{End}}

\renewcommand\sl{\mathfrak{sl}}

\newcommand\Hom{\operatorname{Hom}}

\newcommand\p{\mathfrak{p}}
\renewcommand\a{\mathfrak{a}}
\newcommand\U{\mathcal{U}}
\newcommand\quo{/\!/}

\newcommand\SL{\operatorname{SL}}

\newcommand\C{\mathbb{C}}

\newcommand\param{\mathfrak{p}}
\newcommand{\paramq}{\mathfrak{P}}

\newcommand\ZZ{\mathbb{Z}}
\newcommand{\CC}{\mathsf{CC}}

\newcommand{\WC}{\mathfrak{WC}}

\newcommand{\Q}{\mathbb{Q}}
\newcommand{\VA}{\operatorname{V}}
\newcommand{\B}{\mathcal{B}}
\newcommand{\Weyl}{\mathbf{A}}
\newcommand{\Supp}{\operatorname{Supp}}
\newcommand{\Tor}{\operatorname{Tor}}
\newcommand{\HC}{\operatorname{HC}}
\newcommand{\AL}{\mathfrak{AL}}
\newcommand{\Cat}{\mathcal{C}}
\newcommand{\Coh}{\operatorname{Coh}}
\newcommand{\slf}{\mathfrak{sl}}
\newcommand{\Irr}{\operatorname{Irr}}
\newtheorem{Thm}{Theorem}[section]
\newtheorem{Prop}[Thm]{Proposition}
\newtheorem{Cor}[Thm]{Corollary}
\newtheorem{Lem}[Thm]{Lemma}
\theoremstyle{definition}

\newtheorem{defi}[Thm]{Definition}
\newtheorem{Rem}[Thm]{Remark}
\newtheorem{Conj}[Thm]{Conjecture}
\numberwithin{equation}{section}
\address{R.B.: Department of Mathematics, Massachusetts Institute of Technology,
Cambridge MA  USA}
\email{bezrukav@math.mit.edu}
\address{I.L.: Department
of Mathematics, Yale University, New Haven CT USA}
\email{ivan.loseu@gmail.com}
\begin{document}
\begin{abstract}
We compute the number of finite dimensional irreducible modules for the algebras quantizing
Nakajima quiver varieties. We get a lower bound for all quivers and vectors of framing. We provide
an exact count in the case when the quiver is of finite type or is of affine type and the framing
is the coordinate vector at the extending vertex. The latter case precisely covers Etingof's conjecture
on the number of finite dimensional irreducible representations for Symplectic reflection algebras
associated to wreath-product groups. We use several different techniques, the two principal ones are
categorical Kac-Moody actions and wall-crossing functors.
\end{abstract}
\maketitle
\tableofcontents
\section{Introduction}
\subsection{Counting problem}\label{SSS_count}
Studying irreducible representations of algebraic objects, say of  associative algebras, is the most fundamental
problem in Representation theory. A basic question is  how many there are. For most infinite dimensional
algebras, the set of all irreducible representations is wild, in particular, the number is
infinite. So it makes sense to restrict the class of  representations. The most basic choice is to consider
only finite dimensional ones. This is a restriction we impose in the present paper.

A classical infinite dimensional algebra appearing in Representation theory is the universal enveloping algebra
$U(\g)$ of a finite dimensional Lie algebra $\g$ over $\C$. Let us consider the case when the Lie algebra
$\g$ is semisimple. In this case, the number of finite dimensional irreducible representations is still infinite: they are classified by dominant weights. More precisely, we consider the Cartan subalgebra $\h\subset\g$ and fix a system of simple roots. We say that $\lambda\in \h^*$ is {\it  dominant} if $\langle \alpha^\vee,\lambda\rangle\in \Z_{\geqslant 0}$
for any simple root $\alpha$. Then to $\lambda$ we can assign the irreducible module with highest weight $\lambda$.
Those form a complete and irredundant collection of irreducible finite dimensional representations of $\g$.

However,  we can modify the algebra $U(\g)$ to make the counting problem finite. Namely, recall that the center
of $U(\g)$ is identified with $S(\h)^W=\C[\h^*]^W$ via the Harish-Chandra isomorphism.  Here $W$ is the Weyl group acting on $\h^*$ by $w\bullet\lambda=w(\lambda+\rho)-\rho$, where $\rho$, as usual, is half the sum of all positive roots.
Then, for each $\lambda\in \h^*/W$, we can consider the corresponding central reduction, $\U_\lambda$, of $U(\g)$.
The classification result above can be restated as follows: the algebra $\U_\lambda$ has a single finite
dimensional irreducible representation if $\langle\lambda+\rho,\alpha^\vee\rangle$ is a nonzero integer for
every root $\alpha$. Otherwise, there are no finite dimensional representations.

Another classical feature of the algebras $\U_\lambda$  is that they have a very nice underlying geometry.
These algebras are filtered and the associated graded algebras $\gr \U_\lambda$ are all identified
with $\C[\mathcal{N}]$, where $\mathcal{N}$ stands for the nilpotent cone in
$\g$. Recall the Springer resolution of singularities $\rho:\widetilde{\mathcal{N}}\twoheadrightarrow \mathcal{N}$,
where $\widetilde{\mathcal{N}}$ is the cotangent bundle of the flag variety $\mathcal{B}$ of $\g$. The variety $\widetilde{\mathcal{N}}$ is  smooth and symplectic, while $\mathcal{N}$ is a singular Poisson variety. The morphism $\rho$ is therefore a symplectic resolution
of singularities.

There is a non-commutative analog of this resolution. Namely, for $\lambda\in \h^*$,
we can consider the sheaf $\mathcal{D}_\lambda$ of $\lambda$-twisted differential operators on
$\mathcal{B}$. Then $\Gamma(\mathcal{B}, \mathcal{D}_\lambda)=\mathcal{U}_\lambda$, while all higher
cohomology groups of $\mathcal{D}_\lambda$ vanish. So we have the global section
functor $\Gamma_\lambda: \mathcal{D}_\lambda\operatorname{-mod}\rightarrow \mathcal{U}_\lambda\operatorname{-mod}$
as well as its derived version $R\Gamma_\lambda: D^b(\mathcal{D}_\lambda\operatorname{-mod})
\rightarrow D^b(\mathcal{U}_\lambda\operatorname{-mod})$. The former is an equivalence if $\langle\lambda+\rho, \alpha^\vee\rangle\not\in \Z_{\leqslant 0}$ for all positive roots $\alpha$, this is the celebrated Beilinson-Bernstein
theorem, \cite{BB}.  Its derived version, \cite{BB_derived}, states that $R\Gamma_\lambda$ is an equivalence if
$\langle \lambda+\rho, \alpha^\vee\rangle\neq 0$ for all $\alpha$.

Using the results of the previous paragraph one can give a geometric interpretation of the classification
of finite dimensional irreducible representations. Namely, under the abelian Beilinson-Bernstein equivalence,
 the finite dimensional modules correspond to the $\mathcal{D}_\lambda$-modules whose singular support is contained in $\mathcal{B}\subset \widetilde{\mathcal{N}}$, i.e., to the $\mathcal{O}$-coherent $\mathcal{D}_\lambda$-modules.
It is easy to see that such a module exists if and only if $\lambda$ is integral, in which case
it is the line bundle on $\mathcal{B}$ corresponding to $\lambda$.

\subsection{Etingof's conjecture}
Another interesting  class of associative algebras is Symplectic reflection algebras
introduced by Etingof and Ginzburg in \cite{EG}. Those are filtered deformations of the skew-group algebras $S(V)\#\Gamma$,
where $V$ is a symplectic vector space and $\Gamma$ is a finite subgroup of $\operatorname{Sp}(V)$. The symplectic
reflection algebras $\mathcal{H}_c$ for the pair $(V,\Gamma)$ form a family depending on a collection $c$ of complex numbers.

One especially interesting class of groups $\Gamma$ comes from finite subgroups of $\SL_2(\C)$. Namely, pick
such a subgroup  $\Gamma_1\subset \SL_2(\C)$ and form the semidirect product
$\Gamma(=\Gamma_n):=\mathfrak{S}_n\ltimes \Gamma_1^n$, where $\mathfrak{S}_n$ stands for the symmetric group
on $n$ letters.
The group $\Gamma_n$ naturally acts on $\C^{2n}=(\C^2)^{\oplus n}$ by symplectomorphisms. Here elements of $\mathfrak{S}_n$
permute the $n$ summands  $\C^{2}$, the $n$ copies of $\Gamma_1$ act each on its own summand, and the symplectic form
on $(\C^2)^{\oplus n}$ is obtained as the direct sum of the $n$ copies of a $\Gamma_1$-invariant symplectic form on $\C^2$.
For $n>1$, the algebra $\mathcal{H}_c$ depends on $r$ parameters, where $r$ is the number of conjugacy classes in
$\Gamma_1$ (for $n=1$, the number of parameters
is $r-1$). So one can ask, how many finite dimensional irreducibles does the algebra $\mathcal{H}_c$ have? The answer,
of course, should depend on the parameter $c$.

In \cite[Section 6]{Etingof_affine}, Etingof proposed a conjectural answer to this and more general questions. The conjecture
takes the following form. Recall that the finite subgroups of $\SL_2(\C)$ are in one-to-one (McKay) correspondence with
the affine Dynkin diagrams. Take the affine Dynkin diagram, say $Q$, corresponding to $\Gamma_1$ and form the  Kac-Moody
algebra $\g(Q)$ from this diagram. Then Etingof defines a certain subalgebra $\a\subset \g(Q)\times \mathfrak{heis}$
depending on $c$, where $\mathfrak{heis}$
stands for the Heisenberg Lie algebra. Next, he considers the module $\mathbf{V}\otimes \mathcal{F}$, where $\mathbf{V}$ is the basic representation of $\g(Q)$ (whose highest weight is the fundamental weight corresponding to the extending vertex of $Q$)
and $\mathcal{F}$ is the Fock space representation of $\mathfrak{heis}$. Then Etingof takes an appropriate weight subspace  in that representation and considers its intersection with the sum of certain $\a$-isotypic components.
The conjecture is that the number of finite dimensional irreducibles is the dimension of the resulting
intersection.

Etingof's conjecture (in fact, its more general version dealing with the number of irreducibles
with given support in a category $\mathcal{O}$) was proved in the case when $\Gamma_1$ is cyclic
by Shan and Vasserot, \cite[Section 6]{shanvasserot} (under some technical restrictions on $c$ that were removed in \cite[Appendix]{VV_proof}). The techniques used in \cite{shanvasserot} are based on the representation
theory of rational Cherednik algebras  and do not  generalize to the case of non-cyclic $\Gamma_1$.

The main goal of this paper is to prove Etingof's conjecture on counting finite dimensional irreducibles
for all groups $\Gamma_1$. But,  first, we put this problem into a more general context: counting finite dimensional
irreducible representations over quantizations of symplectic resolutions.

\subsection{Quantizations of symplectic resolutions}
Inside $\mathcal{H}_c$ we can consider the spherical subalgebra, $e\mathcal{H}_ce$, where $e$ is the averaging idempotent.
This algebra is a filtered deformation of $S(V)^\Gamma$. By \cite[Theorem 5.5]{Etingof_affine}, $e\mathcal{H}_c e$ is
Morita equivalent to $\mathcal{H}_c$ if and only if $e\mathcal{H}_c e$ has finite homological dimension (the parameter $c$ is called {\it spherical} in this
case). Under this assumption,  the numbers of finite dimensional irreducibles for $\mathcal{H}_c$ and $e\mathcal{H}_c e$ coincide.

When $\Gamma=\Gamma_n$, the variety $V/\Gamma_n$ can be realized as an affine Nakajima quiver variety.  It admits a symplectic resolution of singularities that is a smooth Nakajima quiver variety.
The algebra $e\mathcal{H}_c e$ can be realized as a quantum Hamiltonian reduction, see \cite{EGGO,quant} and references therein (we briefly recall this below in Section \ref{SSS_SRA}).
Also we can quantize the symplectic resolution getting a sheaf of algebras on that symplectic variety. So we again have a nice geometry as in the case of universal enveloping algebras.

There are other algebras that quantize (i.e., are filtered deformations of) affine Poisson varieties
admitting symplectic resolutions and it is natural to expect that the counting problems for these
algebras have some nice answers that have to do with the geometry of the resolution. There are three
known large classes of resolutions  giving rise to interesting algebras. First, there are more general Nakajima
quiver varieties, the corresponding algebras are obtained as quantum Hamiltonian reductions of algebras
of differential operators. Second, there are Slodowy varieties that generalize cotangent bundles
to (partial) flag varieties. The corresponding algebras are finite W-algebras generalizing the universal
enveloping algebras. The counting problem for W-algebras was studied by the second author and Ostrik
in \cite{LO} (in the case of integral central characters) and by the authors of this paper in
\cite{BLo} (in general), below we will briefly mention
how the answer looks like in that case. Third, there are hypertoric varieties that are similar to but much easier
than Nakajima quiver varieties, this case is treated in \cite{BLPW_ht}.

In this paper we concentrate
on the case of Nakajima quiver varieties. In Section \ref{SS_intro_Nak} we recall necessary definitions.

\subsection{Nakajima quiver varieties and their quantizations}\label{SS_intro_Nak}
In this section we briefly recall Nakajima quiver varieties and their quantizations.
We will elaborate more on their properties in Section \ref{S_prelim}.

Let $Q$ be a quiver (=oriented graph, we allow loops and multiple edges). We can formally represent $Q$ as a quadruple
$(Q_0,Q_1,t,h)$, where $Q_0$ is a finite set of vertices, $Q_1$ is a finite set of arrows,
$t,h:Q_1\rightarrow Q_0$ are maps that to an arrow $a$ assign its tail and head.

Pick vectors $v,w\in \ZZ_{\geqslant 0}^{Q_0}$ and vector spaces $V_i,W_i$ with
$\dim V_i=v_i, \dim W_i=w_i$. Consider the (co)framed representation space
$$R=R(Q,v,w):=\bigoplus_{a\in Q_1}\Hom(V_{t(a)},V_{h(a)})\oplus \bigoplus_{i\in Q_0} \Hom(V_i,W_i).$$
We will also consider the cotangent bundle $T^* R=R\oplus R^*$ that can be identified with
$$\bigoplus_{a\in Q_1}\left(\Hom(V_{t(a)},V_{h(a)})\oplus \Hom(V_{h(a)}, V_{t(a)})\right)\oplus \bigoplus_{i\in Q_0} \left(\Hom(V_i,W_i)\oplus \Hom(W_i,V_i)\right).$$
The space $T^*R$ carries  a natural symplectic form, denote it by $\omega$.
On $R$ we have a natural action of the group $G:=\prod_{i\in Q_0} \GL(v_i)$. This action extends to
an action on $T^*R$ by linear symplectomorphisms.
As any action by linear symplectomorphisms, the $G$-action
on $T^*R$ admits a moment map, i.e., a $G$-equivariant morphism $\mu:T^*R\rightarrow \g^*$ with the
property that $\{\mu^*(x),\bullet\}=x_{T^*R}$ for any $x\in \g$. Here $\mu^*:\g\rightarrow \K[T^*R]$
denotes the dual map to $\mu$, $\{\bullet,\bullet\}$ is the Poisson bracket on $\K[T^*R]$ induced by $\omega$,
and $x_{T^*R}$ is the vector field on $T^*R$ induced by $x$ via the $G$-action.
Also we consider the dilation action of the one-dimensional torus $\K^\times$ on $T^*R$ given by $t.r=t^{-1}r$.
We specify the moment map uniquely by requiring that it is quadratic:  $\mu(t.r)=t^{-2}\mu(r)$. In this case $\mu^*(x)=x_R$, where we view $x_R$, an element of $\operatorname{Vect}_R$, as a function on $T^*R$.

In what follows,  $Q$ and $w$ are often fixed, but $v$ will vary.

Now let us proceed to  the definition of Nakajima quiver varieties.
Pick a character $\theta$ of $G$ (below we will often call $\theta$ a {\it stability condition})
and also an element $\lambda\in (\g/[\g,\g])^*$.
To $\theta$ we associate an open subset $(T^*R)^{\theta-ss}$ of $\theta$-semistable
points in $T^*R$ (that may be empty). Recall that a point $r\in T^*R$ is called
{\it $\theta$-semistable} if there is a $(G,n\theta)$-semiinvariant (with $n>0$) polynomial $f\in \C[T^*R]$
such that $f(r)\neq 0$.

We can form the GIT quotient $\M_\lambda^{\theta}(v):=(\mu^{-1}(\lambda)\cap (T^*R)^{\theta-ss})\quo G$ (we omit the subscript when $\lambda=0$). This variety is smooth provided $(\lambda,\theta)$ is {\it generic} (we will explain the precise meaning of this condition in  \ref{SSS_gen_param}).  The variety $\M_\lambda^0(v)$ is affine and
there is a projective morphism $\rho:\M^{\theta}_\lambda(v)\rightarrow \M_\lambda^0(v)$. There is a sufficient
condition for this morphism to be a resolution of singularities that will  be recalled in
\ref{SSS_sing_resol}.
We remark that all varieties $\M^\theta_\lambda(v)$ carry natural Poisson structures because they are defined as Hamiltonian reductions.
For a generic pair  $(\lambda,\theta)$,
the variety $\M^{\theta}_\lambda(v)$ is symplectic. Also we remark that we have an action of
$\C^\times$ on $\M^\theta(v)$ that comes from the dilation action on $T^*R$ and so rescales the
symplectic form.

Now let us briefly recall Nakajima's construction of a geometric $\g(Q)$-action on the middle homology groups
of the varieties $\M^\theta(v)$, we assume $Q$ has no loops and $\theta$ is generic. Consider the space $\bigoplus_v H_{mid}(\M^\theta(v))$,
where the subscript ``mid'' means the middle dimension, i.e., $\dim_\C \M^\theta(v)$.
We remark that these spaces are naturally identified for different $\theta$,
see \cite[Section 9]{Nakajima}, this result is recalled  in \ref{SS_homol_ident}.

Nakajima, \cite{Nakajima}, defined an action of $\g(Q)$ on $\bigoplus_v H_{mid}(\M^\theta(v))$ 
turning that space into the irreducible
integrable $\g(Q)$-module $L_\omega$ with highest weight \begin{equation}\label{eq:omega}\omega:=\sum_{i\in Q_0} w_i\omega^i,\end{equation} where we write $\omega^i$
for the fundamental weight corresponding to the vertex $i$. The individual space $H_{mid}(\M^\theta(v))$
gets identified with the weight space $L_\omega[\nu]$ of weight $\nu$, where \begin{equation}\label{eq:nu}\nu:=\omega-\sum_{i\in Q_0}v_i\alpha^i\end{equation}
(we write $\alpha^i$ for the simple root corresponding to $i$). We note that the identification is determined by the choice of a  highest weight vector in $L_\omega[\omega]$  and so is defined up to a scalar multiple.

Now we proceed to the quantum part of this story. Let us start by constructing quantizations of $\M^\theta(v)$
that will be certain sheaves of filtered algebras on $\M^\theta(v)$. Namely, consider the algebra $D(R)$ of differential operators on $R$. We can localize this algebra to a {\it microlocal} (the sections are only defined on $\C^\times$-stable open subsets) sheaf on $T^*R$ denoted by $D_R$. We have a quantum comoment map $\Phi: \g\rightarrow D(R)$
quantizing the classical comoment map $\g\rightarrow \C[T^*R]$, still $\Phi(x)=x_R$.

Now fix $\lambda\in \C^{Q_0}$ and a generic $\theta$. We get the quantum Hamiltonian reduction sheaf
$$\A^\theta_\lambda(v):=\pi_*[D_R/D_R\{\Phi(x)-\langle\lambda,x\rangle| x\in \g\}|_{(T^*R)^{\theta-ss}}]^G$$
on $\M^\theta(v)$, here $\pi$ is the quotient morphism
$\mu^{-1}(0)^{\theta-ss}\twoheadrightarrow \M^\theta(v)$.
This is a sheaf of filtered algebras with $\gr \A^\theta_\lambda(v)=\mathcal{O}_{\M^\theta(v)}$.
In fact, because of this, it has no higher cohomology, and $\Gamma(\A^\theta_\lambda(v))$ satisfies $\gr \Gamma(\A^\theta_\lambda(v))=\C[\M^\theta(v)]$. Further, one can show that $\Gamma(\A^\theta_\lambda(v))$
is independent of $\theta$, we will write $\A_\lambda(v)$
for that algebra. When $\mu$ is flat or $\lambda$
is Zariski generic, then $\A_\lambda(v)$ coincides with $\A_\lambda^0(v):=
[D(R)/D(R)\{\Phi(x)-\langle\lambda,x\rangle| x\in \g\}]^G$, we will elaborate on this
in Section \ref{SS_quant_prop}.

\subsection{Main conjecture}\label{SS_main_conj}
Our ultimate goal is to compute the number of finite dimensional irreducible representations of $\A_\lambda(v)$
or, equivalently, to compute $K_0(\A_\lambda(v)\operatorname{-mod}_{fin})$ (we consider all $K_0$'s over $\C$).
For this, we need to relate the categories of modules for $\A_\lambda(v)$ and for $\A_\lambda^\theta(v)$ with generic
$\theta$. Consider the category $\A_\lambda^\theta(v)\operatorname{-mod}$ of all coherent $\A_\lambda^\theta(v)$-modules
and also its derived analog $D^b(\A_\lambda^\theta(v)\operatorname{-mod})$
(see Section \ref{SS_Coh} for the definitions). Then we have the derived global sections
functor $R\Gamma^\theta_\lambda:D^b(\A_\lambda^\theta(v)\operatorname{-mod})\rightarrow D^b(\A_\lambda(v)\operatorname{-mod})$.
As McGerty and Nevins checked in \cite[Theorem 1.1]{MN} (under some technical conditions,
\cite[Assumptions 3.2]{MN}), this functor is an equivalence if and only if the algebra
$\A_\lambda(v)$ has finite homological dimension. The inverse of $R\Gamma_\lambda^\theta$ is the derived localization
functor $L\Loc_\lambda^\theta:=\A_\lambda^\theta(v)\otimes^L_{\A_\lambda(v)}\bullet$.
In  the majority of  interesting cases, the precise locus of $\lambda$, where the homological dimension is finite (such $\lambda$
are called {\it regular}), is not known.
Conjecturally, the regular locus should be the complement of a finite union of hyperplanes, see Section \ref{SS_loc_conj}. In this paper we deal
with the counting problem only in the case when $\lambda$ is regular, and we make a conjecture on the answer in general,
Conjecture \ref{Conj:inf_homol_dim}.

For an object in $\A_\lambda^\theta(v)\operatorname{-mod}$, one can define its (singular) support,
a $\C^\times$-stable coisotropic subvariety of $\M^\theta_0(v)$, and the characteristic
cycle, $\CC_v(M)$, see Section \ref{SSS_Supp_CC}. The equivalence $R\Gamma_\lambda^\theta$
identifies the following two categories:
\begin{itemize}
\item
 the full subcategory $D^b_{fin}(\A_\lambda(v)\operatorname{-mod})\subset
D^b(\A_\lambda(v)\operatorname{-mod})$ of all complexes with finite dimensional homology
\item
and the full subcategory  $D^b_{\rho^{-1}(0)}(\A_\lambda^\theta(v)\operatorname{-mod})\subset
D^b(\A_\lambda^\theta(v)\operatorname{-mod})$ of all complexes whose homology is supported on (i.e., has support
contained in)
$\rho^{-1}(0)$.\end{itemize}
 Since $\rho^{-1}(0)$ is an isotropic subvariety in $\M^\theta(v)$, the characteristic cycle of a module
supported on $\rho^{-1}(0)$ is a combination of the lagrangian irreducible components of $\rho^{-1}(0)$
(this is a consequence of the Gabber involutivity theorem recalled in
Section \ref{SSS_Supp_CC}), let us denote the set
of lagrangian components by $\mathsf{comp}$. So we
get a linear map
\begin{equation}
\CC_{\lambda,v}: K_0(\A_\lambda(v)\operatorname{-mod}_{fin})\rightarrow \C^{\mathsf{comp}}.
\end{equation}
The space $\C^{\mathsf{comp}}$ is identified with $H_{mid}(\M^\theta(v))=L_\omega[\nu]$ and so is canonically
independent of $\theta$. We will see below, Section \ref{SS_WC_K0}, that $\CC_{\lambda,v}$ is actually independent of $\theta$.
By an unpublished result of Baranovsky and Ginzburg, \cite{BarGin}, the map $\CC_{\lambda,v}$ is injective (since this result is not published yet,  we actually
give a proof of this result for quantized quiver varieties in the cases we consider in this paper).
So, to solve
our counting problem, we just need to describe the image of $\CC_\lambda=\bigoplus_v \CC_{\lambda,v}$.

Our conjectural description is inspired by Etingof's conjecture. Namely, consider the subalgebra $\a\subset\g(Q)$ constructed from $\lambda$ as follows: the algebra $\a$ is generated by the Cartan $\h\subset \g(Q)$ and all root subspaces $\g(Q)_\beta$ for \underline{real} roots $\beta=\sum_{i\in Q_0}b_i\alpha^i$ with $\sum_{i\in Q_0}b_i\lambda_i\in \Z$. For instance, if $\lambda$ is generic, then $\a=\h$, while if all $\lambda_i\in \Z$, then $\a=\g(Q)$ provided $Q$ contains no loops. Let $L_\omega^\a$ be the $\a$-submodule of $L_\omega$ generated by the weight spaces $L_\omega[\sigma\omega]$ for $\sigma\in W(Q)$, where $W(Q)$ stands for the Weyl group of $\g(Q)$.

\begin{Conj}\label{Conj:main}
Assume that $Q$ has no loops. Then we have $\operatorname{Im}\CC_\lambda= L_\omega^\a$.
\end{Conj}

Let us point out that the case when $Q$ has a loop is non-interesting for
our counting problem as stated: the answer is $0$ (provided the dimension in the corresponding vertex is positive,
if it is zero, then the loop does not matter anyway).
In this case, the algebra $\A_\lambda(v)$ decomposes into the tensor product
of $D(\C^k)$, the algebra of differential operators on $\C^k$, where $k$ is the number of loops,
and of another algebra, say $\underline{\A}$. The former has no finite
dimensional representations. Of course, it is an interesting question to
count the finite dimensional irreducible representations of $\underline{\A}$.
At present, the answer to this question is known for the case of Jordan quiver,
see \cite{Gies}.

We remark that the dimension vectors $v$ corresponding to $\nu=\sigma \omega$ are precisely those with $\M^\theta(v)=\{\operatorname{pt}\}$ and hence $\A_\lambda(v)=\C$. In particular, if $\lambda$ is Weil generic
(which, by definition, means a parameter lying outside countably many proper subvarieties), then our conjecture predicts that a non-trivial algebra $\A_\lambda(v)$ has no finite dimensional representations, as expected. The other extreme is when $\lambda$ is integral. Here our conjecture
predicts that $\operatorname{Im}\CC=L_\omega$. This follows from the work of Webster, \cite[Section 3]{Webster},
see Section \ref{SS_W_fun} for details.

Here is the main result of the present paper.

\begin{Thm}\label{Thm:verymain}
Conjecture \ref{Conj:main} is true
\begin{itemize}
\item when $Q$ is of finite type,
\item or when $Q$ is an affine quiver,
$v=n\delta, w=\epsilon_0$.
\end{itemize}
\end{Thm}

Here and below we write $\delta$ for the indecomposable imaginary root of $Q$ and $\epsilon_0$
for the coordinate vector at the extending vertex.

Let us notice that (ii) precisely covers the algebras
of interest for Etingof's conjecture. In fact, that conjecture
follows from Theorem \ref{Thm:verymain} and results of \cite{GL},
we will elaborate on that in Section \ref{SS_CC_inj}.

In a subsequent paper \cite{perv} the second author proves Conjecture \ref{Conj:main} for affine type
quivers with arbitrary framing.

We also would like to point out that there is a very similar result for finite W-algebras $U(\g,e)$,
see \cite[Theorem 1.1]{LO}, \cite[Theorem 1.1]{BLo} (here $\g$ is a semisimple Lie algebra and $e\in \g$ is a nilpotent element).
The role of $\M^\theta(v)$ is played by the Slodowy variety $\widetilde{S}$ that is obtained as follows. We take the transversal {\it Slodowy slice} $S$ to the $G$-orbit of $e$ in $\g$, and for $\widetilde{S}$ take the preimage of $S$ in $\widetilde{\mathcal{N}}$.
The zero fiber $\rho^{-1}(0)$ becomes the Springer fiber $\mathcal{B}_e$. Therefore $H_{mid}(\mathcal{B}_e)$
is the Springer representation of the Weyl group $W(\g)$ of $\g$. So instead of $\g(Q)$ we need to consider
$W(\g)$, and instead of $\a$ we take the integral Weyl group $W'$ corresponding to a given central character.
Then $K_0$ of the finite dimensional representations  coincides with the sum of certain isotypic components for $W'$, see
\cite[Introduction]{BLo} for details. One could expect that for a general symplectic resolution $X$ one should
be able to state a conjecture on $K_0$ of finite dimensional representations using the monodromy representation
associated to the quantum connection but specifics of this are not clear at the moment, at least to the authors.

\subsection{Content of the paper}
The paper roughly splits into two parts. The first part, consisting of Sections \ref{S_prelim}, \ref{S_HC},
\ref{S_loc}, \ref{S_WWC} is (mostly) preparatory:
there we recall known results (or their generalizations to settings we need)
as well as prove some technical generalizations that we will need.
In Section \ref{S_prelim} we recall preliminaries on Nakajima quiver
varieties, their quantizations, coherent  modules over sheaf quantizations,
supports and characteristic cycles and Hamiltonian reduction functors. In Section \ref{S_HC}
we recall Harish-Chandra (HC) bimodules, construct  restriction functors for HC bimodules
in our setting and describe some applications of these functors. In Section \ref{S_loc} we discuss
abelian and derived localization theorems. In Section \ref{S_WWC} we will introduce two main
players in the proof of Theorem \ref{Thm:verymain} -- the Webster functors and
the wall-crossing functors.

Section \ref{S_outline} outlines the main ideas and steps of the proof of Theorem \ref{Thm:verymain}. We will describe the content
of the subsequent sections in the end of Section \ref{S_outline}.

\subsection{Notation}\label{SS_notation}
The following table contains various notation used in the paper (we first list the notation starting with Roman
letters in, roughly, the alphabetical order and then list the notation starting with a Greek letter).

\setlongtables

\begin{longtable}{p{3.4cm} p{12cm}}
$\widehat{\otimes}$&the completed tensor product of complete topological vector spaces/ modules.\\
$\A^{opp}$& the opposite algebra of $\A$.\\
$(a_1,\ldots,a_k)$& the two-sided ideal in an associative algebra generated by  elements $a_1,\ldots,a_k$.\\
 $A^{\wedge_\chi}$&
the completion of a commutative (or ``almost commutative'') algebra $A$ with respect to the maximal ideal
of a point $\chi\in \Spec(A)$.\\
$\A_\lambda^\theta(v)$&$:=[\mathcal{Q}_\lambda|_{(T^*R)^{\theta-ss}}]^G$\\
$\A_\lambda(v)$&$:=\Gamma(\A_\lambda^\theta(v))$.\\
$\A^\theta_{\lambda,\chi}(v)$&$:=[\mathcal{Q}_\lambda|_{(T^*R)^{\theta-ss}}]^{G,\chi}$,
where $\chi$ is a character of $G$, and the superscript $(G,\chi)$ means taking
$\chi$-semiinvariants. \\
$\mathfrak{a}^\lambda$& the subalgebra in $\g(Q)$ generated by Cartan $\h$ and real root subspaces $\g(Q)_\beta$
with $\beta\cdot \lambda\in \Z$.\\
$\A^{(\theta)}_{\lambda,\chi}(v)$&$:=\Gamma(\A^\theta_{\lambda,\chi}(v))$.\\
$\A\operatorname{-mod}$& the category of finitely generated modules over an associative algebra $\A$.\\
$\mathsf{AC}(Y)$& the asymptotic cone of a subvariety $Y\subset \C^n$.\\
$\mathfrak{AL}(v)$& the set of $\lambda\in \paramq$ such that $\Gamma_\lambda^\theta$
is an abelian equivalence. \\
$\operatorname{Ann}_\A(M)$& the annihilator of an $\A$-module $M$ in an algebra
$\A$.\\
$\mathcal{C}$& the full subcategory of $\bigoplus_{v}\A_\lambda^\theta(v)\operatorname{-mod}_{\rho^{-1}(0)}$
defined in Section \ref{SS_upper_bound_outline}.\\
$\CC(M)$& the characteristic cycle of a module/sheaf of modules $M$.\\
$D^b_{fin}(\A_\lambda(v)\operatorname{-mod})$&$:=\{M\in D^b(\A_\lambda(v)\operatorname{-mod})|\dim H_*(M)<\infty\}$.\\
$D^b_{\rho^{-1}(0)}(\A^\theta_\lambda(v)\operatorname{-mod})$&$:=\{M\in D^b(\A^\theta_\lambda(v)\operatorname{-mod})|\operatorname{Supp} H_*(M)\subset \rho^{-1}(0)\}$.\\
$D(R)\operatorname{-mod}^{G,\lambda}$&the category of $(G,\lambda)$-equivariant finitely generated $D(R)$-modules.\\
$D(X)$& the algebra of differential operators on a smooth affine variety $X$.\\
$D_X$& the differential operators on a smooth variety $X$ viewed as microlocal sheaf on $T^*X$.\\
$\operatorname{Frac}(A)$&the fraction field of a commutative domain $A$.\\
$G^\circ$& the connected component of unit in an algebraic group $G$.\\
$(G,G)$& the derived subgroup of a group $G$.\\
$G_x$& the stabilizer of $x$ in $G$.\\
$G*_HV$& the homogeneous bundle on $G/H$ with fiber $V$.\\
$\g(Q)$& the Kac-Moody algebra associated to a quiver $Q$. \\
$\gr \A$& the associated graded vector space of a filtered
vector space $\A$.\\
$\operatorname{Irr}(C)$& the set of simple objects in an abelian category $\mathcal{C}$.\\
$K_0(C)$& the (complexified) Grothendieck group of an abelian (or triangulated)
category $\mathcal{C}$.\\
$L_\omega$&the irreducible integrable representation of $\g(Q)$ with highest weight $\omega$.\\
$L_\omega[\nu]$& the $\nu$-weight space in $L_\omega$.\\
$\Loc_\lambda^\theta$& localization functor $\A_\lambda(v)\operatorname{-mod}\rightarrow \A_\lambda^\theta(v)\operatorname{-mod}$.\\
$L\pi^0_\lambda(v)^!$& the derived left adjoint to $\pi^0_\lambda(v)$.\\
$\M^\theta_\lambda(v)$&$:=\mu^{-1}(\lambda)^{\theta-ss}\quo G$.\\
$\M_\lambda(v)$&$:=\Spec(\C[\M^\theta_\lambda(v)])$ for generic $\theta$.\\
$\param$&$:=\C^{Q_0}$, the parameter space for classical reduction.\\
$\paramq$&$:=\C^{Q_0}$, the parameter space for quantum reduction.\\
$\paramq^{iso}$&the locus of $\lambda\in \paramq$ such that $\A^0_\lambda(v)\xrightarrow{\sim} \A_\lambda(v)$.\\
$\paramq^{ISO}$&a locus of $\lambda\in \paramq^{iso}$  with $\operatorname{Tor}_i^{U(\g)}(D(R),\C_\lambda)=0$ for $i>0$.\\
$\mathcal{Q}_\lambda$&$:=D(R)/D(R)\{\Phi(x)-\langle \lambda,x\rangle, x\in \g\}$.\\
$R$&$(=R(Q,v,w)):=\bigoplus_{a\in Q_1}\operatorname{Hom}(V_{t(a)},V_{h(a)})\oplus \bigoplus_{k\in Q_0}\operatorname{Hom}(V_k,W_k)$, the coframed representation space of a quiver $Q$ with dimension $v$
and framing $w$.\\
$R_\hbar(\A)$&$:=\bigoplus_{i\in
\mathbb{Z}}\hbar^i \F_i\A$ :the Rees $\K[\hbar]$-module of a filtered
vector space $\A$.\\
$\mathfrak{S}_n$& the symmetric group on $n$ letters. \\
$S(V)$& the symmetric algebra of a vector space $V$.\\
$\Supp(M)$& the support of the module/sheaf of modules $M$.\\
$\Supp_{\paramq}^r(\B)$&$:=\{\lambda\in \paramq| \B\otimes_{\C[\paramq]}\C_\lambda\neq \{0\}$.\\
$W(Q)$& the Weyl group of $\g(Q)$.\\
$\WC_{\lambda\rightarrow \lambda'}$& a wall-crossing functor. \\
$\tilde{w}_i$&$:=w_i+\sum_{a, t(a)=i}v_{h(a)}$ (for a source $i\in Q_0$).\\
$X^{\theta-ss}$& the open locus of $\theta$-semistable points for an action of
a reductive group $G$ on an affine algebraic variety $X$, where $\theta$
is a character of $G$.\\
$X^{\theta-uns}$&$:=X\setminus X^{\theta-ss}$.\\
$x\cdot y$& $\sum_{i\in Q_0}x_iy_i$.\\
$(x,y)$&$=2\sum_{k\in Q_0}x_ky_k-\sum_{a\in Q_1}(x_{t(a)}y_{h(a)}+x_{h(a)}y_{t(a)})$,
the symmetrized Tits form. \\
$\Gamma_n$&$=\mathfrak{S}_n\ltimes \Gamma_1^n$ for a finite subgroup $\Gamma_1\subset \operatorname{SL}_2(\C)$.\\
$\Gamma_\lambda^\theta$& global section functor $\A_\lambda^\theta(v)\operatorname{-mod}
\rightarrow \A_\lambda(v)\operatorname{-mod}$.\\
$\mu$& the moment  map $T^*R\rightarrow \g$. \\
$\nu$& the weight of $\g(Q)$ determined from dimension vector $v$ and framing $w$.\\
$\pi^0_\lambda(v)$&the natural functor $D(R)\operatorname{-mod}^{G,\lambda}\rightarrow \A_\lambda^0(v)\operatorname{-mod}$
(or its derived analog). \\
$\pi^\theta_\lambda(v)$&the natural functor $D_R\operatorname{-mod}^{G,\lambda}\rightarrow \A_\lambda^\theta(v)\operatorname{-mod}$ (or its derived analog).\\
$\rho$& the natural projective morphism $\M^\theta_\lambda(v)\rightarrow \M^0_\lambda(v)$
or $\M^\theta_\lambda(v)\rightarrow \M_\lambda(v)$.\\
$\sigma\bullet v$& dimension vector corresponding to $\sigma \nu, \sigma\in W(Q)$.\\
$\sigma\bullet^v\lambda$& the quantization parameter determined by $\sigma\in W(Q), v\in \Z_{\geqslant 0}^{Q_0},
\lambda\in \paramq$.\\
$\omega$& the dominant weight of $\g(Q)$ determined from framing $w$.
\end{longtable}

\subsection{Acknowledgements} We are grateful to Pavel Etingof, Joel Kamnitzer, Hiraku Nakajima,
Andrei Okounkov, and Ben Webster for numerous helpful discussions.
We also would like to thank Dmitry Korb, Kevin McGerty and Yaochen Wu  for remarks
on the previous versions of this text. We are also grateful to the referees
for their comments that helped us to improve the exposition and fix several
gaps in the paper. R.B. was
supported by the NSF under Grant DMS-1102434. I.L. was supported
by the NSF under Grant  DMS-1161584.

\section{Preliminaries on quiver varieties and their quantizations}\label{S_prelim}
\subsection{Properties of quiver varieties}\label{SS_quiv_class}
\subsubsection{Generic parameters}\label{SSS_gen_param}
First of all, let us recall a description of generic values of $(\lambda,\theta)$, which, by definition,
means that the $G$-action on $\mu^{-1}(\lambda)^{\theta-ss}$ is free, due to Nakajima, \cite[Theorem 2.8]{Nakajima}.
Namely, $(\lambda,\theta)$ is generic provided there is no $v'\in \Z_{\geqslant 0}^{Q_0}$ such that
\begin{itemize}\item
$v'\leqslant v$ (component-wise), \item $\sum_{i\in Q_0}v'_i\alpha^i$ is a root for $\g(Q)$, \item and
$v'\cdot \theta= v'\cdot \lambda=0$ (where we write $v'\cdot \lambda$ for $\sum_{i\in Q_0}v'_i\lambda_i $).
\end{itemize}
We say that $\lambda$ (resp., $\theta$) is generic if $(\lambda,0)$ (resp., $(0,\theta)$)
is generic. We write $\param$ for $\C^{Q_0}$, the space of parameters $\lambda$.
The set of non-generic $\lambda$'s will be denoted
by $\param^{sing}$ (or $\param^{sing}(v)$ when we need to indicate the dependence on $v$).
It is easy to see that $\param^{sing}$ is a union of hyperplanes all of which have the form
$\{\lambda| \lambda\cdot v'=0\}$ for $v'$ as above, however, a priori, not all $v'$
as above actually appear.

We note that by results of Crawley-Boevey, \cite[Section 1, Remarks]{CB}, $\M_\lambda^\theta(v)$ is connected when $(\lambda,\theta)$ is generic.

Since $\M_\lambda^0(v)$ is defined as a categorical quotient, and $\M_\lambda^\theta(v)$
is defined as a GIT quotient, we have a natural projective morphism $\rho: \M^\theta_\lambda(v)
\rightarrow \M^0_\lambda(v)$.
We  note that, for a generic $\lambda$, the open locus $\mu^{-1}(\lambda)^{\theta-ss}$
coincides with $\mu^{-1}(\lambda)$.
In particular, for a generic $\lambda$,  the morphism
$\rho$ is an isomorphism.

\subsubsection{Line bundles}
For a character $\chi$ of $G$,
we consider the line bundle $\mathcal{O}(\chi)$ on $\M^\theta(v)$ whose sections are given by
$$\Gamma(U,\mathcal{O}(\chi))=\C[\pi^{-1}(U)]^{G,\chi}:=\{f\in \C[\pi^{-1}(U)]| g.f=\chi(g)f, \forall g\in G\}.$$
Here $g.f$ stands for the function defined by $(g.f)(x):=f(g^{-1}x)$,
$U\subset \M^\theta(v)$ is an affine open subset, and $\pi$ stands for the quotient morphism
$\mu^{-1}(0)^{\theta-ss}\rightarrow \M^\theta(v)$. By the very definition, $\mathcal{O}(\theta)$
is an ample line bundle.

\subsubsection{LMN isomorphisms}\label{SSS_LMN}
Now let us discuss certain  isomorphisms of quiver varieties.
For $\sigma\in W(Q)$, we have an isomorphism
$\M^\theta_\lambda(v)\cong \M^{\sigma\theta}_{\sigma\lambda}(\sigma\bullet v)$, where we assume that
$(\lambda,\theta)$ is generic. Here we write $\sigma\bullet v$
for the dimension vector that produces the weight $\sigma\nu$ via (\ref{eq:nu}). For a simple reflection
$\sigma=s_k$, we have $(s_k\bullet v)_\ell=v_\ell$ for $\ell\neq k$ and $(s_k\bullet v)_k=w_k+\sum_{a, t(a)=k}v_{h(a)}+\sum_{a, h(a)=k}v_{t(a)}-v_k$.

The existence of such
isomorphisms was conjectured by Nakajima in \cite{Nakajima} and first proved by Maffei in \cite{Maffei_Weyl}
and, independently, by Nakajima, \cite{Nakajima_Weyl}, a closely related construction was found
by Lusztig, \cite{Lusztig_Weyl}. So we call those {\it LMN isomorphisms}.

Below we will need a (slightly rephrased) construction of LMN isomorphisms due to Maffei.
Let us construct an isomorphism corresponding to a simple reflection $s_i\in W(Q)$ following
\cite{Maffei_Weyl}.
We may assume that the vertex $i$ is a source and, since $(\lambda,\theta)$ is generic, that either $\lambda_i\neq 0$ or $\theta_i>0$
(if $\theta_i<0$, then we just construct an isomorphism for $s\theta$). Let \begin{equation}\label{eq:tildew}\tilde{W}_i:=W_i\oplus \bigoplus_{a, t(a)=i}V_{h(a)},
\tilde{w}_i:=\dim \tilde{W}_i\end{equation} so that $v_i+(s_i\bullet v)_i=\tilde{w}_i$. Set
\begin{equation}\label{eq:underl}
\underline{R}(=\underline{R}^i):=\bigoplus_{a, t(a)\neq i}\Hom(V_{t(a)},V_{h(a)})\oplus \bigoplus_{j\neq i} \Hom(V_j,W_j),
\underline{G}:=\prod_{j\neq i}\GL(V_j),
\end{equation}
so that $R=\underline{R}\oplus \Hom(V_i,\tilde{W}_i)$.
Consider the Hamiltonian reduction $$T^*R\red_{\lambda_i}^{\theta_i}\GL(v_i)= T^* \Hom(V_i,\tilde{W}_i)\red^{\theta_i}_{\lambda_i}\GL(v_i)\times T^*\underline{R}.$$
Let us remark that if $\lambda_i=0$, the reduction $T^* \Hom(V_i,\tilde{W}_i)\red^{\theta_i}_{\lambda_i}\GL(v_i)$
is just $T^*\operatorname{Gr}(v_i,\tilde{w}_i)$.

An easy special case of Maffei's construction is an isomorphism $$T^* \Hom(V_i,\tilde{W}_i)\red^{\theta_i}_{\lambda_i}\GL(v_i)\xrightarrow{\sim}
T^* \Hom(\tilde{W}_i, V_i')\red^{-\theta_i}_{-\lambda_i}\GL(v'_i),$$
where $v'_i=(s_i\bullet v)_i=\tilde{w}_i-v_i$ and $V'_i$ is a vector space of dimension $v_i'$. When $\lambda_i=0$,
we just have two realizations of $T^*\operatorname{Gr}(v_i,\tilde{w}_i)$ (where $\operatorname{Gr}(v_i,\tilde{w}_i)$
is thought as the variety of $v_i$-dimensional subspaces in $\C^{\tilde{w}_i}$ and as the variety of
$(\tilde{w}_i-v_i)$-dimensional quotients), while for $\lambda_i\neq 0$,
we get two equal twisted cotangent bundles on the Grassmanian. These isomorphisms are clearly symplectomorphisms,
$\C^\times$-equivariant when $\lambda_i=0$.

As a consequence, we get a  $\underline{G}$-equivariant symplectomorphism
\begin{equation}\label{eq:part_red_iso} T^*R\red_{\lambda_i}^{\theta_i}\GL(v_i)\xrightarrow{\sim}
T^*R'\red_{-\lambda_i}^{-\theta_i}\GL(v_i'),\end{equation} where $R':=\Hom(\tilde{W}_i,V_i')\oplus \underline{R}$. According to
\cite[Section 3.1]{Maffei_Weyl}, this isomorphism does not intertwine the moment maps for the $\underline{G}$-actions. Rather, if $\underline{\mu},\underline{\mu}'$ are the two moment
 maps, then $\underline{\mu}-\lambda=\underline{\mu}'-s_i\lambda$.
The isomorphism does not intertwine the stability conditions either, instead it maps $(T^*R)^{\theta-ss}\red_{\lambda_i}\GL(v_i)$
to $(T^*R')^{s_i\theta-ss}\red_{-\lambda_i}\GL(v_i')$. So, by reducing the $\underline{G}$-action,
we do get a symplectomorphism $\M^\theta_{\lambda}(v)\xrightarrow{\sim}
\M^{s_i\theta}_{s_i\lambda}(s_i\bullet v)$. This isomorphism is  $\C^\times$-equivariant, if $\lambda=0$.

We will need a compatibility of the LMN isomorphisms with certain $T$-actions. Namely, the torus $T:=(\C^\times)^{Q_1}\times (\C^\times)^{Q_0}$ naturally acts on $R$ (the copy of $\C^\times$ corresponding to an arrow $a$ acts by scalars on $\Hom(V_{t(a)},V_{h(a)})$,
the copy corresponding to $i\in Q_0$ acts on $\operatorname{Hom}(V_i,W_i)$). The lift of this  $T$-action to
$T^*R$ commutes with $G$ and preserves the moment map and so descends to $\M^\theta(v,w)$. An isomorphism $s_i$  is
$T$-equivariant, \cite[Proposition 4.13]{Webster_O}.

We will also need to understand the behaviour of line bundles under the LMN isomorphism.
Namely, by tracking the construction, we see  that $s_i$ maps the line bundle $\mathcal{O}(\chi)$ to $\mathcal{O}(s_i\chi)$.

\subsubsection{Properties of $\M^0(v)$}\label{SSS_M0_prop}
Now let us turn to the affine quiver varieties $\M^0(v)$. In \cite{CB} Crawley-Boevey found a combinatorial
criterion on $v$ for $\mu$ to be flat. Let us state this criterion. Recall the symmetrized Tits
form $(\cdot,\cdot)$ for $Q$: $(v^1,v^2):=2\sum_{k\in Q_0} v^1_k v^2_k-\sum_{a}(v^1_{t(a)}v^2_{h(a)}+v^1_{h(a)}v^2_{t(a)})$.
We set $p(v):=1-\frac{1}{2}(v,v)$ (so if $v$ is a root, then $p(v)\geqslant 0$). According to
\cite[Theorem 1.1]{CB}, the map $\mu$ is flat if and only if
\begin{equation}\label{eq:CB_cond}
p(v)+w\cdot v- (w\cdot v^0+\sum_{i=0}^k p(v^i))\geqslant 0
\end{equation}
for all decompositions $v=v^0+\ldots+v^k$ with all $v^i>0$, equivalently, for the  decompositions, where $v^1,\ldots,v^k$
are roots and $v^0\geqslant 0$. Note that Crawley-Boevey works with unframed quivers. As usual, one can
translate from our setting to his by adding a new vertex $\infty$, as explained in the end of
\cite[Section 1]{CB}. In (\ref{eq:CB_cond}) we have restated the Crawley-Boevey inequalities
in terms of the original quiver $Q$.

Also if all inequalities for proper decompositions in (\ref{eq:CB_cond}) are strict, then $\mu^{-1}(0)$
is irreducible and contains a free closed orbit, \cite[Theorem 1.2]{CB}.

We want to analyze condition (\ref{eq:CB_cond}) in the case when $Q$ is finite or affine and
$\nu$ is dominant.

\begin{Lem}\label{Lem:free_closed_orbit}
Suppose that $\nu$ is dominant.
\begin{enumerate}
\item If $Q$ is finite, then $\mu$ is flat, $\mu^{-1}(0)$ is irreducible and contains a free closed
orbit.
\item If $Q$ is affine, then $\mu$ is flat. Moreover, if $(\omega,\delta)>1$, then
$\mu^{-1}(0)$ is irreducible and contains a free closed orbit.
\end{enumerate}
\end{Lem}
\begin{proof}
When $Q$ is finite, then $p(v^i)=0$ for all $i>0$ and so
the left hand side of (\ref{eq:CB_cond})
becomes $\frac{1}{2}((v^0,v^0)-(v,v))+w\cdot(v-v^0)=(\nu^0-\nu, \frac{1}{2}(\nu+\nu^0))=
\frac{1}{2}(\nu^0-\nu,\nu^0-\nu)+(\nu, \nu^0-\nu)$. Here, in the second and the third expressions, $(\cdot,\cdot)$ is the usual form on $\h^*$. The first summand is positive if $v\neq v^0$, while the second is non-negative. We conclude that $\mu^{-1}(0)$ is irreducible and  contains a free closed orbit.

Now consider the case when $Q$ is affine. Here  $p(v^i)=1$ if  $v^i=a_i\delta$
and $p(v^i)=0$ else, for $i>0$. The left hand side of (\ref{eq:CB_cond})
is minimized when all $a_i=1$ and we will assume this. So the left hand side
becomes $$\frac{1}{2}(\nu^0-\nu,\nu^0-\nu)+(\nu, \nu^0-\nu)-s=\frac{1}{2}(\nu^0-\nu,\nu^0-\nu)+(\nu, \nu^0-s\delta-\nu)+s((\omega,\delta)-1),$$ where $v\geqslant v^0+s\delta$. The first summand is non-negative, it equals
$0$ if and only if $v-v^0$ is a multiple of $\delta$. The second summand is non-negative, it equals $0$ if and only
if $\nu=\nu^0+s\delta$. Finally, the third summand is
nonnegative, it is $0$ if and only if $(\omega,\delta)=1$. So we see that $\mu$ is flat. The
subvariety $\mu^{-1}(0)$ is irreducible and contains a free closed orbit provided $(\omega,\delta)>1$.
\end{proof}

\subsubsection{Families}\label{SSS_fam_quiv}
Set $\param:=\C^{Q_0}\cong (\g^*)^G$ and consider the varieties $\M^0_{\param}(v):=\mu^{-1}(\g^{*G})\quo G,$
$\M^\theta_{\param}(v):=\mu^{-1}(\g^{*G})^{\theta-ss}\quo G$.


For a vector subspace $\param_0\subset \param$, we consider the specializations $\M^0_{\param_0}(v):=\param_0\times_{\param} \M^0_{\param}(v), $ $ \M^\theta_{\param_0}(v)$.

\subsubsection{Structure of neighborhoods}\label{SSS_class_slice}
Pick a point $x\in \M^0_{\param}(v)$. Consider the completion $\C[\M^0_{\param}(v)]^{\wedge_x}$
of the algebra $\C[\M^0_\param(v)]$ at $x$. Set $\M^0_{\param}(v)^{\wedge_x}:=\operatorname{Spec}(\C[\M^0_{\param}(v)]^{\wedge_x})$
(we emphasize that we view $\M^0_\param(v)^{\wedge_x}$ as a scheme and not as a formal scheme).
Further, set
$$\M^\theta_{\param}(v)^{\wedge_x}:=\M^0_{\param}(v)^{\wedge_x}\times_{\M^0_\param(v)}\M^\theta_{\param}(v)$$
In this section, we give  a description of the schemes $\M^0_{\param}(v)^{\wedge_x},
\M^\theta_{\param}(v)^{\wedge_x}$ essentially following Nakajima, \cite[Section 6]{Nakajima}.

Let $r\in T^*R$ be a point with closed $G$-orbit mapping to $x$. Then $r$ is a semisimple representation
of the following quiver $\overline{Q}^w$. We first adjoin the vertex $\infty$ to $Q$ and connect
each vertex $i\in Q_0$ to $\infty$ with $w_i$ arrows. Then we add an opposite arrow to each existing arrow of $Q^w$.
The dimension of $r$ is $(v,1)$. Let us decompose $r$ into the sum $r=r_0\oplus r_1\otimes U_1\oplus\ldots \oplus r_k\otimes U_k$, where $r_0$ is an irreducible representation with dimension vector of the form $(v^0,1)$,  $r_1,\ldots,r_k$ are pairwise non-isomorphic irreducible representations with dimensions $(v^i,0), i=1,\ldots,k$,
and $U_i$ is the multiplicity space of $r_i$. In particular, the stabilizer $G_r$ of $r$ is $\prod_{i=1}^k\GL(U_i)$.

Let us define a new quiver $\hat{Q}$, a dimension vector $\hat{v}$ and a framing $\hat{w}$. For the set of vertices $\hat{Q}_0$ we take $\{1,\ldots,k\}$ and we set $\hat{v}=(\dim U_i)_{i=1}^k$. The number of arrows between $i,j\in \{1,\ldots,k\}$ is determined as follows.  The subspace $T_r(Gr)\subset T_r(T^*R)$ is contained in its
skew-orthogonal complement $T_r(Gr)^{\angle}$. So we get a symplectic $G_r$-module
$T_r(Gr)^{\angle}/T_r(Gr)$.
We want the $G_r$-module $T^*R_x$, where we write $R_x$ for $R(\hat{Q},\hat{v},\hat{w})$, to be isomorphic
to $T_r(Gr)^{\angle}/T_r(Gr)$.  So $T^*R_x\oplus T^*(\g/\g_r)=T^*R$.

For $i\neq j$, the multiplicity of the $G_r$-module $\operatorname{Hom}(U_i,U_j)$ in $T^*R$ equals $\sum_{a} (v^i_{t(a)}v^j_{h(a)}+ v^j_{t(a)}v^i_{h(a)})$, while the multiplicity in $T^*(\g/\g_r)$ equals $2\sum_{k\in Q_0}v^i_kv^j_k$.
So the multiplicity of $\operatorname{Hom}(U_i,U_j)$ in
the $G_r$-module $T^*R_x$ has to be equal to  $-(v^i,v^j)$ if $i\neq j$ and to
$2-(v^i,v^i)$ if $i=j$. Hence the number of arrows between $i$ and  $j$ in $\hat{Q}$ has to be
$-(v^i,v^j)$ if $i\neq j$ and $p(v^i)=1-\frac{1}{2}(v^i,v^i)$ if $i=j$. Similarly, for $\hat{w}_i$ we need
to take $w\cdot v^i-(v^0,v^i)$. Finally, we need to add some loops at $\infty$ but those are just spaces
with trivial action of $G_r$. We will  treat them separately: so the symplectic part of the slice
module at $r$ can be written as $T^*R_x\oplus R_0$, where $R_x=R(\hat{Q},\hat{v},\hat{w})$
and $R_0$ is a symplectic vector space with trivial action of $G_r$.
We choose an orientation on $\hat{Q}$ in such a way that the $G_r$-modules
$R_x\oplus \g/\g_r$ and $R$ are isomorphic up to a trivial summand. We remark, however, that this choice
may violate the condition that the vertex $\infty$ (corresponding to $r^0$) in $\hat{Q}$ is a sink.

Consider the homogeneous vector bundle $G*_{G_r}(\g/\g_r\oplus T^*R_x\oplus R_0)$.
The symplectic form on the latter comes from a natural identification of that homogenous bundle
with $[T^*G\times (T^*R_x\oplus R_0)]\red_0 G_r$ (the action of $G_r$
is diagonal with $G_r$ acting on $T^*G$ from the right). The moment map on the homogeneous bundle is given by $[g,(\alpha, \beta, \beta_0)]\mapsto \operatorname{Ad}(g)(\alpha+\hat{\mu}(\beta))$. Here $[g,(\alpha,\beta,\beta_0)]$ stands for the class in $G*_{G_r}(\g/\g_r\oplus T^*R_x\oplus R_0)$ of a point $(g,\alpha, \beta,\beta_0)\in
G\times(\g/\g_r\oplus T^*R_x\oplus R_0)$, and $\hat{\mu}:T^*R_x
\rightarrow \g_r$ is the moment map.

Let $\pi:T^*R\rightarrow T^*R\quo G$ and $\pi':G*_{G_r}(\g/\g_r\oplus T^*R_x\oplus R_0)\rightarrow
(\g/\g_r\oplus T^*R_x\oplus R_0)\quo G_r$ denote the quotient morphisms.
The symplectic slice theorem (see, for example, \cite{slice} where analytic neighborhoods instead of formal
ones were used) asserts that there is an isomorphism of formal neighborhoods $U$ of $\pi(r)$ in $T^*R\quo G$
and $U'$ of $\pi'([1,(0,0,0)])$ in $(\g/\g_r\oplus T^*R_x\oplus R_0)\quo G_r$ that lifts to a $G$-equivariant
symplectomorphism $\pi^{-1}(U)\cong \pi'^{-1}(U')$  intertwining the moment maps.

We have the restriction map $\param=\g^{*G}\rightarrow \hat{\param}=
\g_r^{*G_r}$. We set $\hat{\M}^0_{\param}(\hat{v}):=
\param\times_{\hat{\param}}\hat{\M}^0_{\hat{\param}}(\hat{v})$.
So we see that (compare with \cite[Section 6]{Nakajima})
\begin{equation}\label{eq:affine_decomp} \M^0_{\param}(v)^{\wedge_x}= (\hat{\M}^0_{\param}(\hat{v})\times R_0)^{\wedge_0}\end{equation}
(an equality of Poisson  schemes).

We have a similar decomposition for smooth quiver varieties. First, observe that
$$G*_{G_r}(\g/\g_r\oplus T^*R_x\oplus R_0)^{\theta-ss}=G*_{G_r}([\g/\g_r\oplus T^*R_x]^{\theta-ss}\times R_0),$$
where in the right hand side we slightly abuse the notation and write $\theta$  for the restriction of $\theta$ to $G_r$.
From here it follows that
\begin{equation}\label{eq:smooth_decomp}
\M^\theta_{\param}(v)^{\wedge_x}=(\hat{\M}^\theta_{\param}(\hat{v})\times R_0)^{\wedge_0},\end{equation}
where, recall, by definition, $\M^\theta_{\param}(v)^{\wedge_x}:=\M^0_{\param}(v)^{\wedge_x}\times_{\M^0_{\param}(v)}\M^\theta_{\param}(v)$.

Moreover, by the construction, the following diagram commutes

\begin{picture}(90,30)
\put(3,2){$\M^0(v)^{\wedge_x}$}
\put(2,22){$\M^\theta(v)^{\wedge_x}$}
\put(53,2){$(\hat{\M}(\hat{v})\times R_0)^{\wedge_0}$}
\put(52,22){$(\hat{\M}^\theta(\hat{v})\times R_0)^{\wedge_0}$}
\put(7,20){\vector(0,-1){14}}
\put(57,20){\vector(0,-1){14}}
\put(8,15){\tiny $\rho$}
\put(58,15){\tiny $\hat{\rho}\times \operatorname{id}$}
\put(19,4){\vector(1,0){32}}
\put(19,24){\vector(1,0){32}}
\put(30,5){\tiny $\cong$}
\put(30,25){\tiny $\cong$}
\end{picture}

Let us finish this discussion with a remark.


\begin{Rem}\label{Rem:sing_param}
Let us write $\varpi$ for the restriction map $\param\rightarrow \hat{\param}$.
It follows from (\ref{eq:affine_decomp}) that $\varpi^{-1}(\hat{\param}^{sing})
\subset \param^{sing}$.
\end{Rem}

\subsubsection{Resolution of singularities}\label{SSS_sing_resol}
\begin{Prop}\label{Prop:symp_resol}
The algebra $\C[\M^\theta(v)]$ is finitely generated.
The morphism $\rho:\M^\theta(v)\rightarrow \operatorname{Spec}(\C[\M^\theta(v)])$ is
a symplectic resolution of singularities.
\end{Prop}
\begin{proof}
Fix a generic $\lambda$ and consider the varieties $\M^{\theta}_{\C\lambda}(v)$
and $\M^0_{\C\lambda}(v)$. Both are schemes over $\C\lambda$.
We have a natural morphism $\phi_{\C\lambda}:\M^\theta_{\C\lambda}(v)
\rightarrow \M^0_{\C\lambda}(v)$ of schemes over $\C\lambda$
that is an isomorphism over $\C^\times \lambda$
by \ref{SSS_gen_param}. Since  $\M^\theta_{z\lambda}(v)$ is smooth, connected and
has the same dimension  for all $z\in \C$, we see that $\M^\theta_{\C\lambda}(v)$
has dimension $\dim \M^\theta(v)+1$ and the  morphism
$\M^\theta_{\C\lambda}(v)\rightarrow \C\lambda$ is dominant.

Let $\bar{\M}_{\C\lambda}(v)$ be the image of $\phi_{\C\lambda}$, this is a closed
subvariety in $\M^0_{\C\lambda}(v)$ because $\phi_{\C\lambda}$ is projective.
So it coincides with the closure of $\C^\times \lambda\times_{\C\lambda}\bar{\M}_{\C\lambda}(v)$ and
has dimension $\dim \M^\theta(v)+1$.

Hence the fiber $\bar{\M}(v)$ of $\bar{\M}_{\C\lambda}(v)$
over $0$ has dimension $\dim \M^\theta(v)$ and admits a surjective projective
morphism from $\M^\theta(v)$. Applying the Stein decomposition to this morphism
we decompose it to the composition of a birational morphism $\rho:\M^\theta(v)\rightarrow X$
with connected fibers and some finite dominant morphism $X\rightarrow \bar{\M}(v)$,
where $X$ is necessarily $\operatorname{Spec}(\C[\M^\theta(v)])$.
So  $\rho$ has to be  (an automatically symplectic) resolution of singularities.
\end{proof}


\begin{Cor}\label{Cor:prop_Mv}
The following claims are true.
\begin{enumerate}
\item The higher cohomology of $\mathcal{O}_{\M^\theta(v)}$ vanish.
\item The algebra $\C[\M^\theta(v)]$ is the specialization to $0$
of $\C[\M^\theta_{\param}(v)]$. In particular, the latter is finitely
generated.
\item The algebra $\C[\M^\theta(v)]$ coincides with the associated graded
of $\C[\M_\lambda^0(v)]$ for a generic $\lambda$ and, in particular, is independent of $\theta$.
\item The variety $\operatorname{Spec}(\C[\M^\theta(v)])$ is Cohen-Macaulay.
\end{enumerate}
\end{Cor}
\begin{proof}
(1) is a corollary of the Grauert-Riemenschneider theorem, and (2) is a corollary of
(1). (4) follows from Proposition \ref{Prop:symp_resol} and general properties of
symplectic resolutions.

Let us prove (3). Consider the $\C^\times$-equivariant morphism
$\M^\theta_{\param}(v)\rightarrow \M^0_\param(v)$. The $\C^\times$
action is induced from the action on $T^*R$. Specializing to a generic
parameter $\lambda\in \param$, we therefore get an isomorphism
$\C[\M^0_\lambda(v)]\rightarrow \C[\M^\theta_\lambda(v)]$. The filtrations
on both algebras are induced by the grading on $\C[T^*R]$, hence the
isomorphism intertwines these filtrations. From (2), we deduce that
$\C[\M^\theta(v)]=\gr \C[\M^\theta_\lambda(v)]$. So $\C[\M^\theta(v)]=\gr \C[\M^0_\lambda(v)]$
and hence $\C[\M^\theta(v)]$ is independent of $\theta$.  \end{proof}

We write $\M(v)$ for $\operatorname{Spec}(\C[\M^\theta(v)])$ and $\M^{(\theta)}_{\param}(v)$
for $\operatorname{Spec}(\C[\M^\theta_\param(v)])$, we will later see that
$\M^{(\theta)}_{\param}(v)$ is independent of the choice of $\theta$ so we will
write $\M_{\param}(v)$ for this variety.

\begin{Prop}\label{Prop:flatness_resol}
Suppose $\mu$ is flat. Then $\rho^*: \C[\M^0_{\param}(v)]\rightarrow \C[\M_{\param}(v)]$
is an isomorphism.
\end{Prop}
\begin{proof}
It is enough to show that $\rho^*: \C[\M^0(v)]\rightarrow \C[\M(v)]$ is an isomorphism
because $\C[\M^0_{\param}(v)],$ $\C[\M_{\param}(v)]$ are graded free over $\C[\param]$
and  $\C[\M_0(v)]=\C[\M_0^{\param}(v)]/(\param),\C[\M(v)]=\C[\M_{\param}(v)]/(\param)$.
Now note that both $\C[\M_0(v)],\C[\M(v)]$ are identified with the associated graded
of $\C[\M^0_\lambda(v)]=\C[\M_\lambda(v)]$ for $\lambda$ generic and, under this identification,
$\rho^*$ becomes the identity.
\end{proof}

\subsubsection{Identification of homology}\label{SS_homol_ident}
The purpose of this part is to establish an identification of the homology groups $H_*(\M_\lambda^\theta(v))$
for different generic $(\lambda, \theta)$.

First, there is a classical way to produce the identification, \cite[Section 9]{Nakajima}. We can view $\theta$
as an element in $\mathbb{R}^{Q_0}$, in this case we define $\M^\theta_\lambda(v)$ as a hyper-K\"{a}hler
reduction.  We get the same varieties as before, the complex structure on $\M^\theta_\lambda(v)$
depends only on the chamber of $\theta$. As we have mentioned in Section \ref{SS_intro_Nak},
this shows that all varieties
$\M^\theta_\lambda(v)$ with generic $(\lambda,\theta)$ are diffeomorphic  as $C^\infty$-manifolds. Consider the
generic locus and a bundle with fiber $H_*(\M^\theta_\lambda(v))$ on this locus. This is a flat bundle with respect to the
Gauss-Manin connection. But the generic locus of $(\lambda, \theta)$
is simply connected so the connection is trivial. Therefore all
fibers are canonically identified.

We will need a slightly different description. Pick a generic parameter $\lambda\in \param$. By the end of \ref{SSS_gen_param}, we have $\M_\lambda^\theta(v)=\M_\lambda^0(v)$.
Let $D$ denote the line through $\lambda$.
The inclusions $\M^\theta(v)\hookrightarrow \M^\theta_D(v),
\M^\theta_\lambda(v)\hookrightarrow \M^\theta_D(v)$ induce maps
$H^*(\M^\theta_D(v))\rightarrow H^*(\M^\theta(v)), H^*(\M^\theta_D(v))
\rightarrow H^*(\M^\theta_\lambda(v))$.
The former is an isomorphism because $\M_D^\theta(v)$ gets contracted to $\M^\theta(v)$
by a $\C^\times$-action. The latter is also an isomorphism because the resulting
map $H^*(\M^\theta(v))\rightarrow H^*(\M^\theta_\lambda(v))$ is precisely
the identification in the previous paragraph.

Note that the identification $H^2(\M^\theta(v))\cong H^2(\M^{\theta'}(v))$
intertwines the maps from $\param$ by the construction. The variety
$\M^\theta(v)$ admits a universal deformation over $H^2(\M^\theta(v))$
whose algebra of global functions is known to be independent of the choice
of generic $\theta$. We conclude that $\C[\M^\theta_\param(v)]$
is independent of $\theta$.

\subsection{Properties of quantizations}\label{SS_quant_prop}
In this section, we describe some properties of the algebras $\A_\lambda(v)$ and $\A_\lambda^0(v)$.

\subsubsection{Filtrations}
The algebras $\A_\lambda(v),\A_\lambda^0(v):=[D(R)/D(R)\{\Phi(x)-\langle \lambda,x\rangle\}]^G$
can be filtered in different ways, depending on a
filtration on $D(R)$ we consider. First of all, there is the {\it Bernstein filtration} on
$\A_\lambda(v), \A_\lambda^0(v)$ that is induced from the eponymous filtration on $D(R)$
(where $\deg R=\deg R^*=1$).
Let us write $\F_i\A_\lambda(v)$ for the $i$th filtration component with respect to this
filtration.  Note that $[\F_i\A_\lambda(v), \F_j\A_\lambda(v)]\subset \F_{i+j-2}\A_\lambda(v)$.

Sometimes,  it will be more convenient for us to work with filtrations, where the commutator
decreases degrees by $1$. Namely, equip $D(R)$ with the filtration by the order of
differential operator (where $\deg R^*=0, \deg R=1$). We have induced filtrations on
$\A_\lambda(v), \A_\lambda^0(v)$ to be denoted by $\F^Q_i$ (the superscript indicates that these filtrations depend on
the orientation). Note that $[\F^Q_i\A_\lambda(v), \F^Q_j\A_\lambda(v)]\subset \F^Q_{i+j-1}
\A_\lambda(v)$.

The two filtrations are related to each other. Namely, let $\mathsf{eu}$ denote the Euler
vector field in $D(R)$ (so that $[\mathsf{eu},\cdot]$ acts by $1$ on $R^*$
and by $-1$ on $R$). Since this element is $G$-invariant, it descends to
$\A_\lambda(v), \A_\lambda^0(v)$, we denote the images again by $\mathsf{eu}$.
So we can consider the inner $\Z$-gradings on the algebras of interest by eigenvalues
of $[\mathsf{eu},\cdot]$, let us write $\A_\lambda(v)=\bigoplus_{i}\A_\lambda(v)_i$
and $\A^0_\lambda(v)=\bigoplus_{i}\A^0_\lambda(v)_i$
for these gradings. The gradings are compatible with the filtrations $\F_i,\F_i^Q$ and we have
$$\F_{i}\A_\lambda(v)=\bigoplus_{k\in \Z} \F^Q_{k}\A_\lambda(v)_{i-2k}.$$
Thanks to this equality, the associated graded algebras for the two filtrations are the same.

The same considerations apply to $\A_\lambda^0(v)$.


In Section \ref{SS_intro_Nak} we have mentioned that $\gr \A_\lambda(v)=\C[\M(v)]$ and $H^i(\A_\lambda^\theta(v))=0$.
This is because $\gr\A^\theta_\lambda(v)=\mathcal{O}_{\M^\theta(v)}$ and
$H^i(\mathcal{O}_{\M^\theta(v)})=0$ for $i>0$.

\subsubsection{$\A_\lambda^0(v)$ vs $\A_\lambda(v)$, I}
Now we want to relate the algebra $\A_\lambda(v)$ to
$\A^0_\lambda(v)$.
We have a natural epimorphism $\C[\M^0(v)]\twoheadrightarrow \gr \A^0_\lambda(v)$
to be denoted by $\eta$.
Besides, we have a natural homomorphism $\kappa:\A^0_\lambda(v)\rightarrow \A_\lambda(v)$
coming from restricting elements of $D(R)$ to $(T^*R)^{\theta-ss}$.
It is clear that $\rho^*:\C[\M^0(v)]\rightarrow \C[\M^\theta(v)]$ coincides with the composition  $\gr\kappa\circ \eta:\C[\M^0(v)]\twoheadrightarrow \gr  \A_\lambda(v)$.
It follows that $\A^0_\lambda(v)=\A_\lambda(v)$ and $\gr\A^0_\lambda(v)=\C[\M^0(v)]$
when $\mu$ is flat. In particular,
$\A_\lambda(v)$ is independent of $\theta$ in this case. This is also true
for an arbitrary vector $v$ by \cite[Proposition 3.8]{BPW}.

\subsubsection{Variations}
Now let us consider some related constructions.
We can consider the  quantization
$$\A^\theta_{\paramq}(v)=D_R\red^\theta G:=[(D_R/ D_R\{x_R, x\in [\g,\g]\})|_{T^*R^{\theta-ss}}]^G$$
of $\M^\theta_{\param}(v)$ and its global section $\A_{\paramq}(v)$.
For an affine subspace $\paramq_0\subset \paramq$, we consider pull-backs
$\A^\theta_{\paramq_0}(v):=\C[\paramq_0]\otimes_{\C[\paramq]}\A^\theta_{\paramq}(v),
\A_{\paramq_0}(v):=\C[\paramq_0]\otimes_{\C[\paramq]}\A_{\paramq}(v)$.
Those are quantizations of $\M^\theta_{\param_0}(v), \M_{\param_0}(v)$, where $\param_0\subset \param$ is the vector subspace
corresponding to $\paramq_0$. Note that $\A_{\paramq_0}(v)=\Gamma(\A^\theta_{\paramq_0}(v))$.
It also makes sense to speak about $\A_{\paramq_0}^0(v)$.

We also consider homogenized
versions. Namely, we take the Rees sheaf $D_{R,\hbar}$ of $D_R$ (for the filtration by the order of a differential
operator) and its reduction
$\A^\theta_{\paramq}(v)_\hbar$, it is related to $\A^\theta_{\paramq}(v)$ via
$\A^\theta_{\paramq}(v)=\A^\theta_{\paramq,\hbar}(v)/(\hbar-1)$. Also consider
the global sections $\A_{\paramq}(v)_\hbar$. This is a graded (with positive grading) deformation of $\C[\M(v)]$ over the space $\param\oplus \C$. Here we consider the grading coming
from the action of $\C^\times$ on $T^*R$ by dilations: $t.(r,\alpha)=(t^{-1}r,t^{-1}\alpha)$
so that the parameter space $\param\oplus \C$ is in degree $2$.


\subsubsection{Quantized LMN isomorphisms}\label{SSS_LMN_quant}
The LMN isomorphisms discussed in \ref{SSS_LMN} can be quantized. This was
done in \cite{quant} in a special case (but the construction generalizes in
a straightforward way). In fact, the quantum isomorphisms
can be obtained by the same reduction in stages construction as before. One either
quantizes the steps of that argument or argues similarly to \cite[Section 6.4]{quant}:
for $\theta_i>0$, isomorphism (\ref{eq:part_red_iso}) can be regarded as an isomorphism of  symplectic schemes $\mathcal{X}:=T^*R\red^{\theta_k} \GL(v_k), \mathcal{X'}:=T^*R'\red^{-\theta_k}\GL(\tilde{w}_k-v_k)$
over $\mathbb{A}^1$ that gives the multiplication by $-1$ on the base. Here we write
$R'$ for $\operatorname{Hom}(\tilde{W}_k,V_k)\oplus \underline{R}$. So (\ref{eq:part_red_iso})
extends to an isomorphism of the canonical (=even+ $\C^\times$-equivariant) deformation quantizations $\mathcal{D},\mathcal{D}'$ of the schemes $\mathcal{X},\mathcal{X}'$ that are defined as follows:
\begin{align*}
&\mathcal{D}:=[D_{R,\hbar}/D_{R,\hbar}\Phi_k^{sym}(\mathfrak{sl}(v_k))|_{T^*R^{\theta_k-ss}}]^{\GL(v_k)},\\
&\mathcal{D}':=[D_{R',\hbar}/D_{R',\hbar}\Phi_k^{sym}(\mathfrak{sl}(\tilde{w}_k-v_k))|_{T^*R'^{\theta_k-ss}}]^{\GL(\tilde{w}_k-v_k)}
\end{align*}
Here $\Phi^{sym}$, the {\it symmetrized quantum comoment map}, stands the composition of $\g\rightarrow \mathfrak{sp}(T^*R)$ and the natural embedding $\mathfrak{sp}(T^*R)\hookrightarrow \mathbf{A}_\hbar(T^*R)$, where
$\mathbf{A}$ denotes the Weyl algebra.
For the discussion of canonical and even quantizations and connections between them see \cite[Sections 2.2,2.3]{quant}.
For the definition of a symmetrized quantum comoment map, see \cite[Section 5.4]{quant}. The isomorphism  $\mathcal{D}\xrightarrow{\sim}\mathcal{D}'$
does not intertwine the symmetrized quantum comoment maps  for the $\underline{G}$-actions on
$\mathcal{D},\mathcal{D}'$ but rather does the same change as the with the classical comoment maps.

So we get an isomorphism $\A^\theta_\lambda(v) \xrightarrow{\sim}
\A^{\sigma \theta}_{\sigma\bullet^v \lambda}(\sigma\bullet v)$, where the
parameter $\sigma\bullet^v \lambda$ is determined as follows.
Let $\varrho(v)$ be the character of $\g$ equal $-\frac{1}{2}\chi_{\bigwedge^{top} R}$,
where $\chi_{\bigwedge^{top} R}$ is the character of  the action of $\g$ on $\bigwedge^{top} R$. Then $\Phi(x)-\Phi^{sym}(x)=\langle \varrho(v),x\rangle$. Hence we have
\begin{equation}\label{eq:param_Weyl_action}
\sigma\bullet^v\lambda=\sigma(\lambda-\varrho(v))+\varrho(\sigma\bullet v)
\end{equation}

We remark that $\Phi(x)$ depends on the orientation of $Q$ (while $\Phi^{sym}(x)$ does not)
and we have  \begin{equation}\label{eq:rho}\varrho(v)_k=\frac{1}{2}(\sum_{a, h(a)=k}v_{t(a)}-\sum_{a,t(a)=k}v_{h(a)}-w_k),\quad k\in Q_0.\end{equation}
When we change an orientation of $Q$, the character $\varrho(v)$ changes by an element from $\Z^{Q_0}$.

We compute $s_i\bullet^v \lambda$  in the case when $i$ is a source so that $\varrho(v)_i=-\frac{1}{2}\tilde{w}_i$.  By (\ref{eq:param_Weyl_action}), $s_i\bullet^v \lambda=s_i\lambda+\varrho(s_i\bullet v)-s_i\varrho(v)$
and what we need to compute is $\varrho(s_i\bullet v)-s_i\varrho(v)$.
We have $\rho(s_i\bullet v)_k=(s_i\varrho(v))_k$ when $k$ is different from $i$ and is not adjacent to $i$.
When $k=i$, we have $(s_i\varrho(v))_k=-\varrho(v)_k=-\varrho(s_i\bullet v)_k$. Finally, let us consider
the case when $k$ is adjacent to $i$, say there are $q$ arrows from $i$ to $k$.
Then $(s_i\varrho(v))_k=\varrho(v)_k+q \varrho(v)_i=\varrho(v)_k-\frac{q}{2}\tilde{w}_i$ and
$\varrho(s_i\bullet v)_k=\varrho(v)_k+ \frac{q}{2}(\tilde{w}_i-v_i-v_i)$. In particular, we deduce
that $\varrho(s_i\bullet v)-s_i \varrho(v)\in \Z^{Q_0}$.

\begin{Rem}\label{Rem:integr_difference_action}
One conclusion that will be used below is that $\varrho(\sigma\bullet v)-\sigma\varrho(v)$ is integral and
hence $\sigma\bullet^v\lambda-\lambda$ is integral if and only if $\sigma\lambda-\lambda$ is.
\end{Rem}

An important corollary of $\A^\theta_\lambda(v)\xrightarrow{\sim}
\A^{\sigma \theta}_{\sigma\bullet \lambda}(\sigma\bullet v)$ is an
isomorphism $\A_\lambda(v)\xrightarrow{\sim} \A_{\sigma\bullet\lambda}(\sigma\bullet v)$.

We also would like to point out that the quantum LMN isomorphisms are $T$-equivariant,
this also follows from \cite[Proposition 4.13]{Webster_O}.

\subsubsection{$\A_\lambda^0(v)$ vs $\A_\lambda(v)$, II}
Recall that for a subvariety $Y\subset V$, where $V$ is a vector space, one can define its
{\it asymptotic cone} $\mathsf{AC}(Y)$ as $\operatorname{Spec}(\gr \C[Y])\subset V$,
where we take the filtration on $\C[Y]$ induced by the epimorphism $\C[V]\twoheadrightarrow \C[Y]$.

A Zariski open subset $\paramq^0\subset \paramq$ will be called {\it asymptotically
generic} if $\mathsf{AC}(\paramq\setminus \paramq^0)\subset \param^{sing}$.

Recall that we write $\paramq^{iso}$ for the set of $\lambda\in \paramq$
such that $\A_\lambda^0(v)\rightarrow \A_\lambda(v)$ is an isomorphism.
The following proposition (to be proved in Section \ref{SS_compl2}) should be
thought as a quantum analog of the isomorphism $\M^0_\lambda(v)\cong \M_\lambda(v)$
for a  generic $\lambda$.

\begin{Prop}\label{Prop:alg_iso}
The subvariety $\paramq^{iso}\subset \paramq$ is Zariski open and asymptotically  generic.
\end{Prop}

\subsubsection{Spherical symplectic reflection algebras}\label{SSS_SRA}
Here we will discuss the special case when $Q$ is an affine quiver. Let $0$
denote the extending vertex (so that $Q\setminus 0$ is a finite Dynkin quiver),
and $w=\epsilon_0$, the coordinate vector at the extending vertex.

It is a classical fact that all weights of the irreducible $\g(Q)$-module
$L_{\omega_0}$ (a.k.a. the basic representation) are conjugate to the
weights of the form $\omega_0-n\delta, n\in \Z_{\geqslant 0}$, under the
action of $W(Q)$. Thanks to the quantum LMN isomorphisms it is enough
to consider $v=n\delta$. Here the algebra $\A_\lambda(v)$ is known
to be isomorphic to a certain {\it spherical symplectic reflection algebra}.
Let us recall some basics about these algebras.

Let $\Gamma$ be a finite subgroup in $\operatorname{Sp}(V)$, where $V$ is a symplectic vector space. We choose
independent variables ${\bf c}=({\bf c}_0,\ldots,{\bf c}_r)$, one for each conjugacy class of symplectic reflections in $\Gamma$.
Then we can consider the algebra ${\bf H}$, the quotient of $T(V)\#\Gamma[{\bf c}_0,\ldots,{\bf c}_r]$ by
the relations of the form
$$[u,v]=\omega(u,v)+ \sum_{i=0}^r {\bf c}_i \sum_{s\in S_i}\omega_s(u,v), u,v\in V.$$
Here $S_1,\ldots,S_r$
are the conjugacy classes of symplectic reflections in $\Gamma$, $\omega$ is the symplectic form on $V$, and $\omega_s(u,v)=\omega(\pi_s u, \pi_s v)$,
where we write $\pi_s$ for the $s$-invariant projection from $V$ to $\operatorname{im}(s-1)$. Inside ${\bf H}$ we can consider the spherical subalgebra $e{\bf H}e$, where $e=\frac{1}{|\Gamma|}\sum_{\gamma\in \Gamma}\gamma$.
Also, for numerical values of ${\bf c}$, say $c$, we can consider the specializations
$\mathcal{H}_{c}$ of ${\bf H}$. Recall that a parameter $c$ is called {\it spherical} if $e\mathcal{H}_{c}e$ and
$\mathcal{H}_{c}$ are Morita equivalent (via the bimodule $\mathcal{H}_{c}e$).

Examples of $\Gamma$ that are of most interest for us are as follows. Take a finite subgroup $\Gamma_1\subset \SL_2(\C)$
and a positive integer $n$. Then we can form the group $\Gamma=\Gamma_n:=\mathfrak{S}_n\ltimes \Gamma_1^n$ that acts
on $V=\C^{2n}$ by linear symplectomorphisms. We have two kinds of symplectic reflections: the conjugacy class $S_0$ containing
transpositions in $\mathfrak{S}_n$, and conjugacy classes $S_1,\ldots,S_r$ containing elements from the $n$ copies of $\Gamma$
(here $r$ is the number of the nontrivial conjugacy classes in $\Gamma_1$).
We will use the notation $\mathcal{H}_{\kappa,c}(n)$ for the algebra corresponding to $c_0=2\kappa$ and $c_1,\ldots,c_r$.

Now recall that, by the McKay correspondence, to $\Gamma_1$ we can assign
an affine Dynkin quiver $Q$. Take $v=n\delta$, where $\delta$ is the indecomposable imaginary root, and $w=\epsilon_0$,
where $0$ stands for the extending vertex of $Q$. Then we have isomorphisms $e\mathcal{H}_{\kappa,c}(n)e\cong \A_\lambda(v)$, where $\lambda$ can obtained from $c$ by formulas explained in \cite[1.4]{EGGO}.  In particular,
$\kappa=\langle\lambda,\delta\rangle$. For example, for $\Gamma_1=\{1\}$  we just get $e\mathcal{H}_{\kappa, \varnothing}(n)e=\A_\kappa(n)$.

The following lemma gives a characterization of spherical values of $e\mathcal{H}_{\kappa,c}(n)e$.

\begin{Lem}\label{Lem:spher}
The following claims are true.
\begin{enumerate}
\item The parameter $(\kappa,c)$ is spherical if and only if $e\mathcal{H}_{\kappa,c}(n)e$ has finite homological dimension.
\item The parameter $\kappa$ of the type A Rational Cherednik algebra $\mathcal{H}_\kappa(n)$ is not spherical if and only if
$\kappa=-\frac{s}{m}$ with $1<m\leqslant n$ and $0<s<m$.
\end{enumerate}
\end{Lem}

(1) follows from \cite[Theorem 5.5]{Etingof_affine} and (2) is proved in \cite[Corollary 4.2]{BE}.

\subsection{Coherent  modules}\label{SS_Coh}
Let us proceed to defining suitable categories of sheaves of modules over the sheaves of algebras
$\A_\lambda^\theta(v)$. We follow \cite[Section 4]{BPW}.

\subsubsection{Coherent modules}
Now let $X$ be a smooth symplectic variety (with a $\C^\times$-action rescaling the Poisson
bracket) and $\mathcal{D}$ be its filtered quantization. Recall that this means that
$\mathcal{D}$ is a filtered sheaf of algebras in the conical topology on $X$
together with an isomorphism $\gr\mathcal{D}\cong \mathcal{O}_X$ of sheaves of graded Poisson
algebras. We also require that the filtration on $\mathcal{D}$ is complete and separated.

Let $M$ be a  sheaf of $\mathcal{D}$-modules in the conical topology.

\begin{defi}\label{Def:coh}
We say that a $\mathcal{D}$-module $M$ is coherent if it can be equipped with a global complete and separated
filtration such that $\gr M$ is a coherent $\mathcal{O}_X$-module (this filtration is called good).
\end{defi}

Let us write $\operatorname{Coh}(\mathcal{D})$ for the category of coherent $\mathcal{D}$-modules.

We can also consider the completed Rees sheaf $\mathcal{D}_\hbar$ of $\mathcal{D}$. By definition,
a coherent $\mathcal{D}_\hbar$-module $M_\hbar$ is a coherent sheaf of modules such that the $\hbar$-adic
filtration on $M_\hbar$ is complete and separated and $M_\hbar/\hbar M_\hbar$ is a coherent $\Str_X$-module.

The notion of coherent modules we use is equivalent to that of \cite[Section 4]{BPW}. They consider
the case when $X$ is a conical symplectic resolution and the action is contracting but the definition
generalizes to our setting in a straightforward way. Namely, \cite{BPW}
considers $\C^\times$-equivariant $\mathcal{D}_\hbar[\hbar^{-1}]$-modules. Such a module is called coherent there
if it has a coherent $\mathcal{D}_\hbar$-lattice. From a coherent $\mathcal{D}$-module $M$ we can produce a $\C^\times$-equivariant
coherent $\mathcal{D}_\hbar[\hbar^{-1}]$-module via $M\mapsto \mathcal{D}_\hbar[\hbar^{-1}]\otimes_{\mathcal{D}}M$ (note that
$\mathcal{D}=\mathcal{D}_\hbar[\hbar^{-1}]^{\C^\times}$). A quasi-inverse functor is given by taking $\C^\times$-invariant.

The following lemma contains some basic facts about coherent $\mathcal{D}$-modules.

\begin{Lem}\label{Lem:coh}
The following statements are true:
\begin{enumerate}
\item Let $X$ be affine, and $\A:=\Gamma(\mathcal{D})$.
Then the functors $M\mapsto M^{loc}:=\mathcal{D}\otimes_{\A}M$
and $N\mapsto \Gamma(N)$ are mutually inverse equivalences between $\A\operatorname{-mod}$ and $\operatorname{Coh}(\mathcal{D})$. A similar claim holds for the coherent
$\mathcal{D}_\hbar$-modules.
\item $\operatorname{Coh}(\mathcal{D})$ is an abelian subcategory in the category of sheaves
of all $\mathcal{D}$-modules.
\end{enumerate}
\end{Lem}
\begin{proof}
Let us prove (1). Note that $\gr(M^{loc})$ is the coherent sheaf on $X$
associated to $\gr M$ and $\gr \Gamma(N)=\Gamma(\gr N)$, the latter is true because
$H^1(X,\gr N)=0$. This shows that
the natural homomorphisms $M\mapsto \Gamma(M^{loc}), \Gamma(N)^{loc}\rightarrow N$
are isomorphisms after passing to the associated graded modules, hence are isomorphisms
because  all the filtrations involved are complete and separated.

Let us prove (2). This amounts to prove that for a morphism $\varphi: M\rightarrow N$ of coherent
sheaves, the kernel and the cokernel are coherent. Choose good filtrations $M=\bigcup_{i\in \Z}M_{\leqslant i},
N=\bigcup_{i\in \Z} N_{\leqslant i}$. After shifting a filtration on $N$, we can assume that
$\varphi(M_{\leqslant i})\subset N_{\leqslant i}$: this is a local condition and in the case of
modules over algebras, this claim is classical. Let $M_\hbar, N_\hbar$ denote the completed
Rees modules of $M,N$. Then $\varphi$ gives rise to a $\C^\times$-finite homomorphism $\varphi_\hbar:
M_\hbar\rightarrow N_\hbar$. Then $\ker \varphi_\hbar$ is clearly coherent.
The module $\ker \varphi$ is obtained from $\ker \varphi_\hbar$ by first taking
locally $\C^\times$-finite sections and then taking the quotient by $\hbar-1$.
This procedure endows $\ker\varphi$ with a filtration that satisfies the conditions
of Definition \ref{Def:coh}. So $\ker\varphi$ is coherent.

Now let us explain why $\operatorname{coker}\varphi$ is coherent. Similarly to the case of
kernel, it is enough to show that $\operatorname{coker}\varphi_\hbar$ is coherent.
This will follow once we know that $\operatorname{im}\varphi_\hbar$ is closed in
the $\hbar$-adic topology on $N_\hbar$. Again, this is a local condition and for modules
over an algebra this is easy to show.
\end{proof}

\subsubsection{Derived categories and functors}
We consider $X=\M^\theta(v)$ with its contracting $\C^\times$-action (we could also consider
the action  induced by $\deg R=0,\deg R^*=1$).
Below we write $\A_\lambda^\theta(v)\operatorname{-mod}$ for $\operatorname{Coh}(\A_\lambda^\theta(v))$.
Here we investigate the derived category $D^b(\A_\lambda^\theta(v)\operatorname{-mod})$
and the derived global section functor. Let us write $\operatorname{Sh}(\A_\lambda^\theta(v))$
for the category of all sheaves of $\A_\lambda^\theta(v)$-modules.

\begin{Lem}\label{Lem:coh_derived_cat}
The following statements are true.
\begin{enumerate}
\item The natural functor $D^b(\A_\lambda^\theta(v)\operatorname{-mod})\rightarrow
D^b(\operatorname{Sh}(\A_\lambda^\theta(v)))$ is a full embedding.
\item The derived global section functor $R\Gamma$ takes
$D^b(\A_\lambda^\theta(v)\operatorname{-mod})$ to $D^b(\A_\lambda(v)\operatorname{-mod})$,
where we write $\A_\lambda(v)\operatorname{-mod}$ for the category of all
finitely generated $\A_\lambda(v)$-modules.
\end{enumerate}
\end{Lem}
\begin{proof}
(1) is \cite[Corollary 5.11]{BPW} and (2) is \cite[Proposition 4.12]{BPW}.
\end{proof}

Note that (1) implies that $R\Gamma$ is given by taking the \v{C}ech complex.

\subsection{Supports and characteristic cycles}\label{SSS_Supp_CC}
Now let us define the supports of  objects in $\A_\lambda(v)\operatorname{-mod}$ and supports and
characteristic cycles for objects in $\operatorname{Coh}(\A_\lambda^\theta(v))$.
We also define holonomic modules.

For $M\in \A_\lambda(v)\operatorname{-mod}$ we can define the  support, $\operatorname{Supp}M$, to be the support of the coherent sheaf $\gr M$ with respect to any good filtration. Similarly, we can define the support of an object in $\A_\lambda^\theta(v)\operatorname{-mod}$.

We remark that the support of an $\A_\lambda^\theta(v)$-module (resp., an $\A_\lambda(v)$-module)
$M$ is a coisotropic subvariety in $\M^\theta(v)$ (resp., $\M(v)$) by the Gabber involutivity theorem, \cite{Ga}, an easier
proof due to Knop can be found in \cite[Section 1.2]{Ginzburg_D_mod}.
If the support of $M\in \A_\lambda^\theta(v)\operatorname{-mod}$ is lagrangian, then we call $M$ {\it holonomic}.
An object $N\in \A_\lambda(v)\operatorname{-mod}$ is called holonomic if the intersection of $\operatorname{Supp}(N)$
with every symplectic leaf in $\M(v)$ is isotropic. By \cite[Appendix]{B_ineq}, this is equivalent to
$\rho^{-1}(\operatorname{Supp}(N))$ to be isotropic.

Let us proceed to characteristic cycles. Suppose $Y\subset \M^\theta(v)$ is a $\C^\times$-stable isotropic  subvariety. Recall that to a coherent sheaf
$M_0$ on $\M^\theta(v)$ supported on $Y$ on can assign its characteristic cycle $\CC(M_0)$ equal to
the following formal linear combination of the  irreducible components of $Y$:
$$\CC(M_0):=\sum_{Y'\subset Y}(\operatorname{grk}_{Y'}M_0)Y',$$
where $\operatorname{grk}_{Y'}$ stands for the rank in the generic point of a component $Y'$.
We can define the characteristic
cycle $\CC$ of a coherent $\A^\theta_\lambda(v)$-module $M$ supported on $Y$ by $\CC(M):=\CC(\gr M)$,
this is easily seen to be well-defined. An alternative definition is given in \cite[Section 6.2]{BPW}. Yet another description
of $\CC(M)$ is as follows. The object $M$
gives rise to a well-defined class in $K_0(\operatorname{Coh}_Y(\M^\theta(v)))$, that of $\gr M$.
The map $[M]\rightarrow [\gr M]: K_0(\A_\lambda^\theta(v)\operatorname{-mod}_Y)\rightarrow
K_0(\operatorname{Coh}_Y \M^\theta(v))$ will be called the {\it degeneration map} in what follows.
Applying the Chern character map, we get an element  $\CC'(M)\in H^*(\M^\theta(v), \M^\theta(v)\setminus Y)=H^{BM}_*(Y)$.
Then $\CC(M)$ coincides with the projection of $\CC'(M)$ to $H^{BM}_{top}(Y)$.

When $Y=\rho^{-1}(0)$, we have $H^{BM}_*(Y)=
H_*(Y)=H_*(\M^\theta(v))$. The first equality holds because $\rho^{-1}(0)$ is compact, the second one is true because
$\M^\theta(v)$ is contracted onto $\rho^{-1}(0)$ by the $\C^\times$-action (induced by the dilation
action on $T^*R$).

\begin{Prop}[\cite{BarGin}]\label{Prop:CC_inj}
The map $\CC: K_0(\A^\theta_\lambda(v)\operatorname{-mod}_{\rho^{-1}(0)})\rightarrow H_{mid}(\M^\theta(v))$ is injective.
\end{Prop}

The proof of this proposition has not appeared yet, so we will give an independent proof later in the paper
in the case when $Q$ is finite and when $Q$ is affine with $w=\epsilon_0$.

\subsection{Various functors}
In this section we will study Hamiltonian reduction functors from the category of $(G,\lambda)$-equivariant D-modules on $R$ to the categories of modules over $\A_\lambda^0(v),\A_\lambda^\theta(v)$. We will also study the localization and global section functors and their connection to Hamiltonian reduction functors.

\subsubsection{Twisted equivariant D-modules}
By a  $(G,\lambda)$-equivariant $D(R)$-module one means a weakly $G$-equivariant module $\M$ such that
$x_{\M}m=\Phi(x)m-\lambda(x)m$ for all $x\in \g, m\in \M$.
We consider the category $D(R)\operatorname{-Mod}^{G,\lambda}$ of all  $(G,\lambda)$-equivariant modules
over $D(R)$ and its full subcategory $D(R)\operatorname{-mod}^{G,\lambda}$ of finitely generated modules.
Note that, for $\chi\in \Z^{Q_0}$,
the categories $D(R)\operatorname{-Mod}^{G,\lambda}$ and $D(R)\operatorname{-Mod}^{G,\lambda+\chi}$
are equivalent, via $M\mapsto M\otimes \C_{-\chi}: D(R)\operatorname{-Mod}^{G,\lambda}
\rightarrow D(R)\operatorname{-Mod}^{G,\lambda+\chi}$, where $\C_{-\chi}$ is the one-dimensional
$G$-module corresponding to the character $-\chi$.

\subsubsection{Functors for abelian categories}\label{SSS_abel_red_fun}
Let us write $\A^0_\lambda(v)\operatorname{-Mod}$ for the category of all $\A^0_\lambda(v)$-modules.
 We have a functor $\pi^0_{\lambda}(v):D(R)\operatorname{-Mod}^{G,\lambda}\rightarrow \A^0_{\lambda}(v)\operatorname{-Mod}$ of taking $G$-invariants that restricts to $D(R)\operatorname{-mod}^{G,\lambda}\rightarrow \A^0_{\lambda}(v)\operatorname{-mod}$. It is  a quotient functor,
it kills precisely the modules without nonzero $G$-invariants.
It has a left adjoint (and right inverse) $$\pi^0_{\lambda}(v)^!: \A^0_{\lambda}(v)\operatorname{-Mod}\rightarrow D(R)\operatorname{-Mod}^{G,\lambda},$$
given by taking the tensor product with the $D(R)$-$\A^0_\lambda(v)$-bimodule $$\mathcal{Q}_\lambda:=D(R)/D(R)\{\Phi(x)-\langle\lambda,x\rangle, x\in \g\}.$$
This functor restricts to $\A^0_\lambda(v)\operatorname{-mod}
\rightarrow D(R)\operatorname{-mod}^{G,\lambda}$.

 Note that $\mathcal{Q}_\lambda$ is
$(G,\lambda)$-equivariant as a $D(R)$-module so $\pi^0_\lambda(v)^!$ indeed maps to $(G,\lambda)$-equivariant
$D(R)$-modules.

Recall that we assume that $\theta$ is generic. We have a functor
$$\pi^\theta_\lambda(v):D_R\operatorname{-mod}^{G,\lambda}\rightarrow \A_\lambda^\theta(v)\operatorname{-mod}$$
it first restricts a $D$-module to the $\theta$-semistable locus and then takes the $G$-invariants.
The image of $D(R)\operatorname{-mod}^{G,\lambda}$ consists of coherent modules.

\begin{Prop}\label{Prop:quot_Ham_loc}
The functor $\pi^\theta_\lambda(v)$  is a quotient functor.
\end{Prop}

In the case when $\mu$ is flat the proof was given in \cite[Section 5.5]{BPW} and also
announced in \cite[Proposition 4.9]{MN}. Below, in Section \ref{SS_transl_bimod}, we will explain
how to generalize the proof from \cite{BPW} without the flatness assumption.

\subsubsection{Reminder on equivariant derived categories}\label{SSS_equiv_der_remind}
We will need derived versions of the reduction functors
considered in \ref{SSS_abel_red_fun}. We can form the derived
categories $D^?(D(R)\operatorname{-mod}^{G,\lambda})$
(the naive derived categories; here $?$ stands for $+,-$ or $b$) but we will also need the equivariant derived categories $D^?_{G,\lambda}(D(R)\operatorname{-mod})$.
Here we recall some basics regarding equivariant derived categories.

Let $\A$ be an associative algebra equipped with a rational action of a
connected reductive algebraic group $G$. Assume that this action is
Hamiltonian with quantum comoment map $\Phi$ so it makes sense to speak
about weakly $G$-equivariant and $G$-equivariant $\A$-modules.
Then the equivariant derived category $D^b_G(\A\operatorname{-mod})$
is defined as follows. Consider the Chevalley-Eilenberg complex $\bar{U}(\g)$,
a standard resolution of the trivial one-dimensional $\g$-module,
and form the tensor product $\A\otimes \bar{U}(\g)$. This is a differential
graded algebra equipped with a Hamiltonian $G$-action (the diagonal
action together with the diagonal quantum comoment map). So it makes sense
to speak about $G$-equivariant differential graded $\A\otimes \bar{U}(\g)$-modules.
The category $D^b_G(\A\operatorname{-mod})$ is obtained from the category of those
modules by passing to the homotopy category and localizing
the quasi-isomorphisms. Consider
the natural homomorphism $\varpi:\A\otimes \bar{U}(\g)\rightarrow \A$ of differential
graded algebras (taking the 0th homology). The pull-back functor
$\varpi^*$ is a natural functor $D^b(\A\operatorname{-mod}^G)\rightarrow
D^b_G(\A\operatorname{-mod})$. On the other hand, the category
of $G$-equivariant $\A\otimes U(\g)$-modules is the same as the category
of weakly $G$-equivariant $\A$-modules. We have a $G$-equivariant
homomorphism $\iota:\A\otimes U(\g)\rightarrow \A\otimes \bar{U}(\g)$
intertwining the quantum moment maps. This gives a pull-back functor
$\iota^*:D^b_G(\A\operatorname{-mod})\rightarrow D^b(\A\otimes U(\g)\operatorname{-mod}^G)$.
The composition $\iota^*\circ \varpi^*: D^b(\A\operatorname{-mod}^G)\rightarrow
D^b(\A\otimes U(\g)\operatorname{-mod}^G)$ comes from the forgetful functor  between
abelian categories -- from the strongly equivariant category
to a weakly equivariant one. Besides, we have  left adjoints of $\iota^*, \varpi^*$, the functors
$\iota_!(\bullet):=\bullet\otimes^L_{U(\g)}\C$ and $\varpi_!(\bullet):=\A\otimes^L_{\A\otimes \bar{U}(\g)}\bullet$.

This discussion implies the following lemma to be used in what follows.

\begin{Lem}\label{Lem:der_equiv_comput}
Let $V$ be a $G$-module. For $M\in D^b_G(\A\operatorname{-mod})$, we have a natural isomorphism
$$\Hom_{D^b_G(\A\operatorname{-mod})}((\A\otimes V)\otimes^L_{U(\g)}\C, M)\cong\Hom_G(V, H_0(M)).$$
\end{Lem}

\subsubsection{Functors for derived categories}\label{SSS_Ham_der_fun}
Let us proceed to   derived analogs of $\pi^\theta_\lambda(v),\pi^0_\lambda(v)$
and $\pi^0_\lambda(v)^!$.

The functor $\varpi^*:D^b(D(R)\operatorname{-mod}^{G,\lambda})\rightarrow
D^b_{G,\lambda}(D(R)\operatorname{-mod})$  is an equivalence provided $\mu$ is flat,
see \cite[Theorem 1.6]{BLu}. This generalizes to any filtered algebra $\A$, not just $D(R)$, provided
the filtration is complete and separated.

In general, i.e. without additional assumptions on $\mu$, we have the following lemma. Consider the subcategories $$D^b_{\theta-uns}(D(R)\operatorname{-mod}^{G,\lambda}),
D^b_{G,\lambda,\theta-uns}(D(R)\operatorname{-mod})$$ of all all objects with (singular) supports of homology contained
in $(T^*R)^{\theta-uns}$.

\begin{Lem}\label{Lem:equiv_der_fun}
The induced functor
$$D^b(D(R)\operatorname{-mod}^{G,\lambda})/
D^b_{\theta-uns}(D(R)\operatorname{-mod}^{G,\lambda})\rightarrow
D^b_{G,\lambda}(D(R)\operatorname{-mod})/D^b_{G,\lambda,\theta-uns}(D(R)\operatorname{-mod}).$$
is a category equivalence.
\end{Lem}
\begin{proof}
Let $U$ be a $G$-stable open affine subvariety of $T^*R$.
Recall from \cite[Theorem 1.6]{BLu} that if $G$ acts freely on $U$, then
the functors $\varpi^*_U,\varpi_{U!}$ for the algebra $\A=D_R(U)$
are mutually inverse equivalences. We have t-exact functors
$$D^b(D(R)\operatorname{-mod}^{G,\lambda})\rightarrow D^b(\A\operatorname{-mod}^{G,\lambda}),
D^b_{G,\lambda}(D(R)\operatorname{-mod})\rightarrow D^b_{G,\lambda}(\A\operatorname{-mod})$$
induced by microlocalization. They intertwine $\varpi^*$ with $\varpi^*_U$ and
$\varpi_!$ with $\varpi_{U!}$.

We need to show that for all $M\in D^b(D(R)\operatorname{-mod}^{G,\lambda}),
N\in D^b_{G,\lambda}(D(R)\operatorname{-mod})$  the cones of the adjunction
morphisms $\omega_!\circ \omega^*(M)\rightarrow M$ and $N\rightarrow \omega^*\circ \omega_!(N)$
have cohomology supported on $\mu^{-1}(0)^{\theta-uns}$. Indeed, this precisely means
that $\omega^*, \omega_!$ give mutually quasi-inverse equivalences between the
quotient categories.

Let $\omega_!\circ \omega^*(M)\rightarrow M\rightarrow M_0\xrightarrow{+1}$ be a distinguished
triangle. By the first paragraph of the proof, $\omega_{U!}\circ \omega_U^*(M|_U)\xrightarrow{\sim} M_U$.
It follows that the microlocaliztion $M_{0}|_U$ is zero, equivalently, the cohomology
of $M_0$ are supported away from $U$. Note that $(T^*R)^{\theta-ss}$ is covered by
open affine $G$-stable subvarieties of the form $(T^*R)_f$, where $f$ is a $G$-semiinvariant.
It follows that the cohomology of $M_0$ are supported
on $\mu^{-1}(0)^{\theta-uns}$. The statement for $N\rightarrow \omega^*\circ \omega_!(N)$
is established similarly.
\end{proof}

So we can extend the  functor $\pi_\lambda^\theta(v)$ to a t-exact functor
$D^b_{G,\lambda}(D_R\operatorname{-Mod})\twoheadrightarrow D^b(\A_\lambda^\theta(v)\operatorname{-Mod})$.
Assuming the former is a quotient functor, so is the latter.

Let us consider a derived version  of  $\pi^0_\lambda(v)$. This functor
extends to $D^b(D(R)\operatorname{-mod}^{G,\lambda})\twoheadrightarrow D^b(\A_\lambda^0(v)\operatorname{-mod})$
and we have the derived left adjoint functor $L\pi_\lambda^0(v)^!:
D^b(\A_\lambda^0(v)\operatorname{-mod})\rightarrow D^b(D(R)\operatorname{-mod}^{G,\lambda})$.

When $\lambda$ is Zariski generic, we can also lift $\pi^0_\lambda(v)$ to a quotient functor
$D^-_{G,\lambda}(D(R)\operatorname{-mod})\twoheadrightarrow D^-(\A_\lambda^0(v)\operatorname{-mod})$.
For this, we need the following proposition.

\begin{Prop}\label{Prop:der_Ham_quot_glob}
The following is true:
\begin{enumerate}
\item There is a Zariski open asymptotically generic
subset $\paramq^{ISO}\subset\paramq^{iso}$ such that
$\operatorname{Tor}^i_{U(\g)}(D(R),\C_\lambda)=0$ for all $i>0$ and $\lambda\in \paramq^{iso}$.
\item For $\lambda\in \paramq^{ISO}$, the functor $$\pi_\lambda^0(v):=\Hom_{D^b_{G,\lambda}(D(R)\operatorname{-mod})}(\mathcal{Q}_\lambda,\bullet)$$
maps  $M\in D(R)\operatorname{-mod}^{G,\lambda}$ to $H_0(M)^G$. It is a quotient functor $D^b_{G,\lambda}(D_R\operatorname{-mod})\rightarrow
D^b(\A_\lambda^0(v)\operatorname{-mod})$ with a left adjoint and right inverse
functor $L\pi_\lambda^0(v)^!$  given by $\mathcal{Q}_\lambda\otimes^L_{\A_\lambda^0(v)}\bullet$.
\end{enumerate}
\end{Prop}
\begin{proof}
(1) will be proved below, see Section \ref{SS_der_Ham_red}. (2) follows from (1) and Lemma \ref{Lem:der_equiv_comput}
applied to the trivial $G$-module $V$.
\end{proof}

\subsubsection{Global section and localization functors}
We write $R\Gamma_\lambda^\theta:D^b(\A_\lambda^\theta(v)\operatorname{-mod})\rightarrow D^b(\A_\lambda(v)\operatorname{-mod})$ for the derived global sections functor, this makes
sense by (2) of Lemma \ref{Lem:coh_derived_cat}. The functor extends also to bounded above
derived category.  There it has a  left adjoint, the derived localization functor, $$L\Loc_\lambda^\theta:\A_\lambda^\theta(v)\otimes^L_{\A_\lambda(v)}\bullet:
D^-(\A_\lambda(v)\operatorname{-mod})\rightarrow D^-(\A_\lambda^\theta(v)\operatorname{-mod}).$$
We will also consider the abelian versions of these functors: $\Gamma_\lambda^\theta$ and its
left adjoint $\Loc_\lambda^\theta$.

\begin{Lem}\label{Lem:der_glob_descr}
Assume that $\lambda\in \paramq^{ISO}$. Then
$L\Loc_\lambda^\theta=\pi_\lambda^\theta(v)\circ L\pi^0_\lambda(v)^!$.
\end{Lem}
\begin{proof}
The functor $L\pi^0_\lambda(v)^!$ is the derived tensor product with $\mathcal{Q}_\lambda$. The
functor $\pi^\theta_\lambda(v)$ is the composition of  three functors, $\pi^\theta_{\lambda}(v)=\pi_3\circ\pi_2\circ\pi_1$,
 where the functors $\pi_1,\pi_2,\pi_3$ are as follows. First, we have the quotient
functor $$\pi_1:D^-_{G,\lambda}(D(R)\operatorname{-mod})\twoheadrightarrow
D^-_{G,\lambda}(D(R)\operatorname{-mod})/D^-_{G,\lambda}(D(R)\operatorname{-mod})_{\theta-uns}.$$
Second, we have the identification
$$\pi_2:D^-_{G,\lambda}(D_R\operatorname{-mod})/D^-_{G,\lambda,\theta-uns}(D_R\operatorname{-mod})
\xrightarrow{\sim} D^-(D_R\operatorname{-mod}^{G,\lambda})/
D^-_{\theta-uns}(D_R\operatorname{-mod}^{G,\lambda}),$$
see Lemma \ref{Lem:equiv_der_fun}. Third,  we have the equivalence
$$\pi_3:D^-(D_R\operatorname{-mod}^{G,\lambda})/
D^-_{\theta-uns}(D_R\operatorname{-mod}^{G,\lambda})\xrightarrow{\sim}
D^-(\A_\lambda^\theta(v)\operatorname{-mod})$$
that is realized by taking $G$-invariants. The  functor
$$\pi_2\circ \pi_1\circ L\pi^0_\lambda(v)^!:D^-(\A_\lambda(v)\operatorname{-mod})\rightarrow D^-(D_R\operatorname{-mod}^{G,\lambda})/
D^-_{\theta-uns}(D_R\operatorname{-mod}^{G,\lambda})$$
is isomorphic to $\pi_3^{-1}(\mathcal{Q}_\lambda\otimes^L_{\A_\lambda(v)}\bullet)$.
From here we deduce that
$$\pi_\lambda^\theta(v)\circ  L\pi^0_\lambda(v)^!=
[\mathcal{Q}_\lambda|_{T^*R^{\theta-ss}}\otimes_{\A_\lambda^0(v)}(\bullet)]^G=
\A_\lambda^\theta(v)\otimes^L_{\A_\lambda^0(v)}\bullet.$$
But the functor $L\Loc_\lambda^\theta$ is $\A_\lambda^\theta(v)\otimes^L_{\A_\lambda^0(v)}\bullet$, by its definition.
\end{proof}

Let $D^b_Y(\A_\lambda(v)\operatorname{-mod})\subset D^b(\A_\lambda(v)\operatorname{-mod}), D^b_{\rho^{-1}(Y)}(\A^\theta_\lambda(v)\operatorname{-mod})\subset D^b(\A^\theta_\lambda(v)\operatorname{-mod})$ denote the full subcategories consisting of all objects whose homology have support contained in $Y,\rho^{-1}(Y)$, respectively.

\begin{Lem}\label{Lem:support_preservation}
The functor $R\Gamma_\lambda^\theta$ maps $D^b_{\rho^{-1}(Y)}(\A^\theta_\lambda(v))$ to
$D^b_Y(\A_\lambda(v)\operatorname{-mod})$, while $L\Loc^\theta_\lambda$ sends
$D^-_Y(\A_\lambda(v)$ to $D^-_{\rho^{-1}(Y)}(\A^\theta_\lambda(v))$.
\end{Lem}
\begin{proof}
Note that for an open affine $U\subset \M(v)$ the microlocalization functors to $U$ and $\rho^{-1}(U)$
intertwine the functor $L\operatorname{Loc}_\lambda^\theta$ with its
counterpart for the restriction of $\rho$ to $\rho^{-1}(U)$. So if $M\in \A_\lambda(v)\operatorname{-mod}$
is supported away from $U$, then $L\Loc_\lambda^\theta(M)$ is supported away from $\rho^{-1}(U)$.

To prove the corresponding statement for $R\Gamma_\lambda^\theta$, we need to use that
$\rho$ is proper so that $R\rho_*$ maps coherent sheaves to coherent ones.
Using this it is easy to prove that $R\Gamma$ maps sheaves supported away from
$\rho^{-1}(U)$ to complexes with homology supported away from $U$.
\end{proof}

In particular, $R\Gamma^\theta_\lambda$ maps $D^b_{\rho^{-1}(0)}(\A_\lambda^\theta(v)\operatorname{-mod})$
to $D^b_0(\A_{\lambda}(v)\operatorname{-mod})$, while $L\Loc_\lambda^\theta$ maps
$D^-_0(\A_{\lambda}(v)\operatorname{-mod})$ to $D^-_{\rho^{-1}(0)}(\A_\lambda^\theta(v)\operatorname{-mod})$.
In the case when $R\Gamma_\lambda^\theta$ is an equivalence
$$D^b(\A_\lambda^\theta(v)\operatorname{-mod})\xrightarrow{\sim}
D^b(\A_\lambda^\theta(v)\operatorname{-mod}),$$
it restricts to an equivalence
$$D^b_{\rho^{-1}(0)}(\A_\lambda^\theta\operatorname{-mod})\xrightarrow{\sim}
D^b_0(\A_\lambda^\theta\operatorname{-mod})$$
with quasi-inverse $L\Loc_\lambda^\theta$.

In what follows, we will write $D^b_{fin}$ instead of $D^b_0$.


\section{Harish-Chandra bimodules and restriction functors}\label{S_HC}
Harish-Chandra (shortly, HC) bimodules and restriction functors between the categories of HC bimodules
play a crucial role in this paper. In this section we review a definition and basic properties of these bimodules
(Section \ref{SS_HC},\ref{SS_HC_fam}). In the remaining sections, we
construct restriction functors for HC bimodules over quantized quiver varieties,
study their basic properties and provide some applications. In particular, we prove
Proposition \ref{Prop:alg_iso} and part (1) of Proposition \ref{Prop:der_Ham_quot_glob}.

\subsection{Harish-Chandra bimodules}\label{SS_HC}
Let us start with a general definition of a Harish-Chandra bimodule, compare to \cite{HC,Ginzburg_HC,sraco,BPW}.
Let $\A=\bigcup_{i\leqslant 0}\A^{\leqslant i},\A'=\bigcup_{i=0}\A'^{\leqslant i}$ be
$\Z_{\geqslant 0}$-filtered algebras such that
the algebras $\gr\A,\gr\A'$ are identified with
graded Poisson quotients of the same finitely generated commutative graded Poisson algebra $A$.
Below we will always consider graded Poisson algebras, where the bracket has degree $-1$.

We will take $\A=\A_{\lambda}^0(v), \A'=\A_{\lambda'}^0(v)$
or sometimes  $\A=\A_\lambda(v),\A'=\A_{\lambda'}(v)$ (the filtration on $\A$ is induced from the differential operator filtration on $D(R)$). In the first case, we take $A:=\C[\M_0^0(v)]$ (where we consider $\M^0_0(v)$
with its natural scheme structure),
in the second case put $A:=\C[\M(v)]$ so that $\gr\A=\gr\A'=A$.

\subsubsection{Definition}\label{SSS_HC_def}
By a Harish-Chandra (HC) $\A'$-$\A$-bimodule we mean a bimodule $\B$ that can be equipped with  a
bimodule $\Z$-filtration bounded from below, $\B=\bigcup_i \B^{\leqslant i}$, such that
$\gr\B$ is a finitely generated $A$-module (meaning, in particular, that the left and the
right actions of $A$ coincide). Such a filtration on $\B$ is called {\it good}. We remark
that every HC bimodule is finitely generated
both as a left $\A'$-module and as a right $\A$-module. We also remark that, although $\gr\B$ does depend on the choice of
a filtration on $\B$, the support of $\gr\B$ in $\operatorname{Spec}(A)$ depends only on $\B$,
this support is called the {\it associated variety} of $\B$ and is denoted by $\VA(\B)$.
We remark that $\VA(\B)$ is always a Poisson subvariety of $\operatorname{Spec}(A)$.

By a homomorphism of HC bimodules we mean a bimodule homomorphism.
Given a homomorphism $\varphi:\B\rightarrow \B'$ we can find good
filtrations $\B=\bigcup_i \B^{\leqslant i}$ and $\B'=\bigcup_i \B'^{\leqslant i}$
with $\varphi(\B^{\leqslant i})\subset \B'^{\leqslant i}$ for all $i$.
Indeed, if $\gr\B$ is generated by homogeneous elements of degree up to $d$
then we can use any good filtration on $\B'$ such that $\varphi(\B^{\leqslant i})
\subset \B'^{\leqslant i}$ for $i\leqslant d$.

For example, both $\A_\lambda^0(v),\A_\lambda(v)$ are HC $\A_\lambda^0(v)$-bimodules.
It follows that any HC $\A_{\lambda'}(v)$-$\A_\lambda(v)$-bimodule is HC
also when viewed as a $\A_{\lambda'}^0(v)$-$\A^0_\lambda(v)$-bimodule.

\subsubsection{Rees construction}
Starting from $\A$, we can form the Rees algebra $\A_{\hbar}:=\bigoplus_{i}\A^{\leqslant i}\hbar^{i}$
that is graded with $\deg \hbar=1$.

We can introduce a notion of a Harish-Chandra $\A'_\hbar$-$\A_\hbar$-bimodule: those are finitely generated graded
$\A'_\hbar$-$\A_\hbar$-bimodules $\B_\hbar$ with $a'm-ma\subset \hbar \B_\hbar$ (for $a,a'$ such that
$a+\hbar\A_{\hbar}, a'+\hbar\A'_{\hbar}$ are the images of a single element $\tilde{a}\in A$)
that are free over $\C[\hbar]$. To pass from HC $\A_\hbar$-bimodules to HC $\A$-bimodules with a fixed good filtration, one mods out $\hbar-1$. To get back, one takes the Rees bimodule.

\subsubsection{Derived categories}
 Consider the  category $D^-_{HC}(\A'\operatorname{-}\A\operatorname{-bimod})$ consisting of all bounded above complexes of  $\A'$-$\A$-bimodules whose homology are Harish-Chandra. Similarly to \cite[Proposition 6.3]{BPW},
the subcategories $D^-_{HC}(\ldots)\subset D^-(\ldots)$ are closed with respect to $$\otimes^L_{\A'}: D^-(\A''\operatorname{-}\A'\operatorname{-bimod})\times D^-(\A'\operatorname{-}\A\operatorname{-bimod})\rightarrow
D^-(\A''\operatorname{-}\A\operatorname{-bimod}).$$ The same argument implies that
$R\operatorname{Hom}_{\A}$ sends $D^-_{HC}(\A\operatorname{-}\A'\operatorname{-bimod})\times
D^+_{HC}(\A\operatorname{-}\A''\operatorname{-bimod})$ to $D^+_{HC}(\A'\operatorname{-}\A''\operatorname{-bimod})$.

\subsubsection{Translation bimodules}\label{SSS_HC_transl}
Let us provide two closely related examples of HC bimodules over the algebras $\A_?(v),\A^0_{?}(v)$:
translation bimodules.

Recall the $D(R)$-$\A^0_\lambda(v)$-bimodule $\mathcal{Q}_\lambda$ from \ref{SSS_abel_red_fun}. Pick $\chi\in \Z^{Q_0}$.
We can consider the $\A^0_{\lambda+\chi}(v)$-$\A^0_\lambda(v)$ bimodule $\A^0_{\lambda,\chi}(v)=\mathcal{Q}_\lambda^{G,\chi}$. This bimodule is HC, the filtration on
$\A^0_{\lambda,\chi}(v)$ induced from the filtration on $D(R)$ by the order of a differential operator is good.

Now consider the restriction $\mathcal{Q}_\lambda|_{T^*R^{\theta-ss}}$ and set
$\A_{\lambda,\chi}^\theta(v):=[\mathcal{Q}_{\lambda}|_{T^*R^{\theta-ss}}]^{G,\chi}$.
This is a sheaf on $\M^\theta(v)$ that is an  $\A_{\lambda+\chi}^\theta(v)$-$\A_\lambda^\theta(v)$-bimodule.
Set $\A^{(\theta)}_{\lambda,\chi}(v):=\Gamma(\A^\theta_{\lambda,\chi}(v))$. That it is HC was demonstrated in
\cite[Section 6.3]{BPW} but we want to sketch a proof.
Namely, notice that $\gr\A_{\lambda,\chi}^\theta(v)=\mathcal{O}(\chi)$.
Consider the Rees bimodule $\A_{\lambda,\chi}^\theta(v)_\hbar$ that is a deformation
of $\mathcal{O}(\chi)$. Then $\Gamma(\A_{\lambda,\chi}^\theta(v)_\hbar)$ is the Rees bimodule
for $\A^{(\theta)}_{\lambda,\chi}(v)$. But $\Gamma(\A_{\lambda,\chi}^\theta(v)_\hbar)/(\hbar)$ embeds into
$\Gamma(\mathcal{O}(\chi))$, the latter  is a  $\C[\M^\theta(v)]$-module
rather than just a bimodule. This completes the proof.

We have a natural bimodule homomorphism \begin{equation}\label{eq:hom_0_to_theta}\A^0_{\lambda,\chi}(v)\rightarrow
\A_{\lambda,\chi}^{(\theta)}(v)\end{equation} induced by the restriction map $\mathcal{Q}_\lambda
\rightarrow \mathcal{Q}_\lambda|_{T^*R^{\theta-ss}}$.
A priori, (\ref{eq:hom_0_to_theta}) is neither injective, not surjective. In Section \ref{S_loc}
we will get some sufficient conditions for (\ref{eq:hom_0_to_theta}) to be an isomorphism.

\subsubsection{Further properties}
Finally, we need some results from  \cite{B_ineq}.
The next lemma follows from Theorems 1.2, 1.3 or Section 4.3 there.

\begin{Lem}\label{Lem:fin_length}
Every HC $\A_{\lambda'}(v)$-$\A_\lambda(v)$-bimodule has finite length.
\end{Lem}

The following claim is \cite[Lemma 4.2]{B_ineq}.

\begin{Lem}\label{Cor:assoc_var}
Let $\mathcal{B}$ be a HC $\A_{\lambda'}(v)$-$\A_\lambda(v)$ bimodule and $\J_\ell,\J_r$ be its left and right annihilators.
Then $\VA(\mathcal{B})=\VA(\A_{\lambda'}(v)/\J_\ell)=\VA(\A_{\lambda}(v)/\J_r)$.
\end{Lem}

\subsection{Families of Harish-Chandra bimodules}\label{SS_HC_fam}
Recall from \ref{SSS_fam_quiv} that we have the scheme $\M^\theta_{\param}(v)=\mu^{-1}(\g^{*G})^{\theta-ss}/G$ over $\param$. In Section \ref{SS_quant_prop} we have introduced  the sheaf of $\C[\paramq]$-algebras $$\A^\theta_{\paramq}(v):=[\mathcal{Q}_{\paramq}|_{T^*R^{\theta-ss}}]^G,$$ where we write
$\mathcal{Q}_{\paramq}$ for $D(R)/D(R)\Phi([\g,\g])$,  on $\M^\theta_{\param}(v)$,
and the $\C[\paramq]$-algebra $\A_{\paramq}(v)=\Gamma(\A^\theta_{\paramq}(v))$.
We also consider the  global Hamiltonian reduction  $\A^0_{\paramq}(v):=[\mathcal{Q}_{\paramq}]^G$.
Also, for a vector subspace $\param_0\subset\param$, we can
consider the specialization $\M_{\param_0}^\theta(v)$ and, for an affine subspace $\paramq_0\subset\paramq$,
we consider the specializations $\A^\theta_{\paramq_0}(v),\A_{\paramq_0}(v),\A^0_{\paramq_0}(v)$.

The algebra $\A_{\paramq_0}(v)$ is filtered with commutative associated graded (equal to $\C[\M_{\param_0}(v)]$,
where $\param_0$ is the vector subspace of $\param$ parallel to $\paramq_0$). The algebra $\A^0_{\paramq_0}(v)$
is filtered as well with $\C[\M^0_{\param_0}(v)]\twoheadrightarrow \gr\A^0_{\paramq_0}(v)$.
So it makes sense to speak about HC $\A_{\paramq_0}(v)$-bimodules or HC  $\A^0_{\paramq_0}(v)$-bimodules.
Also for two parallel affine subspaces $\paramq_0,\paramq_0'$ one can speak
about HC $\A_{\paramq_0'}(v)$-$\A_{\paramq_0}(v)$  bimodules
or about HC $\A^0_{\paramq_0'}(v)$-$\A^0_{\paramq_0}(v)$-bimodules.

For $\A^0_{\paramq_0'}(v)$-$\A^0_{\paramq_0}(v)$-bimodules we can still consider the corresponding derived category
of all complexes with HC homology. These categories are  closed under derived tensor products
or under $R\Hom$'s of left or right modules.
The proofs are as for $\A^0_{\lambda'}(v)$-$\A^0_\lambda(v)$-bimodules.

\subsubsection{Translation bimodules}
For example, we have the  HC
$\A^0_{\paramq_0'}(v)$-$\A^0_{\paramq_0}(v)$-bimodule $\A^0_{\paramq_0,\chi}(v)$
(where $\paramq_0'=\chi+\paramq_0$). This is the most important family of HC bimodules considered in this paper.
Obviously, the specialization of $\A^0_{\paramq_0,\chi}(v)$ to $\lambda\in \paramq_0$ coincides with $\A_{\lambda,\chi}^0(v)$.

Yet another family that we will need for technical reasons is $\A^{(\theta)}_{\paramq_0,\chi}(v)$ defined analogously
to $\A^{(\theta)}_{\lambda,\chi}(v)$. This is a HC  $\A_{\paramq_0+\chi}(v)$-$\A_{\paramq_0}(v)$-bimodule
and hence also  a HC $\A^0_{\paramq_0+\chi}(v)$-$\A^0_{\paramq_0}(v)$-bimodule.
An important result here is as follows, \cite[Proposition 6.23]{BPW}.

\begin{Prop}\label{Prop:univ_wc}
The $\A_{\paramq}(v)$-bimodule $\A^{(\theta)}_{\paramq,\chi}(v)$ is independent of $\theta$.
\end{Prop}
Let us provide the proof for readers convenience, since it is omitted in \cite{BPW}.
\begin{proof}
Recall the variety $\M_{\param}(v)$ introduced after Corollary \ref{Cor:prop_Mv} that comes
with a projective morphism $\M^\theta_{\param}(v)\rightarrow \M_{\param}(v)$. This morphism is Poisson
by construction and is a fiberwise resolution of singularities.
Let us write $\M_{\param}(v)^{reg}$ for the union of open symplectic leaves in
the fibers of $\M_{\param}(v)\rightarrow \param$. This is an open subvariety of $\M_{\param}(v)$
because the dimension of a symplectic leaf is an upper-semicontinuous function. Let
$\M^\theta_{\param}(v)^{reg}$ denote its preimage in $\M^\theta_{\param}(v)$.
Then $\M^\theta_{\param}(v)^{reg}\xrightarrow{\sim} \M_{\param}(v)^{reg}$ is an isomorphism.
In particular, the varieties $\M^\theta_{\param}(v)^{reg}$ are identified for different
generic $\theta$.

The sheaf $\A^{\theta}_{\paramq,\chi}|_{\M_\param(v)^{reg}}$ is a quantization of $\mathcal{O}_{\param}(\chi)|_{\M_\param(v)^{reg}}$.

The 1st cohomology of the structure sheaf of $\M_{\param}(v)^{reg}$ vanish. This is
because $H^1(\M_{\param}^\theta(v),\mathcal{O})=0$, $\M_{\param}(v)$ is Cohen-Macaulay, and the complement
of $\M_\param(v)^{reg}$ in $\M_\param(v)$ has codimension $3$, compare to the proof of Proposition
of \cite[Proposition 3.7]{BPW}. It follows that there is a unique microlocal deformation
of $\mathcal{O}(\chi)|_{\M_\param(v)^{reg}}$ to an
$\A^\theta_{\paramq,\chi}(v)|_{\M_\param(v)^{reg}}$-bimodule.
So the restrictions of all $\A^\theta_{\paramq,\chi}(v)$ to $\M_\param(v)^{reg}$ coincide. Since the codimension of $\M^\theta_\param(v)\setminus \M^\theta_\param(v)^{reg}$ is bigger
than $2$, we have $\Gamma(\M^\theta_\param(v)^{reg}, \A^\theta_{\paramq,\chi}(v))=\Gamma(\M^\theta_\param(v), \A^\theta_{\paramq,\chi}(v))$.
The left hand side is independent of $\theta$ and so we get the claim of the proposition.
\end{proof}

We would like to point out that the specialization $\A^{(\theta)}_{\paramq,\chi}(v)_\lambda$ admits a natural
homomorphism to $\A^{(\theta)}_{\lambda,\chi}(v)$. This homomorphism is injective because $\Gamma$ is left exact.
We do not know if this is an isomorphism in general, but this is so under additional assumptions:

\begin{Lem}\label{Prop:trans_spec}
Let $\paramq_0\subset \paramq$ be an affine subspace and pick $\chi\in \Z^{Q_0}$.  Suppose that one of the following conditions holds:
\begin{enumerate}
\item $H^1(\M^{\theta}(v),\mathcal{O}(\chi))=0$.
\item $(\lambda+\chi,\theta)\in \mathfrak{AL}(v)$.
\end{enumerate}
Then $\A^{(\theta)}_{\lambda,\chi}(v)=\A^{(\theta)}_{\paramq_0,\chi}(v)_\lambda$.
\end{Lem}
\begin{proof}
We can view $\A^\theta_{\lambda,\chi}(v)$ as a sheaf on $\M^\theta(v)$ quantizing $\mathcal{O}(\chi)$.
Let us show that both our assumptions imply that $H^1(\A^\theta_{\lambda,\chi}(v))=0$.
Then we can apply \cite[Proposition 6.26]{BPW} to deduce $\A^{(\theta)}_{\lambda,\chi}(v)=\A^{(\theta)}_{\paramq_0,\chi}(v)_\lambda$.

The filtration on $\A^\theta_{\lambda,\chi}(v)$ induces a separated filtration
on $H^1(\M^\theta(v),\A^\theta_{\lambda,\chi}(v))$ (the claim that the filtration
is separated is proved similarly to the proof of \cite[Lemma 5.6.3]{GL}) with
$H^1(\M^\theta(v), \mathcal{O}(\chi))\twoheadrightarrow
\gr H^1(\M^\theta(v),\A^\theta_{\lambda,\chi}(v))$. So the equality
$H^1(\M^{\theta}(v),\mathcal{O}(\chi))=0$ implies  $H^1(\M^\theta(v),\A^\theta_{\lambda,\chi}(v))=0$.

Now assume that $(\lambda+\chi,\theta)\in \AL(v)$. Then any object in $\A^\theta_\lambda(v)\operatorname{-mod}$
has no higher cohomology, and we are done.
\end{proof}

\subsubsection{Supports in parameters}\label{SSS_param_supp}
We also have the following elementary but important property. By the right  $\paramq$-support of an
$\A_{\paramq}(v)$-bimodule $\mathcal{B}$ (denoted by $\Supp_{\paramq}^r(\mathcal{B})$),
we mean the set of all $\lambda\in \paramq$ such that the specialization $\mathcal{B}_\lambda$
is nonzero. Analogously, we can speak about the left support $\Supp_{\paramq}^\ell(\mathcal{B})$.

\begin{Lem}\label{Lem:HC_gen_flat}
For a reduced closed subscheme $Y$ of $\paramq$, set $\A^0_Y(v):=\C[Y]\otimes_{\C[\paramq]}\A^0_{\paramq}(v)$.
Any finitely generated right $\A^0_{Y}(v)$-module $\mathcal{B}$ is generically free over $\C[Y]$,
i.e., there is a non zero divisor $f\in \C[Y]$ such that the localization $\mathcal{B}_f$ is a free $\C[Y]_f$-module.
\end{Lem}
\begin{proof}
We will need to modify a filtration on $\A^0_Y(v)$ so that $\C[Y]$
lives in degree $0$. Consider the Rees algebra $\A^0_{\paramq}(v)_\hbar$ and its base change $\tilde{\A}^0_{\paramq}(v)_\hbar=
\C[\paramq,\hbar]\otimes_{\C[\paramq,\hbar]}\A^0_{\paramq}(v)_\hbar$, where  the endomorphism
of $\C[\paramq,\hbar]$ used to form the tensor product is given by $\hbar\mapsto \hbar, \alpha\mapsto \alpha\hbar$ for $\alpha\in \paramq^*$
(here we consider $\paramq=\C^{Q_0}$ as a vector space, not as an affine space).
The algebra $\tilde{\A}^0_{\paramq}(v)_{\hbar}$ is graded with $\deg\hbar=1, \deg\C[\paramq]=0$. Also the specializations
of $\tilde{\A}^0_{\paramq}(v)_\hbar,\A^0_{\paramq}(v)_\hbar$ at $\hbar=1$ are the same and so coincide with
$\A^0_{\paramq}(v)$. We equip $\A^0_{\paramq}(v)$ with the filtration coming from the grading on $\tilde{\A}^0_{\paramq}(v)_\hbar$
and we equip the quotient $\A^0_Y(v)$ of $\A^0_{\paramq}(v)$ with the induced filtration.
We remark that $\gr \A^0_{Y}(v)$ is now a quotient of  $\C[\M^0(v)]\otimes \C[Y]$.

A finitely generated right module $\mathcal{B}$ admits a good filtration. By a general commutative algebra result, \cite[Theorem 14.4]{Eisenbud}, $\gr\mathcal{B}$ is generically free over $\C[Y]$. So there is a non zero divisor $f$ such that $(\gr\mathcal{B})_f$ is free over $\C[Y]_f$. It follows that $\mathcal{B}_f$ and $(\gr\mathcal{B})_f$ are isomorphic free $\C[Y]_f$-modules, and we are done.
\end{proof}

There is a trivial but very important corollary of this lemma.

\begin{Cor}\label{Cor:HC_supp}
Let $\mathcal{B}$ be a Harish-Chandra $\A^0_{\paramq_0'}(v)$-$\A^0_{\paramq_0}(v)$-bimodule. Then the
following claims hold:
\begin{enumerate}
\item There is $f\in \C[\paramq_0]$ such that $\mathcal{B}_f$ is a free
$\C[\paramq_0]_f$-module.
\item $\Supp^r_{\paramq_0}(\B)$  is a constructible set.
\end{enumerate}
We also have left-handed analogs of these claims.
\end{Cor}
\begin{proof}
A  Harish-Chandra bimodule is finitely generated as a right $\A_{\paramq_0}(v)$-module (this was noted in the beginning
of Section \ref{SS_HC}). So (1) follows from Lemma \ref{Lem:HC_gen_flat}.

To prove (2) we note that the support of any finitely generated right $\A_Y(v)$-module is a constructible subset of
$Y$. This follows from Lemma \ref{Lem:HC_gen_flat}.
\end{proof}

Below, Proposition \ref{Prop:HC_support}, we will see that
$\Supp^r_{\paramq_0}(\B)$ is actually a closed subset.

Here is how we are going to use (2). Let $\B$ be a HC $\A_{\paramq_0+\chi}$-$\A_{\paramq_0}$-bimodule,
where $\paramq_0\subset \paramq$ is an affine subspace. Then if $\B_\lambda=0$ for
a Weil generic $\lambda\in \paramq_0$ (we say that a parameter is Weil generic
if it lies outside of the countable union of algebraic subvarieties), then $\B_\lambda=0$ for a Zariski generic $\lambda$
as well.  HC $\A_{\lambda+\chi}(v)$-$\A_{\lambda}(v)$-bimodules
for $\lambda$ Weil generic are easier then for an arbitrary (even Zariski generic) $\lambda$. We will use this observation many times in our discussion of short wall-crossing functors through the affine wall, Section \ref{S_affine}.

\subsection{Restriction functors: construction}\label{SS_compl}
We want to define restriction functors for Harish-Chandra bimodules over $\A^0_{\paramq}(v)$
(or over $\A_{\paramq}(v)$) similar to the functors $\bullet_\dagger$ used in \cite{HC,sraco}.
Those will be exact $\C[\paramq]$-linear functors mapping HC bimodules over $\A^0_{\paramq}(v)$
to those over $\hat{\A}^0_{\paramq}(\hat{v})$, an algebra defined similarly to $\A^0_{\paramq}(v)$
 but for the quiver $\hat{Q}$ and vectors $\hat{v},\hat{w}$ that were constructed in Section \ref{SSS_class_slice}
(in fact, we will sometimes need to modify the algebras $\hat{\A}^0_{\paramq}(\hat{v})$,
see below).



\subsubsection{Algebras $\hat{\A}^0_{\paramq}(\hat{v})$, etc.}\label{SSS_restr_alg}
Let us proceed to the construction of $\hat{\A}^0_{\paramq}(\hat{v})$.
Let $\hat{\paramq}=\hat{\g}^{\hat{G}*}$ be the parameter space for
the quantizations associated to $(\hat{Q},\hat{v},\hat{w})$.
Let us define an affine map $\hat{r}:\paramq\rightarrow \hat{\paramq}$ whose differential is the restriction map $r:\g^{G*}\rightarrow \hat{\g}^{\hat{G}*}$. Namely, recall that we have elements
$\varrho(v), \hat{\varrho}(\hat{v})$ (the former is defined by (\ref{eq:rho}) and the latter is defined
analogously). Now set \begin{equation}\label{eq:quant_restr_map}\hat{r}(\lambda):=r(\lambda-\varrho(v))+\hat{\varrho}(\hat{v}).\end{equation}
Further, set $\hat{\A}^0_{\paramq}(\hat{v}):=
\C[\paramq]\otimes_{\C[\hat{\paramq}]}\hat{\A}^0_{\hat{\paramq}}(\hat{v})$
and define $\hat{\A}^\theta_{\paramq}(\hat{v})$ in a similar way. Here $$\hat{\A}^0_{\hat{\paramq}}(\hat{v}):=[D(\hat{R})/D(\hat{R})
\Phi([\hat{\g},\hat{\g}])]^{\hat{G}},$$
where $\hat{R}=R(\hat{Q},\hat{v},\hat{w})$.

We want to get   decompositions similar to (\ref{eq:affine_decomp}),(\ref{eq:smooth_decomp})
at the quantum level. For this, we consider the Rees sheaves and algebras $\A^{\theta}_{\paramq}(v)_\hbar, \A_{\paramq}(v)_\hbar,
\A^0_{\paramq}(v)_\hbar$ defined for the filtrations by the order of a differential operator.
We can complete those at $x$ getting the algebras
$\A_{\paramq}(v)_\hbar^{\wedge_x},
\A^0_{\paramq}(v)_\hbar^{\wedge_x}$ with $\A_{\paramq}(v)_\hbar^{\wedge_x}/(\hbar)=\C[\M_{\param}(v)^{\wedge_x}],
\C[\M^0_{\param}(v)^{\wedge_x}]\twoheadrightarrow \A^0_{\paramq}(v)_\hbar^{\wedge_x}/(\hbar)$
and the sheaf of algebras
$\A^{\theta}_{\paramq}(v)_\hbar^{\wedge_x}$ on
$\M^{\theta}_{\param}(v)^{\wedge_x}$ obtained by the $\hbar$-adic completion of
$$\A^0_{\paramq}(v)_{\hbar}^{\wedge_x}\otimes_{\A^0_{\paramq}(v)_\hbar}\A^\theta_{\paramq}(v)_\hbar$$
Note that $\A^{\theta}_{\paramq}(v)_\hbar^{\wedge_x}/(\hbar)=\mathcal{O}_{\M^{\theta}_{\param}(v)^{\wedge_x}}$.

\begin{Lem}\label{Lem:quant_decomp}
We have the following decompositions.
\begin{align}\label{eq:alg_decomp0}
&\A^0_{\paramq}(v)_\hbar^{\wedge_x}=
\hat{\A}^0_{\paramq}(v)_\hbar^{\wedge_0}\widehat{\otimes}_{\C[[\hbar]]}\Weyl_\hbar^{\wedge_0},\\\label{eq:alg_decomp}
&\A_{\paramq}(v)_\hbar^{\wedge_x}=
\hat{\A}_{\paramq}(v)_\hbar^{\wedge_0}\widehat{\otimes}_{\C[[\hbar]]}\Weyl_\hbar^{\wedge_0},\\\label{eq:sh_decomp}
&\A^\theta_{\paramq}(v)_\hbar^{\wedge_x}=
\left(\hat{\A}^\theta_{\paramq}(v)_\hbar\otimes_{\C[[\hbar]]}\Weyl_\hbar\right)^{\wedge_0}.
\end{align}
Isomorphisms (\ref{eq:alg_decomp0}),(\ref{eq:sh_decomp}) become (\ref{eq:affine_decomp}),(\ref{eq:smooth_decomp})
after setting $\hbar=0$. (\ref{eq:alg_decomp}) is obtained from (\ref{eq:sh_decomp})
by taking global sections.
\end{Lem}
By $\Weyl_\hbar$ we denote the homogenized Weyl algebra of $R_0$ and we write $\Weyl_\hbar^{\wedge_0}$ for the quantization of the symplectic formal polydisk $R_0^{\wedge_0}$.
\begin{proof}
The proof follows that of \cite[Lemma 6.5.2]{quant}. We provide it for reader's convenience.

Let $U$ denote the symplectic part of the slice module for $r$. Then, as we have mentioned
in \ref{SSS_class_slice},  \begin{equation}\label{eq:quant_slice} (T^*R)^{\wedge_{Gr}}\cong \left((T^*G\times U)\red G_r\right)^{\wedge_{G/G_r}},\end{equation} where $G_r$ acts diagonally on $T^*G\times U$. We can consider the quantization $D_\hbar(R)^{\wedge_{Gr}}$ of $(T^*R)^{\wedge_{Gr}}$ obtained by the completion of the homogenized Weyl algebra on $T^*R$. Also we can consider the quantization
$$[D_\hbar(G)^{\wedge_G}\widehat{\otimes}_{\C[[\hbar]]}\Weyl_\hbar(U)^{\wedge_0}]\red_0 G_r$$ of
$([T^*G\times U]\red G_r)^{\wedge_{G/G_r}}$ (where we use the symmetrized quantum comoment  map for $G_r$).
Those are canonical quantizations in the sense of \cite{BK} (for the second quantization
this follows from \cite[Section 5.4]{quant}) and so they are isomorphic. Consequently, their reductions
(both affine and GIT) for the $G$-action (again, with respect to the symmetrized quantum comoment map $\Phi^{sym}$)
are isomorphic. But the reduction of the quantization of  the right hand side of (\ref{eq:quant_slice})
coincides with $$\C[[\param,\hbar]]\widehat{\otimes}_{\C[[\hat{\param},\hbar]]}\left[\Weyl_\hbar^{\wedge_0}(U)/\Weyl_\hbar^{\wedge_0}(U)
\hat{\Phi}^{sym}([\hat{\g},\hat{\g}])\right]^{\hat{G}}.$$ Since $\Phi-\varrho(v), \hat{\Phi}-\hat{\varrho}(\hat{v})$ are the symmetrized quantum comoment maps, (\ref{eq:alg_decomp0}) and (\ref{eq:sh_decomp}) follow. (\ref{eq:alg_decomp}) is obtained from (\ref{eq:sh_decomp}) via taking the global sections
on both sides.
\end{proof}

Let us observe that
\begin{equation}\label{eq:integr}
r(\varrho(v))-\hat{\varrho}(\hat{v})\in \Z^{\hat{Q}_0}.
\end{equation}
Indeed,  by reversing some arrows in $\hat{Q}^{\hat{w}}$ (the quiver obtained from $\hat{Q}$  by adjoining the new vertex $\infty$, see \ref{SSS_class_slice}), we can arrange (by reversing some arrows, perhaps, including
arrows coming from $\infty$) that $R, R_x\oplus \g/\g_r$
are isomorphic up to a trivial direct summand. Since $\g/\g_r$ is an orthogonal $G_r$-module, we see that $\bigwedge^{top}R\cong \bigwedge^{top}R_x$ as $G_r$-modules. Then we need to turn  $\infty$
back to a sink, so we have to reverse some arrows. Reversing an arrow in a quiver results in adding an integral character
to the quantum comoment map, and so (\ref{eq:integr}) follows.


\subsubsection{Euler derivations}
The sheaf $\A^\theta_{\paramq}(v)_\hbar$ comes with a $\C^\times$-action (that is induced now by the
fiberwise dilation action on $T^*R$) and hence with the Euler derivation
$\mathsf{eu}$ satisfying $\mathsf{eu}(\hbar)=\hbar$. This derivation extends to the completion $\A^\theta_{\paramq}(v)_\hbar^{\wedge_x}$.
On the other hand, the product $\hat{\A}^\theta_{\paramq}(\hat{v})_\hbar^{\wedge_0}\widehat{\otimes}_{\C[[\hbar]]}\Weyl_\hbar^{\wedge_0}$
comes with a $\C^\times$-action, and hence with the Euler derivation $\hat{\mathsf{eu}}$ again satisfying $\hat{\mathsf{eu}}(\hbar)=\hbar$.
We want to compare derivations $\mathsf{eu}$ and $\hat{\mathsf{eu}}$
of $\A_{\paramq}^\theta(v)_{\hbar}^{\wedge_x}$ (and similarly defined derivations
of $\A_{\paramq}^0(v)_{\hbar}^{\wedge_x}$).

\begin{Lem}\label{Lem:Euler_comp}
There is an element $a\in \A^0_{\paramq}(v)_\hbar^{\wedge_x}$ such that $\mathsf{eu}-\hat{\mathsf{eu}}=\frac{1}{\hbar}[a,\cdot]$
on $\A^0_{\paramq}(v)_\hbar^{\wedge_x}$ and on
$\A^\theta_{\paramq}(v)_\hbar^{\wedge_x}$.
\end{Lem}
\begin{proof}
Consider a more general setting.  Let $R$ be a  vector space, $G$ be a reductive group
acting on $R$, $v\in \mu^{-1}(0)\subset T^*R$ be a point such that $Gv$ is closed
and $G_v$ is connected (for simplicity). Let $\Phi:\g\rightarrow D_\hbar(R)$
be a symmetrized quantum comoment map and let $\theta:G\rightarrow \C^\times$ be a character.
Consider the quantum Hamiltonian reduction $(D_\hbar(R)\red^\theta_\lambda G)^{\wedge_v}$
(there the completion is taken at the image $x$ of $v$ in $T^*R\red_0 G$).
Let $d$ be a $G$-invariant $\C[\hbar]$-linear derivation of
$D_\hbar(R)^{\wedge_{Gv}}$  such that $d\circ \Phi=0$ so that $d$
induces  derivations $d^\theta$ on $(D_\hbar(R)\red^\theta_\lambda G)^{\wedge_v}$
and $d^0$ on $(D_\hbar(R)\red^0_\lambda G)^{\wedge_v}$.
We claim that
\begin{itemize}
\item[(*)] there is an element $a\in (D_\hbar(R)\red^0_\lambda G)^{\wedge_v}$
such that $d^\theta=\frac{1}{\hbar}[a,\cdot]$ and $d^0=\frac{1}{\hbar}[a,\cdot]$.
\end{itemize}

To apply (*) in our situation, we take $d=\mathsf{Eu}-\hat{\mathsf{Eu}}$.
Here $\mathsf{Eu}$ is the derivation of $D_\hbar(R)^{\wedge_{Gv}}$ induced by the fiberwise
$\C^\times$-action on $T^*R$ and     $\hat{\mathsf{Eu}}$ is
the derivation induced by the fiberwise $\C^\times$-action on $T^*(G*_{G_v}R_x)$.

To prove (*) note that we can replace $G$ with a finite central extension  and assume that
$G=G_0\times T$, where $T$ is a torus and $G_0$ satisfies $G_0=(G_0,G_0)G_v$.
So  $D_\hbar(R)^{\wedge_{Gv}}=D_\hbar(T)^{\wedge_T}\widehat{\otimes}_{\C[[\hbar]]}
D_\hbar(Y)^{\wedge_{G_0y}}$, where $Y=G_0*_{G_r}R_x$ and $y$ is the point $[1,0]\in Y$.
The algebra $D_\hbar(Y)^{\wedge_{G_0y}}$ is the reduction of  $D_\hbar(R)^{\wedge_{Gv}}$
by the action of $T$ and so $d$ descends to $D_\hbar(Y)^{\wedge_{G_0y}}$.
Furthermore, $(D_\hbar(R)\red^\theta_\lambda G)^{\wedge_v}=D_\hbar(Y)^{\wedge_{G_0y}}\red_\lambda^\theta G_0$.
Let us note that $H^1_{DR}(T^*Y)=0$ because of the assumption $G_0=(G_0,G_0)G_v$.
Modulo $\hbar$, the derivation $d$  is a symplectic vector field
on the formal neighborhood of $G_0y$ in $T^*Y$. So it is Hamiltonian.
From here we deduce that $d=\frac{1}{\hbar}[\tilde{a},\cdot]$ for some
element $\tilde{a}\in D_\hbar(Y)^{\wedge_{G_0y}}$. This element commutes
with $\Phi(\g_0)$ and hence is $G_0$-invariant. For $a$ we take its
image in   $(D_\hbar(R)\red^0_\lambda G)^{\wedge_x}$. It is straightforward
to see that this element satisfies (*).
\end{proof}


\subsubsection{Construction of $\bullet_{\dagger,x}$}
Let us now proceed to constructing $\bullet_{\dagger,x}:\HC(\A_{\paramq}(v))\rightarrow
\HC(\hat{\A}_{\paramq}(\hat{v}))$.
Define the category $\HC(\A_{\paramq}(v)_\hbar^{\wedge_x})$ as the category of $\A_{\paramq}(v)_\hbar^{\wedge_x}$-bimodules $\B'_\hbar$ that are
\begin{itemize}
\item finitely generated as bimodules,
\item flat over $\C[[\hbar]]$ and complete and separated in the $\hbar$-adic topology,
\item satisfy $[a,b]\in \hbar \B'_\hbar$
for all $a\in \A_{\paramq}(v)_{\hbar}^{\wedge_x},b\in \B'_\hbar$,
\item and come equipped with a derivation $\mathsf{Eu}$ compatible with
$\mathsf{eu}$ on $\A_{\paramq}(v)_\hbar^{\wedge_x}$.
\end{itemize}

Similarly, we can define the category $\HC(\hat{\A}_{\paramq}(\hat{v})_\hbar^{\wedge_0})$
(we need to have a derivation compatible with $\hat{\mathsf{eu}}$). The categories
$\HC(\A_{\paramq}(v)_\hbar^{\wedge_x})$ and $\HC(\hat{\A}_{\paramq}(\hat{v})_\hbar^{\wedge_0})$
are equivalent as follows. Using the decomposition (\ref{eq:alg_decomp}), we view $\B_\hbar'\in
\HC(\A_{\paramq}(v)_\hbar^{\wedge_x})$ as a bimodule over $\hat{\A}_{\paramq}(\hat{v})_\hbar^{\wedge_0}\widehat{\otimes}_{\C[[\hbar]]}\Weyl_\hbar^{\wedge_0}$.
Similarly to \cite[Proposition 3.3.1]{HC}, this bimodule splits as $\hat{\B}_\hbar'\widehat{\otimes}_{\C[[\hbar]]}\Weyl_\hbar^{\wedge_0}$,
where $\hat{\B}_\hbar'$ is an $\hat{\A}_{\paramq}(\hat{v})^{\wedge_0}_\hbar$-bimodule.
The derivation $\hat{\mathsf{Eu}}:=\mathsf{Eu}-\frac{1}{\hbar}[a,\cdot]$ on $\B'_\hbar$ is compatible
with the derivation $\hat{\mathsf{eu}}$ on $\hat{\A}_{\paramq}(\hat{v})_\hbar^{\wedge_0}\widehat{\otimes}_{\C[[\hbar]]}\Weyl_\hbar^{\wedge_0}$
and so restricts to $\hat{\B}'_\hbar$ making it an object of $\HC(\hat{\A}_{\paramq}(v)_\hbar^{\wedge_0})$.
An equivalence $\HC(\A_{\paramq}(v)_\hbar^{\wedge_x})\xrightarrow{\sim}\HC(\hat{\A}_{\paramq}(\hat{v})_\hbar^{\wedge_0})$
we need maps $\B'_\hbar$ to $\hat{\B}_\hbar'$. A quasi-inverse equivalence sends
$\hat{\B}_\hbar'$ to $\hat{\B}'_\hbar\widehat{\otimes}_{\C[[\hbar]]}\Weyl_\hbar^{\wedge_0}$.

Now we construct the functor $\bullet_{\dagger}$.
Pick $\B\in \operatorname{HC}(\A_{\paramq}(v))$. Choose a good filtration on $\B$
and let $\B_\hbar\in \HC(\A_{\paramq}(v)_\hbar)$ be the Rees bimodule. So the completion $\B_\hbar^{\wedge_x}$ is an $\A_{\paramq}(v)_\hbar^{\wedge_x}$-bimodule. By the construction, $\B_\hbar$ comes with  the
derivation $\mathsf{Eu}:=\hbar \partial_{\hbar}$ compatible with the derivation $\mathsf{eu}$
on $\A_{\paramq}(v)_\hbar$. The derivation $\mathsf{Eu}$ extends to   $\B_\hbar^{\wedge_x}$
that makes the latter an object of $\HC(\A_{\paramq}(v)_\hbar^{\wedge_x})$. From this object
we get $\hat{\B}'_\hbar\in  \HC(\hat{\A}_{\paramq}(\hat{v})_\hbar^{\wedge_0})$.

By \cite[Proposition 3.3.1]{HC}, the $\hat{\mathsf{Eu}}$-finite part $\hat{\B}_\hbar$ is dense in $\hat{\B}_\hbar'$.
Since $\hat{\B}'_\hbar$ is a finitely generated bimodule over $\hat{\A}_\paramq(\hat{v})^{\wedge_0}_\hbar$,
and $\hat{\B}_\hbar$ is dense, we can choose generalized $\hat{\mathsf{Eu}}$-eigen-vectors for generators
of $\hat{\B}'_\hbar$. Now it is easy to see that $\hat{\B}_\hbar$ is finitely generated over $\hat{\A}_\paramq(\hat{v})_\hbar$. In  its turn, this implies that $\hat{\B}_\hbar$ can be made into a graded $\hat{\A}_\paramq(\hat{v})_\hbar$-bimodule.

We set $\B_{\dagger,x}:=\hat{\B}_\hbar/(\hbar-1)$, it is a Harish-Chandra
$\hat{\A}_{\paramq}(\hat{v})$-bimodule, a good filtration comes from the $\C^\times$-action
on $\hat{\B}_\hbar$. Similarly to \cite[Section 3.4]{HC}, we see that the assignment
$\B\rightarrow \B_{\dagger,x}$ is functorial.

Let us note that the functor is independent
(up to an isomorphism) of the choice of $a$ (which is defined uniquely up to a summand from
$\C[[\paramq,\hbar]]$). This is because the spaces of $\C^\times$-finite sections arising from
$a$ and $a+f$ with $f\in \C[[\paramq,\hbar]]$  are obtained from one another by applying
$\exp([F,\cdot])$, where $F:=\frac{1}{\hbar}\int_0^\hbar f d\hbar$.

So we have constructed $\bullet_{\dagger,x}: \operatorname{HC}(\A_{\paramq}(v))\rightarrow
\operatorname{HC}(\hat{\A}_{\paramq}(\hat{v}))$.

\subsubsection{Variations}
The functor $\bullet_{\dagger,x}: \operatorname{HC}(\A^0_{\paramq}(v))\rightarrow
\operatorname{HC}(\hat{\A}^0_{\paramq}(\hat{v}))$ is constructed completely
analogously. Similarly to Section \ref{SSS_HC_def},
any HC $\A_{\paramq}(v)$-bimodule $\B$ is also HC over $\A^0_{\paramq}(v)$ and
$\B_{\dagger,x}$ does not depend on whether we consider $\B$ as an
$\A_{\paramq}(v)$-bimodule or as a $\A^0_{\paramq}(v)$-bimodule.

Note also that above we have established a functor $\HC(\A_{\paramq}(v)_\hbar^{\wedge_x})\rightarrow
\HC(\hat{\A}_{\paramq}(\hat{v}))$. Denote it by $\Psi$. We also have a version of this functor
for the $\A^0$-algebras (again, denoted by $\Psi$).

In the case of affine quivers, we sometimes will need a slight modification of the target category for $\bullet_{\dagger,x}$. Namely, we remark that $0$ does not need to be a single symplectic leaf in $\hat{\M}(\hat{v})$. This happens, for example, when the quiver $\hat{Q}$ is a single loop or is a union of such. Let $\mathcal{L}_0$ be a leaf through $0\in \hat{\M}(\hat{v})$, this is an affine space. So the algebra $\hat{\A}_{\paramq}(v)$
splits into the product of the Weyl algebra $\mathbf{A}_{0}$ quantizing $\mathcal{L}_0$ and of some other algebra $\bar{\A}_{\paramq}(\hat{v})$. The latter is obtained by the same reduction
but from the space where we replace all summands of the form $\operatorname{End}(\C^{\hat{v}_i})$
with $\mathfrak{sl}_{\hat{v}_i}$. We have a category equivalence
$\operatorname{HC}(\hat{\A}_{\paramq}(\hat{v}))\xrightarrow{\sim}\operatorname{HC}(\bar{\A}_{\paramq}(\hat{v}))$
sending $\hat{\B}$ to the centralizer $\bar{\B}$ of $\mathbf{A}_0$ in $\hat{\B}$
(so that $\hat{\B}=\mathbf{A}_0\otimes \bar{\B}$). We will view
$\bullet_{\dagger,x}$ as a functor with target category
$\operatorname{HC}(\bar{\A}_{\paramq}(\hat{v}))$.

\subsection{Restriction functors: properties}\label{SS_compl1}
It is straightforward from the construction that $\bullet_{\dagger,x}$ is exact and $\C[\paramq]$-linear,
compare to \cite[Section 3.4]{HC} or \cite[Section 4.1.4]{W_dim}.

Now  let us describe the behavior of the functor $\bullet_{\dagger,x}$ on the associated varieties.
The following lemma follows straightforwardly from the construction (compare with (4) of \cite[Proposition 3.6.5]{sraco}).

\begin{Lem}\label{Lem:dag_assoc}
Let $\B$ be a HC $\A_{\paramq}(v)$-bimodule.  Then the associated variety of $\B_{\dagger,x}$
is uniquely characterized by $(\VA(\B_{\dagger,x})\times \mathcal{L})^{\wedge_x}=\VA(\B)^{\wedge_x}$, where $\mathcal{L}$
is the symplectic leaf through $x$. A similar claim holds for HC $\A_{\paramq}^0(v)$-bimodules.
\end{Lem}

Now let us proceed to the compatibility of $\bullet_{\dagger,x}$ with the Tor's and Ext's.

\begin{Lem}\label{Lem:tens_dag_intertw}
We have a functorial isomorphism $$\Tor^{\A^0_{\paramq_0}(v)}_i(\mathcal{B}^1,\mathcal{B}^2)_{\dagger,x}=
\Tor^{\hat{\A}_{\paramq_0}(\hat{v})}_i(\mathcal{B}^1_{\dagger,x},\mathcal{B}^2_{\dagger,x}).$$
Here $\mathcal{B}^1\in \operatorname{HC}(\A^0_{\paramq_0}(v)\text{-}\A^0_{\paramq'_0}(v))$
and $\mathcal{B}^2\in \operatorname{HC}(\A^0_{\paramq'_0}(v)\text{-}\A^0_{\paramq''_0}(v))$,
where $\paramq_0,\paramq_0',\paramq_0''$ are three parallel affine subspaces in $\paramq$.
Similarly, we have
$$\Ext_{\A^0_{\paramq_0}(v)}^i(\mathcal{B}^1,\mathcal{B}^2)_{\dagger,x}=
\Ext_{\hat{\A}_{\paramq_0}(\hat{v})}^i(\mathcal{B}^1_{\dagger,x},\mathcal{B}^2_{\dagger,x}),$$
where $\mathcal{B}^1\in \operatorname{HC}(\A^0_{\paramq_0}(v)\text{-}\A^0_{\paramq'_0}(v))$
and $\mathcal{B}^2\in \operatorname{HC}(\A^0_{\paramq_0}(v)\text{-}\A^0_{\paramq''_0}(v))$.
\end{Lem}
\begin{proof}
We will deal with the case when $\paramq_0=\paramq$, the general case is similar.
We will do  Tor's, the case of Ext's is similar.

Consider the bounded derived category $D^b(\A^0_{\paramq}(v))$
of the category $\A^0_{\paramq}(v)\operatorname{-bimod}$ of finitely generated $\A^0_{\paramq}(v)$-bimodules and its subcategory $D^b_{HC}(\A^0_{\paramq}(v))$ of all complexes with HC homology. Similarly, consider
the bounded derived category $D^b(\A^0_{\paramq}(v)_\hbar)$ of the category $\A^0_{\paramq}(v)_\hbar\operatorname{-grbimod}$
of graded finitely generated graded $\A^0_{\paramq}(v)_\hbar$-bimodules and its subcategory $D^b_{HC}(\A^0_{\paramq}(v)_\hbar)$
of all complexes whose homology mod $\hbar$ are $\C[\M^0_{\param}(v)]$-modules (rather than just arbitrary bimodules).
We have a functor $\C_1\otimes_{\C[\hbar]}\bullet: \A^0_{\paramq}(v)_\hbar\operatorname{-grbimod}\rightarrow
\A^0_{\paramq}(v)\operatorname{-bimod}$ whose kernel is the subcategory $\A^0_{\paramq}(v)_\hbar\operatorname{-grbimod}_{tor}$
of all bimodules where $\hbar$ acts  nilpotently. This gives rise to the equivalence
\begin{equation}\label{eq:equi1}\C_1\otimes_{\C[\hbar]}\bullet: D^b(\A^0_{\paramq}(v)_\hbar)/D^b_{tor}(\A^0_{\paramq}(v)_\hbar)\rightarrow D^b(\A^0_{\paramq}(v))\end{equation}
that restricts to an equivalence of the HC subcategories and clearly intertwines the derived tensor product
(or $\Hom$) functors.

Let us proceed to the completed setting. Consider the algebra $$\Af:=\C[\mathsf{eu}]\ltimes
(\A^0_{\paramq}(v)_\hbar^{\wedge_x}\widehat{\otimes}_{\C[[\hbar]]}\A^0_{\paramq}(v)_\hbar^{\wedge_x,opp}),$$ where
$[\mathsf{eu},a]=\hbar \partial_\hbar a$ for $a\in \A^0_{\paramq}(v)_\hbar^{\wedge_x}\widehat{\otimes}_{\C[[\hbar]]}\A^0_{\paramq}(v)_\hbar^{\wedge_x,opp}$.
Any module over $\mathfrak{A}$ is a  $\A^0_{\paramq}(v)_\hbar^{\wedge_x}$-bimodule equipped with an Euler derivation
(but not vice versa).
Let $D^b(\A^0_{\paramq}(v)_\hbar^{\wedge_x})\subset D^b(\mathfrak{A}\operatorname{-mod})$ stand for the full subcategory  of all objects whose homology is
a HC $\A^0_{\paramq}(v)_\hbar^{\wedge_x}$-bimodule. We have the completion functor
\begin{align*}\bullet^{\wedge_x}:=
(\A^0_{\paramq}(v)_\hbar^{\wedge_x}\widehat{\otimes}_{\C[[\hbar]]}\A^0_{\paramq}(v)_\hbar^{\wedge_x,opp})
\otimes_{\A^0_{\paramq}(v)_\hbar\otimes_{\C[\hbar]}\A^0_{\paramq}(v)_\hbar^{opp}}\bullet:
\A^0_{\paramq}(v)_\hbar\operatorname{-grbimod}
\rightarrow \Af\operatorname{-mod}\end{align*}
We remark that, for a HC bimodule $\M$, we have $\M^{\wedge_x}=\A^0_{\paramq}(v)_\hbar^{\wedge_x}\otimes_{\A^0_{\paramq}(v)_\hbar}\M$
because the right hand side is already complete as a right $\A^0_{\paramq}(v)_\hbar$-module.
The completion functor restricts to a functor
\begin{equation}\label{eq:equi2}\bullet^{\wedge_x}: D^b_{HC}(\A^0_{\paramq}(v)_\hbar)\rightarrow D^b_{HC}(\A^0_{\paramq}(v)_\hbar^{\wedge_x}).\end{equation}
This functor  preserves the
$\hbar$-torsion subcategories. It intertwines $\bullet\otimes_{\A^0_{\paramq}(v)_\hbar}\bullet$
with $\bullet\otimes_{\A^0_{\paramq}(v)_\hbar^{\wedge_x}}\bullet$.  It is t-exact.
And since it sends $\A^0_{\paramq}(v)_\hbar$ to $\A^0_{\paramq}(v)_\hbar^{\wedge_x}$,
it intertwines the derived tensor product functors as well.

Now let us equip $\mathcal{B}^1,\mathcal{B}^2$ with good filtrations and consider the corresponding Rees bimodules
$\mathcal{B}^1_\hbar, \mathcal{B}^2_\hbar$. Since $\bullet^{\wedge_x}$ is a t-exact functor, we see that
\begin{equation}\label{eq:homol} H_i(\mathcal{B}^1_\hbar\otimes^L_{\A^0_{\paramq}(v)_\hbar}
\mathcal{B}^2_\hbar)^{\wedge_x}=
H_i(\mathcal{B}^{1\wedge_x}_\hbar\otimes^L_{\A^0_{\paramq}(v)_\hbar^{\wedge_x}}\mathcal{B}^{2\wedge_x}_\hbar),\end{equation} the equality of HC $\A^0_{\paramq}(v)_\hbar^{\wedge_x}$-bimodules.

Recall the functor
$\Psi: \HC(\A^0_{\paramq}(v)_\hbar^{\wedge_x})\rightarrow \HC(\hat{\A}_{\paramq}(\hat{v}))$
from Section \ref{SS_compl}. Applying $\Psi$ to the left hand side of (\ref{eq:homol}), we get $H_i(\mathcal{B}^1\otimes^L_{\A^0_{\paramq}(v)}\mathcal{B}^2)_{\dagger,x}$.

Let us see what happens when we apply $\Psi$ to the right hand side.
Note that $\mathcal{B}_\hbar^{i\wedge_x}=\Weyl_\hbar^{\wedge_0}\widehat{\otimes}_{\C[[\hbar]]}R_\hbar(\mathcal{B}^i_{\dagger,x})^{\wedge_0}$
that yields $$\mathcal{B}_\hbar^{1\wedge_x}\otimes^L_{\A^0_{\paramq}(v)_\hbar^{\wedge_x}}\mathcal{B}_\hbar^{2\wedge_x}=
\Weyl_\hbar^{\wedge_0} \widehat{\otimes}_{\C[[\hbar]]}\left(R_\hbar(\mathcal{B}^1_{\dagger,x})^{\wedge_0}\otimes^L_
{\hat{\A}_{\paramq}(\hat{v})^{\wedge_0}_\hbar}R_\hbar(\mathcal{B}^2_{\dagger,x})^{\wedge_0}\right).$$
So if we apply $\Psi$ to the right hand side of (\ref{eq:homol}) we get $H_i(\mathcal{B}^1_{\dagger,x}\otimes_{\hat{\A}_{\paramq}(\hat{v})}\mathcal{B}^2_{\dagger,x})$.
This completes the proof.
\end{proof}

Another important property of the restriction functor is the equality \begin{equation}\label{eq:transl_restr}\A^0_{\paramq,\chi}(v)_{\dagger,x}=\hat{\A}^0_{\paramq,\chi}(\hat{v}).
\end{equation}
This follows from the decomposition $\A^0_{\paramq,\chi}(v)_\hbar^{\wedge_x}\cong \hat{\A}^0_{\paramq,\chi}(\hat{v})_\hbar^{\wedge_0}
\widehat{\otimes}_{\C[[\hbar]]}\Weyl_\hbar^{\wedge_0}$ that is proved similarly to (\ref{eq:alg_decomp0}).

We finish this section with two remarks.

\begin{Rem}\label{Rem:derived_restr}
Let us explain why in Lemma \ref{Lem:tens_dag_intertw} we deal with $\Tor$'s rather than with the derived tensor products.
The reason is that we do not have the derived version of the functor $\bullet_{\dagger,x}$. The difficulty
here is to pass between the derived version of the category $\HC(\hat{\A}_{\paramq}(\hat{v})^{\wedge_0}_\hbar)$
to that of the category $\HC(\hat{\A}_{\paramq}(\hat{v})_\hbar)$. For the latter derived version
we take the subcategory in the derived category of the category of graded $\hat{\A}_{\paramq}(\hat{v})_\hbar$-bimodules
with HC homology. For the former derived version we need to use
the subcategory in the derived category of modules over $$\C[\hat{\mathsf{eu}}]\ltimes\left( \C((\hbar))\otimes_{\C[[\hbar]]}\left(\hat{\A}_{\paramq}(\hat{v})^{\wedge_0}_\hbar\widehat{\otimes}_{\C[[\hbar]]}
\hat{\A}_{\paramq}(\hat{v})^{\wedge_0,opp}_\hbar\right)\right)$$
with homology that is a localization (from $\C[[\hbar]]$ to $\C((\hbar))$) of a HC bimodule. We need to localize to
$\C((\hbar))$ because  the operator
$\frac{1}{\hbar}[a,\cdot]$ is not defined on an arbitrary $\hat{\A}_{\paramq}(\hat{v})^{\wedge_0}_\hbar$-bimodule.
Of course, we still have a completion functor \begin{align*}&D^b_{HC}(\hat{\A}_{\paramq}(\hat{v})_\hbar\operatorname{-grbimod})/
D^b_{HC}(\hat{\A}_{\paramq}(\hat{v})_\hbar\operatorname{-grbimod})_{tor}\rightarrow \\
&D^b_{HC}(\C[\hat{\mathsf{eu}}]\ltimes \C((\hbar))\otimes_{\C[[\hbar]]}\left[\hat{\A}_{\paramq}(\hat{v})^{\wedge_0}_\hbar\widehat{\otimes}_{\C[[\hbar]]}
\hat{\A}_{\paramq}(\hat{v})^{\wedge_0,opp}_\hbar\right])\end{align*}
(here $\operatorname{grbimod}$ means graded bimodules). A problem with this functor is that it is not an equivalence,
the target category has more $\Hom$'s, which has to do with the fact that we do not require the action of a derivation
$\hat{\operatorname{eu}}$ to be diagonalizable (and we do not see any way to impose this condition).

Let us point out that this problem does not occur in the W-algebra setting, \cite{HC,W_dim} because there we have a Kazhdan torus action that fixes a point where we complete. So in that case it is enough to deal with weakly  $\C^\times$-equivariant derived categories.
\end{Rem}

\begin{Rem}\label{Rem:adjoint}
Let $\tilde{\operatorname{HC}}(\A^0_{\paramq}(v))$ denote the category of {\it locally HC} $\A^0_{\paramq}(v)$-bimodules
(i.e., bimodules that are sums of their Harish-Chandra subbimodules), the ind completion of $\operatorname{HC}(\A^0_{\paramq}(v))$.
Then, similarly to \cite[Section 3.4]{HC},\cite[Section 3.7]{sraco}, we have a functor $\bullet^{\dagger,x}:
\operatorname{HC}(\hat{\A}^0_{\paramq}(\hat{v}))\rightarrow \tilde{\operatorname{HC}}(\A^0_{\paramq}(v))$
that is right adjoint to $\bullet_{\dagger,x}$. This functor is automatically $\C[\paramq]$-linear.
It is likely that the image of $\bullet^{\dagger,x}$ actually lies in $\operatorname{HC}(\A^0_{\paramq}(v))$
but we do not know the proof of this claim. Below we will see $\underline{\mathcal{B}}^{\dagger,x}$
lies in $\operatorname{HC}(\A^0_{\paramq}(v))$ provided $\underline{\mathcal{B}}$ is finitely generated
over $\C[\paramq]$.
\end{Rem}

\subsection{Restriction functors: applications}\label{SS_compl2}
Our first application will be to  $\paramq$-supports of HC bimodules.

\begin{Prop}\label{Prop:HC_support}
Let $\B$ be a HC $\A^0_{\paramq}(v)$-bimodule. Then
$\operatorname{Supp}^r_{\paramq}(\B)$  is closed and
$$\mathsf{AC}(\operatorname{Supp}^r_{\paramq}(\B))=\Supp_{\param}(\gr\B).$$
\end{Prop}
Recall that $\mathsf{AC}$ stands for the asymptotic cone.
\begin{proof}
Pick a generic point $x$ in an irreducible
component of $\VA(\B)\cap \M_0^0(v)$ and consider the HC $\hat{\A}_{\paramq}^0(\hat{v})$-bimodule
$\B_{\dagger,x}$. By the choice of $x$, $\B_{\dagger,x}$ is finitely generated over
$\C[\paramq]$, this follows from Lemma \ref{Lem:dag_assoc}.
Moreover, since $\bullet_{\dagger,x}$ is $\C[\paramq]$-linear (and is $\C[\param]$-linear after
passing to the associated graded bimodules) by the construction, we have
$$\Supp^r_{\paramq}(\B_{\dagger,x})\subset \Supp^r_{\paramq}(\B),\quad
\Supp_{\param}(\gr\B_{\dagger,x})\subset \Supp_{\param}(\gr\B).$$
Since $\B_{\dagger,x}$ is finitely generated over
$\C[\paramq]$, we see that $\mathsf{AC}(\Supp^r_{\paramq}(\B_{\dagger,x}))=
\Supp_{\param}(\gr\B_{\dagger,x})$. Hence
$\mathsf{AC}(\Supp^r_{\paramq}(\B_{\dagger,x}))\subset
\Supp_{\param}(\gr\B)$.
There is the unique maximal subbimodule $\B'\subset \B$ with $\B'_{\dagger,x}=0$.
Clearly, $\operatorname{Supp}^r_{\paramq}(\B_{\dagger,x})\subset
\operatorname{Supp}^r_{\paramq}(\B/\B')$. On the other hand, let
$I$ be the right annihilator of $\B_{\dagger,x}$ in $\C[\paramq]$.
Then $\B I\subset \B'$ and $\operatorname{Supp}^r_{\paramq}(\B/\B')
\subset \operatorname{Supp}^r_{\paramq}(\B/\B I)
\subset \operatorname{Supp}^r_{\paramq}(\B_{\dagger,x})$.
So we see that $\operatorname{Supp}^r_{\paramq}(\B_{\dagger,x})=
\operatorname{Supp}^r_{\paramq}(\B/\B')$ is a closed subvariety in
$\paramq$ whose asymptotic cone coincides with
$\operatorname{Supp}_{\param}(\gr(\B/\B'))=
\operatorname{Supp}_{\param}(\gr\B_{\dagger,x})$.

Now let us observe that
\begin{equation}\label{eq:supp_union1}\operatorname{Supp}^r_{\paramq}(\B)=
\operatorname{Supp}^r_{\paramq}(\B/\B')\cup \operatorname{Supp}^r_{\paramq}(\B').\end{equation}
The inclusion of the left hand side into the right hand side is clear.
Now we just need to show that if $z\in \operatorname{Supp}^r_{\paramq}(\B')\setminus \operatorname{Supp}^r_{\paramq}(\B/\B')$, then $z\in \operatorname{Supp}^r_{\paramq}(\B)$.
Recall that $\operatorname{Supp}^r_{\paramq}(\B/\B')$ is closed.
So if $z\not\in \operatorname{Supp}^r_{\paramq}(\B/\B')$, then
$\operatorname{Tor}^1_{\C[\paramq]}(\B/\B',\C_z)=0$. Hence if $z\in \operatorname{Supp}^r_{\paramq}(\B')$,
then $z\in \operatorname{Supp}^r_{\paramq}(\B)$.
Similarly,
\begin{equation}\label{eq:supp_union2}\operatorname{Supp}_{\param}(\gr\B)=
\operatorname{Supp}_{\param}(\gr\B/\B')\cup \operatorname{Supp}_{\param}(\gr\B').
\end{equation}

Thanks to (\ref{eq:supp_union1}) and (\ref{eq:supp_union2}),
it remains to prove that $\operatorname{Supp}^r_{\paramq}(\B')$
is closed and its asymptotic cone is $\operatorname{Supp}_{\param}(\gr\B')$.
The variety $\M_0^0(v)$ has finitely many symplectic leaves.
Since $\operatorname{Supp}^r_{\paramq}(\B')\subsetneq
\operatorname{Supp}^r_{\paramq}(\B)$, we can use the induction on
the maximal dimension of  a symplectic leaf in the support to finish
the proof of the proposition.
\end{proof}

Now we are ready to prove Proposition \ref{Prop:alg_iso}.

\begin{proof}[Proof of Proposition \ref{Prop:alg_iso}]
Consider the natural homomorphism $\A^0_{\paramq}(v)\rightarrow \A_{\paramq}(v)$
and let $K,C$ denote its kernel and cokernel. Both $\A^0_{\paramq}(v), \A_{\paramq}(v)$
are HC bimodules over $\A^0_{\paramq}(v)$ and therefore $K,C$ are HC bimodules as well.
By Proposition \ref{Prop:HC_support}, $\operatorname{Supp}^r_{\paramq}(K),\operatorname{Supp}^r_{\paramq}(C)$
are closed.

 The homomorphism $\A^0_{\lambda}(v)
\rightarrow \A_\lambda(v)$ is an isomorphism if and only  if
$\lambda\not\in \operatorname{Supp}^r_{\paramq}(K)\cup \operatorname{Supp}^r_{\paramq}(C)$.
Indeed, the homomorphism
is surjective if and only if $\lambda\not\in \operatorname{Supp}^r_{\paramq}(C)$.
Further, if $\lambda\not\in \operatorname{Supp}^r_{\paramq}(C)$, then, similarly to
the proof of Proposition \ref{Prop:HC_support},
 we get
$\A^0_{\lambda}(v)\xrightarrow{\sim} \A_\lambda(v)$ if and only if
$\lambda\not\in \operatorname{Supp}^r_{\paramq}(K)$.

Consider the homomorphism  $\gr\A^0_{\paramq}(v)\rightarrow \gr\A_{\paramq}(v)=\C[\M_{\param}(v)]$
and compose it with the epimorphism $\C[\M_{\param}^0(v)]\twoheadrightarrow \gr\A^0_{\paramq}(v)$.
Let $K^0,C^0$ denote the kernel and the cokernel of the resulting homomorphism
$\C[\M_{\param}^0(v)]\rightarrow \C[\M_{\param}(v)]$. The latter coincides with
$\rho^*$. It follows that   $\Supp_{\param}(K^0\oplus C^0)\subset \param^{sing}$,
where, recall, $\param^{sing}$ denotes the locus of non-generic parameters in $\param$.
Note  that $C^0\twoheadrightarrow \gr C$, while $\gr K$ is a subquotient of $K^0$.
Because of this, we have $\mathsf{AC}(\Supp^r_{\paramq}(C))\subset \Supp_{\param}(C^0)$
and $\mathsf{AC}(\Supp^r_{\paramq}(K))\subset \Supp_{\param}(K^0)$.
The claim of the proposition follows.
\end{proof}

Next we will show that the algebra $\A_\lambda(v)$ is simple for a Weil generic $\lambda$, compare with \cite[Section 4.2]{sraco}.

\begin{Prop}\label{Prop:gen_simpl}
The algebra $\A_\lambda(v)$ is simple for a Weil generic $\lambda$.
\end{Prop}
We will obtain a more precise description of the locus, where
$\A_\lambda(v)$ is simple, using wall-crossing functors below,
Proposition \ref{Prop:gen_simpl_strong}.
\begin{proof}
{\it Step 1.} Let us show that, for a Weil generic $\lambda$, the algebra $\A_\lambda(v)$ has no finite dimensional representations. Let $\paramq_d$ denote the set of points $\lambda\in \paramq$ such that $\A_\lambda(v)$ has a $d$-dimensional representation or, in other words, there is a homomorphism $\A_\lambda(v)\rightarrow \operatorname{Mat}_d(\C)$. Consider the ideal $I^d\subset \A_{\paramq}(v)$ generated by the elements
$$\alpha_{2d}(x_1,\ldots,x_{2d})=\sum_{\sigma\in \mathfrak{S}_{2d}}\operatorname{sgn}(\sigma) x_{\sigma(1)}\ldots x_{\sigma(2d)}.$$
Any homomorphism $\A_{\paramq}(v)\rightarrow \operatorname{Mat}_d(\C)$ factors through $\A_{\paramq}(v)/I^d$, this is the Amitsur-Levitski theorem. The support
of $\A_{\paramq}(v)/I^d$ in $\paramq$ is closed by Proposition \ref{Prop:HC_support}.
If a Weil generic element of $\paramq$ belongs to $\bigcup_{d} \operatorname{Supp}_{\paramq}(\A_{\paramq}(v)/I^d)$,  then  $\operatorname{Supp}_{\param}(\A_{\paramq}(v)/I^d)=\paramq$ for some $d$.
By Proposition \ref{Prop:HC_support}, $\Supp_{\param}(\C[\M_{\param}(v)]/\operatorname{gr}I^d)=\param$. However, this is impossible. Indeed, for a Zariski generic $\lambda$, the variety $\M_\lambda(v)$ is symplectic, so $\C[\M_\lambda(v)]$  has no proper Poisson ideals. Since  $\gr I^d$ is a Poisson ideal, we get a required contradiction.

{\it Step 2.} By the previous step, for a Weil generic $\lambda$ and all $x\in \M(v)\setminus \M(v)^{reg}$,
the algebra $\hat{\A}_\lambda(\hat{v})$ defined from $x$ has no finite
dimensional irreducible representations. It follows from Lemma \ref{Lem:dag_assoc} that  the algebra $\A_\lambda(v)$
has no ideals $I$ such that $\VA(\A_\lambda(v)/I)$ is a proper subvariety of $\M(v)$. Indeed,
for $x$ that is generic in an irreducible component of $\VA(\A_\lambda(v)/I)$, the ideal $I_{\dagger,x}\subset\hat{\A}_\lambda(\hat{v})$ is of finite codimension. On the other hand, if
$I$ is a proper ideal, then $\VA(\A_\lambda(v)/I)$ is also proper, this is consequence of \cite[Corollar 3.6]{BoKr}.
The proposition follows.
\end{proof}

\subsection{Applications to derived Hamiltonian reduction}\label{SS_der_Ham_red}
In this section we prove part (1) of Proposition \ref{Prop:der_Ham_quot_glob}. The proof does not have to
do with HC bimodules but involves techniques similar to what was used in Sections \ref{SS_compl}-\ref{SS_compl2}.

We will prove the following claim that implies (1) of Proposition \ref{Prop:der_Ham_quot_glob}:
\begin{itemize}
\item[(*)] There is an asymptotically generic open affine subset $U\subset \paramq^{iso}$
such that $\mathcal{Q}_U:=\mathcal{Q}_{\paramq}\otimes_{\C[\paramq]}\C[U]$ is flat over $\C[U]$ and $\operatorname{Tor}_i^{U(\g)}(D(R),\C[U])=0$ for $i>0$.
\end{itemize}
Let $r\in T^*R$ be a point with closed $G$-orbit and let $\hat{R}, R_0$
have the same meaning as in \ref{SSS_class_slice}. We need to relate
$\operatorname{Tor}_i^{U(\g)}(D(R),\C[\paramq])$ to
$\operatorname{Tor}_i^{U(\g_r)}(D(\hat{R}),\C[\paramq])$, where $\C[\paramq]$
becomes a $U(\g_r)$-module via the inclusion $U(\g_r)\hookrightarrow U(\g)$.

\begin{Lem}\label{Lem:Tor_compl}
We have a natural $\C[\paramq,\hbar]$-linear isomorphism
\begin{equation}\label{eq:Tor_compl}
\begin{split}
&\operatorname{Tor}_i^{U_\hbar(\g)}(D_\hbar(R),\C[\paramq,\hbar])^{\wedge_{Gr}}\cong\\
&\Gamma\left(G/G_r, G*_{G_r}\operatorname{Tor}_i^{U_\hbar(\g_r)}\left(D_\hbar(\hat{R}),\C[\paramq,\hbar]\right)\right)^{\wedge_{G/G_r}}
\widehat{\otimes}_{\C[[\hbar]]}\Weyl_{\hbar}(R_0)^{\wedge_0}.
\end{split}
\end{equation}
\end{Lem}
\begin{proof}
Recall the isomorphism
\begin{equation}\label{eq:compl_iso}
D_\hbar(R)^{\wedge_{Gr}}\cong \left([D_\hbar(G)^{\wedge_G}\widehat{\otimes}_{\C[[\hbar]]}D_\hbar(\hat{R})^{\wedge_0}]\red_0 G_r\right)\widehat{\otimes}_{\C[[\hbar]]}\Weyl_\hbar(R_0)^{\wedge_0}
\end{equation}
that has appeared in the proof of Lemma \ref{Lem:quant_decomp}.

Note also that
$$\operatorname{Tor}_i^{U_\hbar(\g)}(D_\hbar(R),\C[\paramq,\hbar])^{\wedge_{Gr}}\cong
\operatorname{Tor}_i^{U_\hbar(\g)}(D_\hbar(R)^{\wedge_{Gr}},\C[\paramq,\hbar]).$$
So we need to check that the right hand side of (\ref{eq:Tor_compl}) coincides with
the $i$th Tor of the right hand side of (\ref{eq:compl_iso}). This will follow if we check
that
\begin{equation}\label{eq:Tor_coinc2}
\begin{split}
&\operatorname{Tor}_i^{U_\hbar(\g)}(\left([D_\hbar(G)\otimes_{\C[\hbar]}D_\hbar(\hat{R})]\red_0 G_r\right),
\C[\paramq,\hbar])\cong\\
& \Gamma\left(G/G_r, G*_{G_r}\operatorname{Tor}_i^{U_\hbar(\g_r)}(D_\hbar(\hat{R}),\C[\paramq,\hbar])\right).
\end{split}
\end{equation}
Since the actions of
$G_r$ and $U_\hbar(\g)$ commute and
$$[D_\hbar(G)\otimes_{\C[\hbar]}D_\hbar(\hat{R})]\red_0 G_r=\left([D_\hbar(G)\otimes_{\C[\hbar]}D_\hbar(\hat{R})]\otimes^L_{U_\hbar(\g_r)}\C[\paramq,\hbar]\right)^{G_r},$$ we see that the left hand side of (\ref{eq:Tor_coinc2}) coincides with
$$\operatorname{Tor}_i^{U_\hbar(\g\times \g_r)}(D_\hbar(G)\otimes_{\C[\hbar]}D_\hbar(\hat{R}),
\C[\paramq,\hbar])^{G_r}.$$
Here  $\C[\paramq,\hbar]$ is viewed as the diagonal $\g\times \g_r$-module.
But to compute $\operatorname{Tor}_i^{U_\hbar(\g\times \g_r)}(D_\hbar(G)\otimes_{\C[\hbar]}D_\hbar(\hat{R}),
\C[\paramq,\hbar])$ we can take the derived tensor product with $U_\hbar(\g)$
and after that the derived tensor product with $U_\hbar(\g_r)$. What we get is exactly
the right hand side of (\ref{eq:Tor_coinc2}).
\end{proof}

\begin{Lem}\label{Lem:Tor_supp}
We have $\mathsf{AC}\left(\overline{\Supp^r_{\paramq}(\operatorname{Tor}_i^{U(\g)}(D(R),\C[\paramq]))}\right)\subset \param^{sing}$
provided $i>0$.
\end{Lem}
\begin{proof}
Set $M:=\operatorname{Tor}_i^{U(\g)}(D(R),\C[\paramq])$, we view
it as a $D(R)\otimes \C[\paramq]$-module. It is supported on
$\mu^{-1}(\param)\times_{\param}\param$. Let
$N$ be the maximal submodule of $M$ with the property that
$\VA(N)\cap (\mu^{-1}(0)\times \{0\})$ is contained in the nilpotent cone
of $\mu^{-1}(0)$, equivalently, $N_{\hbar}^{\wedge_x}=0$ for
all nonzero $x\in \M^0(v)$. Note that we have only finitely many possible $G_r\subset G$ and hence
finitely many possible spaces $\hat{\paramq}$. Moreover,
under the natural projection $\param\rightarrow \hat{\param}$, the preimage of
$\hat{\param}^{sing}$ lies in $\param^{sing}$ by Remark \ref{Rem:sing_param}.
From this observation combined with Lemma \ref{Lem:Tor_compl} and an induction argument, it follows that
$\mathsf{AC}(\overline{\Supp^r_{\paramq}(M/N)})\subset \param^{sing}$.

The space $M$ is naturally filtered and there is an inclusion $\gr M\hookrightarrow M^0$,
where $M^0$ is a $G$-equivariant quotient of $\operatorname{Tor}_i^{U(\g)}(\C[T^*R],\C[\param])$. Inside $M^0$ we can consider the maximal submodule $N^0$ defined similarly to $N\subset M$.
Note that $\Supp_{\param}(\operatorname{Tor}_i^{U(\g)}(\C[T^*R],\C[\param]))\subset \param^{sing}$
and therefore $\Supp_{\param}(M^0)\subset \param^{sing}$ and
$\Supp_{\param}(M^0/N^0)\subset \param^{sing}$. From here we deduce that
\begin{equation}\label{eq:supp_cont11}
\Supp_{\param}(N^0)\subset \param^{sing}.
\end{equation}
Clearly, $\gr N\subset N^0$ (a $G$-equivariant embedding).

The $D(R)$-module $N$ is weakly $G$-equivariant and finitely generated.
So it is generated by finitely many $G$-isotypic components, say, corresponding to
$G$-irreps $V_1,\ldots,V_k$. We can assume that the corresponding isotypic components
generate $N^0$ as well ($N^0$ is also finitely generated). Let $N_V\subset
N, N_V^0\subset N^0$ denote the sum of these isotypic components
so that $\gr N_V\subset N_V^0$. We have
\begin{equation}\label{eq:Supp_eq11}\Supp^r_{\paramq}(N_V)=
\Supp^r_{\paramq}(N), \Supp_{\param}(N^0_V)=
\Supp^r_{\param}(N^0).\end{equation} Since $\VA(N^0)\cap (\mu^{-1}(0),0)$ lies in the nilpotent
cone, we see that any $G$-isotypic component in $N^0$ is finitely generated over
$\C[\param]$. A similar claim holds for $N$. From here and the inclusion $\gr N_V
\subset N_V^0$ we deduce that $\mathsf{AC}(\Supp^r_{\paramq}(N_V))
\subset \Supp_{\param}(N^0_V)$. Combining (\ref{eq:supp_cont11})
with (\ref{eq:Supp_eq11}), we see that $\mathsf{AC}(\overline{\Supp^r_{\paramq}(N)})\subset \param^{sing}$.
Since $\Supp^r_{\paramq}(M)\subset \Supp^r_{\paramq}(N)\cup
\Supp^r_{\paramq}(M/N)$, we get
$\mathsf{AC}(\overline{\Supp^r_{\paramq}(M)})\subset \param^{sing}$.
\end{proof}

Now let us show that there is  an asymptotically generic $U\subset \paramq$ such that
$\mathcal{Q}_{U}$ is flat over $U$. For this, we consider  various modules
$\operatorname{Tor}_i^{\C[\paramq]}(\mathcal{Q}_{\paramq},\C[Y])$,
where $Y$ is a closed irreducible subvariety in $\paramq$.
Since $\mathcal{Q}_{\paramq}$
is a finitely generated $D(R)$-module, there is an affine Zariski open
subset  $U\subset\paramq$ (not asymptotically generic, a priori) such that $\mathcal{Q}_{U}$
is free over $U$, this is proved analogously to Lemma \ref{Lem:HC_gen_flat}.
Let $Z$ denote the Zariski closure of the union of
the supports of various  $\operatorname{Tor}_i^{\C[\paramq]}(\mathcal{Q}_{\paramq},\C[Y])$.
So $Z\subset \paramq\setminus U$.
We need to check  that $\mathsf{AC}(Z)\subset \param^{sing}$. This is done as in the proof of
Lemma \ref{Lem:Tor_supp}, now we need to consider $M:=\operatorname{Tor}_i^{\C[\paramq]}(\mathcal{Q}_{\paramq},\C[Z])$. This finishes the proof of
(1) of Proposition \ref{Prop:der_Ham_quot_glob}.

\section{Localization theorems and translation bimodules}\label{S_loc}
In this section we deal with (abelian and derived) localization theorems. These theorems
allow to relate the category
of finitely generated modules over $\A_\lambda(v)$ to the category of coherent sheaves over $\A_\lambda^\theta(v)$.
We also study more closely translation bimodules introduced in \ref{SSS_HC_transl} that play
a crucial role in the abelian localization theorems.

\subsection{Abelian and derived localization}
Let $\theta$ be a generic stability condition and $\lambda\in \paramq$.
We say that $(\lambda,\theta)$ satisfies abelian (resp., derived) localization
if the functors $\Gamma_\lambda^\theta$ and $\Loc^\theta_\lambda$ are mutually
inverse equivalences between $\A_\lambda(v)\operatorname{-mod}$ and
$\A_\lambda^\theta(v)\operatorname{-mod}$ (resp.,
$R\Gamma_\lambda^\theta$ and $L\Loc^\theta_\lambda$ are mutually
inverse equivalences between $D^b(\A_\lambda(v)\operatorname{-mod})$ and
$D^b(\A_\lambda^\theta(v)\operatorname{-mod})$). We will write $\AL(v)$
for the set of all $(\lambda,\theta)$ satisfying abelian localization.

First, let us recall the derived localization theorem for $\A_\lambda(v)$.

\begin{Prop}\label{Prop:der_MN}
Suppose that the moment map $\mu$ is flat or $Q$ has finite or affine type. Then
$(\lambda,\theta)$ satisfies derived localization if and only if the homological
dimension of $\A_\lambda(v)$ is finite.
\end{Prop}
\begin{proof}
The case when $\mu$ is flat follows from \cite[Theorem 1.1]{MN}.

In general, we can apply a quantum LMN isomorphism and assume that $\nu$ is dominant.
If $Q$ is of finite or affine type, then by \ref{SSS_M0_prop}, $\mu$ is flat.
\end{proof}

Sufficient conditions (in greater generality) for abelian localization to hold were studied in
\cite[Section 5.3]{BPW}. Let us recall some results from there. For this we need some terminology.
By a {\it classical wall} for $v$ we mean a hyperplane of the form $\{\theta| \theta\cdot v'=0\}$, where $v'$ is as in  \ref{SSS_gen_param}. So if $\theta$  does not lie
on a classical wall, it is generic. By a {\it classical chamber} we mean the closure of a connected component of the complement to the union of classical walls in $\mathbb{R}^{Q_0}$. Let $C=C_\theta$ be the classical chamber of $\theta$.

\begin{Prop}[Corollary 5.17 in \cite{BPW}]\label{Prop:abelian_loc_BPW}
For every $\lambda$ and any $\chi\in \Z^{Q_0}\cap \mathsf{int}C$
there  is $n_0\in \Z$ such that the  $(\lambda+n\chi,\theta)\in \mathfrak{AL}(v)$ for any $n>n_0$.
\end{Prop}
Here we write $\mathsf{int}C$ for the interior of $C$.

Unfortunately, Proposition \ref{Prop:abelian_loc_BPW}
is not good enough for our purposes, as we will need a stronger version.
We will also need to relate abelian localization to the functors $\pi_\lambda^0(v),\pi_\lambda^\theta(v)$.
Recall the open subset $\paramq^{iso}$ of all parameters $\lambda$ such that $\A_\lambda^0(v)\xrightarrow{\sim}
\A_\lambda(v)$.

\begin{Prop}\label{Lem:ab_loc} The following statements are true.
\begin{enumerate}
\item Suppose $\lambda\in \paramq^{iso}$. We have $(\lambda,\theta)\in \mathfrak{AL}(v)$ if and only if the functors $\pi^0_\lambda(v)$ and $\pi_\lambda^\theta(v)$ are isomorphic.
\item For every $\lambda$, there is $\chi\in \Z^{Q_0}$ such that $(\lambda',\theta)\in \AL(v)$ for every
$\lambda'\in \lambda+\chi+(C\cap \Z^{Q_0})$.
\end{enumerate}
\end{Prop}

To prove (2) (that will be used when  we  discuss wall-crossing functors),
we will also need a more technical version of (2), which is the following lemma.

\begin{Lem}\label{Lem:ab_loc_techn}
For every $\lambda$, there are
\begin{itemize}
\item
$\lambda'\in \lambda+\Z^{Q_0}$,
\item and a subset $Y(\lambda')\subset \lambda'+(C\cap \Z^{Q_0})$
\end{itemize}
such that $\lambda''\in \paramq^{iso}, (\lambda'',\theta)\in \AL(v)$ for all $\lambda''\in Y(\lambda')$
and the  intersection of $Y(\lambda')$ with every codimension $1$ face of the cone $\lambda'+C$ is Zariski dense in that face.
\end{Lem}


\subsection{Translation bimodules}\label{SS_transl_bimod}
In this subsection we will apply results from Sections \ref{SS_compl1} and \ref{SS_compl2} to studying translation bimodules $\A_{\lambda,\chi}^0(v), \A_{\lambda,\chi}^{(\theta)}(v)$ and a connection between them. In particular, here we will prove Propositions \ref{Prop:quot_Ham_loc} and \ref{Lem:ab_loc}(1) as well as Lemma \ref{Lem:ab_loc_techn}.

The next two propositions investigate when various versions of translations coincide.

\begin{Prop}\label{Prop:transl_coinc}
Let $\chi,\chi'\in \Z^{Q_0}$. Then the following subsets of $\paramq$ are Zariski open and
asymptotically generic.
\begin{enumerate}
\item The set of $\lambda$ such that $\A_{\lambda,\chi}^0(v)\rightarrow \A^{(\theta)}_{\paramq,\chi}(v)_\lambda$
is an isomorphism.
\item The set of $\lambda$ such that the multiplication homomorphism $\A^0_{\lambda+\chi,\chi'}(v)\otimes_{\A^0_{\lambda+\chi}(v)}\A_{\lambda,\chi}^0(v)\rightarrow \A^0_{\lambda, \chi+\chi'}(v)$ is an isomorphism.
\end{enumerate}
\end{Prop}
\begin{proof}
Let us prove (1). We have a natural surjection $\C[\mu^{-1}(\param)]^{G,\chi}\twoheadrightarrow \gr \A^0_{\paramq,\chi}(v)$
and  a natural inclusion $\gr \A^{(\theta)}_{\paramq,\chi}(v)\hookrightarrow  \C[\mu^{-1}(\param)^{\theta-ss}]^{G,\chi}$.
Further, the following diagram is commutative (microlocalization commutes with taking the associated graded)
\begin{equation}\label{eq:transl_bimod_commut}
\begin{picture}(80,30)
\put(2,2){$\gr \A^0_{\paramq,\chi}(v)$}
\put(2,20){$\C[\mu^{-1}(\param)]^{G,\chi}$}
\put(50,2){$\gr \A^{(\theta)}_{\paramq,\chi}(v)$}
\put(50,20){$\C[\mu^{-1}(\param)^{\theta-ss}]^{G,\chi}$}
\put(11,19){\vector(0,-1){12}}
\put(22,3){\vector(1,0){27}}
\put(25,22){\vector(1,0){24}}
\put(60,7){\vector(0,1){12}}
\end{picture}
\end{equation}
Now the top horizontal arrow becomes an isomorphism when localized to the generic locus in $\param$.
(1) follows from Proposition \ref{Prop:HC_support}, as in the proof of Proposition \ref{Prop:alg_iso}.

Let us  prove (2). We have natural epimorphisms
$$\C[\mu^{-1}(\param)]^{G,\chi}\twoheadrightarrow \gr\A^0_{\paramq,\chi}(v),
\C[\mu^{-1}(\param)]^{G,\chi'}\twoheadrightarrow \gr\A^0_{\paramq,\chi'}(v),
\C[\mu^{-1}(\param)]^{G,\chi+\chi'}\twoheadrightarrow \gr\A^0_{\paramq,\chi+\chi'}(v).$$ 
Note that
\begin{equation}\label{eq:microloc_equal}\gr\mathcal{Q}_{\paramq}|_{\M^0_{\param^{reg}}(v)}=
\C[\mu^{-1}(\param)]|_{\M^0_{\param^{reg}}(v)}\end{equation}
as $\mu:\mu^{-1}(\param)\rightarrow \param$ is flat over $\param^{reg}:=\param\setminus \param^{sing}$.
So the kernels of these epimorphisms are supported on $\param^{sing}$. From here we see that
the kernel of $$\C[\mu^{-1}(\param)]^{G,\chi'}\otimes_{\C[\M^0_{\param}(v)]}\C[\mu^{-1}(\param)]^{G,\chi}
\twoheadrightarrow \gr\A^0_{\paramq+\chi,\chi'}(v)\otimes_{\C[\M^0_\param(v)]}\gr\A^0_{\paramq,\chi}$$
is also supported on $\param^{sing}$. (\ref{eq:microloc_equal}) also shows that the kernel of
$$\gr\A^0_{\paramq+\chi,\chi'}(v)\otimes_{\C[\M^0_\param(v)]}\gr\A^0_{\paramq,\chi}
\twoheadrightarrow \gr\left(\A^0_{\paramq+\chi,\chi'}(v)\otimes_{\A^0_{\paramq}(v)}\A^0_{\paramq,\chi}(v)\right)$$
is supported on $\param^{sing}$.
So the kernel of the composition
$$\eta:\C[\mu^{-1}(\param)]^{G,\chi'}\otimes_{\C[\M^0_{\param}(v)]}\C[\mu^{-1}(\param)]^{G,\chi}
\twoheadrightarrow \gr\left(\A^0_{\paramq+\chi,\chi'}(v)\otimes_{\A^0_{\paramq}(v)}\A^0_{\paramq,\chi}(v)\right)$$
is supported on $\param^{sing}$.  Let $\varpi,\varpi^0$ denote the natural homomorphisms
\begin{align*}
&\A^0_{\paramq+\chi,-\chi}(v)\otimes_{\A^0_{\paramq}(v)}\A^0_{\paramq,\chi}(v)\rightarrow \A^0_{\paramq,\chi+\chi'}(v),\\
&\C[\mu^{-1}(\param)]^{G,\chi'}\otimes_{\C[\M^0_{\param}(v)]}\C[\mu^{-1}(\param)]^{G,\chi}\twoheadrightarrow
\gr \A^0_{\paramq,\chi+\chi'}(v).
\end{align*}
We have $\varpi^0=\gr\varpi\circ \eta$. It follows that both the kernel and the cokernel of
$\gr\varpi$ are supported on $\param^{sing}$.
Now we can argue as in the proof of Proposition \ref{Prop:alg_iso}
to finish the proof of (2).
%
\end{proof}

\begin{Prop}\label{Prop:trans_reln}
Suppose that $\chi$ lies in the interior of the chamber of $\theta$ and satisfies $H^1(\M^\theta(v), \mathcal{O}(\chi))=0$.
Then we have $\A^0_{\paramq,\chi}(v)\xrightarrow{\sim} \A^{(\theta)}_{\paramq,\chi}(v)$. Moreover, this isomorphism is filtered and induces an isomorphism $\gr\A^0_{\paramq,\chi}(v)\xrightarrow{\sim}
\gr\A^{(\theta)}_{\paramq,\chi}(v)$. Both these $\C[\mu^{-1}(\param)]^G$-modules are identified with $\C[\mu^{-1}(\param)]^{G,\chi}$.
\end{Prop}
\begin{proof}
Note that the restriction map $\C[\mu^{-1}(\param)]^{G,\chi}
\rightarrow \C[\mu^{-1}(\param)^{\theta-ss}]^{G,\chi}$ is injective because
$\chi$ is in the chamber of $\theta$. The left vertical
arrow in diagram (\ref{eq:transl_bimod_commut}) is surjective.
We conclude that it is an isomorphism.
From $H^1(\M^\theta(v), \mathcal{O}(\chi))=0$ it follows that
the right vertical arrow in (\ref{eq:transl_bimod_commut}) is an isomorphism. For the
same reasons, the same true for the specialization of (\ref{eq:transl_bimod_commut})
to any value of $\lambda$. For $\lambda$ Zariski generic,
the natural map $\A^0_{\lambda,\chi}(v)\rightarrow \A^{(\theta)}_{\lambda,\chi}(v)$
is an isomorphism. This follows from (1) of Proposition \ref{Prop:transl_coinc}
combined with Lemma \ref{Prop:trans_spec}. Since the induced map of the associated graded modules is
an embedding, it is forced to be an isomorphism. So
$\A^0_{\lambda,\chi}(v)\rightarrow \A^{(\theta)}_{\lambda,\chi}(v)$ is an isomorphism
for all $\lambda$. From here we deduce that $\A^0_{\paramq,\chi}(v)\hookrightarrow
\A^{(\theta)}_{\paramq,\chi}(v)$ is an isomorphism. The claim about the associated
graded follows from here.
\end{proof}

Let us deduce a corollary  of the previous proposition.

\begin{Cor}\label{Cor:shifts_compar}
Let $\theta$ be a generic stability condition.
Let $\chi$ be generic. Then $\A^0_{\paramq,\chi}(v)=
\A^{(\theta)}_{\paramq,\chi}(v)$ provided $H^1(\M^\theta(v),\mathcal{O}(\chi))=0$.
\end{Cor}
\begin{proof}
This is a consequence of  Propositions \ref{Prop:univ_wc},\ref{Prop:trans_reln}.
\end{proof}

\begin{proof}[Proof of Proposition \ref{Prop:quot_Ham_loc}]
The claim of the proposition is equivalent for any two parameters with difference
in $\Z^{Q_0}$. Recall, Proposition \ref{Prop:alg_iso}, that $\paramq^{iso}$ is
asymptotically generic. Therefore, after adding an element of $\Z^{Q_0}$ to $\lambda$, we
can assume that $\lambda+n\chi\in \paramq^{iso}$ for all $n\geqslant 0$.
By Proposition \ref{Prop:trans_reln},  $\A^0_{\lambda+m\chi,n\chi}(v)\xrightarrow{\sim} \A^{(\theta)}_{\lambda+m\chi,n\chi}(v)$
for all $n,m\geqslant 0$. So the $\Z$-algebra $\mathsf{Z}_{\lambda,\chi}:=\bigoplus_{n,n'\geqslant 0}\A^0_{\lambda+n\chi, n'\chi}$ is the same as the one appearing in \cite[Section 5.3]{BPW}.

Thanks to Proposition \ref{Prop:abelian_loc_BPW},
replacing $\lambda$ with $\lambda+m\chi$ for some $m$, we may assume that the
$\Z$-algebra $\mathsf{Z}_{\lambda,\chi}$ is Morita, see, e.g., \cite[Section 5.3]{BPW}
for the definition. Therefore, $\A_{\lambda}^\theta(v)\operatorname{-mod}$ is equivalent to
the category $\mathsf{Z}_{\lambda,\chi}\operatorname{-mod}^{gr}$ of finitely generated graded $\mathsf{A}$-modules.

Now the claim that $\pi_\lambda^\theta(v)$ is a quotient functor is proved as in \cite[Section 5.5]{BPW}.
Let us provide details of the argument. Consider the functor
$\pi^{\Z}:=\bigoplus_{i=0}^\infty \pi^0_{\lambda+n\chi}(v): D(R)\operatorname{-mod}^{G,\lambda}
\rightarrow \mathsf{Z}_{\lambda,\chi}\operatorname{-mod}^{gr}$. By \cite[Proposition 5.28]{BPW} (that only uses
the assumption $\lambda+n\chi\in \paramq^{iso}$ for all $n\geqslant 0$ and not the flatness of
the moment map), the equivalence $\A_\lambda^\theta(v)\operatorname{-mod}\xrightarrow{\sim}
\mathsf{Z}_{\lambda,\chi}\operatorname{-mod}^{gr}$ intertwines the functors $\pi^\theta_\lambda(v)$
and $\pi^{\Z}$. The latter is a quotient functor by \cite[Lemma 5.29]{BPW}.
\end{proof}

\begin{proof}[Proof of (1) Proposition \ref{Lem:ab_loc}]
The proof is in several steps.

{\it Step 1}.
We can choose $\chi$ in the interior of the chamber of $\theta$ satisfying the following three
conditions:
\begin{itemize}
\item[(i)]
$H^1(\M^\theta(v),\mathcal{O}(n\chi))=0$ for all $n\geqslant 1$.
\item[(ii)] $H^1(\M^{-\theta}(v),\mathcal{O}(-n\chi))=0$ for all $n\geqslant 1$.
\item[(iii)] $\lambda+n\chi\in \paramq^{iso}$ and $(\lambda+n\chi,\theta)\in \mathfrak{AL}(v)$
for all $n\geqslant 0$.
\end{itemize}
Namely, choose $\chi$ in the chamber of $\theta$. Multiplying $\chi$ by a positive integer, we achieve
(i) and (ii). Then we can rescale $\chi$ again and achieve (iii) thanks to Proposition \ref{Prop:alg_iso}
and Proposition \ref{Prop:abelian_loc_BPW}.

{\it Step 2}.
By Corollary \ref{Cor:shifts_compar}, $\A^{(\theta)}_{\paramq,m\chi}(v)_{\lambda+n\chi}=
\A^0_{\lambda+n\chi,m\chi}(v)$ for all $m\geqslant -n$. Since $\lambda+n\chi\in \mathfrak{AL}(v)$,
we see that $\A^{(\theta)}_{\paramq,m\chi}(v)_{\lambda+n\chi}=
\A^{(\theta)}_{\lambda+n\chi,m\chi}(v)$ thanks to Lemma \ref{Prop:trans_spec}.
Therefore  $\A^{(\theta)}_{\lambda+n\chi,m\chi}(v)=\A^0_{\lambda+n\chi,m\chi}(v)$ for all $m\geqslant -n$.
We are going to deduce (1) of Proposition \ref{Lem:ab_loc} from this equality.

{\it Step 3}.
Consider the $\Z$-algebra $\mathsf{Z}_{\lambda,\chi}:=\bigoplus_{n,m\geqslant 0} \A^0_{\lambda+n\chi,m\chi}(v)$
and an ``extended'' $\Z$-algebra $\tilde{\mathsf{Z}}_{\lambda,\chi}:=\bigoplus_{n\geqslant 0, m\geqslant -n}\A^0_{\lambda+n\chi,m\chi}(v)$.
For $M\in D(R)\operatorname{-mod}^{G,\lambda}$, the sum $\bigoplus_{n\geqslant 0}M^{G,n\chi}$ is a module over $\tilde{\mathsf{Z}}_{\lambda,\chi}$. But since $(\lambda+n\chi,\theta)\in \mathfrak{AL}(v)$ for all $n\geqslant 0$,
all bimodules $\A^{(\theta)}_{\lambda+n\chi,m\chi}(v)=\A^0_{\lambda+n\chi,m\chi}(v)$ are Morita equivalences
with inverse $\A^{(\theta)}_{\lambda+(n+m)\chi,-m\chi}(v)=\A^0_{\lambda+(n+m)\chi,-m\chi}(v)$.

{\it Step 4}.
We have an isomorphism $\A^0_{\lambda+(n+m)\chi,-m\chi}(v)\otimes_{\A_{\lambda+(n+m)\chi}(v)} \A^0_{\lambda+n\chi,m\chi}(v)\xrightarrow{\sim}\A_{\lambda+n\chi}(v)$ hence the map
$\A^0_{\lambda+(n+m)\chi,-m\chi}(v)\otimes_{\A_{\lambda+(n+m)\chi}(v)} \A^0_{\lambda+n\chi,m\chi}(v)\otimes_{\A_{\lambda+n\chi}(v)}M^{G,n\chi}
\rightarrow M^{G,n\chi}$ is an isomorphism as well.
But this map comes from taking products by elements of $D(R)$ in $M$ and hence factors as
\begin{align*}&\A^0_{\lambda+(n+m)\chi,-m\chi}(v)\otimes_{\A_{\lambda+(n+m)\chi}(v)} \A^0_{\lambda+n\chi,m\chi}(v)\otimes_{\A_{\lambda+n\chi}(v)}M^{G,n\chi}\\
&\rightarrow \A^0_{\lambda+(n+m)\chi,-m\chi}(v)\otimes_{\A_{\lambda+(n+m)\chi}(v)}M^{G,(n+m)\chi}\rightarrow M^{G,n\chi}.\end{align*}
So we see that the second map is surjective. Similarly, so is the first one. It follows that all maps
$\A^0_{\lambda+(n+m)\chi,-m\chi}(v)\otimes_{\A_{\lambda+(n+m)\chi}(v)}M^{G,(n+m)\chi}\rightarrow M^{G,n\chi}$
are isomorphisms. A conclusion is that the spaces $M^{G,n\chi}$ are either all zero or all nonzero.

{\it Step 5}.
As described in \cite[Section 5.3]{BPW}, the category $\A_\lambda^\theta(v)\operatorname{-mod}$ is equivalent to $\mathsf{Z}_{\lambda,\chi}\operatorname{-mod}$, where the latter stands for the quotient of the category of graded $\mathsf{Z}_{\lambda,\chi}$-modules by the subcategory of all bounded modules. Under this equivalence,
the functor $\pi_\lambda^\theta(v)$ becomes $M\mapsto \bigoplus_{n\geqslant 0} M^{G,n\chi}$ by \cite[Proposition 5.28]{BPW}
(we remark that $\pi_{\lambda}^\theta(v)\cong \pi_\lambda^\chi(v)$).
The conclusion of Step 4 now implies that the kernels of $\pi^0_\lambda(v)$ and of $\pi_\lambda^\theta(v)$
coincide. This proves (1).
\end{proof}

\begin{proof}[Proof of Lemma \ref{Lem:ab_loc_techn}]
The proof is in several steps.

{\it Step 1}. Pick $\chi\in \Z^{Q_0}$ lying in the interior of the chamber of $\theta$
and satisfying $H^1(\M^\theta(v), \mathcal{O}(\chi))=0$. Consider the
subset $\paramq^0\subset \paramq$ consisting of all parameters $\lambda$
such that $\lambda,\lambda+\chi\in \paramq^{iso}$ and $\A^0_{\lambda,\chi}(v), \A^0_{\lambda+\chi,-\chi}(v)$ are mutually
inverse Morita equivalences. By Proposition \ref{Prop:alg_iso} and
(2) of Proposition \ref{Prop:transl_coinc},
$\paramq^0$ is Zariski open and asymptotically generic.

{\it Step 2}.
We claim that $(\lambda,\theta)\in \AL(v)$ provided
$\lambda+n\chi\in \paramq^0$ for all $n\geqslant 0$. As was mentioned in Step 2 of the
proof of (1) of Proposition \ref{Lem:ab_loc},
$\A^0_{\lambda+m\chi,n\chi}=\A^{(\theta)}_{\lambda+m\chi,n\chi}$ (we will drop ``$(v)$''
from the notation). So $\A^{(\theta)}_{\lambda+m\chi,\chi}$
is a Morita equivalence bimodule for all $m\geqslant 0$. By \cite[Proposition 5.13]{BPW} what remains
to be checked is that the natural homomorphism
\begin{equation}\label{eq:product1}\A^0_{\lambda+(m+n)\chi,\chi}\otimes_{\A^0_{\lambda+(m+n)\chi}}\A^0_{\lambda+m\chi,n\chi}
\rightarrow \A^0_{\lambda+m\chi, (n+1)\chi}\end{equation}
is an isomorphism. Note that we also have a natural homomorphism
\begin{equation}\label{eq:product2}\A^0_{\lambda+(m+n+1)\chi,-\chi}\otimes_{\A^0_{\lambda+(m+n+1)\chi}}\A^0_{\lambda+m\chi,(n+1)\chi}
\rightarrow \A^0_{\lambda+m\chi, n\chi}\end{equation}
Since $\A^0_{\lambda+(m+n)\chi,\chi}$ is a Morita equivalence bimodule with inverse
$\A^0_{\lambda+(m+n+1)\chi,-\chi}$,  (\ref{eq:product2}) gives rise to
\begin{equation}\label{eq:product3}
\A^0_{\lambda+m\chi, (n+1)\chi}\rightarrow \A^0_{\lambda+(m+n)\chi,\chi}\otimes_{\A^0_{\lambda+(m+n)\chi}}\A^0_{\lambda+m\chi,n\chi}
\end{equation}
It is easy to see from the construction that (\ref{eq:product3}) and (\ref{eq:product1})
are mutually inverse to each other. This completes the proof of the claim in
the beginning of the step.

{\it Step 3}. Let us finish the proof of lemma. Let $f$ denote the product of the linear functions
defining the singular hyperplanes in $\param^{sing}$. Since $\paramq^0$ is asymptotically
generic, there is a polynomial $F\in \C[\paramq]$ that vanishes on $\paramq\setminus \paramq^0$
and has the form $f^d+\ldots$, where $\ldots$ denote the terms of smaller degree. We can assume
that the top degree term of $F$ is positive on the interior of $C$. Pick a codimension $1$
face $\Gamma$ of $C$ and let $\param_0$ denote the hyperplane in $\param$ spanned by $\Gamma$.
We can write an arbitrary element $\lambda\in \paramq$ as $\lambda=\lambda'+z\chi+\lambda_0$ with
$\lambda'$ being a fixed element, $z\in \C, \lambda_0\in \param_0$ (here we use the identifications $\paramq\cong \C^{Q_0}\cong
\param$). So we can view $F$ as an element of $\C[\param_0][z]$. We can shift $\lambda'$
by an integer so that $F|_{z=n}$ is a nonzero element of $\C[\param_0]$ for all $n\in \Z_{\geqslant 0}$.
This is our choice of $\lambda'$. Let $Y(\lambda')$ consist of all $\lambda\in \lambda'+(C\cap \Z^{Q_0})$ such that $\lambda+n\chi\in \paramq^0$ for all $n\in \Z_{\geqslant 0}$.
Now pick a Zariski generic primitive element $\psi\in \Gamma\cap \Z^{Q_0}$. We claim that for infinitely many elements $m\in \Z_{\geqslant 0}$, we have $F(m\psi+n\chi)\neq 0$ for all $n\in \Z_{\geqslant 0}$. This will imply the required
properties of $Y(\lambda')$.

Let $e$ denote the degree
of $F$. Let us write $F_e$ for the homogeneous degree $e$ part of $F$ and $F_{<e}$ for the sum
of degree $<e$ parts so that $F=F_e+F_{<e}$ and $F_e=f^d$. By the construction of $f$, we
have $F_e(m\psi+n\chi)>A_1 n^e$, where $A_1$ is some positive constant. On the other hand,
$|F_{<e}(m\psi+n\chi)|<A_2 \operatorname{max}(m,n)^{e-1}$, where $A_2$ is some positive constant.
We conclude that $F(m\psi+n\chi)\neq 0$ as long as $n\geqslant A_3 m^{1-1/e}$, where $A_3$ is some positive
constant. On the other hand, for fixed $n$, the number of solutions $m$ to $F(m\psi+n\chi)=0$
cannot exceed $e$, because  the degree of $F$ is $e$. So for  $M\in \Z_{>0}$,
the number of pairs $m\in [0,M], n\in \Z_{\geqslant 0}$ such that $F(m\psi+n\chi)=0$
is bounded by $A_4 M^{1-1/e}$, where $A_4$ is some positive constant. This finishes
the proof of the claim in the beginning of the paragraph and the proof of the lemma.

\end{proof}


\subsection{Conjectures on localization}\label{SS_loc_conj}
We would like to finish this section by stating conjectures on the precise loci, where abelian and derived
localizations hold.

Let us state the main conjecture.

\begin{Conj}\label{Conj:loc_main}
The following is true:
\begin{enumerate}
\item The locus $\paramq^{sing}(v)$ of $\lambda\in \paramq$ such that the algebra $\A_\lambda(v)$ has infinite
homological dimension is the union of hyperplanes each parallel to some $\ker\alpha$, where $\alpha$
is a root of $\g(Q)$ with $\alpha\leqslant v$.
\item Let $\theta$ be a generic stability condition lying in the classical chamber $C$. Then $(\lambda,\theta)\in \AL(v)$ if and only
if $(\lambda+(C\cap \Z^{Q_0}))\cap \paramq^{sing}(v)=\varnothing$.
\end{enumerate}
\end{Conj}

We would like to point out that the main challenge in (a) is to prove that $\paramq^{sing}(v)$
is a finite union of hyperplanes, it should not be hard to determine the hyperplanes in $\paramq^{sing}(v)$.
Let us give a more detailed  conjectural description of $\paramq^{sing}(v)$. Assume,
for simplicity, that the moment map $\mu$ is flat so that $\M(v)=\M^0(v)$.

For a root $\alpha$, let $\Sigma_\alpha$ denote the union of hyperplanes parallel to $\ker\alpha$
that are contained in $\paramq^{sing}(v)$ so that, according to Conjecture \ref{Conj:loc_main},
$\paramq^{sing}(v)=\bigcup_{\alpha\leqslant v}\Sigma_\alpha$. Let us explain how to compute
$\Sigma_\alpha$.

Pick a generic point $p\in \ker\alpha$ and assume that $\alpha$ is indecomposable. Let
$k$ be maximal such that $(v,1)-k\alpha$ is a root of the quiver $Q^w$. Then, in the
terminology of \ref{SSS_class_slice}, we can pick $x\in \M_p(v)$ that  corresponds
to the decomposition $r=r_0\oplus r_1\otimes \C^k$, where $\dim r_1=\alpha$ and $\dim r_0=(v^0,1)$.
Then we get the quiver $\hat{Q}$ that has a single vertex and $1-(\alpha,\alpha)/2$
loops. We consider the dimension $\hat{v}=k$ and the framing $\hat{w}=w\cdot \alpha-(v^0,\alpha)$.
Recall the affine map $\hat{r}:\paramq\rightarrow \hat{\paramq}=\C$ from
\ref{SSS_restr_alg}.

\begin{Conj}\label{Conj:sigma}
We have $\Sigma_\alpha=\hat{r}^{-1}(\hat{\paramq}^{sing}(\hat{v}))$ (where the locus
$\hat{\paramq}^{sing}(\hat{v})$ is formed for the framing $\hat{w}$).
\end{Conj}

Conjectures \ref{Conj:loc_main} and \ref{Conj:sigma} reduce the computation of the locus
where abelian/derived localization holds to quivers with a single vertex. Let us explain
what is known there.

In the case when there are no loops, the algebra $\A_\lambda(v)$ is $D^\lambda(\operatorname{Gr}(v,w))$,
the algebra of global $\lambda$-twisted differential operators on the grassmanian $\operatorname{Gr}(v,w)$.
In this case,  analogs of the abelian/derived Beilinson-Bernstein theorems (stated originally
for the flag varieties) hold. We have $\paramq^{sing}(v)=\{-1,-2,\ldots, 1-w\}$ and (2) of Conjecture
\ref{Conj:loc_main} holds.

Let us consider the situation when there is one loop. A classical case is when $w=1$. Here
$\A_\lambda(v)$ is the spherical rational Cherednik  algebra for $(S_n,\C^n)$,
see \ref{SSS_SRA}.
The subset $\paramq^{sing}(v)$ consists of all rational $\lambda\in (-1,0)$
with denominator not exceeding $n$, see e.g. \cite[Corollary 4.2]{BE}. Moreover,
(2) of Conjecture \ref{Conj:loc_main} holds, this follows from \cite{GS,KashRouq}, the
case of half-integer parameters was completed in \cite{BE}.

When $w>1$, we have that $\paramq^{sing}(v)$ consists of all rational
numbers $\lambda\in (-w,0)$ with denominator not exceeding $v$.
Moreover, (2) of Conjecture \ref{Conj:loc_main} holds. These results
are obtained in the subsequent paper \cite{Gies} by the
second named author.

Finally, let us mention that derived and/or abelian localization is known
for some $(Q,v,w)$ with $|Q_0|>1$. For example, much is known about the case when $Q$ is of finite
Dynkin type A. The case when the corresponding variety $\M^\theta(v)$
is the cotangent bundle to a partial flag variety follows similarly to
the Beilinson-Bernstein theorem. The case when $\M^\theta(v)$ is the preimage of
the Slodowy slice in the cotangent bundle of the full flag variety  follows
as in \cite{Ginzburg_HC}.

Now let $Q$ be an affine quiver with extending vertex $0$. Assume that
$w=\epsilon_0$ so that $\A_\lambda(v)$ is the spherical subalgebra
in an SRA. In this case there was a conjecture describing the singular
(=aspherical) locus in $\paramq^{sing}(n\delta)$, \cite[Conjecture 5.3]{Etingof_affine},
based on the cyclic case done before that in \cite{DG}. It is easy to see that
(after relating the parameterizations) (1) of Conjecture \ref{Conj:loc_main}
reduces to \cite[Conjecture 5.3]{Etingof_affine} for $v=n\delta$. Moreover, part (2)
in the cyclic case should follow from results of \cite{cycl_ab_loc}.

\section{Wall-crossing and Webster functors}\label{S_WWC}
In this section we will recall/introduce  two different (but related) families of functors
that are the main ingredients of the proof of Theorem \ref{Thm:verymain}.

We will consider functors categorifying the action of $\a$ on the cohomology
of $\bigsqcup_\theta \M^\theta(v)$. In special cases (for example,
when $\lambda$ is integral and all components  of $\theta$ are positive) these functors were constructed by Webster in \cite{Webster} and our general construction is built on his. So we call these {\it Webster
functors}. The second family of functors was introduced in \cite[Section 6]{BPW}
under the name of twisting functors. In this paper we call them {\it wall-crossing functors}.

Roughly, the Webster functors  should be thought as induction
functors that allow to produce new finite dimensional modules from existing ones,
proving a ``lower bound'' of Theorem \ref{Thm:verymain}.   The wall-crossing functors are used
to establish the ``upper bound''.

\subsection{Wall-crossing functors}\label{SS:WC}
Here we define wall-crossing functors and study some of their properties.
In what follows we assume that the functor $L\Loc_\lambda^\theta$ is a derived
equivalence provided the homological dimension of $\A_\lambda(v)$ is finite
(this is always the case when the quiver $Q$ is of finite or affine type,
see Proposition \ref{Prop:der_MN}).

\subsubsection{Construction of the functor}\label{SS_WC_constr}
Pick $\lambda\in \paramq, \chi\in \Z^{Q_0}$. Recall, \ref{SSS_HC_transl},
the $\A_{\lambda+\chi}^{\theta}(v)$-$\A_\lambda^\theta(v)$-bimodule
$\A_{\lambda,\chi}^\theta(v):=[\mathcal{Q}_{\lambda}|_{T^*R^{\theta-ss}}]^{G,\chi}$
and its global sections $\A_{\lambda,\chi}^{(\theta)}(v)$. Note that the functor
$\mathcal{T}_{\lambda,\chi}:\A^\theta_{\lambda,\chi}(v)\otimes_{\A_\lambda^\theta(v)}\bullet:
\A_\lambda^\theta(v)\operatorname{-mod}\rightarrow \A_{\lambda+\chi}^\theta(v)\operatorname{-mod}$
is an equivalence.
We remark that
\begin{equation}\label{eq:shift}
\mathcal{T}_{\lambda,\chi}\circ \pi_\lambda^\theta(v)=\pi_{\lambda+\chi}^\theta(v)\circ (\C_{-\chi}\otimes\bullet).
\end{equation}

Now let $\lambda',\lambda$ be such that $\chi:=\lambda-\lambda'\in \Z^{Q_0}$, the algebra $\A_{\lambda'}(v)$ has finite homological dimension (and so $L\Loc_{\lambda'}^\theta$ is a derived equivalence), and $(\lambda,\theta)\in \AL(v)$.
Following \cite[Section 6.4]{BPW}, we define a functor $\WC_{\lambda'\rightarrow\lambda}: D^b(\A_{\lambda'}(v)\operatorname{-mod})
\xrightarrow{\sim} D^b(\A_\lambda(v)\operatorname{-mod})$ by \begin{equation}\label{eq:WC_def}\WC_{\lambda'\rightarrow\lambda}:=\Gamma_\lambda^\theta\circ \mathcal{T}_{\lambda',\chi}\circ L\Loc_{\lambda'}^\theta.\end{equation} We remark that this functor is right $t$-exact.
If $(\lambda',\theta')\in \mathfrak{AL}(v)$, then we can also consider
the functor $\WC_{\lambda'\rightarrow \lambda}=\mathcal{T}_{\lambda',\chi}\circ (R\Gamma^\theta_{\lambda'})^{-1}\circ \Gamma_{\lambda'}^{\theta'}: D^b(\A_{\lambda'}^{\theta'}(v)\operatorname{-mod})\xrightarrow{\sim}
D^b(\A_{\lambda}^{\theta}(v)\operatorname{-mod})$. When $(\lambda',\theta')\in \mathfrak{AL}(v)$,
we often write $\WC_{\theta'\rightarrow \theta}$ instead of $\WC_{\lambda'\rightarrow \lambda}$.
We note that under the identifications $$\A_{\lambda_1}^{\theta}(v)\operatorname{-mod}
\xrightarrow{\sim}\A_\lambda^\theta(v)\operatorname{-mod}, \A_{\lambda_1'}^{\theta}(v)\operatorname{-mod}
\xrightarrow{\sim}\A_{\lambda'}^\theta(v)\operatorname{-mod}$$ with $\lambda_1\in \lambda+\Z^{Q_0},
\lambda_1'\in \lambda'+\Z^{Q_0}$, the functor $\WC_{\theta'\rightarrow \theta}$ is independent of
the choice of $\lambda_1,\lambda_1'$ provided $(\lambda_1,\theta),(\lambda_1',\theta')\in
\AL(v)$.

\subsubsection{Alternative realizations}
Here is another formula for  $\WC_{\theta'\rightarrow\theta}$
that holds when $\lambda\in \paramq^{ISO}$ (and $(\lambda,\theta)\in \AL(v)$):
\begin{equation}\label{eq:wc_formula}
\WC_{\theta'\rightarrow\theta}= \pi^{\theta}_{\lambda}(v)\circ (\C_{\lambda'-\lambda}\otimes\bullet) \circ L\pi_{\lambda'}^{\theta'}(v)^!.
\end{equation}
This formula follows from (\ref{eq:shift}),(\ref{eq:WC_def}) and Lemma \ref{Lem:der_glob_descr}.
Here we use the isomorphism $\pi_{\lambda'}^{\theta'}(v)=\pi^0_{\lambda'}(v)$ to produce the functor $L\pi_{\lambda'}^{\theta'}(v)^!$.

A connection of $\A_{\lambda,\chi}^{(\theta)}$ to the wall-crossing functor is provided by the following assertion.

\begin{Lem}[Proposition 6.31 in \cite{BPW}]\label{Lem:wc_bimod}
If $(\lambda,\theta)\in \mathfrak{AL}(v)$, then $$\WC_{\lambda\rightarrow\lambda'}(\bullet)=\A_{\lambda,\chi}^{(\theta)}(v)\otimes^L_{\A_{\lambda}(v)}\bullet.$$
\end{Lem}

In fact, under the assumptions of Lemma \ref{Lem:wc_bimod}, $\A_{\lambda,\chi}^{(\theta)}(v)=\A_{\lambda,\chi}^0(v)$,
as the following proposition shows.

\begin{Prop}\label{Lem:shift_coinc}
Suppose that $\lambda,\lambda+\chi\in \paramq^{iso}$ and  $(\lambda+\chi,\theta)\in \mathfrak{AL}(v)$. Then the
natural homomorphism $\A^0_{\lambda,\chi}(v)\rightarrow \A^{(\theta)}_{\lambda,\chi}(v)$ is an isomorphism.
\end{Prop}
\begin{proof}
By (1) of Proposition \ref{Lem:ab_loc}, the functors
$\pi^\theta_{\lambda+\chi}(v),\pi^0_{\lambda+\chi}(v)$ are isomorphic (below we suppress $v$ and write
$\pi^\theta_{\lambda+\chi}$, etc.). Moreover, $\pi^0_{\lambda+\chi}=\Gamma^\theta_{\lambda+\chi}\circ \pi^\theta_{\lambda+\chi}$. Then we have \begin{align*}
&\A^{(\theta)}_{\lambda,\chi}(v)\otimes_{\A_\lambda(v)}\bullet\cong  [\text{Lemma }\ref{Lem:wc_bimod}] \\  &\Gamma_{\lambda+\chi}^\theta\circ (\A_{\lambda,\chi}^\theta\otimes_{\A_\lambda^\theta}\bullet)\circ \operatorname{Loc}_\lambda^\theta \cong  [\text{Lemma }\ref{Lem:der_glob_descr}]\\&\Gamma^\theta_{\lambda+\chi}\circ\pi^\theta_{\lambda+\chi}\circ (\C_{-\chi}\otimes\bullet)\circ (\pi^0_\lambda)^!.
\end{align*}
Also it is easy to see that $\A^0_{\lambda,\chi}(v)\otimes_{\A_\lambda(v)}\bullet=\pi^0_{\lambda+\chi}\circ (\C_{-\chi}\otimes\bullet)\circ (\pi_{\lambda}^0)^!$. So the functors $\A^{(\theta)}_{\lambda,\chi}(v)\otimes_{\A_\lambda(v)}\bullet$
and $\A^0_{\lambda,\chi}(v)\otimes_{\A_\lambda(v)}\bullet$ are isomorphic. It follows that the bimodules
$\A^{(\theta)}_{\lambda,\chi}(v)$ and $\A^0_{\lambda,\chi}(v)$ are isomorphic. Let us see why the corresponding isomorphism
coincides with (\ref{eq:hom_0_to_theta}).

Consider the functor $(\pi^{\theta}_\lambda)^\diamondsuit: \A_\lambda^\theta(v)\operatorname{-mod}
\rightarrow D(R)\operatorname{-Mod}^{G,\lambda}$ defined by
$$(\pi_{\lambda}^\theta)^\diamondsuit:=\Gamma(\mathcal{Q}_\lambda|_{(T^*R)^{\theta-ss}}\otimes_{\A_\lambda^\theta(v)}\bullet)$$
We have a natural isomorphism
\begin{equation}\label{eq:adj_Hom_iso}\operatorname{Hom}_{\A_\lambda^\theta(v)}(\pi_\lambda^\theta(N),M)\cong
\operatorname{Hom}_{D(R)}(N, (\pi_{\lambda}^\theta)^\diamondsuit(M)), N\in D(R)\operatorname{-mod}^{G,\lambda},
M\in \A_\lambda^\theta(v)\operatorname{-mod}. \end{equation}

By the definition, (\ref{eq:hom_0_to_theta}) coincides with the natural homomorphism
\begin{equation}\label{eq:homom_11}\pi_{\lambda+\chi}^0\circ (\C_{-\chi}\otimes\bullet)\circ (\pi_\lambda^0)^!(\A_\lambda(v))\rightarrow
[\pi_{\lambda+\chi}^0\circ (\pi^{\theta}_{\lambda+\chi})^\diamondsuit]\circ \pi_{\lambda+\chi}^\theta \circ (\C_{-\chi}\otimes\bullet)\circ (\pi_\lambda^0)^!(\A_\lambda(v)).\end{equation}
It follows from Lemma \ref{Lem:der_glob_descr} and (\ref{eq:adj_Hom_iso})  that
the composition in the brackets in (\ref{eq:homom_11}) is $\Gamma_{\lambda+\chi}^\theta$.
Also (\ref{eq:adj_Hom_iso}) gives rise to a functor morphism  $\operatorname{id}\rightarrow (\pi^{\theta}_{\lambda+\chi})^\diamondsuit\circ \pi_{\lambda+\chi}^\theta$, which is nothing else but the restriction homomorphism of a module in $D(R)\operatorname{-mod}^{G,\lambda}$ to its sections on the semistable locus.
(\ref{eq:homom_11}) is induced by this functor morphism.
So (\ref{eq:homom_11}) is the homomorphism $\A_{\lambda,\chi}^0(v)\rightarrow \A^{(\theta)}_{\lambda,\chi}(v)$ constructed before in this proof.
%
\end{proof}

The importance of this proposition is that the bimodules $\A_{\lambda,\chi}^0(v)$ are better
than $\A_{\lambda,\chi}^{(\theta)}(v)$ in several aspects: for example, the former
behave well under restriction functors, (\ref{eq:transl_restr}). This will allow to study
wall-crossing functors inductively.

\subsubsection{Composition of wall-crossing functors}
It turns out that, under additional restrictions, a composition of wall-crossing functors is
again a wall-crossing functor.

More precisely, suppose that we have
two generic stability conditions $\theta,\theta'$. Suppose that $\theta_i, i=0,\ldots,q,$
are such that $\theta_0=\theta, \theta_q=\theta'$, $\theta_i$ and $\theta_{i+1}$
are separated by a single wall and $q$ is minimal with these properties.

\begin{Thm}\label{Thm:wc_decomp_short}
We have an isomorphism  of functors
$$\WC_{\theta_0\rightarrow \theta_q}=\WC_{\theta_{q-1}\rightarrow \theta_q}\circ\ldots
\WC_{\theta_1\rightarrow \theta_2}\circ \WC_{\theta_0\rightarrow \theta_1}.$$
\end{Thm}
\begin{proof}
This is established in the proof of \cite[Theorem 6.35]{BPW}.
\end{proof}

\subsubsection{Non-essential walls}
Sometimes a wall-crossing functor between two different chambers happens to be an
abelian equivalence. We will be interested in the situation when this happens
for two chambers sharing a wall.

Namely, we say that a classical wall $\ker\alpha$ is {\it non-essential} (for the parameter
$\lambda$) if for every
two classical chambers $C,C'$ separated by $\ker\alpha$ only, the wall-crossing functor
$\WC_{\lambda\rightarrow \lambda'}$ is an abelian equivalence for $\theta\in C, \theta'\in C'$
(and $(\lambda,\theta),(\lambda',\theta')\in \AL(v)$).

Here is an important example of a non-essential wall.

\begin{Prop}\label{Prop:wall_non_essent}
Suppose $\alpha$ is a real root and $\langle\alpha,\lambda\rangle\not\in \Z$.
Assume also that the intersections of $\paramq^{iso}$ with $\lambda+\ker\alpha, \lambda'+\ker\alpha$
are nonempty.  Then the wall $\ker\alpha$ is non-essential for $\lambda$.
\end{Prop}
\begin{proof}
Let $\paramq_0:=\lambda+\ker\alpha,\paramq_0':=\lambda'+\ker\alpha$. Consider the
translation bimodules $\B_{\paramq_0}:=\A^0_{\paramq_0,\chi}(v), \B'_{\paramq_0}:=\A^0_{\paramq'_0,-\chi}(v)$
and the algebras $\A_{\paramq_0}:=\A^0_{\paramq_0}(v), \A'_{\paramq_0}:=\A^0_{\paramq'_0}(v)$, where $\chi=\lambda'-\lambda$.

{\it Step 1}.
We claim that for a Zariski generic $\lambda_1\in \paramq_0$
the specializations $\B_{\lambda_1}, \B'_{\lambda_1}$ are mutually inverse Morita equivalences.
Similarly to the proof of Proposition \ref{Prop:alg_iso} (Section \ref{SS_compl2}),
this amounts to checking that the kernels
and the cokernels of  the natural homomorphisms
\begin{equation}\label{eq:nat_homs}\B'_{\paramq_0}\otimes_{\A'_{\paramq_0}}\B_{\paramq_0}\rightarrow \A_{\paramq_0},\quad
\B_{\paramq_0}\otimes_{\A_{\paramq_0}}\B_{\paramq_0'}\rightarrow \A'_{\paramq_0}
\end{equation}
have proper supports in $\paramq_0$. This will be proved in subsequent steps.

{\it Step 2}. Pick a Zariski generic $p\in \ker\alpha$ and consider $x\in \M^0_p(v)$.
The corresponding representation $r$ of $\overline{Q}^w$ decomposes into irreducibles as
$r=r_0\oplus  r_1\oplus\ldots\oplus r_k$, where $r_0$ has dimension $(v-k\alpha,1)$ and all $r_i$ have dimension
$(\alpha,0)$. By \cite[Theorem 1.2]{CB}, the representations $r_1,\ldots,r_k$
are all isomorphic. The corresponding quiver $\hat{Q}$ (see \ref{SSS_class_slice}) has
one vertex and no loops, $\hat{v}=k$ and $\hat{w}=w\cdot \alpha- (v-k\alpha,\alpha)=
(\nu,\alpha)+2k$. We claim that  $\hat{r}(\lambda)=\langle\alpha, \lambda\rangle+s$,
where $s\in \Z$. Indeed, this boils down to $\langle \varrho(v),\alpha\rangle-\frac{1}{2}(\nu,\alpha)\in \Z$
that is a straightforward check. So the parameter $\hat{\lambda}:=\hat{r}(\lambda)$
is not an integer. Also since $\lambda\in \paramq^{iso}$, we see that $\hat{\lambda}\in \hat{\paramq}^{iso}$
(this follows from Lemma \ref{Lem:quant_decomp}).
We conclude that the algebra $\hat{\A}^0_{\hat{\lambda}}(\hat{v})$ is the algebra
$D^{\hat{\lambda}}(\operatorname{Gr}(k,\hat{w}))$ of $\hat{\lambda}$-twisted differential operators on
the grassmanian $\operatorname{Gr}(k,\hat{w})$.

{\it Step 3}. Let $\hat{\chi}=\langle \chi,\alpha\rangle$, this is an integer. So both
$\hat{\lambda}, \hat{\lambda}+\hat{\chi}$ are not integers. A version of the  Beilinson-Bernstein
abelian localization theorem for twisted differential operators on grassmanians
implies that abelian localization holds for $(\hat{\lambda},\hat{\theta})$
and $(\hat{\lambda}+\hat{\chi},\hat{\theta})$. Therefore the bimodules
$\hat{\A}^{(\hat{\theta})}_{\hat{\lambda},\hat{\chi}}(\hat{v}),
\hat{\A}^{(\hat{\theta})}_{\hat{\lambda}+\hat{\chi},-\hat{\chi}}(\hat{v})$
are mutually inverse Morita equivalences. Proposition \ref{Lem:shift_coinc}
implies that the bimodules $\hat{\A}^{0}_{\hat{\lambda},\hat{\chi}}(\hat{v}),
\hat{\A}^{0}_{\hat{\lambda}+\hat{\chi},-\hat{\chi}}(\hat{v})$ are Morita
equivalences. Equivalently, the kernels and cokernels of the homomorphisms
in (\ref{eq:nat_homs}) vanish under the functor $\bullet_{\dagger,x}$.
By Lemma \ref{Lem:dag_assoc}, the associated varieties of these kernels and cokernels
do not intersect $\M_p(v)$. Since $p$ was chosen to be Zariski generic,
the $\paramq_0$-supports of the kernels and cokernels are proper.
The claim in the beginning  of Step 1 follows.

{\it Step 4}. Let us finish the proof. By applying integral shifts to
$\paramq_0,\paramq'_0$ (and, in particular, modifying $\chi$), thanks to Lemma
\ref{Lem:ab_loc_techn}, we may assume that there is
$\lambda_0\in \lambda+(\Z^{Q_0}\cap \ker\alpha)$ such that $(\lambda_1,\theta),
(\lambda_1+\chi,\theta')\in \AL(v)$ for  $\lambda_1\in \lambda_0+(C\cap \ker\alpha)$
in some Zariski dense subset of $\ker\alpha$.
As was mentioned in \ref{SS_WC_constr},
the functor $\WC_{\lambda_1\rightarrow \lambda_1+\chi}$ becomes $\WC_{\lambda\rightarrow \lambda+\chi}$
up to pre- and post-composing with equivalences of abelian categories. Since $\lambda_0+(C\cap \ker\alpha)$ is Zariski dense,
we use Proposition \ref{Lem:shift_coinc} together with Lemma \ref{Lem:wc_bimod} to see
that $\WC_{\lambda_1\rightarrow \lambda_1+\chi}=\A^0_{\lambda_1,\chi}(v)\otimes^L_{\A^0_{\lambda_1}(v)}\bullet$
is an abelian equivalence. This completes the proof.
\end{proof}

So the only non-trivial wall-crossing functors corresponding to the walls $\ker\alpha$
with real $\alpha$ are for $\alpha$ that are roots of $\mathfrak{a}^\lambda$. This is
the first indication that the representation theory of the algebras $\A_\lambda(v)$
is controlled by the algebras $\mathfrak{a}^\lambda$.

\subsection{Webster functors}\label{SS_W_fun}
\subsubsection{Special case: Webster's construction}
In \cite{Webster}, Webster introduced a quantum categorical version of Nakajima's construction, \cite[Section 10]{Nakajima}. In the case when all
$\theta_k$ are positive and for $i\in Q_0$
such that $\lambda_i\in \Z,$ he produced functors $F_i: D^b(\A^\theta_\lambda(v)\operatorname{-mod})
\leftrightarrows D^b(\A^\theta_\lambda(v+\epsilon_i)\operatorname{-mod}):E_i$ and studied their properties. We will need the construction
so we recall it first.

We start with the simplest possible case when $Q$ is a single vertex without arrows.
In this case, $\M^\theta(v)=T^*\operatorname{Gr}(v,w)$ and $\lambda$ has to be an integer
(for $\mathfrak{a}$ to be different from the Cartan subalgebra).
Our exposition follows \cite{Kaetc}.

Pick $r>0$ and set $d=w-2v+r$. Consider the incidence subvariety $C^r(d):=\operatorname{Fl}(v,v-r,w)\subset\operatorname{Gr}(v,w)\times \operatorname{Gr}(v-r,w)$.
Consider the $\delta$-function $D_{\operatorname{Gr}(v,w)\times \operatorname{Gr}(v-r,w)}$-module
$\delta_{C^r(d)}$ on $C^r(d)$ (the image of the structure sheaf on $C^r(d)$ under the Kashiwara
equivalence). Then we consider the following objects:
\begin{align*}
&\mathcal{E}^{(r)}(d)=\delta_{C^r(d)}[v(w-v)]\in D^b(D_{\operatorname{Gr}(v,w)\times \operatorname{Gr}(v-r,w)}\operatorname{-mod}),\\
&\mathcal{F}^{(r)}(d)=\delta_{C^r(d)}[(v-r)(w-v+r)]
\in D^b(D_{\operatorname{Gr}(v-r,w)\times \operatorname{Gr}(v,w)}\operatorname{-mod}).
\end{align*}
The object $\mathcal{E}^{(r)}(d)$ defines a functor $E^{(r)}(d):
D^b(D_{\operatorname{Gr}(v,w)}\operatorname{-mod})\rightarrow
D^b(D_{\operatorname{Gr}(v-r,w)}\operatorname{-mod})$ by convolving with
$\mathcal{E}^{(r)}(d)$. Similarly, we get a functor $F^{(r)}(d):
D^b(D_{\operatorname{Gr}(v-r,w)}\operatorname{-mod})\rightarrow
D^b(D_{\operatorname{Gr}(v,w)}\operatorname{-mod})$.
We write $E^{(r)}$ for $\bigoplus_d E^{(r)}(d)$,
and $F^{(r)}$ for $\bigoplus_d F^{(r)}(d)$.

The functors $E^{(r)}(d), F^{(r)}(d)$ are adjoint to one another
up to homological shifts. Namely, let us write $E^{(r)}(d)_L,E^{(r)}(d)_R$
for the left and right adjoint functors of $E^{(r)}(d)$. We have
\begin{equation}\label{eq:EF_adj}
E^{(r)}(d)_L\cong F^{(r)}(d)[-rd],\quad  E^{(r)}(d)_R\cong F^{(r)}(d)[rd].
\end{equation}


This construction has several extensions. For example, we get functors $$F:
D^b(D_{\operatorname{Gr}(\bullet,w)}\otimes D_{\underline{R}}\operatorname{-mod})
\leftrightarrows D^b(D_{\operatorname{Gr}(\bullet+1,w)}\otimes D_{\underline{R}}\operatorname{-mod}):E$$
for any vector space $\underline{R}$. Also if $H$ is a reductive group equipped with homomorphisms
$H\rightarrow \operatorname{GL}(w), \operatorname{GL}(\underline{R})$ and $\underline{\lambda}$
is a character of $\mathfrak{h}$, then we get functors
\begin{equation}\label{eq:W_fun}F:D^b_{H,\underline{\lambda}}(D_{\operatorname{Gr}(\bullet,w)}\otimes D_{\underline{R}}\operatorname{-mod})
\leftrightarrows D^b_{H,\underline{\lambda}}(D_{\operatorname{Gr}(\bullet+1,w)}\otimes D_{\underline{R}}\operatorname{-mod}):E
\end{equation}

Now let us proceed to the  case of a general quiver. We assume that $\theta_k>0$ for all $k\in Q_0$. Let $\underline{R},
\underline{G}$ be as in (\ref{eq:underl}), $\underline{\theta}$ be the collection of $\theta_j$ with $j\neq i$ and $\underline{\lambda}$  have the similar meaning to $\underline{\theta}$. We  reverse arrows if necessary and  assume that  $i$ is a source in $Q$.

Since $\theta_i>0$, we have $R\quo^{\theta_i} G_i=\operatorname{Gr}(v_i,\tilde{w}_i)\times \underline{R}$, where $\tilde{w}_i$ is defined by (\ref{eq:tildew}). The group $\underline{G}$ acts on $\operatorname{Gr}(v_i,\tilde{w}_i)\times \underline{R}$ diagonally, the action on $\operatorname{Gr}(v_i,\tilde{w}_i)$ is via a natural action of
$\underline{G}$ on $\tilde{W}_i$. Then we have $$
D_R\red^{\theta_i}_{\lambda_i} \GL(v_i)=D^{\lambda_i}_{\operatorname{Gr}(v_i,\tilde{w}_i)}\otimes D_{\underline{R}},\quad
\A_\lambda^\theta(v)=[D^{\lambda_i}_{\operatorname{Gr}(v_i,\tilde{w}_i)}\otimes D_{\underline{R}}]\red_{\underline{\lambda}}^{\underline{\theta}}\underline{G}.$$
Let us write $\A^{\theta_i}_{\lambda_i}(v)$ for the former reduction.

It follows from Proposition \ref{Prop:quot_Ham_loc}
that the category $\A^\theta_{\lambda}(v)\operatorname{-mod}$  is the quotient of
the category $\A^{\theta_i}_{\lambda_i}(v)\operatorname{-mod}^{\underline{G},\underline{\lambda}}$
(of $(\underline{G},\underline{\lambda})$-equivariant $\A_{\lambda_i}^{\theta_i}(v)$-modules)
by the Serre subcategory  of all modules whose singular support
is contained in the image of $\mu^{-1}(0)^{\theta_i-ss}\setminus \mu^{-1}(0)^{\theta-ss}$ in $T^*R\red^{\theta_i}\GL(v_i)$.
Lemma \ref{Lem:equiv_der_fun} shows that the same is true on the level of (equivariant)
derived categories.

As was checked by Webster, \cite[Section 4]{Webster}, the functors
$$F: D^b_{\underline{G},\underline{\lambda}}(\A_{\lambda_i}^{\theta_i}(v)\operatorname{-mod})
\rightleftarrows D^b_{\underline{G},\underline{\lambda}}(\A_{\lambda_i}^{\theta_i}(v+\epsilon_i)\operatorname{-mod}):E$$
preserve the subcategories  of all complexes whose homology are supported on the image of $\mu^{-1}(0)^{\theta_i-ss}\setminus\mu^{-1}(0)^{\theta-ss}$ in $T^*R\red^{\theta_i}\GL(v_i)$ (it is important
here that all $\theta_k$ are positive). So they descend to  endo-functors of
$\bigoplus_v D^b(\A_\lambda^\theta(v)\operatorname{-mod})$ to be denoted by $E_i,F_i$.

\begin{Rem}\label{Rem:Webster_functors_def}
Let $\chi\in \Z^{Q_0}$. By the very definition of the functors $E_i,F_i$, the equivalences
$$\A_\lambda^\theta(v)\operatorname{-mod}\xrightarrow{\sim}
\A_{\lambda+\chi}^\theta(v)\operatorname{-mod},\quad
\A_\lambda^\theta(v+\epsilon_i)\operatorname{-mod}\xrightarrow{\sim}
\A_{\lambda+\chi}^\theta(v+\epsilon_i)\operatorname{-mod}$$
intertwine these functors.
\end{Rem}

\subsubsection{Properties}\label{SSS_Webster_prop}
Let us explain some properties of the functors $E,F$ (and also of $E_i,F_i$).

The proof of the next lemma is a part of that for \cite[Theorem 3.1]{Webster}.

\begin{Lem}\label{Lem:categorif}
The functors (\ref{eq:W_fun}) define a categorical action of the 2-Kac-Moody algebra $\mathcal{U}(\sl_2)$
(we use the same version as in \cite[Section 1]{Webster}) on the category
\begin{equation}\label{eq:cat_rep}
\bigoplus_{v=0}^w D_{H,\underline{\lambda}}^b(D^\lambda_{\operatorname{Gr}(v,w)}\otimes D_{\underline{R}}\operatorname{-mod}).
\end{equation}
The divided power functors are $E^{(r)},F^{(r)}$.
\end{Lem}

Consider a 2-category $\mathsf{Q}$, a ``single vertex analog'' of the 2-category $\mathcal{Q}^\lambda$ introduced
in the end of \cite[Section 2]{Webster} (so that the 1-morphisms $E^{(r)},F^{(r)}$
in $\mathcal{U}(\slf_2)$
map to $\mathcal{E}^{(r)},\mathcal{F}^{(r)}$) that we are going to define now.
In our definition of $\mathsf{Q}$ we will need
a version of the Steinberg variety. By definition, this is the subvariety
$\mathsf{St}\subset T^*\operatorname{Gr}(v,w)\times T^*\operatorname{Gr}(v',w)$ that is
the preimage of the diagonal in $\gl(w)\times \gl(w)$ under the moment map $T^*\operatorname{Gr}(v,w)\times T^*\operatorname{Gr}(v',w)\rightarrow \gl(w)\times \gl(w)$.

To define $\mathsf{Q}$, it is enough to restrict to the case $\lambda=0$.
The collection of objects in that category is $\{0,\ldots,w\}$,
and the 1-morphisms from $v'$ to $v$ are the objects from $$D^b_{H}(\mathcal{D}_{\operatorname{Gr}(v',w)}\otimes
\mathcal{D}_{\operatorname{Gr}(v,w)}\otimes D_{\underline{R}^2})$$
with homology supported (in the sense of the singular support) on $\mathsf{St}\times \mathcal{N}^0_{diag}$,
where $\mathcal{N}^0_{diag}$ stands for the conormal bundle to the diagonal
$\underline{R}\subset \underline{R}^2$.  We remark that Webster's 2-category has, in a sense, more 1-morphisms but what we have above is sufficient for our purposes. The description of 2-morphisms in $\mathsf{Q}$ is similar to \cite[Section 2]{Webster}.  The action of $\mathsf{Q}$ on (\ref{eq:cat_rep})  (as well as the tensor structure on $\mathsf{Q}$) is defined via convolution of D-modules. The grading shift in $\mathcal{U}(\mathfrak{sl}_2)$ corresponds to the homological shift
in $\mathsf{Q}$.

The construction of \cite[Section 4]{Webster} shows that the  action in  Lemma \ref{Lem:categorif}
factors through a homomorphism $\psi$ of $2$-algebras $\mathcal{U}(\sl_2)\rightarrow \mathsf{Q}$.

Let us point out several other important properties of the functors $E_i, F_i$ that are due to Webster.

\begin{Lem}\label{Lem:W_funct_prop}
The following claims are true:
\begin{enumerate}
\item The functors $E_i,F_i$ preserve the subcategory $$\bigoplus_{v} D^b_{\rho^{-1}(0)}(\A_\lambda^{\theta}(v)\operatorname{-mod})
\subset \bigoplus_{v} D^b(\A_\lambda^{\theta}(v)\operatorname{-mod}).$$
\item Moreover, for $M\in D^b_{\rho^{-1}(0)}(\A_\lambda^{\theta}(v)\operatorname{-mod})$, we have
$$\CC(E_i M)=e_i\CC(M), \CC(F_i M)=f_i \CC(M),$$ where $e_i,f_i$ stand for the Nakajima operators.
\end{enumerate}
\end{Lem}

For the proof, see \cite[Corollary 3.4, Proposition 3.5]{Webster}.
In particular, if $\lambda\in \Z^{Q_0}$, we see that $\CC: K_0(\A^\theta_\lambda(v)
\operatorname{-mod})\rightarrow L_\omega[\nu]$ is surjective.

\subsubsection{General case}\label{SSS_Webster_general}
Now let us explain how to generalize Webster's functors to the case when $\lambda$ is not necessarily
integral and $\theta$ is not necessarily positive. We will assume that $\theta$ lies in the Tits
cone (which puts no restrictions for finite type $Q$ and results in $\langle\theta,\delta\rangle>0$
for affine type $Q$).  We will also assume that $\theta$ is generic for all $v$, in fact, for every
given $v$ and every classical chamber there is such an element there.

The element $\theta$ defines a Weyl chamber for the algebra $\a$ and hence a system of simple
roots $\Pi^\theta$ for $\a$.  For $\alpha\in \Pi^\theta$ (and all $v$) we will define functors
$$F_\alpha: D^b(\A_\lambda^\theta(v)\operatorname{-mod})\rightleftarrows
D^b(\A_\lambda^\theta(v+\alpha)\operatorname{-mod}):E_\alpha$$
generalizing the functors constructed by Webster.

Now we proceed to constructing $E_\alpha, F_\alpha$.
Let $\theta^+$ denote a stability condition with all entries positive.
Let $\sigma\in W(Q)$ be such that $\sigma \theta^+$ lies in the same
Weyl chamber for $\a$ as $\theta$. The stability conditions $\theta,\sigma\theta^+$
are separated only by non-essential walls of the form $\ker\beta$ for real roots $\beta$
and so, by Proposition \ref{Prop:wall_non_essent} and Theorem \ref{Thm:wc_decomp_short}, we can identify the categories
$\A_\lambda^\theta(v)\operatorname{-mod}$ and $\A_{\lambda}^{\sigma\theta^+}(v)\operatorname{-mod}$
by means of the wall-crossing functor $\WC_{\theta\rightarrow \sigma\theta^+}$ (for all $v$).
Also recall, \ref{SSS_LMN_quant}, that $\sigma$ gives rise to the quantum LMN isomorphism $\sigma:
\A^{\theta^+}_{\lambda'}(v')\xrightarrow{\sim}\A^{\sigma \theta^+}_\lambda(v)$, where
$v':=\sigma^{-1}\bullet v, \lambda':= \sigma^{-1}\bullet^{v}\lambda$, and hence an abelian equivalence
$\sigma_*:\A^{\theta^+}_{\lambda'}(v')\operatorname{-mod}\xrightarrow{\sim}
\A^{\sigma \theta^+}_\lambda(v)\operatorname{-mod}$.

Now let $\alpha\in \Pi^\theta$. In general,  $\sigma^{-1}\alpha$ is not
a simple root. However, we can modify $\theta$ staying in the same Weyl chamber
for $\a$ and in the Tits cone for $\g(Q)$ (using a wall-crossing functor through
non-essential walls) so that the $\g(Q)$-chamber of $\theta$
is adjacent to a wall for $\a$. Then we can, in addition, assume $\sigma^{-1}(\alpha)$ is a simple root
for $\g(Q)$, say $\alpha^i$.

\begin{defi} The functors $F_\alpha,E_\alpha$ are, by definition,
obtained by transferring Webster's functors $F_i,E_i$ using equivalences
\begin{align*}&\sigma_*:\A^{\theta^+}_{\lambda'}(v')\operatorname{-mod}\xrightarrow{\sim}
\A^{\sigma \theta^+}_\lambda(v)\operatorname{-mod}, \\
&\sigma_*:\A^{\theta^+}_{\lambda''}(v'+\alpha^i)\operatorname{-mod}\xrightarrow{\sim}
\A^{\sigma \theta^+}_\lambda(v+\alpha)\operatorname{-mod}.\end{align*}
Here  $\lambda',\lambda''$ are given
by $\lambda'=\sigma^{-1}\bullet^v \lambda, \lambda''=\sigma^{-1}\bullet^{v+\alpha}\lambda$.
Note that $\lambda''-\lambda'\in \Z^{Q_0}$.

More precisely,
$$F_\alpha:= \sigma_*\circ F_i\circ\mathcal{T}_{\lambda',\lambda''-\lambda'}\circ\sigma_*^{-1}$$
and $E_\alpha$ is defined in  a similar fashion.
\end{defi}

By the construction and (1) of Lemma \ref{Lem:W_funct_prop} the functors $E_\alpha,F_\alpha$ preserve
$$\bigoplus_v D^b_{\rho^{-1}(0)}(\A_\lambda^\theta(v)\operatorname{-mod}).$$

\begin{Rem}
Of course, our construction of the functors $E_\alpha,F_\alpha$ depends on the choice
of a suitable $\g(Q)$-chamber in the $\a$-chamber of $\theta$. In a subsequent paper
the second named author plans to check that the functors $E_\alpha, F_\alpha$
are well-defined and, in fact, give a categorical $\a$-action on
$\bigoplus_v D^b(\A_\lambda^\theta(v)\operatorname{-mod})$, at least when $\a$
is simply laced. It is expected that this result will allow to compute the Euler
characteristics of $R\Gamma(M)$ for
$M\in \operatorname{Irr}(\A_\lambda^\theta(v)\operatorname{-mod}_{\rho^{-1}(0)})$.
\end{Rem}


\section{Roadmap}\label{S_outline}
In this section we will explain key ideas and steps in the proof of Theorem
\ref{Thm:verymain} that will be carried out in the subsequent sections.

Recall that $\lambda\in \paramq$ is chosen in such a way that $\A_\lambda(v)$ has finite homological dimension.
By Proposition \ref{Prop:abelian_loc_BPW} we can always achieve that by replacing $\lambda$ with some
element of $\lambda+\Z^{Q_0}$.

Conjecture \ref{Conj:main} boils down to the following three claims.
\begin{itemize}
\item[(I)] The image of $\CC_\lambda^\theta$ contains $L^{\a}_\omega[\nu]$.
\item[(II)] The image of $\CC^\theta_\lambda$ is contained in $L^{\a}_\omega[\nu]$.
\item[(III)] The map $\CC^\theta_\lambda: K_0(\A_\lambda(v)\operatorname{-mod}_{fin})\rightarrow L_\omega[\nu]$
is injective.
\end{itemize}

\subsection{Outline of proof: lower bound on the image}\label{SS_CC_image_lower_bound}
We proceed to outlining the key ideas of our proof of Theorem \ref{Thm:verymain} starting with
(I): $\operatorname{im}\CC^\theta_\lambda\supset L_\omega^{\a}$ (we will also see that $\CC^\theta_\lambda$
does not depend on the choice of $\theta$). For this we need to establish a generalization of
(2) of Lemma \ref{Lem:W_funct_prop} that will be proved in Section \ref{SS:K_0_action}.

\begin{Prop}\label{Prop:a_CC_intertw}
For $M\in D^b(\A_\lambda^\theta(v)\operatorname{-mod}_{\rho^{-1}(0)})$, we have
$\CC(E_\alpha M)=\pm e_\alpha \CC(M)$ and $\CC(F_\alpha M)=\pm f_\alpha \CC(M)$.
\end{Prop}

In fact, we can do even better and prove an analog of this proposition for $K_0(\Coh_{\rho^{-1}(0)}\M^\theta(v))$
instead of $H_{mid}(\M^\theta(v))$ and the degeneration map $[M]\mapsto [\gr M]$ instead of $\CC$,
this is done in Section \ref{SS:K_0_action}.

Let us explain ideas behind a proof of Proposition \ref{Prop:a_CC_intertw} (similar ideas are used
to prove a stronger result from Section \ref{SS:K_0_action}, Proposition \ref{Prop:a_actions_intertwined}).
Since the homologies of all $\M^\theta(v)$, for generic $\theta$, are identified (see \ref{SS_homol_ident}), the LMN isomorphisms give rise to a $W(Q)$-action on $L_\omega$. We also have an action of a suitable  extension (by a 2-torsion group) of $W(Q)$ on $L_\omega$ coming from the $\g(Q)$-action. If we knew that on each weight space
the two actions coincide up to a sign, then we would have $\sigma_*^{-1} \CC(E_\alpha) \sigma_*=c_v e_{\sigma^{-1}(\alpha^i)}=c_v e_\alpha$ (for $c_v=\pm 1$) and similarly for $F$'s. So, to prove
Proposition \ref{Prop:a_CC_intertw}, we need to establish the coincidence of the two group actions. Of course, it is enough to check the equality for simple reflections, $s_i$.

We will check that rather indirectly: on a categorical level. Namely, assume that $\theta_k>0 $ for all $k$
and $\lambda_i$ is integral. Set $v':=s_i\bullet v$. Recall that the functors $E_i,F_i$ give rise to a categorical $\sl_2$-action and
hence produce derived equivalences (convolutions with Rickard complexes)
$\Theta_i: D^b(\A_\lambda^\theta(v)\operatorname{-mod})
\xrightarrow{\sim} D^b(\A_\lambda^\theta(s_i\bullet v)\operatorname{-mod})$.
On the other hand, suppose that $(\lambda,\theta)\in \mathfrak{AL}(s_i\bullet v)$
(and hence $(\lambda',s_i\theta)\in \mathfrak{AL}(v)$ for $\lambda'=s_i\bullet^{v'}\lambda$).
Then we can consider the wall-crossing functor $$\WC_{\lambda\rightarrow \lambda'}: D^b(\A_\lambda^\theta(v)\operatorname{-mod})\xrightarrow{\sim} D^b(\A^{s_i\theta}_{\lambda'}(v)\operatorname{-mod}).$$

\begin{Thm}\label{Thm:WC}
Assume that $(\lambda,\theta)\in \mathfrak{AL}(v), \lambda_i\in \Z_{\geqslant 0}$ and $\theta_k>0$
for all $k\in Q_0$. Then we have an isomorphism of functors $\Theta_i=s_{i*}\circ\WC_{\lambda\rightarrow
\lambda'}$.
\end{Thm}


We are going to use Theorem \ref{Thm:WC} to show that $\CC(E_\alpha)$ acts as $\pm e_\alpha$ on $\operatorname{Im}\CC_v$
and that $\CC(F_\alpha)$ acts as $\pm f_\alpha$ on $\operatorname{Im}\CC_{v+\epsilon_i}$.
For this we need to check that $\CC(\Theta_i)$ acts by $\operatorname{Im}\CC_v$ by $\pm s_i$
(by $s_i$ we denote an operator on $L_\omega$ induced by the $\g(Q)$-action, it is defined
up to a sign). This follows from (2) of Lemma \ref{Lem:W_funct_prop}
when we use the stability condition $\theta^+$ but this is not straightforward for $\sigma\theta$.
On the other hand, we need to show that $\CC(s_{i*}\circ\WC_{\lambda\rightarrow \lambda'})$ coincides with $s_{i*}$ (the operator of the $W(Q)$-action on the middle homology). Both claims follow
from the next proposition. Let us write $\CC^\theta_{\lambda,v}$ for the characteristic cycle map
$K_0(\A_\lambda(v)\operatorname{-mod}_{fin})\rightarrow L_\omega[\nu]$
defined using the stability condition $\theta$.

\begin{Prop}\label{Prop:WC_vs_CC}
The following statements are true:
\begin{enumerate}
\item
If the homological dimension
of $\A_\lambda(v)$ is finite and $\lambda\in \paramq^{ISO}$, then the map $\CC^\theta_{\lambda,v}$
is independent of $\theta$.
\item For $M\in D^b_{\rho^{-1}(0)}(\A_{\lambda'}^{\theta'}(v)\operatorname{-mod})$ we have $\CC(M)=\CC(\WC_{\lambda'\rightarrow \lambda} M)$.
\end{enumerate}
\end{Prop}

Again, we have an analog of this proposition (and of the equality $s_i=\pm [s_{i*}]$) on
$\bigoplus_v K_0(\operatorname{Coh}_{\rho^{-1}}(0)(\M^\theta(v)))$, see Sections
\ref{SS_WC_K0} and \ref{SS:K_0_action}.

We will deduce the coincidence of the two group actions from Theorem \ref{Thm:WC} and Proposition \ref{Prop:WC_vs_CC}.
Also Theorem \ref{Thm:WC} that can be regarded as a formula for $\WC_{\lambda\rightarrow s_i\bullet\lambda}$
will play an important role in proving (II), see the next section.

\subsection{Outline of proof: upper bound on the image}\label{SS_upper_bound_outline}
We prove (II), the inclusion $\operatorname{im}\CC_\lambda\subset L_\omega^{\a}$, using  wall-crossing
functors.

More precisely, we will prove the following claim. Let $\mathcal{C}$ denote the full subcategory
of $\bigoplus_v \A_\lambda^\theta(v)\operatorname{-mod}_{\rho^{-1}(0)}$ consisting of all modules
that appear in the homology of complexes of the form $\mathcal{F}L_0$,
where $\mathcal{F}$ is a monomial in the functors $E_\alpha, F_\alpha, \alpha\in \Pi^\theta,$
(here, as usual, $\Pi^\theta$ is the simple root system for $\a$
defined in \ref{SSS_Webster_general})
and $L_0\in \A_\lambda^\theta(\sigma\bullet w)$, where $\sigma\in W(Q)$ is such that $\sigma \omega$
is dominant for $\a$.

\begin{Prop}\label{Prop:fin_dim_cryst}
We have $\mathcal{C}=\bigoplus_v \A_\lambda^\theta(v)\operatorname{-mod}_{\rho^{-1}(0)}$.
\end{Prop}

This proposition will be proved in Section  \ref{S_proof_compl}.
In Section \ref{S_fin_WC}, we will see that Proposition \ref{Prop:fin_dim_cryst}
implies (II). Conversely, one can show that, modulo (III), (II) implies
Proposition \ref{Prop:fin_dim_cryst}.

A basic strategy of proving Proposition \ref{Prop:fin_dim_cryst} is as follows (we will elaborate on
the strategy below in this section).

\begin{enumerate}
\item Characterize the finite dimensional modules using a ``long wall-crossing functor''.
\item Reduce the study of the ``long wall-crossing functor'' to the study of ``short wall-crossing
functors'' -- between chambers sharing an essential wall.
\item Study short wall-crossing functors using Theorem \ref{Thm:WC} and categorical $\slf_2$-actions.
\item In the case of affine quivers, study the wall-crossing functor crossing the affine wall
$\ker\delta$.
\end{enumerate}

\subsubsection{Long wall-crossing functor}
Our first goal is to characterize the dimension of support of a simple $\A_\lambda(v)$-module in terms of
a functor. It turns out that the functor we need is a ``long wall-crossing functor'' defined as follows.
Let $\lambda,\theta$ be such that $(\lambda+k\theta,\theta)\in \mathfrak{AL}(v)$
for any $k\in \Z_{\geqslant 0}$. Thanks to Proposition \ref{Prop:abelian_loc_BPW}, we can find
$\lambda^-\in \lambda+\Z^{Q_0}$ such that $(\lambda^--k\theta,-\theta)\in \mathfrak{AL}(v)$ for any
$k\geqslant 0$. By a long wall-crossing functor we mean
$\WC_{\lambda\rightarrow \lambda^-}:D^b(\A_\lambda(v)\operatorname{-mod})\xrightarrow{\sim}
D^b(\A_{\lambda^-}(v)\operatorname{-mod})$.

We have recalled the definition of  holonomic modules for $\A_\lambda(v),\A_\lambda^\theta(v)$ in
Section \ref{SSS_Supp_CC}. In particular, modules from $\A_\lambda(v)\operatorname{-mod}_{fin},
\A_\lambda^\theta(v)\operatorname{-mod}_{\rho^{-1}(0)}$ are holonomic.

The following proposition is inspired by a similar result on
the BGG category $\mathcal{O}$, see \cite[Proposition 4.7]{BFO}.

\begin{Prop}\label{Prop:long_shift}
Let $M$ be a  holonomic $\A_\lambda(v)$-module. Then
\begin{enumerate} \item $H_i(\WC_{\lambda\rightarrow\lambda^-}M)=0$ if
$i<\frac{1}{2}\dim \M^\theta(v)-\dim \operatorname{Supp} M$ or $i>\frac{1}{2}\dim \M^\theta(v)$.
 Moreover, $H_i(\WC_{\lambda\rightarrow\lambda^-}M)\neq 0$ for $i=\frac{1}{2}\dim \M^\theta(v)-
\dim \operatorname{Supp} M$.
\item The functor $\WC_{\lambda\rightarrow\lambda^-}[-\frac{1}{2}\dim \M^\theta(v)]$
is an abelian equivalence $\A_\lambda(v)\operatorname{-mod}_{fin}\xrightarrow{\sim} \A_{\lambda^-}(v)\operatorname{-mod}_{fin}$.
\end{enumerate}
\end{Prop}

The minimal number $i$ such that $H_i(\WC_{\lambda\rightarrow \lambda^-}L)\neq 0$
will be called the {\it homological shift} (of $L$ under the functor
$\WC_{\lambda\rightarrow \lambda^-}$).

The proof to be given below, Section \ref{S_long_WC}, is based on using abelian localization as well as a connection
between the long wall-crossing functor and a homological duality functor.

One problem with the wall-crossing functor is that it is quite hard to study it (in particular,
computing the homological shifts) directly. Instead, we will decompose $\WC_{\lambda\rightarrow \lambda^-}$
into a composition of short wall-crossing functors (i.e., functors crossing a single wall between
two adjacent chambers) using Theorem \ref{Thm:wc_decomp_short}.  The information
about homological shifts under the short wall-crossing functors together with Theorem
\ref{Thm:wc_decomp_short} turn out to be enough to establish (II) for finite and affine quivers $Q$,
see Section \ref{SS_extrem_abs}
(in this paper we only deal with very special framing in  the affine case but
this can be generalized to arbitrary framing by suitably generalizing techniques that
we use here, see the subsequent paper \cite{perv}). The case of wild
quivers poses some additional essential difficulties, we will discuss why in
\ref{SSS_short_affine_WC}.

\subsubsection{Short wall-crossing through real wall}\label{SSS_short_wc_outline}
We want to characterize objects with zero homological shift  under the short
wall-crossing functor through a wall defined by a real root.

Let $\alpha\in \Pi^\theta$. We say that
an object $L\in \Irr(\A_\lambda^\theta(v)\operatorname{-mod}_{\rho^{-1}(0)})$ is
{\it $\alpha$-singular} if
\begin{itemize}
\item $(\nu,\alpha)\geqslant 0$ and $[L]\not \in \operatorname{im}[F_\alpha]$ or
\item $(\nu,\alpha)\leqslant 0$ and $[L]\not \in \operatorname{im}[E_\alpha]$.
\end{itemize}
Here and below we write $[L]$ for the class of $L$ in $K_0(\A_\lambda^\theta(v)\operatorname{-mod}_{\rho^{-1}(0)})$.
We write $[E_\alpha],[F_\alpha]$ for the maps between the $K_0$ spaces induced by the functors $E_\alpha,F_\alpha$.


There is an important alternative  characterization of $\alpha$-singular objects
based on Theorem \ref{Thm:WC}. Namely, suppose that $\theta'$ is a generic
stability condition separated from $\theta$ by the single wall $\ker\alpha$
(meaning the chambers of $\theta,\theta'$ share the common wall $\ker\alpha$).

\begin{Prop}\label{Prop:singular_equiv}
Let $L\in \operatorname{Irr}(\A_\lambda^\theta\operatorname{-mod}_{\rho^{-1}(0)})$. The following conditions are equivalent:
\begin{enumerate}
\item $L$ is $\alpha$-singular.
\item $H_0(\WC_{\theta\rightarrow \theta'}L)\neq 0$.
\end{enumerate}
Moreover, under these equivalent conditions there is a unique $\alpha$-singular simple constituent of $H_*(\WC_{\theta\rightarrow \theta'}L)$, say $L'$, and it is a quotient of $H_0(\WC_{\theta\rightarrow \theta'}L)$. The map $L\mapsto L'$ is a bijection between the sets of $\alpha$-singular simples in $\A_\lambda^\theta\operatorname{-mod}_{\rho^{-1}(0)}, \A_\lambda^{\theta'}\operatorname{-mod}_{\rho^{-1}(0)}$.
\end{Prop}
This proposition will be proved in Section \ref{SS_fin_sing}.

\subsubsection{Short wall-crossing through affine wall}\label{SSS_short_affine_WC}
Thanks to the previous paragraph, we only need to show (modulo technicalities to be
addressed later) that the wall-crossing functor through $\ker\delta$ cannot homologically shift
a simple by more than $\dim \M^\theta(v)/2-1$.

Let us explain an idea of the proof of this claim, which is an extension of what was
done in the proof of Proposition \ref{Prop:wall_non_essent}. By Proposition \ref{Lem:shift_coinc},
the wall-crossing  functor through $\ker\delta$ is given by
$\A^0_{\lambda,\chi}(v)\otimes^L_{\A^0_\lambda(v)}\bullet$ (here we assume that $\lambda,\lambda+\chi\in \paramq^{iso}$
and there are generic stability conditions $\theta,\theta'$ separated by $\ker\delta$ such that
$(\lambda,\theta),(\lambda+\chi,\theta')\in \AL(v)$). Now let us set $\paramq_0:=\lambda+\ker\delta,
\paramq_0':=\lambda+\chi+\ker\delta$. Then $\A^0_{\lambda,\chi}(v)$ is the specialization
of $\A^0_{\paramq_0,\chi}(v)$. We will show that ``homological shift behavior'' of the
functors $\A^0_{\lambda_1,\chi}(v)\otimes_{\A^0_{\lambda_1}(v)}\bullet$ is the same for
Zariski generic parameters $\lambda_1\in \paramq_0$ (this is the most non-trivial
part of the proof; the very first step here is results
from \ref{SSS_param_supp}). Then we need to show that for a {\it Weil generic}
$\lambda_1$ the homological shifts under the functor $\A^0_{\lambda_1,\chi}(v)\otimes^L_{\A^0_{\lambda_1}(v)}\bullet$
are less than $\frac{1}{2}\dim \M^\theta(v)$ (recall that ``Weil generic'' means ``lying outside
of countably many proper closed algebraic subvarieties''). The point of considering Weil generic
parameters is that here the functor $\A^0_{\lambda_1,\chi}(v)\otimes^L_{\A^0_{\lambda_1}(v)}\bullet$
becomes the long wall-crossing functor (indeed, all walls but possibly $\ker\delta$
are non-essential). Basically, we prove that the algebra $\A^0_{\lambda_1}(v)$ has no finite dimensional simples, and our claim about homological shifts follows from Proposition \ref{Prop:long_shift}.

In fact, we show more than the bound for homological shifts, we prove that $\WC_{\lambda\rightarrow \lambda+\chi}$
is {\it perverse} in the sense of Chuang and Rouquier, \cite[Section 2.6]{rouquier_ICM}. We
characterize filtrations on the categories $\A^0_\lambda(v)\operatorname{-mod}, \A^0_{\lambda+\chi}(v)\operatorname{-mod}$
that make $\WC_{\lambda\rightarrow \lambda+\chi}$ perverse and deduce our claim about homological
shifts from there. We use results on the representation theory of type A Rational Cherednik
algebras to establish the perversity, which leads to restrictions on the framing.

Let us explain the most essential reason why we restrict to finite and affine quivers.
In fact, as shown in the subsequent paper \cite{perv}, wall-crossing functors are perverse
in a much more general situation (including all wall-crossings through hyperplanes for wild quivers).
A difficulty of dealing with wild quivers is that one may need to cross many
walls defined by imaginary roots and we do not know how to control the homological
shifts of compositions in that case.

\subsubsection{Completion of the proof}
Let us define  {\it extremal} objects.

\begin{defi}\label{defi:extremal}
We say that $L\in \operatorname{Irr}(\A_\lambda^\theta(v)\operatorname{-mod}_{\rho^{-1}(0)})$
is {\it extremal} if $L$ does not lie in the category $\mathcal{C}$ from \ref{SS_upper_bound_outline}
and $v$ is minimal such that $L$ exists.
\end{defi}

Of course, (II) is equivalent to the claim that no extremal objects exist.

Here are two important properties of extremal objects.

\begin{Lem}\label{Lem:extremal_sing}
An extremal object is singular for all $\alpha$.
\end{Lem}
\begin{proof}
Let $L\in \operatorname{Irr}(\A_\lambda^\theta(v)\operatorname{-mod}_{\rho^{-1}(0)})$ be extremal.
By the minimality assumption on $v$, we see that $\nu$ is dominant and
$[L]\not\in \sum_\alpha \operatorname{im} [F_\alpha]$. In particular, $L$ is $\alpha$-singular
for all  $\alpha\in \Pi^\theta$.
\end{proof}

The following is a crucial property of extremal objects. It will be proved
in Section \ref{S_proof_compl}.

\begin{Prop}\label{Prop:extremal_bij}
The bijection $L\mapsto L'$ from Proposition \ref{Prop:singular_equiv}
restricts to a bijection between the sets of extremal simples in
$\A_\lambda^\theta\operatorname{-mod}_{\rho^{-1}(0)},
\A_\lambda^{\theta'}\operatorname{-mod}_{\rho^{-1}(0)}$.
\end{Prop}

We conclude that extremal objects are not homologically shifted by short wall-crossing functors
through real walls (=walls defined by real roots). This finishes the proof of (II)
in the case when $Q$ is finite. To deal with the case of affine $Q$ (under our restrictions
on the framing) one uses results outlined in \ref{SSS_short_affine_WC}.

%

\subsection{Outline of proof: injectivity of $\CC$}
Now we will explain how Proposition \ref{Prop:fin_dim_cryst} implies (III).

Suppose first, that we know that
\begin{itemize}
\item[(*)]
the endomorphisms $[E_\alpha],[F_\alpha], \alpha\in \Pi^\theta,$
define a representation of the Lie algebra $\a$ in $\bigoplus_{v}K_0(\A_\lambda^\theta(v)\operatorname{-mod}_{\rho^{-1}(0)})$
\end{itemize}
(so far we know that each pair $([E_\alpha],[F_\alpha])$ defines a representation of $\slf_2$).
So $\CC_\lambda$ becomes an epimorphism of $\a$-modules. Using Proposition \ref{Prop:fin_dim_cryst} we will show
that it is an isomorphism.

It remains to establish (*). In fact, this reduces to the case when
$\lambda$ is rational (where we have a powerful tool -- reduction to positive characteristic).
The general case will be deduced from there and Proposition \ref{Prop:fin_dim_cryst}.

To prove (*) we will argue more or less as follows.  We will show that the degeneration map, $[M]\mapsto [\gr M]$
defines an embedding $K_0(\A_\lambda^\theta(v)\operatorname{-mod}_{\rho^{-1}(0)})\hookrightarrow
K_0(\Coh_{\rho^{-1}(0)}\M^\theta(v))$. We will see that $\a$ naturally acts on
$K_0(\Coh_{\rho^{-1}(0)}\M^\theta(v))$ (in fact, the whole algebra $\g(Q)$ does)
and our embedding (after some twist) intertwines $[E_\alpha]$ with $e_\alpha$, $[F_\alpha]$ with $f_\alpha$.
The proofs here are based on $K$-theory version of results mentioned in
Section  \ref{SS_CC_image_lower_bound}. This proves (*) and finishes the proof
of Theorem \ref{Thm:verymain}.

\subsection{Subsequent content}
Let us describe the content of the following sections. Section \ref{S_quant_pos} is preparatory,
there we discuss some relatively standard results based on the reduction to characteristic $p$.
Then, in Section \ref{S_lower_bound_proof} we prove (I) (the inclusion $L_\omega^{\a}[\nu]\subset
\operatorname{im}\CC_\lambda$) as well as some stronger results needed in the proofs of (II)
and (III) (concerning $K$-theory rather than middle homology).

In the subsequent three sections we study wall-crossing functors. In Section \ref{S_long_WC}
we study the long wall-crossing functor and relate the homological shifts under this functor
to codimension of support. In Section \ref{S_fin_WC} we study short wall-crossing functors
through walls defined by real roots and their interactions with singular simples.
In Section \ref{S_affine} we study the short wall-crossing functor through the wall $\ker\delta$
and prove that it is a perverse equivalence.

Finally, in Section \ref{S_proof_compl} we finish the proofs of (II) and (III) and hence
of Theorem \ref{Thm:verymain}. We also discuss some generalizations of Conjecture \ref{Conj:main}.

\section{Quantizations in positive characteristic and applications}\label{S_quant_pos}
In this section we deal with quantizations in positive characteristic and their applications
to characteristic $0$.

Let us explain two main applications first. They concern the existence of tilting generators
on $\M^\theta(v)$ with some special properties and the injectivity of a natural map $K_0(\A_\lambda(v)\operatorname{-mod}_{fin})\rightarrow
K_0(\Coh_{\rho^{-1}(0)}(\M^\theta(v)))$ that is the composition of the localization
and degeneration (here we assume that the homological
dimension of $\A_\lambda(v)$ is finite and $\lambda\in \Q^{Q_0}$).

Let us state a result about a tilting bundle. We say that  a vector bundle
$\mathcal{P}$ on a smooth algebraic variety $X$ is a {\it tilting generator}
if $\operatorname{Ext}^i(\mathcal{P},\mathcal{P})=0$ for $i>0$ and the
homological dimension of the algebra $\operatorname{End}(\mathcal{P})$ is finite.
When $X$ is a Nakajima quiver variety (in fact, under some more general assumptions) and $\mathcal{P}$
is a tilting generator, the functor $R\Hom_{\mathcal{O}_X}(\mathcal{P},\bullet)$
is an equivalence $D^b(\operatorname{Coh}X)\xrightarrow{\sim} D^b(\operatorname{End}(\mathcal{P})^{opp}\operatorname{-mod})$
(see \cite[Proposition 2.2]{BK2}).

%

Here is our main result on the existence of compatible tilting generators
on the quiver varieties $\M^\theta(v)$.

\begin{Prop}\label{Prop:tilting_gen}
There is a $\C^\times$-equivariant (with respect to the contracting action)
tilting generator $\mathcal{P}^\theta$ on $\M^\theta(v)$ such that the
algebra $\End(\mathcal{P}^\theta)$ is independent of $\theta$.
\end{Prop}

Such a bundle $\mathcal{P}^\theta$ is constructed by Kaledin in \cite{Kaledin}.
Our construction is quite similar to Kaledin's and is also inspired by
an earlier construction in \cite{BK2}. Namely, one fixes a suitable quantization
of $\M^\theta(v)$ over an algebraically closed field $\Fi$ of positive characteristic.
It is an Azumaya algebra on the Frobenius twist  $\M^\theta(v)^{(1)}$
which then can be shown to split on $$\M^\theta(v)^{(1)\wedge_0}=\operatorname{Spec}(\Fi[\M(v)]^{\wedge_0}).$$
The splitting bundle then extends to a vector bundle that is shown to be tilting.
Our proof of the splitting result is easier than Kaledin's. Besides, for our
next main result of this section we need to use a particular choice of a quantization:
one obtained by Hamiltonian reduction.

Now let us proceed to the second main result in this section: on the injectivity
of the natural map $K_0(\A_\lambda(v)\operatorname{-mod}_{fin})\rightarrow
K_0(\Coh_{\rho^{-1}(0)}(\M^\theta(v)))$. Let $\lambda$
be such that $\A_\lambda(v)$ is regular so that the localization functor
$L\Loc_\lambda^\theta$ gives rise to an identification $K_0(\A_\lambda(v)\operatorname{-mod}_{fin})
\xrightarrow{\sim} K_0(\A_\lambda^\theta(v)\operatorname{-mod}_{\rho^{-1}(0)})$. We have a well-defined
map $$[M]\mapsto [\gr M]: K_0(\A_\lambda^\theta(v)\operatorname{-mod}_{\rho^{-1}(0)})
\rightarrow K_0(\Coh_{\rho^{-1}(0)}(\M^\theta(v))).$$ Let us denote the composition
$$K_0(\A_\lambda(v)\operatorname{-mod}_{fin})\rightarrow
K_0(\A_\lambda^\theta(v)\operatorname{-mod}_{\rho^{-1}(0)})\rightarrow
K_0(\Coh_{\rho^{-1}(0)}(\M^\theta(v)))$$ by $\gamma_\lambda^\theta$.

\begin{Prop}\label{Prop:inj_K0_rat}
Suppose, in addition, that $\lambda\in \Q^{Q_0}$. Then $\gamma_\lambda^\theta$
is injective.
\end{Prop}

The proposition is true for any $\lambda$ and, in fact, follows from the injectivity
of the characteristic cycle map. But at this point we are only  prove
it for rational $\lambda$.

\subsection{Quiver varieties and quantizations in characteristic $p$}\label{SS_pos_char}
We can define the GIT quotient $\M^\theta(v)$ in characteristic $p$ for $p$ large enough.
More precisely, the moment map $\mu: T^*R\rightarrow \g$ is defined over $\Z$. So we can
reduce it modulo $p$ and get $\mu_\Fi:T^*R_\Fi\rightarrow \g_\Fi$.
For $p$ large enough, this is still a moment
map and we can form the Hamiltonian reduction $\M^\theta(v)_{\Fi}$ that is a smooth symplectic
algebraic variety over $\Fi$.

\begin{Lem}\label{Lem:quiv_var_base_change}
There is a finite localization $S$ of $\Z$ and a smooth symplectic scheme $\M^\theta(v)_S$
over $\Spec(S)$ with the following properties:
\begin{enumerate}
\item $\M^\theta(v),\M^\theta(v)_{\Fi}$ are obtained from $\M^\theta(v)_S$ by base change
(for every $S$-algebra $\Fi$ that is an algebraically closed field).
\item $\C[\M^\theta(v)],\Fi[\M^\theta(v)_{\Fi}]$ are obtained from $S[\M^\theta(v)_S]$ by base change.
\item $H^i(\M^\theta(v)_S, \mathcal{O}_{\M^\theta(v)_S})=0$,
$H^i(\M^\theta(v)_{\Fi}, \mathcal{O}_{\M^\theta(v)_{\Fi}})=0$ for $i>0$, where $\Fi$ is as in (1).
\end{enumerate}
\end{Lem}
\begin{proof}
We remark that $\mu^{-1}(0)^{\theta-ss}\rightarrow \M^\theta(v)$ is a principal $G$-bundle,
in particular, it is locally trivial in the Zariski topology. It is defined over some
finite localization $S$ of $\Z$. After a finite localization, $\mu_S^{-1}(0)^{\theta-ss}$ --
the stable locus of $\Spec(S[T^*R_S]/(\mu^*_S(\g_S)))$ --
becomes the total space of this principal bundle. (i) follows.

Fix an open affine cover of $\M^\theta(v)_S$. After a finite localization all cocycle groups
in the positive degree part of the  \v{C}ech complex for $\mathcal{O}_{\M^\theta(v)_S}$  coincide with the corresponding
coboundary groups and they are free over $S$.  (2) and (3) follow.
\end{proof}

This result can be generalized to $\M^\theta_{\param}(v)$ in a straightforward way.

Quantizations of $\M^\theta(v)_{\Fi}$ were studied in \cite{BFG}, see Sections 3,4,6 there.
Take $\lambda\in \Fi_p^{Q_0}$. The algebra $D(R)_\Fi$ is Azumaya over the Frobenius twist $\Fi[T^*R]^{(1)}$ so we can view $D_{R,\Fi}$ as a coherent sheaf on $(T^*R_\Fi)^{(1)}$.
According to \cite[Section 3]{BFG}, $$\A_\lambda^\theta(v)_\Fi:=[\mathcal{Q}_{\Fi,\lambda}|_{(T^*R_{\Fi})^{(1),\theta-ss}}]^{G_\Fi}$$
is a sheaf of Azumaya algebras on $\M^\theta(v)_{\Fi}^{(1)}$. If we consider this sheaf in the conical topology,
it becomes filtered, and the associated graded is $\operatorname{Fr}_*\mathcal{O}_{\M^\theta(v)_\Fi}$.

There is an extension of this construction to $\lambda\in \Fi^{Q_0}$. The difference is that $\A_\lambda^\theta(v)_{\Fi}$ is now an Azumaya algebra over $\M^\theta_{\operatorname{AS}(\lambda)}(v)$, where $\operatorname{AS}$ is the Artin-Schreier map, see \cite[Section 3.2]{BFG}. We also have a version that works in families.
We get a sheaf $\A^\theta_{\paramq}(v)_\Fi$ of Azumaya algebra over $\param_{\Fi}\times_{\param_{\Fi}^{(1)}}\M^\theta_{\param}(v)^{(1)}_{\Fi}$
that specializes to $\A^\theta_{\lambda}(v)_\Fi$ for any $\lambda\in \Fi^{Q_0}$.

Let us write $\A_\lambda(v)_{\Fi}$ for the global sections of $\A^\theta_\lambda(v)_\Fi$.

\begin{Lem}\label{Lem:quant_base_change}
Fix $\lambda^\circ\in \Q^{Q_0}$.  There is a finite localization $S$ of $\Z$ with the following
 property: for any $\lambda\in \lambda^\circ+\Z^{Q_0}$ there exists a filtered $S$-algebra $\A_\lambda(v)_S$ such that
\begin{enumerate}
\item $\gr\A_\lambda(v)_S=S[\M^\theta(v)_S]$.
\item $\C\otimes_S \A_\lambda(v)_S=\A_\lambda(v)$.
\item $\Fi\otimes_S \A_\lambda(v)_S=\A_\lambda(v)_{\Fi}$.
\end{enumerate}
\end{Lem}
\begin{proof}
We may assume that Lemma \ref{Lem:quiv_var_base_change} holds for $S$ and moreover that $\lambda^\circ\in S^{Q_0}$ and that $\mu_S$ is flat. We can define the microlocal quantizations $\A^\theta_{\lambda}(v)_S$ of $\M^\theta(v)_S$ in the same way as was done for the complex numbers. Set $\A_\lambda(v)_S:=\Gamma(\A^\theta_\lambda(v)_S)$.

(1) follows from (2) and (3) of Lemma \ref{Lem:quiv_var_base_change}. The microlocal quantization $\A_\lambda^\theta(v)$ is obtained from $\A^\theta_\lambda(v)_S$ by the base change to $\C$ followed by a suitable completion (needed to preserve the condition that the sheaf is still complete and separated with respect to the filtration). Note that the algebras $\A_\lambda(v)_S,\A_\lambda(v),
\A_\lambda(v)_{\Fi}$
are $\Z_{\geqslant 0}$-filtered so completing with respect to the filtration
does not change these algebras.
(2) follows. To prove (3), we notice that we have a natural homomorphism $\A^\theta_\lambda(v)_\Fi\rightarrow \operatorname{Fr}_*(\Fi\otimes_S \A^\theta_\lambda(v)_S)$. On the level of the associated graded sheaves, it is the identity automorphism of
$\operatorname{Fr}_*\mathcal{O}_{\M^\theta(v)_{\Fi}}$.  So it gives rise to an isomorphism
$\A_\lambda(v)_\Fi=\Gamma\left(\M^\theta(v)_\Fi, \Fi\otimes_S \A^\theta_\lambda(v)_S\right)=\Fi\otimes_S \A_\lambda(v)_S$.
\end{proof}

\subsection{Splitting}\label{SS_splitting}
We set
$$\M(v)_{\Fi}^{(1)}:=\Spec(\Fi[\M^\theta(v)_{\Fi}^{(1)}]), \M_{\param}(v)^{(1)}_\Fi:=\param_{\Fi}\times_{\param_{\Fi}^{(1)}}\Spec(\Fi[\M_{\param}^\theta(v)^{(1)}_{\Fi}]).$$
In this section and the next one we write $ \M^\theta_{\param}(v)_\Fi^{(1)\wedge_0}$ for the formal neighborhood of $\M^{\theta}(v)^{(1)\wedge_0}_\Fi$ in \begin{equation}\label{eq:scheme}\M^\theta_\param(v)^{(1)}_\Fi\times_{\M^0_{\param}(v)^{(1)}_\Fi}
\M^0_\param(v)^{(1)\wedge_0}_\Fi.
\end{equation}
In particular, $\M^\theta_{\param}(v)_\Fi^{(1)\wedge_0}$ is a formal scheme and not a scheme.

\begin{Prop}\label{Prop:split}
The restrictions $\A^\theta_\lambda(v)_{\Fi}^{\wedge_0}$ of $\A^\theta_\lambda(v)_{\Fi}$ to $\M^\theta(v)^{(1)\wedge_0}_{\Fi}$ and  $\A^\theta_{\paramq_\Fi}(v)^{\wedge_0}$ of
$\A^\theta_{\paramq_\Fi}(v)$ to $\M_{\param_\Fi}^\theta(v)^{(1)\wedge_0}$
(where the fiber of $\A^\theta_{\paramq_\Fi}(v)$ over $0\in \param$ is $\A^\theta_\lambda(v)_{\Fi}$) split.
\end{Prop}
\begin{proof}
Let us prove the claim about $\A^\theta_\lambda(v)_{\Fi}^{\wedge_0}$ first.

{\it Step 1}. Consider the one-form $\beta$ on $\M^\theta(v)^{(1)}_\Fi$   obtained by pairing of the symplectic form $\omega$ with the vector field for the $\Fi^\times$-action (induced from the fiberwise dilation
action on $T^*R$). We claim that the class of $\A^\theta_\lambda(v)_\Fi$ in the Brauer group $\operatorname{Br}(\M^\theta(v)_{\Fi}^{(1)})$ comes from  $\beta$ (see \cite[III.4]{Milne} for a general discussion of Azumaya algebras coming from 1-forms).  Let us prove this claim. Let $\pi$ denote the quotient morphism $Z:=(\mu^{(1)}_\Fi)^{-1}(0)^{\theta-ss}\twoheadrightarrow\M^\theta(v)_\Fi^{(1)}$
(by the $G_\Fi^{(1)}$-action). By \cite[Remark 4.1.5]{BFG}, $\pi^*(\A^\theta_\lambda(v)_\Fi)$
is Morita equivalent to $D_{R,\Fi}|_Z$. The class of $D_{R,\Fi}$ comes from the canonical 1-form $\tilde{\beta}$
on $(T^*R_\Fi)^{(1)}$. The restriction of $\tilde{\beta}$ to $Z$ coincides with $\pi^*\beta$. It follows that the
class of $\pi^*\A^\theta_\lambda(v)_\Fi$ in the Brauer group coincides with the class defined by $\pi^*\beta$.
Now recall that $Z$ is a principal $G^{(1)}_{\Fi}$-bundle on $\M^\theta(v)_{\Fi}^{(1)}$ hence it is locally
trivial in Zariski topology. It follows that the restriction of the class of $\A^\theta_\lambda(v)$ to an open subset $U\subset \M^\theta(v)_{\Fi}^{(1)}$ (where the bundle trivializes) coincides with the restriction of the class of $\beta$.
Since the restriction induces an embedding $\operatorname{Br}(\M^\theta(v)_{\Fi}^{(1)})\hookrightarrow \operatorname{Br}(U)$ (\cite[III.2.22]{Milne}), the claim in the beginning of the paragraph is proved.

{\it Step 2}. An Azumaya algebra defined by a one-form $\beta'$ splits provided $\beta'=\alpha-\mathsf{C}(\alpha)$ for
some 1-form $\alpha$, where $\mathsf{C}$ stands for the Cartier map $\Omega^1_{cl}\rightarrow \Omega^1$
(here we write $\Omega^1$ for the bundle of 1-forms and $\Omega^1_{cl}$ for the bundle of closed 1-forms).
We claim that $\mathsf{C}: \Gamma(\M^\theta(v)^{(1)}_\Fi,\Omega^1_{cl})\rightarrow \Gamma(\M^\theta(v)^{(1)}_\Fi,\Omega^1)$
is surjective. This follows from  the following exact sequences of sheaves:
\begin{align}\label{eq:exact1}
& 0\rightarrow \Omega^1_{ex}\rightarrow \Omega^1_{cl}\xrightarrow{\mathsf{C}}\Omega^1\rightarrow 0,\\\label{eq:exact2}
& 0\rightarrow \mathcal{O}^p\rightarrow \mathcal{O}\rightarrow \Omega^1_{ex}\rightarrow 0.
\end{align}
Since $H^i(\M^\theta(v)^{(1)}_\Fi, \mathcal{O})=H^i(\M^\theta(v)_\Fi, \mathcal{O})=0$ for $i=1,2$, (\ref{eq:exact2})
implies $H^1(\M^\theta(v)^{(1)}_\Fi, \Omega^1_{ex})=0$.
Hence, using (\ref{eq:exact1}), we see that
$$\mathsf{C}: \Gamma(\M^\theta(v)^{(1)}_\Fi,\Omega^1_{cl})\twoheadrightarrow \Gamma(\M^\theta(v)^{(1)}_\Fi,\Omega^1).$$

{\it Step 3}. The global sections $\Gamma(\M^\theta(v)^{(1)}_\Fi,\Omega^1_{cl}), \Gamma(\M^\theta(v)^{(1)}_\Fi,\Omega^1)$ are graded with respect to the $\Fi^\times$-action (coming from the dilation action on $T^*R$), let $\Gamma(\ldots)_d$ denote the $d$th graded component. The map $\mathsf{C}$ sends $\Gamma(\M^\theta(v)^{(1)}_\Fi,\Omega^1_{cl})_{d}$ to $\Gamma(\M^\theta(v)^{(1)}_\Fi,\Omega^1)_{d/p}$ if $d$ is divisible by $p$ and to 0 else. Pick $\Fi$-linear sections $\mathsf{C}^{-1}: \Gamma(\M^\theta(v)^{(1)}_\Fi,\Omega^1)_{d}\rightarrow \Gamma(\M^\theta(v)^{(1)}_\Fi,\Omega^1_{cl})_{pd}$.

Let us point out that the degree of $\beta$ is $1$. It follows that $\alpha:=\sum_{i=0}^{+\infty} \mathsf{C}^{-i}(\beta)$
is a well-defined 1-form on $\M^\theta(v)_{\Fi}^{(1)\wedge_0}$. Indeed the $i$th summand has degree $p^{i}$ and so
$\mathsf{C}^{-i}(\beta)$ converges to zero in the topology defined by the maximal ideal of $0$ in $\Fi[\M(v)_{\Fi}^{(1)}]$
because  $\Gamma(\M^\theta(v)_{\Fi}^{(1)},\Omega^1)$
is a finitely generated $\Fi[\M^\theta(v)^{(1)}]$-module.
Clearly, $\beta=\alpha-\mathsf{C}(\alpha)$.

This finishes the proof of the claim that $\A^\theta_\lambda(v)^{\wedge_0}_{\Fi}$ splits.

Let us proceed to the splitting of $\A^\theta_{\paramq_\Fi}(v)^{\wedge_0}$.

{\it Step 4}.  We will prove a stronger statement.
Let $\tilde{\A}$ be an Azumaya algebra over scheme (\ref{eq:scheme}) whose  restriction to $\M^\theta(v)_\Fi^{(1)\wedge_0}$ splits.  We will show that then the restriction of $\tilde{\A}$
to $\M^\theta_{\param}(v)^{(1)\wedge_0}$ splits as well.

Let $\M^\theta_{\param}(v)_{\Fi}^{(1)k}$ denote the $k$th infinitesimal neighborhood of $\M^\theta(v)_\Fi^{(1)\wedge_0}$
in (\ref{eq:scheme}), a scheme over $\operatorname{Spec}(\Fi[\param]/\mathfrak{m}^{k+1})$,
where $\mathfrak{m}$ is the maximal ideal of $0$ in $\Fi[\param]$. We remark that $$H^i(\M^\theta(v)_\Fi^{(1)\wedge_0}, \mathcal{O})=0, \text{ for }i>0$$
 (to simplify the notation we just write $\mathcal{O}$ for the structure sheaf). This follows from $H^i(\M^\theta(v)^{(1)}_{\Fi},\mathcal{O})=0$ and the formal function theorem.

{\it Step 5}. We have a short exact sequence of sheaves on $\M^{\theta}_{\param}(v)_{\Fi}^{(1)\wedge_0}$:
\begin{equation}\label{eq:exact_seq_otimes}0\rightarrow S^k\mathcal{N}\rightarrow \mathcal{O}^\times_{\M^\theta_{\param_\Fi}(v)^{(1)k+1}}\rightarrow
\mathcal{O}^\times_{\M^\theta_{\param_\Fi}(v)^{(1)k}}\rightarrow 0,\end{equation}
where $\mathcal{N}$ is the normal bundle to $\M^\theta(v)^{(1)}_{\Fi}$ in $\M^{\theta}_{\param}(v)^{(1)}_{\Fi}$.

We claim that  $\mathcal{N}=\p\otimes \mathcal{O}$.
First, note that the conormal bundle to $\mu^{-1}(0)^{\theta-ss}$ in $T^*R$ is the trivial bundle with the fiber $\g$:
if $\xi_1,\ldots,\xi_m$ is a basis in $\g$, then $d\mu^*(\xi_1),\ldots,d\mu^*(\xi_m)$
is a basis in the conormal bundle. It follows the conormal bundle to $\mu^{-1}(0)^{\theta-ss}$
in $\mu^{-1}(\g^{*G})^{\theta-ss}$ is trivial with fiber $\g^G$. Since $\mathcal{N}$ is the equivariant
descent of the latter bundle, we get $\mathcal{N}=\p\otimes \mathcal{O}$.

This implies \begin{equation}\label{eq:cohom_vanish_for Br}
H^i(\M^\theta(v)^{(1)\wedge_0}_{\Fi}, S^k\mathcal{N})=0\text{ for }i>0.\end{equation}
In particular, the Picard groups of the schemes $\M^\theta_{\param}(v)_{\Fi}^{(1)k}$ -- equal to
$H^1_{et}(\M^\theta_{\param}(v)_{\Fi}^{(1)k},\mathcal{O}^\times)$-- are naturally identified.

{\it Step 6}.
To check  that the restriction of  $\tilde{\A}$ to $\M^\theta_{\param}(v)_\Fi^{(1)\wedge_0}$ splits,
it is enough to show that $\tilde{\A}|_{\M^\theta_{\param}(v)_{\Fi}^{(1)k}}$ splits
for each $k$. Indeed, let $\mathcal{P}_k$ denote a splitting bundle for
$\tilde{\A}|_{\M^\theta_{\param}(v)_{\Fi}^{(1)k}}$. Such a bundle is defined
up to a twist with a line bundle. So the restriction of $\mathcal{P}_{k+1}$
to $\M^\theta_{\param}(v)_{\Fi}^{(1)k}$ is isomorphic to $\mathcal{P}_k\otimes \mathcal{L}_k$
for some line bundle $\mathcal{L}_k$. By the last paragraph of Step 5, $\mathcal{L}_k$
lifts to a line bundle $\mathcal{L}_{k+1}$ on $\M^\theta_{\param}(v)_{\Fi}^{(1)k+1}$.
Replacing $\mathcal{P}_{k+1}$ with $\mathcal{P}_{k+1}\otimes \mathcal{L}_{k+1}^{-1}$
we achieve that the restriction of $\mathcal{P}_{k+1}$ to
$\M^\theta_{\param}(v)_{\Fi}^{(1)k}$ coincides with $\mathcal{P}_k$.

We show that $\tilde{\A}|_{\M^\theta_{\param}(v)_{\Fi}^{(1)k}}$ splits by using induction on $k$.
The base, $k=1$, has been established before in this proof. Let us establish the induction step.
Recall that $\operatorname{Br}(\M^\theta_{\param}(v)_{\Fi}^{(1)k})\hookrightarrow H^2_{et}(\M^\theta_{\param}(v)_{\Fi}^{(1)k}, \mathcal{O}^\times)$, see \cite[Theorem 2.5]{Milne}. From (\ref{eq:exact_seq_otimes}), (\ref{eq:cohom_vanish_for Br}),
it follows that $H^2_{et}(\M^\theta_{\param}(v)_{\Fi}^{(1)k}, \mathcal{O}^\times)$ and $H^2_{et}(\M^\theta_{\param}(v)_{\Fi}^{(1)k+1},\mathcal{O}^\times)$
are naturally identified. In particular, the claim that the restriction of $\tilde{\A}$ to $\M^\theta_{\param}(v)_{\Fi}^{(1)k}$
splits is equivalent to the vanishing of the class of this restriction
in $H^2_{et}(\M^\theta_{\param}(v)_{\Fi}^{(1)k}, \mathcal{O}^\times)$.
Then the class in $H^2_{et}(\M^\theta_{\param}(v)_{\Fi}^{(1)k+1}, \mathcal{O}^\times)$
of the restriction of $\tilde{\A}$  to $\M^\theta_{\param}(v)_{\Fi}^{(1)k+1}$
vanishes as well, hence that restriction splits.
\end{proof}

\subsection{Comparison for different resolutions}\label{SS_Procesi_comparison}
Let $\hat{\mathcal{P}}^\theta_{\param,\Fi}$ denote a splitting bundle for
$\A^\theta_{\paramq}(v)_{\Fi}^{\wedge_0}$.
The bundle $\hat{\mathcal{P}}^\theta_{\param,\Fi}$ has trivial higher self-extensions, compare with
\cite[Section 2.3]{BK2}, and hence has an $\Fi^\times$-equivariant structure, see \cite{Vologodsky}.

We remark that since $\M^\theta(v)_{\Fi}$ is  defined over $\Fi_p$,
we have an isomorphism $\M^\theta(v)_{\Fi}\cong \M^\theta(v)_{\Fi}^{(1)}$ of $\Fi$-schemes.
Therefore we can view $\hat{\mathcal{P}}^\theta_{\Fi}$ (the specialization of
$\hat{\mathcal{P}}^\theta_{\param,\Fi}$ to $0\in \param_{\Fi}$)
as a bundle on $\M^\theta(v)_{\Fi}^{\wedge_0}$.
Similarly, we can view $\hat{\mathcal{P}}^\theta_{\param,\Fi}$ as a bundle on $\M_{\param}^\theta(v)_{\Fi}^{\wedge_0}$.
This is because the Artin-Schreier map $\param_\Fi\rightarrow \param_{\Fi}^{(1)}$
is etale and so induces an isomorphism
of $\Fi[\param]^{\wedge_0}$ with itself.  Since the  bundle
$\hat{\mathcal{P}}^\theta_{\param,\Fi}$ has no higher self-extensions,  we can extend it to a unique $\Fi^\times$-equivariant vector bundle on $\M^\theta_{\param}(v)_{\Fi}$ to be denoted by $\mathcal{P}^\theta_{\param,\Fi}$.

One can lift $\mathcal{P}^\theta_{\param,\Fi}$ to characteristic $0$ as explained in \cite{BK2}.  Let us recall how to do
this.  The bundle $\mathcal{P}^\theta_{\param,\Fi}$  is defined over some finite field $\Fi_q$. Let $\tilde{S}$ be an algebraic extension of the ring $S$ from Section \ref{SS_pos_char} that has  $\Fi_q$ as a quotient field. Set $\M^{\theta}_{\param}(v)_{\tilde{S}}:=\Spec(\tilde{S})\times_{\Spec(S)}\M^{\theta}_{\param}(v)_S$.
Since the bundle  $\mathcal{P}^\theta_{\param,\Fi}$ has no higher Ext's it can be extended
to a unique $\mathbb{G}_m$-equivariant  bundle $\tilde{\mathcal{P}}^\theta_{\param,\Fi}$
on the formal neighborhood $\M^{\theta}_{\param}(v)^{\wedge_q}_{\tilde{S}}$ of
$\M^\theta_{\param}(v)_{\Fi_q}$ in $\M^{\theta}_{\param}(v)_{\tilde{S}}$ (the existence is guaranteed
by vanishing of $\Ext^2$, and the uniqueness by the vanishing of $\Ext^1$). Since the bundle
$\tilde{\mathcal{P}}^\theta_{\param,\Fi}$ is $\mathbb{G}_m$-equivariant, and the action
is contracting, this bundle is the completion of a unique $\mathbb{G}_m$-equivariant
bundle $\mathcal{P}^\theta_{\param_{\tilde{S}^{\wedge_q}}}$ on $\M^\theta_{\param}(v)_{\tilde{S}^{\wedge_q}}$,
where $\tilde{S}^{\wedge_q}$ stands for the completion of $S$ with respect to
the kernel of $S\rightarrow \mathbb{F}_q$.  Since the quotient field
of $\tilde{S}^{\wedge_q}$ embeds into $\C$, we get a bundle $\mathcal{P}^{\theta}_{\param}$
on $\M^\theta_\param(v)$. This bundle is $\C^\times$-equivariant and has no higher self-extensions.
We remark that its restriction to $\M^\theta(v)$ has no higher self-extensions  because
$\mathcal{P}^{\theta}_{\param}$ and $\operatorname{End}(\mathcal{P}^{\theta}_{\param})$ are flat over $\C[\param]$,
and $\mathcal{P}^\theta_{\param}$ has no higher self-extensions.

Now we want to compare the endomorphism algebras of the bundles $\mathcal{P}^\theta_\param$ for different
$\theta$.

\begin{Prop}\label{Prop:end_alg_coinc}
For any (generic) $\theta,\theta'$, we have $\operatorname{End}(\mathcal{P}^\theta_\param)\cong
\operatorname{End}(\mathcal{P}^{\theta'}_\param)$, an isomorphism of graded $\C[\param]$-algebras.
\end{Prop}
\begin{proof}
The proof is in several steps.

{\it Step 1}. Consider the locus $\M_{\param}(v)_{\Fi}^{(1)reg}\subset
\M_{\param}(v)_{\Fi}^{(1)}$ that is the union of open symplectic leaves in
the Poisson varieties $\M_\lambda(v)^{(1)}_{\Fi}$. As in the proof of
Proposition  \ref{Prop:univ_wc}, $\M_{\param}(v)_{\Fi}^{(1)reg}$ is an open
subvariety in $\M_{\param}(v)_{\Fi}^{(1)}$. Further, the morphism
$\M^\theta_{\param}(v)_{\Fi}^{(1)}\rightarrow \M_{\param}(v)_{\Fi}^{(1)}$
is an isomorphism over $\M_{\param}(v)_{\Fi}^{(1)reg}$ because it is a fiberwise
symplectic resolution of singularities. Let $\M^{\theta}_{\param}(v)_{\Fi}^{(1) reg}$
denote the isomorphic preimage of $\M_{\param}(v)_{\Fi}^{(1)reg}$ in
$\M^\theta_{\param}(v)_{\Fi}^{(1)}$. In particular, we see that
$\M^\theta_{\param}(v)_{\Fi}^{(1),reg}, \M^{\theta'}_{\param}(v)_{\Fi}^{(1),reg}$
are naturally identified.

We claim that the restrictions of the bundles $\mathcal{P}^{\theta}_{\param,\Fi}, \mathcal{P}^{\theta'}_{\param,\Fi}$ to this open subvarieties differ by a twist with a line bundle. The latter will follow if we check that
\begin{equation}\label{eq:end_equi}\End(\mathcal{P}^{\theta}_{\param,\Fi})\cong \End(\mathcal{P}^{\theta'}_{\param,\Fi}).\end{equation}
Indeed, the restriction of this endomorphism algebra to $\M^{\theta}_{\param}(v)_{\Fi}^{(1)reg}$ is Azumaya and the restrictions of both
$\mathcal{P}^{\theta}_{\param,\Fi}, \mathcal{P}^{\theta'}_{\param,\Fi}$ are splitting bundles so differ
by a twist with a line bundle.

To establish (\ref{eq:end_equi})
we will first verify that  $\End(\hat{\mathcal{P}}^{\theta}_{\param,\Fi})=\End(\hat{\mathcal{P}}^{\theta'}_{\param,\Fi})$.
By the construction in Section \ref{SS_splitting},
the left hand side is $\Gamma(\A^\theta_{\paramq_\Fi}(v)^{\wedge_0})$, while the right hand side is
$\Gamma(\A^{\theta'}_{\paramq_\Fi}(v)^{\wedge_0})$, both are isomorphic to $\A_{\paramq_\Fi}(v)^{\wedge_0}$.
Now by the formal function theorem, the completions  $\End(\mathcal{P}^{\theta}_{\param,\Fi})^{\wedge_0}, \End(\mathcal{P}^{\theta'}_{\param,\Fi})^{\wedge_0}$ are isomorphic $\Fi[[\param]]$-algebras.
We want to deduce the isomorphism $\End(\mathcal{P}^{\theta}_{\param,\Fi})\cong \End(\mathcal{P}^{\theta'}_{\param,\Fi})$
from here. This will follow if we show the following claim

\begin{itemize}
\item[(*)]
the isomorphism $\End(\mathcal{P}^{\theta}_{\param,\Fi})^{\wedge_0}\cong \End(\mathcal{P}^{\theta'}_{\param,\Fi})^{\wedge_0}$ can be made $\Fi^\times$-equivariant by twisting the actions on the indecomposable
summands of the vector bundles involved  by characters of $\Fi^\times$.
\end{itemize}

To show that we first observe that
the restrictions of the indecomposable summands of $\mathcal{P}^\theta_{\param,\Fi}$ to $\rho^{-1}_\Fi(\M_{\param}(v)_{\Fi}^{(1)\wedge_0,reg})$
are still indecomposable. This follows from the fact  that the complement to $\M_{\param}^\theta(v)^{(1)reg}_{\Fi}$ has codimension bigger than $2$. Also note that
any $\Fi^\times$-equivariantly indecomposable bundle is also indecomposable
as an ordinary bundle.

Let us show that any two $\Fi^\times$-equivariant structures on an indecomposable summand of the restriction of $\mathcal{P}^\theta_{\param,\Fi}$ to $\rho^{-1}_\Fi(\M_{\param}(v)_{\Fi}^{(1)\wedge_0,reg})$, up to an isomorphism,
differ by a twist with a character. This statement is equivalent to the analogous one
for every  indecomposable summand, say $\mathcal{E}$, of $\mathcal{P}^\theta_{\param,\Fi}$. Let us write $\mathfrak{m}$
for the maximal ideal in $\Fi[\M_{\param}(v)^{(1)}]^{\wedge_0}$. Then
$\operatorname{Aut}(\mathcal{E})$ is the preimage of the group of invertible
elements in $\operatorname{End}(\mathcal{E})/(\mathfrak{m})$ under the
epimorphism   $\operatorname{End}(\mathcal{E})\twoheadrightarrow
\operatorname{End}(\mathcal{E})/(\mathfrak{m})$. The algebra
$\operatorname{End}(\mathcal{E})/(\mathfrak{m})$ is finite dimensional,
so the group  $\operatorname{Aut}(\mathcal{E})$ is pro-algebraic. An $\Fi^\times$-equivariant
structure gives rise to a semi-direct product $\Fi^\times\ltimes \operatorname{Aut}(\mathcal{E})$,
where the action of $\Fi^\times$ on $\operatorname{Aut}(\mathcal{E})$ preserves the pro-algebraic
structure. Another equivariant structure defines an embedding
$\Fi^\times\rightarrow \Fi^\times\ltimes \operatorname{Aut}(\mathcal{E})$ of the form
$t\mapsto (t,\gamma(t))$. Conjugating by a suitable element of $\operatorname{Aut}(\mathcal{E})$
we achieve that $\gamma(t)$ commutes with $\Fi^\times$. However, since $\mathcal{E}$
is indecomposable, we have $\gamma(t)$ is a scalar in this case. (*) follows.


{\it Step 2}. So now we can assume that $$\mathcal{P}^\theta_{\param,\Fi}|_{\M^\theta_{\param}(v)_{\Fi}^{reg}}\cong
\mathcal{P}^{\theta'}_{\param,\Fi}|_{\M^\theta_{\param}(v)_{\Fi}^{reg}},$$
where we consider $\mathcal{P}^{\theta}_{\param,\Fi}$ as a bundle on $\M^\theta_\param(v)_{\Fi}$.
We claim that the first self-Ext of these isomorphic bundles vanishes. The variety $\M_{\param}(v)_{\Fi}$ is Cohen-Macaulay by (4) of Corollary \ref{Cor:prop_Mv}. Since $H^i(\M^\theta_{\param}(v)_{\Fi}, \mathcal{E}nd (\mathcal{P}^\theta_{\param,\Fi}))=0$ for $i>0$, we see that $\operatorname{End}(\mathcal{P}^\theta_{\param,\Fi})$ is a Cohen-Macaulay $\Fi[\mathcal{M}^\theta_{\param,\Fi}]$-module. 
Therefore, for a subvariety $Y\subset \M_{\param}(v)_{\Fi}$ of codimension $i$, we have $H^j_Y(\M_{\param}(v)_{\Fi}, \operatorname{End}(\mathcal{P}^\theta_{\param,\Fi}))=0$ for $j<i$.
Since $\M_{\param}(v)_{\Fi}\setminus \M_{\param}(v)_{\Fi}^{reg}$ has codimension $3$, we use a standard exact sequence
for the cohomology with support to  see that
$H^i_{\M_{\param}(v)_{\Fi}\setminus \M_{\param}(v)^{reg}_{\Fi}}(\M_{\param}(v)_{\Fi}, \operatorname{End}(\mathcal{P}^\theta_{\param,\Fi}))=0$
for $i<2$. Therefore, $H^1(\M_{\param}(v)_{\Fi}^{reg}, \mathcal{E}nd(\mathcal{P}^\theta_{\param,\Fi}))=0$ and we are done.

{\it Step 3}. We have a closed subscheme $\M^\theta_{\param}(v)_{\Fi_q}^{reg}\subset \M^\theta_{\param}(v)_{\tilde{S}}^{reg}$.
Consider its formal neighborhood $\M^\theta_{\param}(v)_{\tilde{S}}^{reg,\wedge_q}$. There is a natural morphism
$\iota: \M^\theta_{\param}(v)_{\tilde{S}}^{reg,\wedge_q}\rightarrow \M^\theta_{\param}(v)_{\tilde{S}}^{\wedge_q}$ of formal schemes.
The bundles $\iota^*\tilde{\mathcal{P}}^\theta_{\param,\Fi}, \iota^*\tilde{\mathcal{P}}^{\theta'}_{\param,\Fi}$
are isomorphic by the previous step, because both deform $\mathcal{P}^\theta_{\param,\Fi}|_{\M^\theta_{\param,\Fi}(v)^{reg}}$ that has
no first  self-extensions.
The induced homomorphism $\End(\tilde{\mathcal{P}}^\theta_{\param,\Fi})\rightarrow \End(\iota^*\tilde{\mathcal{P}}^\theta_{\param,\Fi})$
is an isomorphism because both algebras are flat over the complete algebra $S^{\wedge_q}$
 and modulo the maximal ideal of $S^{\wedge_q}$ this homomorphism coincides with the isomorphism
$\End(\mathcal{P}^\theta_{\param_{\Fi}})\rightarrow \End(\mathcal{P}^\theta_{\param_{\Fi}}|_{\M^\theta_{\param_{\Fi}}(v)^{reg}})$.

By Step 1, $\iota^*\tilde{\mathcal{P}}^\theta_{\param,\Fi}$ is independent of $\theta$.
So $\End(\tilde{\mathcal{P}}^\theta_{\param,\Fi})$ is $\mathbb{G}_m$-equivariantly
isomorphic to $\End(\tilde{\mathcal{P}}^{\theta'}_{\param,\Fi})$.
This yields an isomorphism required in this proposition.
\end{proof}

\begin{Rem}\label{Rem:theta_indep} The argument in the above proof implies that
the bundle $\mathcal{P}^\theta_{\param}|_{\M(v)_\param^{reg}}$ is independent of $\theta$
and the first self-extensions vanish.
\end{Rem}

\subsection{Proof of Proposition \ref{Prop:tilting_gen}}\label{SS_tilt_gen_proof}
Let us prove Proposition \ref{Prop:tilting_gen}. For $\mathcal{P}^\theta$
we take the restriction of $\mathcal{P}^\theta_\param$ to $\M^\theta(v)$.
This bundle has no higher self-extensions, this has already been mentioned
in Section \ref{SS_Procesi_comparison}. Because of that $\End(\mathcal{P}^\theta)=\End(\mathcal{P}^\theta_{\param})/(\param)$.
So the algebras $\End(\mathcal{P}^\theta)$ for different $\theta$
are $\C^\times$-equivariantly identified.

It remains to show that $\operatorname{End}(\mathcal{P}^\theta)$
has finite homological dimension. This is what we do in the remainder
of this section. For this algebra to have  finite homological dimension we will
need to make special choices of $\lambda$ and of $p$.

First of all, let us notice that there is $\lambda\in \mathbb{Q}^{Q_0}$ such that the homological
dimension of $\A_\lambda(v)\otimes \A_\lambda(v)^{opp}$ is finite.
Indeed, for any  $\lambda$ there is $k\in \Z$ such that the abelian localization
theorem holds  for $\A_{\lambda+k\theta}(v)\otimes
\A_{\lambda+k\theta}(v)^{opp}$ on $\M^\theta(v)\times \M^{-\theta}(v)$. It follows that
the homological dimension of $\A_{\lambda+k\theta}(v)\otimes
\A_{\lambda+k\theta}(v)^{opp}$ is finite and we replace $\lambda$
with $\lambda+k\theta$.

We claim that for $p\gg 0$, the homological dimension of $\A_\lambda(v)_{\mathbb{F}}$ is finite.
This follows from a more general result.

\begin{Lem}\label{Lem:fin_hom_char}
Let $S$ be a finite localization of $\Z$. Let $A$ be an $S$-algebra such that $A\otimes_{S}A^{op}$ is Noetherian.
Suppose that for $A_{\C}=\C\otimes_S A$, the homological dimension of
$A_\C\otimes_\C A_{\C}^{opp}$ is finite.
Then the algebra $A_{\Fi}:=\Fi\otimes_{S}A$ has finite homological dimension
for all $p\gg 0$. Moreover, the projective dimension of the regular
$A_{\Fi}$-bimodule is finite.
\end{Lem}
\begin{proof}
For an algebra $\mathcal{A}$ over a field, the homological dimension is finite provided the projective dimension of
the regular bimodule $\mathcal{A}$ is. 

Let $F$ be a finitely generated free $A$-bimodule, $M$ a finitely generated $A$-bimodule and $\varphi$ a homomorphism
$F\rightarrow M$  such that the homomorphism $\varphi_{\C}$ is surjective.
Because all bimodules involved are finitely generated,
$\varphi_{S'}$ is an epimorphism for a finite localization $S'$ of $S$. So, using induction,
we reduce to showing that if $M$ is a finitely generated $A$-bimodule such that $M_{\operatorname{Frac}(S)}$
is a projective $A_{\operatorname{Frac}(S)}$-bimodule, then $M_{S'}$ is a projective $A_{S'}$-bimodule for a
finite localization $S'$ of $S$. Fix an epimorphism $\varphi:F\twoheadrightarrow M$ of $A$-bimodules.
Then $\varphi_{\operatorname{Frac}(S)}$ admits a left inverse, $\iota$. Clearly, $\iota$ is defined
over a finite localization $S'$ of $S$. It follows that $M_{S'}$ is projective.

Therefore
the projective dimension of the regular $A_{S'}$-bimodule is finite. So the projective dimension of $A_{\overline{\mathbb{F}}_p}$
is finite for $p\gg 0$ proving the lemma.
\end{proof}

Since the projective dimension of the regular $\A_\lambda(v)_\Fi$-bimodule is finite, so is the projective dimension
of the regular $\A_\lambda(v)_{\Fi}^{\wedge_0}$-bimodule, because $\A_{\lambda}(v)_{\Fi}^{\wedge_0}$
is a flat (left and right) $\A_{\lambda}(v)_{\Fi}$-module.
It follows that $\A_\lambda(v)_{\Fi}^{\wedge_0}$ has finite homological dimension. In other words,
the homological dimension of $\operatorname{End}(\hat{\mathcal{P}}^\theta_\Fi)$ is finite. Because
of the contracting $\Fi^\times$-action, the homological dimension of $\operatorname{End}(\mathcal{P}^\theta_\Fi)$ is also finite. The same holds for the deformation $\End(\tilde{\mathcal{P}}^\theta_\Fi)$ of $\operatorname{End}(\mathcal{P}^\theta_{\Fi_q})$ over $\tilde{S}^{\wedge_q}$  (here
$\tilde{\mathcal{P}}^\theta_\Fi$ is the specialization of the bundle $\tilde{\mathcal{P}}^\theta_{\param,\Fi}$
constructed in the beginning of Section \ref{SS_Procesi_comparison} to
the zero parameter). Again, thanks to the contracting action of $\mathbb{G}_m$, the homological dimension of $\End(\mathcal{P}^\theta_{\tilde{S}^{\wedge_q}})$ is finite, and we are done.

\subsection{Proof of Proposition \ref{Prop:inj_K0_rat}}
In this section we prove Proposition \ref{Prop:inj_K0_rat} using reduction to
characteristic $p$. We start with reducing finite dimensional representations mod $p$.

\subsubsection{Reduction of representations mod $p$}
Let $\lambda\in \Q^{Q_0}$ be such that $\A_\lambda(v)\otimes \A_{\lambda}(v)^{opp}$ has finite
homological dimension. Here we will discuss the reduction of the finite dimensional representations
of $\A_\lambda(v)$ modulo $p$ for $p\gg 0$.  Let $\Fi$ still be an algebraically closed field of characteristic
$p$. Note that by Lemma \ref{Lem:fin_hom_char}, the algebra $\A_\lambda(v)_{\Fi}$
has finite homological dimension. Let $S$ have the same meaning as in Lemma \ref{Lem:quant_base_change}.

Let $M\in \A_\lambda(v)\operatorname{-mod}_{fin}$. Pick an $S$-lattice $M_S\subset M$. If $p$
is invertible in $S$, then $M_{\Fi}:=\Fi\otimes_S M_S$ makes sense. It is a standard result that
$[M_\Fi]\in K_0(\A_{\lambda}(v)_\Fi\operatorname{-mod}_{fin})$ depends only on $[M]$
and the map $[M]\mapsto [M_\Fi]$ is linear.

Inside $\A_\lambda(v)_{\Fi}$ we have a central subalgebra $\Fi[\M(v)^{(1)}]$ known as the
$p$-center, see, e.g., \cite{BFG}. Consider the category $\A_{\lambda}(v)_{\Fi}\operatorname{-mod}_0$ consisting
of all finite dimensional $\A_\lambda(v)_{\Fi}$-modules supported at $0\in \M(v)_{\Fi}^{(1)}$.

Here is the main result of this section.

\begin{Prop}\label{Prop:red_mod_p}
The following is true provided $p\gg 0$:
\begin{enumerate}
\item $M_{\Fi}\in \A_{\lambda}(v)_{\Fi}\operatorname{-mod}_0$.
\item If $M$ is irreducible, then $M_{\Fi}$ is irreducible.
\item If $M^1,M^2$ are two non-isomorphic finite dimensional irreducible modules,
then $M^1_{\Fi}$ and $M^2_{\Fi}$ are non-isomorphic.
\end{enumerate}
\end{Prop}
It follows from Proposition \ref{Prop:red_mod_p} that we get an injective map
$K_0(\A_\lambda(v)\operatorname{-mod}_{fin})\hookrightarrow K_0(\A_\lambda(v)_\Fi\operatorname{-mod}_0)$.
In \ref{SSS:K0_iso} we will see that there is an isomorphism
$K_0(\A_\lambda(v)_\Fi\operatorname{-mod}_0)\xrightarrow{\sim} K_0(\Coh_{\rho^{-1}(0)}(\M^\theta(v)))$
intertwining the injective map above with $\gamma_\lambda^\theta$.

\begin{proof}[Proof of Proposition \ref{Prop:red_mod_p}]
Let $M\in \A_{\lambda}(v)\operatorname{-mod}_{fin}$ be irreducible and let $I$ be its
annihilator in $A:=\A_\lambda(v)$ so that $A/I=\End_{\C}(M)$.
Then the annihilator of $M_S$ in $A_S$ is $I_S:=A_S\cap I$.

We may assume that $M$ is a free finite rank module over $S$.
Then $I_S$ and $A_S/I_S$ are flat over $S$.
Replacing $S$ with a finite algebraic extension we achieve that
that $A_S/I_S\xrightarrow{\sim}\End_S(M_S)$. For $p$ sufficiently large, $\Fi$ is an
$S$-algebra. So, for $I_{\Fi}:=\Fi\otimes_S I_S$, we have $A_\Fi/I_\Fi\xrightarrow{\sim}
\End_\Fi(M_\Fi)$.  It follows that $M_\Fi$ is irreducible and $I_\Fi$ is the annihilator of
$M_\Fi$ in $A_\Fi$. (2) is proved.

Let $M^1,M^2$ be two non-isomorphic simples. Let $I$ be the intersection of their annihilators
in $A$. Then we can form the ideals $I_S\subset A_S$ and $I_\Fi\subset A_\Fi$ similarly to
the above. We will get $A_\Fi/I_\Fi\xrightarrow{\sim} \End_\Fi(M^1_\Fi)\oplus \End_\Fi(M^2_\Fi)$.
This implies (3).

Finally, let us prove (1). Note that $\dim \Fi[\M(v)^{(1)}_\Fi]/(I_\Fi\cap \Fi[\M(v)^{(1)}_\Fi])
\leqslant (\dim M)^2$ so is bounded with respect to $p$. On the other hand, it is easy
to see that $I_\Fi\cap \Fi[\M(v)^{(1)}_\Fi]$ is a Poisson ideal in $\Fi[\M(v)^{(1)}_\Fi]$.
Since the codimension is bounded with respect to $p$, we see that the subvariety
of  $\M(v)^{(1)}_\Fi$ defined by this ideal is the union of points that are symplectic
leaves. Since we have a contracting $\Fi^\times$-action on $\M(v)^{(1)}_\Fi$
we see that the subvariety is actually $\{0\}$. This finishes the proof of (1).
\end{proof}

\subsubsection{Isomorphism $K_0(\A_\lambda(v)_\Fi\operatorname{-mod}_0)\xrightarrow{\sim} K_0(\Coh_{\rho^{-1}(0)}(\M^\theta(v)))$}\label{SSS:K0_iso}
Note that $R\Gamma(\A_\lambda^\theta(v)_\Fi)=\A_\lambda(v)_\Fi$. By the construction,
the algebra $\A_\lambda(v)_{\Fi}$ has finite homological dimension. By \cite[Section 2.2]{BK2},
the functor $R\Gamma$ gives an equivalence $D^b(\A^\theta_\lambda(v)_{\Fi}\operatorname{-mod})
\xrightarrow{\sim} D^b(\A_\lambda(v)_{\Fi}\operatorname{-mod})$. Consider the subcategory
$D^b_0(\A_\lambda(v)_{\Fi}\operatorname{-mod})\subset D^b(\A_\lambda(v)_{\Fi}\operatorname{-mod})$
of all complexes whose homology are finite dimensional with generalized $p$-character $0$
and the similarly defined subcategory $D^b_{(\rho^{(1)})^{-1}(0)}(\A^\theta_\lambda(v)_{\Fi}\operatorname{-mod})
\subset D^b(\A^\theta_\lambda(v)_{\Fi}\operatorname{-mod})$. The equivalence $R\Gamma$
restricts to $D^b_{(\rho^{(1)})^{-1}(0)}(\A^\theta_\lambda(v)_{\Fi}\operatorname{-mod})
\xrightarrow{\sim} D^b_{0}(\A_\lambda(v)_{\Fi}\operatorname{-mod})$.

Since $\A_\lambda^\theta(v)_\Fi$ splits on $\M^\theta(v)_\Fi^{\wedge_0}$, we further get an
equivalence $$\mathcal{F}:D^b_{(\rho^{(1)})^{-1}(0)}(\Coh(\M^\theta(v)_{\Fi}^{(1)}))\xrightarrow{\sim} D^b_{0}(\A_\lambda(v)_{\Fi}\operatorname{-mod})$$ given by $N\mapsto R\Gamma(\mathcal{P}^\theta_\Fi\otimes N)$.
So we get an isomorphism $$[\mathcal{F}]:K_0(\A_\lambda(v)_\Fi\operatorname{-mod}_0)\xrightarrow{\sim} K_0(\Coh_{(\rho^{(1)})^{-1}(0)}(\M^\theta(v)_{\Fi}^{(1)})).$$ The Frobenius push-forward
for $\M^\theta(v)_{\Fi}$ induces the isomorphism $$[\operatorname{Fr}_*]:K_0(\Coh_{\rho^{-1}(0)}(\M^\theta(v)_{\Fi}))\xrightarrow{\sim}
K_0(\Coh_{(\rho^{(1)})^{-1}(0)}(\M^\theta(v)_{\Fi}^{(1)}))$$
(recall that all $K_0$'s we consider are over $\C$, the map $[\operatorname{Fr}_*]$ is not invertible over $\Z$). But $K_0(\Coh_{\rho^{-1}(0)}(\M^\theta(v)))$ is naturally identified with $K_0(\Coh_{\rho^{-1}(0)}(\M^\theta(v)_\Fi))$ since $p\gg 0$.

We get an isomorphism $K_0(\A_\lambda(v)_\Fi\operatorname{-mod}_0)\xrightarrow{\sim}
K_0(\Coh_{\rho^{-1}(0)}(\M^\theta(v)))$ given by $[\operatorname{Fr}_*]^{-1}\circ [\mathcal{F}]^{-1}$ to be denoted by $\iota$. Set $\iota'=[\mathcal{P}^{\theta}]\iota$.
Our $K_0$ is a $\C$-vector space, so the multiplication by $[\mathcal{P}^{\theta}]$ (the class of a vector bundle) is an invertible transformation.

\subsubsection{Completion of proof}
It remains to prove the following lemma.

\begin{Lem}\label{Lem:inj_K0_final}
We have $\gamma_\lambda^\theta([M])=\iota'([M_{\Fi}])$.
\end{Lem}
\begin{proof}
Note that we still have the degeneration map $$K_0(\A_\lambda^\theta(v)_\Fi\operatorname{-mod}_{\rho^{-1}(0)})
\rightarrow K_0(\Coh_{\rho^{-1}(0)}(\M^\theta(v)_{\Fi})), [N]\mapsto [\gr N].$$
So we get the map $\gamma^\theta_{\lambda,\Fi}: K_0(\A_\lambda(v)_\Fi\operatorname{-mod}_{fin})
\rightarrow K_0(\Coh_{\rho^{-1}(0)}(\M^\theta(v)_{\Fi}))$.
Let us check that $\gamma_\lambda^\theta([M_\C])=\gamma^\theta_{\lambda,\Fi}([M_\Fi])$.
Let $M_S^i:=H_i(\A^\theta_\lambda(v)_S\otimes^L_{\A_\lambda(v)_S}M_S)$. By localizing $S$
further, we can achieve that each $\gr M_S^i$ is flat over $S$. Form the base changes
$M^i_\C=\C\otimes_S M^i_S$ and $M^i_{\Fi}=\Fi\otimes_S M^i_S$. We get $[\gr M^i]=[\gr M^i_{\Fi}]$.
It follows that $\gamma_\lambda^\theta([M_{\C}])=\gamma^\theta_{\lambda,\Fi}([M_\Fi])$.

Now take $L\in \Coh_{(\rho^{(1)})^{-1}(0)}(\M^\theta(v)_\Fi^{(1)})$. The corresponding
object  $M_\Fi\in D^b_0(\A_\lambda(v)_{\Fi}\operatorname{-mod})$ is
$R\Gamma(\mathcal{P}^\theta_\Fi\otimes L)$. But \begin{equation}\label{eq:loc_functor_pos}\A^\theta_\lambda(v)_\Fi\otimes^L_{\A_\lambda(v)_\Fi}M_{\Fi}=
\mathcal{P}^\theta_\Fi\otimes L.\end{equation}
The class of the right hand side of (\ref{eq:loc_functor_pos})
is $[\mathcal{P}^\theta_{\Fi}][L]$. So $\gamma^\theta_{\lambda,\Fi}([M_\Fi])=[\mathcal{P}^\theta_{\Fi}][\operatorname{Fr}_*]^{-1}[N]$.
Since $[N]=[\mathcal{F}]^{-1}([M_{\Fi}])$ we get the required equality
$[\gamma_\lambda^\theta]([M_\Fi])=\iota'([M_\Fi])$.
\end{proof}

\section{Proof of the lower bound}\label{S_lower_bound_proof}
In this section we prove (I) from the beginning of Section \ref{S_outline} and various related statements.

In Section \ref{SS:WC_vs_Rickard} we prove Theorem \ref{Thm:WC} relating the wall-crossing functor
to a Rickard functor. In Section \ref{SS_quiv_KH} we study the K-theory of quiver varieties, in
particular, we use results of Section \ref{S_quant_pos} to
identify the $K_0$-groups $K_0(\operatorname{Coh}_{\rho^{-1}(0)}\M^\theta(v))$
for different generic $\theta$. We show that  the Chern character maps are isomorphisms
that intertwine these identifications with  the identification of homology explained in
Section \ref{SS_quiv_class}. In Section  \ref{SS_WC_K0} we prove that the identifications
$K_0(\operatorname{Coh}_{\rho^{-1}(0)}\M^\theta(v))\xrightarrow{\sim}
K_0(\operatorname{Coh}_{\rho^{-1}(0)}\M^{\theta'}(v))$ intertwine the maps
from $K_0(\A_\lambda(v)\operatorname{-mod}_{fin})$. This shows, in particular,
that $\mathsf{CC}:K_0(\A_\lambda(v)\operatorname{-mod}_{fin})\rightarrow L_\omega[\nu]$
is independent of the choice of $\theta$. Finally, in Section \ref{SS:K_0_action}
we equip $\bigoplus_v K_0(\operatorname{Coh}_{\rho^{-1}(0)}\M^\theta(v))$ with an
$\a$-action and modify the degeneration map
$$\bigoplus_v K_0(\A^\theta_\lambda(v)\operatorname{-mod}_{\rho^{-1}(0)})\rightarrow \bigoplus_v K_0(\operatorname{Coh}_{\rho^{-1}(0)}\M^\theta(v))$$
so that it intertwines $[E_\alpha]$ with $e_\alpha$ and $[F_\alpha]$ with $f_\alpha$.
We deduce Proposition \ref{Prop:a_CC_intertw} (and hence (I)) from here.

\subsection{Wall-crossing vs Rickard complexes}\label{SS:WC_vs_Rickard}
In this subsection we prove Theorem \ref{Thm:WC}.
Our proof follows the scheme of construction
of $E_i,F_i$: we use reduction in stages to reduce to what essentially is
the case of a quiver with a single vertex and no loops.

We will use the notation of Section \ref{SS_W_fun} and of \ref{SSS_LMN}.
Recall that we assume that $\theta_k>0$ for all $k$ and $\lambda_i\in \Z_{\geqslant 0}$.
In the proof we will need to deal with various functors that we will now describe.

\subsubsection{Quotient functors}\label{SSS_quot_fun}
Consider the quotient functors \begin{align*} &\pi^{\theta_i}(v):D_R\operatorname{-mod}^{G,\lambda}\twoheadrightarrow D^{\lambda_i}_{\operatorname{Gr}(v_i,\tilde{w}_i)}\otimes D_{\underline{R}}\operatorname{-mod}^{\underline{G},\underline{\lambda}},\\&
\underline{\pi}^\theta(v): D^{\lambda_i}_{\operatorname{Gr}(v_i,\tilde{w}_i)}\otimes D_{\underline{R}}\operatorname{-mod}^{\underline{G},\underline{\lambda}}\twoheadrightarrow \A_\lambda^\theta(v)\operatorname{-mod}.
\end{align*}
so that $\pi_\lambda^\theta(v)=\underline{\pi}^\theta(v)\circ \pi^{\theta_i}(v)$ (below we will omit the subscript).
Recall, \ref{SSS_Ham_der_fun}, that the functor $\pi^\theta(v)$ extends to a quotient functor
$D^b_{G,\lambda}(D_R\operatorname{-mod})\rightarrow D^b(\A_\lambda^\theta(v)\operatorname{-mod})$
still denoted by $\pi^\theta(v)$. For completely similar reasons, we get quotient functors
\begin{align*} &\pi^{\theta_i}(v):D^b_{G,\lambda}(D_R\operatorname{-mod})\twoheadrightarrow D^b_{\underline{G},\underline{\lambda}}(D^{\lambda_i}_{\operatorname{Gr}(v_i,\tilde{w}_i)}\otimes D_{\underline{R}}\operatorname{-mod}),\\&\underline{\pi}^\theta(v): D^b_{\underline{G},\underline{\lambda}}(D^{\lambda_i}_{\operatorname{Gr}(v_i,\tilde{w}_i)}\otimes D_{\underline{R}}\operatorname{-mod})\twoheadrightarrow \A_\lambda^\theta(v)\operatorname{-mod}.
\end{align*}
such that $\pi^\theta(v)$ still decomposes as $\underline{\pi}^\theta(v)\circ \pi^{\theta_i}(v)$.
Assuming that $\lambda\in \paramq^{ISO}$ and $(\lambda,\theta)\in \AL(v)$,
the functor $\pi^\theta(v)$ admits a left adjoint functor $L\pi^\theta(v)^!$.
Under the same assumptions, $\underline{\pi}^\theta(v)$ admits a derived left adjoint functor
$L\underline{\pi}^\theta(v)^!$.  Further, possibly after replacing $\lambda$
with $\lambda+k\theta$ for $k>0$ we may assume, in addition, that
$\pi^{\theta_i}(v)$ admits a derived left adjoint $L\pi^{\theta_i}(v)^!$.
So we have $L\pi^{\theta}(v)^!=L\pi^{\theta_i}(v)^!\circ L\pi^{\underline{\theta}}(v)^!$.

\subsubsection{Wall-crossing functors}\label{SSS_WC_quotients}
Recall that $\lambda'$ stands for $s_i\bullet^{s_i\bullet v}\lambda$.
Consider the functor $$\WC_{\lambda\rightarrow  \lambda'}^i: D^b_{\underline{G},\underline{\lambda}}(\A^{\theta_i}_{\lambda_i}(v)\operatorname{-mod})
\rightarrow D^b_{\underline{G}, \underline{\lambda'}}(\A^{-\theta_i}_{\lambda'_i}(v)\operatorname{-mod})$$ that is the composition of $$\WC_{\lambda_i\rightarrow \lambda'_i}: D^b_{\underline{G},\underline{\lambda}}(\A^{\theta_i}_{\lambda_i}(v)\operatorname{-mod})
\rightarrow D^b_{\underline{G}, \underline{\lambda}}(\A^{-\theta_i}_{\lambda'_i}(v)\operatorname{-mod})$$ and the equivalence $$D^b_{\underline{G}, \underline{\lambda}}(\A^{-\theta_i}_{\lambda'_i}(v)\operatorname{-mod})
\xrightarrow{\sim} D^b_{\underline{G}, \underline{\lambda'}}(\A^{-\theta_i}_{\lambda'_i}(v)\operatorname{-mod})$$ (an integral change of the twisted equivariance condition). 

We have $(\lambda,\theta)\in \mathfrak{AL}(v)$ by our assumptions. Also $(\lambda_i,\theta_i)\in \mathfrak{AL}(v_i)$ because $\lambda_i,\theta_i\geqslant 0$. These two observations  imply $(\lambda', s_i\theta)\in \mathfrak{AL}(s_i\bullet v)$ and $(\lambda'_i,-\theta_i)\in \mathfrak{AL}(\tilde{w}_i-v_i)$.

\begin{Lem}\label{Lem:WC_red_in_stages}
\begin{equation}\label{eq:intertw}
\WC_{\lambda\rightarrow \lambda'}= \underline{\pi}^{s_i\bullet\theta}(v) \circ \WC_{\lambda\rightarrow \lambda'}^i\circ L\underline{\pi}^{\theta}(v)^!.
\end{equation}
\end{Lem}
\begin{proof}
 Note that we have the next four functor isomorphisms \begin{align*}&\WC_{\lambda\rightarrow \lambda'}^i\cong \pi^{-\theta_i}(v)\circ(\C_{ \lambda'-\lambda}\otimes\bullet)\circ L\pi^{\theta_i}(v)^!,\\
&\WC_{\lambda\rightarrow \lambda'}\cong \pi^{s_i\theta}(v)\circ (\C_{\lambda'-\lambda}\otimes\bullet)\circ L\pi^{\theta}(v)^!,\\
&\pi^{s_i\theta}(v)\cong\underline{\pi}^{s_i\theta}(v)\circ \pi^{-\theta_i}(v),\\ &L\pi^\theta(v)^!\cong L\pi^{\theta_i}(v)^!\circ L\underline{\pi}^{\theta}(v)^!.
\end{align*}
The  last two isomorphisms were discussed in
\ref{SSS_quot_fun}. To prove the second one we use Lemma \ref{Lem:der_glob_descr} and isomorphisms of functors
$\pi^0_\lambda(v)\cong \pi^\theta_\lambda(v), \pi^0_{\lambda'}(v)\cong \pi^{s_i\theta}_{\lambda'}(v)$. The first
isomorphism is analogous.

These four equalities imply (\ref{eq:intertw}).
\end{proof}

\subsubsection{LMN isomorphisms}
Tracking the construction of the LMN isomorphisms, see Sections \ref{SSS_LMN} and \ref{SSS_LMN_quant},
we see that \begin{equation}\label{eq:intertw_s}
s_{i*}\circ \underline{\pi}^{\sigma_i\theta}(v)=\underline{\pi}^{\theta}(s_i\bullet v)\circ \tilde{s}_{i*},
\end{equation} where
we write $\tilde{s}_{i*}$ for the equivalence $$D^b_{\underline{G},
\underline{\lambda'}}(\A_{\lambda'_i}^{-\theta_i}(v)\operatorname{-mod})\xrightarrow{\sim} D^b_{{\underline{G},\underline{\lambda}}}(\A_{\lambda_i}^{\theta_i}(s_i\bullet v)\operatorname{-mod})$$
that comes from the quantum LMN isomorphism $\A_{\lambda'_i}^{-\theta_i}(v)\xrightarrow{\sim}\A^{\theta_i}_{\lambda_i}(v)$. 
Combining (\ref{eq:intertw}) with (\ref{eq:intertw_s}), we get
\begin{align}\label{eq:intertw_1}
s_{i*}\circ \WC_{\lambda\rightarrow \lambda'}=\underline{\pi}^{\theta}(s_i\bullet v)\circ (\tilde{s}_{i*}\circ\WC_{\lambda\rightarrow \lambda'}^i)\circ L\underline{\pi}^{\theta}(v)^!.
\end{align}

\subsubsection{Rickard complexes}\label{SSS_Rickard}
We consider Rickard complexes that (in a somewhat different framework) were
suggested by Chuang and Rouquier, \cite[Section 6]{CR}.
We will use the version of  \cite[Section 8]{Kaetc}.

Set $k=v_i, N=\tilde{w}_i$.
We want to define an object $\Theta$ in the homotopy category of 1-morphisms in $\mathcal{U}(\mathfrak{sl}_2)$
(going from the object $N-2k$ to the object $2k-N$). This will be the complex
$$\Theta^m[-m]\rightarrow \Theta^{m-1}[1-m]\rightarrow\ldots \rightarrow
\Theta^1[-1]\rightarrow \Theta^0,$$ where $m=\min(k,N-k)$. Here
$$\Theta^i=\mathcal{F}^{(N-k-i)}\mathcal{E}^{(k-i)},$$
and $[?]$ denotes the grading shift in $\mathcal{U}(\mathfrak{sl}_2)$.
The differentials in the complex come from adjunctions between $\mathcal{E},\mathcal{F}$.
We note that what is denoted by $\Theta$ in \cite[Section 8]{Kaetc} is
the cone of $\psi(\Theta)$ (where $\psi$ was introduced in  \ref{SSS_Webster_prop}).

Recall that Webster's functors $E,F$ give rise to an action, say $\alpha$, of the 2-category $\mathcal{U}(\sl_2)$
on the category \begin{equation}\label{eq:cat_rep2}
\bigoplus_{v_i=0}^{\tilde{w}_i} D_{\underline{G},\underline{\lambda}}^b(\mathcal{D}^\lambda_{\operatorname{Gr}(v_i,\tilde{w}_i)}\otimes D_{\underline{R}}\operatorname{-mod}).
\end{equation}
We get an endofunctor $\alpha(\Theta)$ of (\ref{eq:cat_rep2}).
It follows from \cite[Theorem 8.1]{Kaetc} that it is an equivalence.

\subsubsection{Comparison}
Now let us compare the wall-crossing functors and the functors coming from Rickard complexes.

\begin{Prop}\label{Prop_WC_grass}
The functor $$\alpha(\Theta): D^b_{\underline{G},\underline{\lambda}}(D^{\lambda_i}_{\operatorname{Gr}(v_i,\tilde{w}_i)}\otimes D_{\underline{R}}\operatorname{-mod})\rightarrow
D^b_{\underline{G},\underline{\lambda}}(D^{\lambda_i}_{\operatorname{Gr}(\tilde{w}_i-v_i,\tilde{w}_i)}\otimes D_{\underline{R}}\operatorname{-mod})$$
coincides with $\tilde{s}_{i*}\circ \WC_{\lambda\rightarrow \lambda'}^i$.
\end{Prop}
\begin{proof}
The statement will be deduced from a result of \cite{Kaetc} which provides an isomorphism
$I\cong \alpha'(\Theta)$ (where $\alpha'$ is a non-equivariant  analog of $\alpha$) between two equivalences $$I,\, \alpha'(\Theta):D^b(D_{\operatorname{Gr}(k,N)}\operatorname{-mod})\xrightarrow{\sim}
D^b(D_{\operatorname{Gr}(N-k,N)}\operatorname{-mod})$$ (also for the $\GL_n$-equivariant categories).
Here $I$ is the Radon transform functor given
by the convolution with $j_*(\mathcal{O}_U)[k(N-k)]\in D^b(D_{\operatorname{Gr}(k,N)\times\operatorname{Gr}(N-k,N)}\operatorname{-mod})$, where  $U\subset \operatorname{Gr}(k,N)\times \operatorname{Gr}(N-k,N)$ is the open $\GL_N$-orbit
and $j:U\hookrightarrow \operatorname{Gr}(k,N)\times \operatorname{Gr}(N-k,N)$
is the open embedding.

Consider the homomorphism $$\psi':\mathcal{U}(\mathfrak{sl}_2)\rightarrow
D^b(D_{\operatorname{Gr}(k,N)\times \operatorname{Gr}(N-k,N)}\operatorname{-mod})$$
that is completely analogous to $\psi$ considered in Section \ref{SS_W_fun}.
The object $$\psi'(\Theta)\in D^b(D_{\operatorname{Gr}(k,N)\times \operatorname{Gr}(N-k,N)}\operatorname{-mod})$$
is  shown in \cite[Corollary 8.6]{Kaetc}  to be isomorphic to the complex  $j_*(\mathcal{O}_U)[k(N-k)]$
of $D$-modules, where $j$ is the embedding $U\to \operatorname{Gr}(k,N)\times \operatorname{Gr}(N-k,N)$. The same is true for the $\operatorname{GL}_N$-equivariant derived category.

It remains to show how this statement implies the proposition. Clearly,
$\psi_i(\Theta)$ is identified with $\psi(\Theta)\boxtimes \delta_{\overline{R}*}(\mathcal{O})$, where
we used the obvious identification $\left( \operatorname{Gr}(v_i,\tilde w_i)\times \overline{R}\right)
\times \left( \operatorname{Gr}(\tilde{w}_i-v_i,\tilde w_i)\times \overline{R} \right)= \operatorname{Gr}(v_i,\tilde w_i)\times \operatorname{Gr}(\tilde{w}_i-v_i,\tilde w_i)\times \overline{R} ^2$. Here $\delta_{\overline{R}}:\overline{R}
\to \overline{R}^2$ is the diagonal embedding.
Since $\psi(\Theta)\cong j_*(\mathcal{O}_U[k(N-k)])\boxtimes  \delta_{\overline{R}*}(\mathcal{O})$, to complete the proof it is enough to notice that  $\tilde{s}_{i*}\circ\WC_{\lambda\rightarrow
\lambda'}^i$ is given by the convolution with $j_*(\mathcal{O}_U)\boxtimes  \delta_{\overline{R}}(\mathcal{O})$
(in the $\underline{G}$-equivaraint derived category). This is proved similarly to \cite[Theorem 12]{BB_Casselman}.
\end{proof}

\subsubsection{Completion of proof}
Let us complete the proof of Theorem \ref{Thm:WC}.  We have the endofunctor $\Theta_i$ of $\bigoplus_{v}D^b(\A_\lambda^\theta(v)\operatorname{-mod})$ induced by $\alpha(\Theta)$, so that $$\Theta_i\circ\pi^{\underline{\theta}}(v)\cong \pi^{\underline{\theta}}(s_i\bullet v)\circ \alpha(\Theta).$$
This implies
\begin{equation}\label{eq:theta_form}
\Theta_i= \pi^{\underline{\theta}}(s_i\bullet v)\circ \alpha(\Theta)\circ L\pi^{\underline{\theta}}(v)^!.
\end{equation}
Thanks to (\ref{eq:theta_form}) and (\ref{eq:intertw_1}), we see that Theorem \ref{Thm:WC} follows from  Proposition
\ref{Prop_WC_grass}.

\subsection{K-theory and cohomology of quiver varieties}\label{SS_quiv_KH}
Recall that the homology groups $H_*(\M^\theta(v))$ and the cohomology groups $H^*(\M^\theta(v))$
are independent of $\theta$, see \ref{SS_homol_ident}.
Note that we have Chern character maps $$K_0(\Coh_{\rho^{-1}(0)}(\M^\theta(v)))\rightarrow H_*(\M^\theta(v)), K_0(\Coh(\M^\theta(v)))\rightarrow H^*(\M^\theta(v))$$
\subsubsection{Results of Nakajima}
In \cite{Nakajima_JAMS}, Nakajima has proved the following results.

\begin{Prop}\label{Prop:chern_iso}
Let $\theta,\theta'$ be generic stability conditions. Then the Chern character maps
$K_0(\Coh_{\rho^{-1}(0)}(\M^?(v)))\rightarrow H_*(\M^?(v)),
K_0(\Coh(\M^?(v)))\rightarrow H^*(\M^?(v))$ are isomorphisms.
\end{Prop}

This follows from \cite[Theorem 7.3.5]{Nakajima_JAMS} and the discussion
in the beginning of \cite[Section 7.1]{Nakajima_JAMS}.

The assignment $$(M,N)\mapsto \sum_{i=0}^\infty (-1)^i \dim \Ext^i(M,N), M\in \Coh(\M^\theta(v)),
N\in \Coh_{\rho^{-1}(0)}(\M^\theta(v))$$ descends to a pairing
\begin{equation}\label{eq:K0_pairing}\langle\cdot,\cdot\rangle: K_0(\Coh(\M^\theta(v)))\times K_0(\Coh_{\rho^{-1}(0)}(\M^\theta(v)))
\rightarrow \C.\end{equation}

\begin{Prop}[Theorem 7.4.1 in \cite{Nakajima_JAMS}]\label{Prop:K0_nondeg_pairing}
The pairing $\langle\cdot,\cdot\rangle$ is non-degenerate.
\end{Prop}

\subsubsection{Alternative identifications of $K_0's$}\label{SSS:K0_ident}
Note that the results quoted in \ref{SS_homol_ident} together with
Proposition \ref{Prop:chern_iso} give rise to  identifications
$$K_0(\Coh_{\rho^{-1}(0)}(\M^\theta(v)))\xrightarrow{\sim}
K_0(\Coh_{\rho^{-1}(0)}(\M^{\theta'}(v))), K_0(\Coh(\M^\theta(v)))\xrightarrow{\sim}
K_0(\Coh(\M^{\theta'}(v))).$$
We will call them the {\it Nakajima isomorphisms}.
We will need an alternative description of these identifications.

 Recall that we  have tilting generators
$\mathcal{P}^?$ on $\M^?(v)$ ($?$ is $\theta,\theta'$) with naturally
identified endomorphism algebras, Proposition \ref{Prop:tilting_gen}.
Set $\tilde{A}:=\End(\mathcal{P}^?)^{opp}$. This algebra has a
central subalgebra $\C[\M(v)]$ so we can consider the category
$\tilde{A}\operatorname{-mod}_0$ of all finite dimensional $\tilde{A}$-modules
supported at $0\in \M(v)$.

By \cite[Section 2.2]{BK2}, the functor $R\Hom(\mathcal{P}^?,\bullet)$
is an equivalence $D^b(\Coh(\M^?(v)))\xrightarrow{\sim} D^b(\tilde{A}\operatorname{-mod})$.
This gives an identification $K_0(\Coh(\M^?(v)))\xrightarrow{\sim} K_0(\tilde{A}\operatorname{-mod})$.
Moreover, the functor $R\Hom(\mathcal{P}^?,\bullet)$ restricts to an equivalence
$D^b_{\rho^{-1}(0)}(\Coh(\M^?(v)))\xrightarrow{\sim} D^b_0(\tilde{A}\operatorname{-mod})$.
This gives an identification $K_0(\Coh_{\rho^{-1}(0)}(\M^{?}(v)))\xrightarrow{\sim}
K_0(\tilde{A}\operatorname{-mod}_0)$. We will use the composite identification
$K_0(\Coh_{\rho^{-1}(0)}(\M^\theta(v)))\xrightarrow{\sim}
K_0(\Coh_{\rho^{-1}(0)}(\M^{\theta'}(v)))$. We will call the resulting
isomorphisms
$$K_0(\Coh_{\rho^{-1}(0)}(\M^\theta(v)))\xrightarrow{\sim}
K_0(\Coh_{\rho^{-1}(0)}(\M^{\theta'}(v))), K_0(\Coh(\M^\theta(v)))\xrightarrow{\sim}
K_0(\Coh(\M^{\theta'}(v)))$$
the {\it tilting isomorphisms}.

\subsubsection{Identifications via deformations}
We are going to give one more characterization of the tilting isomorphism and use this characterization
to show that the tilting isomorphisms are the same as the Nakajima isomorphisms.

Note that $\tilde{A}$ is a graded (with respect to the contracting $\C^\times$-action)
algebra of finite homological dimension. Then we have the following classical result.

\begin{Lem}\label{Lem:K_0_ident_deform}
Let $\tilde{\A}$ be a filtered deformation of $\tilde{A}$. Then the degeneration map
$K_0(\tilde{\A}\operatorname{-mod})\rightarrow K_0(\tilde{A}\operatorname{-mod}),
[M]\mapsto [\gr M],$ is an isomorphism. The inverse sends the class $[\tilde{A}e_i]$
of an indecomposable projective module to the class of its unique deformation.
\end{Lem}

Now let us consider the graded $\C[\param]$-algebra $\tilde{A}_{\param}=\End(\mathcal{P}^\theta_\param)^{opp}$.
It is independent of  $\theta$ by Proposition \ref{Prop:end_alg_coinc}
and its specialization at $0\in \param$ coincides with $\tilde{A}$. Let $\tilde{A}_p$
denote the specialization of $\tilde{A}_{\param}$ to $p\in \param$. The grading on
$\tilde{A}_{\param}$ gives rise to a filtration on $\tilde{A}_p$ so that $\gr\tilde{A}_p=\tilde{A}$.
For $p$ generic, we have an equivalence $R\Hom(\mathcal{P}^\theta_p,\bullet):D^b(\C[\M_p(v)]\operatorname{-mod})
\xrightarrow{\sim} D^b(\tilde{A}_p\operatorname{-mod})$ that is independent of $\theta$ because,
by Remark \ref{Rem:theta_indep},
$\mathcal{P}^\theta_p$ is.  We also have the degeneration map $K_0(\C[\M_p(v)]\operatorname{-mod})\rightarrow
K_0(\Coh(\M^\theta(v)))$.

\begin{Lem}\label{Lem:degen_diagr_commut}
The following diagram is commutative.

\begin{picture}(120,30)
\put(5,22){$K_0(\tilde{A}_p\operatorname{-mod})$}
\put(2,2){$K_0(\C[\M_p(v)]\operatorname{-mod})$}
\put(66,22){$K_0(\tilde{A}\operatorname{-mod})$}
\put(62,2){$K_0(\Coh(\M^\theta(v)))$}
\put(15,20){\vector(0,-1){13}}
\put(76,20){\vector(0,-1){13}}
\put(29,23){\vector(1,0){35}}
\put(39,3){\vector(1,0){21}}
\end{picture}

In particular, $K_0(\C[\M_p(v)]\operatorname{-mod})\xrightarrow{\sim}
K_0(\Coh(\M^\theta(v)))$.
\end{Lem}
\begin{proof}
Note that $\mathcal{P}^\theta=\gr \mathcal{P}_p^\theta$ (here we view $\mathcal{P}_p^\theta$
as a filtered sheaf on $\M^\theta(v)$). Now pick $M\in \C[\M_p(v)]\operatorname{-mod}$
and equip it with a good filtration so that we can view it as a sheaf on $\M^\theta(v)$
with coherent associated graded. It follows that we have the following equality in
$K_0(\tilde{A}\operatorname{-mod})$:
$$\sum_{i}(-1)^i[\Ext^i(\mathcal{P}^\theta, \gr M)]=\sum_{i}(-1)^i [\gr \Ext^i(\mathcal{P}^\theta_p,M)]$$
The left hand side is the image of $[M]$ under
$$K_0(\C[\M_p(v)]\operatorname{-mod})\rightarrow K_0(\Coh(\M^\theta(v)))\rightarrow K_0(\tilde{A}\operatorname{-mod}),$$
while the right hand side is the image of $[M]$ under
$$K_0(\C[\M_p(v)]\operatorname{-mod})\rightarrow K_0(\tilde{A}_p\operatorname{-mod})\rightarrow K_0(\tilde{A}\operatorname{-mod}).$$
This finishes the proof.
\end{proof}

\begin{Cor}\label{Cor:chern_char_comm_ident}
The following claims are true:
\begin{enumerate}
\item The tilting isomorphism $K_0(\Coh(\M^\theta(v)))\xrightarrow{\sim}
K_0(\Coh(\M^{\theta'}(v)))$  is the composition
$$K_0(\Coh(\M^\theta(v)))\xrightarrow{\sim} K_0(\C[\M_p(v)]\operatorname{-mod})
\xrightarrow{\sim} K_0(\Coh(\M^{\theta'}(v))).$$
\item The Chern character maps intertwine the tilting isomorphism
$$K_0(\Coh(\M^\theta(v)))\xrightarrow{\sim}
K_0(\Coh(\M^{\theta'}(v)))$$ and the isomorphism $$H^*(\M^\theta(v))\xrightarrow{\sim}
H^*(\M^{\theta'}(v))$$ from \ref{SS_homol_ident}.
\end{enumerate}
\end{Cor}
\begin{proof}
(1) is a direct corollary of Lemma \ref{Lem:degen_diagr_commut}.
(2) reduces to checking that the Chern character maps intertwine
the degeneration maps $K_0(\C[\M_p(v)]\operatorname{-mod})\xrightarrow{\sim}
K_0(\Coh(\M^\theta(v)))$ and $H^*(\M_p(v))\xrightarrow{\sim} H^*(\M^\theta(v))$,
which is a standard property of the Chern character maps (Chern characters
commute with specialization).
\end{proof}

\subsection{$\WC$ vs degeneration to $K_0(\Coh)$}\label{SS_WC_K0}
According to the previous section we have  the identifications
$K_0(\Coh_{\rho^{-1}(0)}(\M^\theta(v)))\xrightarrow{\sim}
K_0(\Coh_{\rho^{-1}(0)}(\M^{\theta'}(v)))$ for different
generic $\theta,\theta'$.

Let $\lambda$ be such that $\A_\lambda(v)$ has finite homological dimension.
Then we have the identification $$[L\Loc_\lambda^\theta]: K_0(\A_\lambda(v)\operatorname{-mod}_{fin})\xrightarrow{\sim}
K_0(\A_\lambda^\theta(v)\operatorname{-mod}_{\rho^{-1}(0)})$$ and the degeneration
map $$K_0(\A_\lambda^\theta(v)\operatorname{-mod}_{\rho^{-1}(0)})\rightarrow
K_0(\Coh_{\rho^{-1}(0)}(\M^\theta(v))).$$
Consider the composed map
$$K_0(\A_\lambda(v)\operatorname{-mod}_{fin})\rightarrow
K_0(\Coh_{\rho^{-1}(0)}(\M^\theta(v))).$$

\begin{Prop}\label{Prop:degen_indep_theta}
The  identification
$K_0(\Coh_{\rho^{-1}(0)}(\M^\theta(v)))\xrightarrow{\sim}
K_0(\Coh_{\rho^{-1}(0)}(\M^{\theta'}(v)))$ intertwines the maps
from $K_0(\A_\lambda(v)\operatorname{-mod}_{fin})$.
\end{Prop}

Note that Propositions \ref{Prop:chern_iso}, \ref{Prop:degen_indep_theta} imply (1) of Proposition
\ref{Prop:WC_vs_CC}. (2) of Proposition \ref{Prop:WC_vs_CC} follows as well because
of the formula $$\WC_{\theta'\rightarrow \theta}\cong  \mathcal{T}_{\lambda',\chi}\circ
L\Loc_{\lambda'}^\theta \circ R\Gamma^{\theta'}_{\lambda'}$$
and the observation that $\mathcal{T}_{\lambda',\chi}$ does not change the characteristic
cycle.

The proof of Proposition \ref{Prop:degen_indep_theta} occupies the rest of the section.

\subsubsection{Algebras $\tilde{\A}_\lambda$}
Consider the vector bundle $\mathcal{P}^\theta_{\param}$ on $\M^\theta_{\param}(v)$.
It has trivial self-extensions and therefore quantizes to a left $\A^\theta_{\paramq}(v)$-module
to be denoted by $\mathcal{P}^\theta_{\paramq}$. Set $\tilde{\A}_{\paramq}(v):=
\End(\mathcal{P}^\theta_{\paramq})^{opp}$. This is a filtered $\C[\paramq]$-algebra
with $\gr\tilde{\A}_{\paramq}(v)=\tilde{A}_{\param}$.

\begin{Lem}\label{Lem:alg_theta_indep} $\tilde{\A}_{\paramq}(v)$ does not depend on $\theta$.
\end{Lem}
\begin{proof}
Indeed, $\mathcal{P}^\theta_{\paramq}|_{\M_\param(v)^{reg}}$ is a filtered deformation
of $\mathcal{P}^\theta_{\param}|_{\M_\param(v)^{reg}}$, unique because the latter
bundle has zero 1st self-extensions (see Remark \ref{Rem:theta_indep}).
By the same remark, $\mathcal{P}^\theta_{\param}|_{\M_\param(v)^{reg}}=
\mathcal{P}^{\theta'}_{\param}|_{\M_\param(v)^{reg}}$.
So  $\mathcal{P}^\theta_{\paramq}|_{\M_\param(v)^{reg}}=
\mathcal{P}^{\theta'}_{\paramq}|_{\M_\param(v)^{reg}}$. By the same arguments
as in Step 1 of Proposition \ref{Prop:end_alg_coinc} and in the proof
of Proposition \ref{Prop:univ_wc}, we see that
$\End(\mathcal{P}^\theta_{\paramq})=\End(\mathcal{P}^\theta_{\paramq}|_{\M_\param(v)^{reg}})$.
This finishes the proof.
\end{proof}

\subsubsection{Equivalence $D^b(\A_\lambda(v)\operatorname{-mod})\xrightarrow{\sim}D^b(\tilde{\A}_\lambda(v)\operatorname{-mod})$}
Let $\mathcal{P}^\theta_\lambda,\tilde{\A}_\lambda(v)$ be the specializations of
$\mathcal{P}^\theta_{\paramq},\tilde{\A}_\paramq(v)$ at $\lambda\in \paramq$. So we have a functor
$R\Hom(\mathcal{P}^\theta_\lambda,\bullet):D^b(\A_\lambda^\theta(v)\operatorname{-mod})
\rightarrow D^b(\tilde{\A}_\lambda(v)\operatorname{-mod})$. Recall that  $\gr \mathcal{P}^\theta_\lambda=\mathcal{P}^\theta$
and that $R\Hom(\mathcal{P}^\theta,\bullet)$ is an equivalence.
Arguing as in \cite[Section 5]{GL},
we deduce  that $R\Hom(\mathcal{P}^\theta_\lambda,\bullet)$ is an equivalence.
By the construction, the functor $R\Hom(\mathcal{P}^\theta_\lambda,\bullet)$
restricts to an equivalence $D^b_{\rho^{-1}(0)}(\A_\lambda^\theta(v)\operatorname{-mod})
\rightarrow D^b_{fin}(\tilde{\A}_\lambda(v)\operatorname{-mod})$.

The composition $R\Hom(\mathcal{P}^\theta_\lambda,L\Loc_\lambda^\theta(\bullet))$
is an equivalence $D^b(\A_\lambda(v)\operatorname{-mod})\xrightarrow{\sim}
D^b(\tilde{\A}_\lambda(v)\operatorname{-mod})$ to be denoted by $\kappa^{(\theta)}$.
The equivalence $\kappa^{(\theta)}$ restricts to
$D^b_{fin}(\A_\lambda(v)\operatorname{-mod})\xrightarrow{\sim}
D^b_{fin}(\tilde{\A}_\lambda(v)\operatorname{-mod})$. Note that
the equivalence $\kappa^{(\theta)}$ is given by $R\Hom_{\A_\lambda(v)}(\mathcal{B}^{(\theta)}_\lambda,\bullet)$,
where
$$\mathcal{B}^{(\theta)}_\lambda:=R\Gamma((\mathcal{P}_\lambda^\theta)^*)\in D^b(\A_\lambda(v)\text{-}\tilde{\A}_\lambda(v)\operatorname{-bimod}).$$

Below we will prove that, for $M\in \A_\lambda(v)\operatorname{-mod}_{fin}$,
the class $[\kappa^{(\theta)}(M)]\in K_0(\tilde{\A}_\lambda(v)\operatorname{-mod}_{fin})$
is independent of $\theta$ and use this independence to prove
Proposition \ref{Prop:degen_indep_theta}.

\subsubsection{Harish-Chandra bimodules}\label{SSS_more_HC}
In the proof of the independence below we will need the notions of Harish-Chandra $\A_\lambda(v)$-$\tilde{\A}_\lambda(v)$ and
$\A_\paramq(v)$-$\tilde{\A}_\paramq(v)$-bimodules.

Recall that $\C[\M(v)]$ sits as a central subalgebra in $\tilde{A}$. We say that a $\A_\lambda(v)$-$\tilde{\A}_\lambda(v)$-bimodule $\B$ is HC if it admits a
{\it good filtration}, i.e., a bimodule filtration such that $\gr\B$
is a finitely generated $\C[\M(v)]$-module (meaning, in particular,
that the left action of $\C[\M(v)]$ coincides with the right action).
In particular, $H^j(\mathcal{B}^{(\theta)}_\lambda)$ are HC for all $j$, compare
to \ref{SSS_HC_transl}.

By \cite[Theorem 1.3]{B_ineq}, the regular $\A_\lambda(v)$-bimodule has finite length.
Then we argue as in the proof of \cite[Theorem 1.2]{B_ineq} to show
that every HC $\A_\lambda(v)$-$\tilde{\A}_\lambda(v)$-bimodule has finite length.


Now consider $\A_\paramq(v)$-$\tilde{\A}_\paramq(v)$-bimodules. We filter
the algebras $\A_{\paramq}(v),\tilde{\A}_{\paramq}(v)$ as in the proof of
Lemma \ref{Lem:HC_gen_flat} so that $\C[\paramq]$ is in degree $0$.
Then we can define HC bimodules the same way as in the previous paragraph.
We can define HC bimodules over $\A_S(v):=S\otimes_{\C[\paramq]}\A_\paramq(v),
\tilde{\A}_S(v)$ for any quotient $S$ of $\C[\paramq]$.

Let us give an example of a HC $\A_{\paramq}(v)$-$\tilde{\A}_\paramq(v)$-bimodule.
Set $\B^{(\theta)}_{\paramq}:=R\Gamma((\mathcal{P}^\theta_{\paramq})^*)$.
So we get $\A_{\paramq}(v)$-$\tilde{\A}_{\paramq}(v)$-bimodules $H^j(\B^{(\theta)}_{\paramq})$.

\begin{Lem}\label{Lem:HC_example}
All $H^j(\B^{(\theta)}_\paramq)$ are HC bimodules.
\end{Lem}
\begin{proof}
Note that the sheaf $\mathcal{P}^\theta_{\paramq}$ acquires a filtration
similar to those on  $\tilde{\A}_\paramq(v),\A_\paramq(v)$ we currently
consider. We have $\gr\mathcal{P}^\theta_\paramq=\C[\paramq]\otimes \mathcal{P}^\theta$.
Now we argue as in the proof of \cite[Theorem 6.5]{BPW} (see also Proposition
\ref{Prop:univ_wc}) and complete the proof of the lemma.
\end{proof}

\subsubsection{$K_0$-class of $\B_\lambda^{(\theta)}$ is independent of $\theta$}
Consider  $\B_{\paramq}^{(\theta)}\in D^b(\A_{\paramq}(v)\text{-}\tilde{\A}_{\paramq}(v)\operatorname{-bimod})$,
by Lemma \ref{Lem:HC_example}, this is an object with Harish-Chandra cohomology.
Note that $\B_\lambda^{(\theta)}=\C_\lambda\otimes^L_{\C[\paramq]}\B^{(\theta)}_{\paramq}$.
We will need the following lemma.

\begin{Lem}\label{Lem:gener_indep}
Let $\A_\lambda(v)$ have finite homological dimension.
The class of $\B_\lambda^{(\theta)}$ in $K_0(\HC(\A_\lambda(v)$-$\tilde{\A}_\lambda(v)))$ is independent
of $\theta$.
\end{Lem}
\begin{proof}
First, let us show that
\begin{itemize}
\item
the $\A_{\paramq}(v)$-$\tilde{\A}_{\paramq}(v)$-bimodule $H^0(\B_{\paramq}^{(\theta)})$
is independent of $\theta$,
\item the higher cohomology are torsion over $\paramq$.
\end{itemize}

The claim about $H^0$ follows from the fact that $\mathcal{P}^\theta_{\paramq}|_{\M_{\param}(v)^{reg}}$
is independent of $\theta$, see the proof of Lemma \ref{Lem:alg_theta_indep}. Let us prove
that the higher cohomology are torsion. Note that $H^i(\B_{\paramq}^{(\theta)})$
is a finitely generated $\A_{\paramq}(v)$-module. By Lemma \ref{Lem:HC_gen_flat},
$\Supp_{\paramq}(H^i(\B_{\paramq}^{(\theta)}))$ is a constructibe subset of $\paramq$.
If $H^i(\B_{\paramq}^{(\theta)})$ is not torsion, then  $\Supp_{\paramq}(H^i(\B_{\paramq}^{(\theta)}))$
contains a principal Zariski open subset $\paramq^0\subset \paramq$. We may assume that
all $\C[\paramq^0]\otimes_{\C[\paramq]}H^j(\B_{\paramq}^{(\theta)})$ are free over $\C[\paramq^0]$.
It follows that for $\lambda\in \paramq^0$, we have
$H^j(\B_\lambda^{(\theta)})=H^j(\B_{\paramq}^{(\theta)})_\lambda$.
On the other hand, $H^j(\B_\lambda^{(\theta)})=0$ for $j>0$ if
$(\lambda,\theta)\in \AL(v)$. By Proposition \ref{Prop:abelian_loc_BPW}, the set
of $\lambda\in \paramq$ such that $(\lambda,\theta)\in \AL(v)$
is Zariski dense. We arrive at a contradiction with $\paramq^0
\subset \Supp_{\paramq}(H^i(\B_{\paramq}^{(\theta)}))$ that finishes
the proof of the claim in the previous paragraph.

Now the class  $[\B_\lambda^{(\theta)}]$ equals
$$\sum_{i=0}^\infty (-1)^i[\C_\lambda\otimes^L_{\C[\paramq]} H^i(\B_{\paramq}^{(\theta)})].$$
The summand with $i=0$ is independent of $\theta$ because $H^0(\B_{\paramq}^{(\theta)})$ is.
The other summands are zero because $H^i(\B_{\paramq})$ are torsion over
$\C[\paramq]$.
\end{proof}


\subsubsection{Completion of the proof}
\begin{proof}[Proof of Proposition \ref{Prop:degen_indep_theta}]
First of all, note that for any $\B\in D^b_{HC}(\A_\lambda(v)\text{-}\tilde{\A}_\lambda(v)
\operatorname{-bimod})$, the functor $R\Hom(\B,\bullet)$ restricts
to $D^b_{fin}(\A_\lambda(v)\operatorname{-mod})\rightarrow D^b_{fin}(\tilde{\A}_\lambda(v)\operatorname{-mod})$.
This gives rise to a bilinear map
$$K_0(\operatorname{HC}(\A_\lambda(v)\text{-}\tilde{\A}_\lambda(v)))\otimes
K_0(\A_\lambda(v)\operatorname{-mod}_{fin})\rightarrow K_0(\tilde{\A}_\lambda(v)\operatorname{-mod}_{fin}).$$

By Lemma \ref{Lem:gener_indep}, $[\kappa^{(\theta)}(M)]\in K_0(\tilde{\A}_\lambda(v)\operatorname{-mod}_{fin})$
depends only on $[M]\in K_0(\A_\lambda(v)\operatorname{-mod}_{fin})$.

Similarly to the proof of Lemma \ref{Lem:degen_diagr_commut}, we see that the
following diagram commutes. The horizontal arrows  are the degeneration maps, and the vertical ones come
from derived equivalences.

\begin{picture}(120,30)
\put(2,22){$K_0(\tilde{\A}_\lambda(v)\operatorname{-mod}_{fin})$}
\put(2,2){$K_0(\A_\lambda^\theta(v)\operatorname{-mod}_{\rho^{-1}(0)})$}
\put(66,22){$K_0(\tilde{A}\operatorname{-mod}_0)$}
\put(62,2){$K_0(\Coh_{\rho^{-1}(0)}(\M^\theta(v)))$}
\put(15,20){\vector(0,-1){13}}
\put(76,20){\vector(0,-1){13}}
\put(37,23){\vector(1,0){28}}
\put(41,3){\vector(1,0){20}}
\end{picture}

So the image of $[L\Loc_\lambda^\theta(M)]$ in $K_0(\Coh_{\rho^{-1}(0)}(\M^\theta(v)))$
is also independent of $\theta$. This completes the proof.
\end{proof}

\begin{Rem}\label{Rem:another_indep}
For similar reasons, $[L\Loc_\lambda^\theta(M)]\in K_0(\A_\lambda^\theta(v)\operatorname{-mod})$
is independent of $\theta$ for $M\in \A_\lambda(v)\operatorname{-mod}$.
\end{Rem}

For  later applications let us note that we also have a well-defined map
$$K_0(\operatorname{HC}(\A_\lambda(v)\text{-}\tilde{\A}_\lambda(v)))\otimes
K_0(\A_\lambda(v)\operatorname{-mod})\rightarrow K_0(\tilde{\A}_\lambda(v)\operatorname{-mod}).$$

\subsection{Actions on $K_0$}\label{SS:K_0_action}
In this section  we will produce an action of the Lie algebra $\a$ on $\bigoplus_{v}K_0(\Coh_{\rho^{-1}(0)}(\M^\theta(v)))$
and show that, after a suitable modification, degeneration maps $\A_\lambda^\theta(v)\operatorname{-mod}_{\rho^{-1}(0)}
\rightarrow K_0(\Coh_{\rho^{-1}(0)}(\M^\theta(v)))$ intertwine $[E_\alpha]$ with $e_\alpha$
and $[F_\alpha]$ with $f_\alpha$. This will imply Proposition \ref{Prop:a_CC_intertw}.

Let us produce an identification $K_0(\A_\lambda^{\theta}(v)\operatorname{-mod})\xrightarrow{\sim}
K_0(\Coh(\M^\theta(v)))$. Note that the element $[O(\chi)]$ of the algebra $ K_0(\Coh(\M^\theta(v)))$
is unipotent for any $\chi\in \Z^{Q_0}$: for any variety the operator of
multiplication by the class of any line bundle is unipotent.
So, as the class in $K_0$, $[\mathcal{O}(\lambda)]$
makes sense for \underline{any} $\lambda\in \C^{Q_0}$. We identify $K_0(\A_\lambda^\theta(v)\operatorname{-mod})$
with $K_0(\Coh(\M^\theta(v)))$  by $[M]\mapsto [\mathcal{O}(\varrho(v)-\lambda)\otimes\gr M]$. We note that
this identification is independent of the choice of the orientation on $Q$: when we change the orientation
we also shift $\lambda$ by the difference of the $\varrho$-vectors and in this sense
$\varrho(v)-\lambda$ is independent of the choice of the filtration. Also the classes of
shift functors $\mathcal{T}_{\lambda,\chi}$ from Section
\ref{SS_WC_constr} are sent to the identity. We modify the degeneration maps
$K_0(\A_\lambda^{\theta}(v)\operatorname{-mod}_{\rho^{-1}(0)})\rightarrow
K_0(\Coh_{\rho^{-1}(0)}(\M^\theta(v)))$ in a similar fashion so that our identifications
intertwine the natural maps $$K_0(\A_\lambda^\theta(v)\operatorname{-mod}_{\rho^{-1}(0)})
\xrightarrow{\sim} K_0(\A_\lambda^\theta(v)\operatorname{-mod})^*,
K_0(\Coh_{\rho^{-1}(0)}(\M^\theta(v)))\xrightarrow{\sim}
K_0(\Coh(\M^\theta(v)))^*.$$

Recall that the algebra $\g(Q)$ acts on $\bigoplus_v D^b(\A_\lambda^\theta(v)\operatorname{-mod})$
for $\lambda\in \Z^{Q_0}$ and $\theta_i>0$ for all $i>0$, see \cite{Webster}.
Now we define an action of $\g(Q)$ on $K_0(\Coh(\M^\theta(v)))$ from the identification
$\bigoplus_v K_0(\Coh(\M^\theta(v)))\cong \bigoplus_v K_0(\A_\lambda^\theta(v)\operatorname{-mod})$.
And then we define a $\g(Q)$-action on $\bigoplus_v K_0(\Coh_{\rho^{-1}(0)}(\M^\theta(v)))=
\bigoplus_v K_0(\Coh(\M^\theta(v)))^*$ so that $e_i,f_i$ act by   $$\epsilon_i f_i^*,  \epsilon_i e_i^*.$$
Here $\epsilon_i$ is the operator that acts by $(-1)^{\nu(h_i)}$ on the $\nu$-weight space.
It is straightforward to check that these operators indeed define a $\g(Q)$-action.

\begin{Prop}\label{Prop:a_actions_intertwined}
There is a choice of Serre generators $e_\alpha, f_\alpha\in \a, \alpha\in \Pi^\theta,$ such that
the modified degeneration map  $\bigoplus_v K_0(\A_\lambda^\theta(v)\operatorname{-mod}_{\rho^{-1}(0)})
\rightarrow \bigoplus_v K_0(\Coh_{\rho^{-1}(0)}(\M^\theta(v)))$ intertwines $[E_\alpha]$
with $e_\alpha$ and $[F_\alpha]$ with $f_\alpha$.
\end{Prop}
\begin{proof}
First note that $[E_i],[F_i]$ are independent of $\lambda$, this can be deduced
from Remark \ref{Rem:Webster_functors_def} and our identification of $K_0$'s. From Proposition
\ref{Prop:degen_indep_theta} and our identifications of $K_0$'s it follows that
the wall-crossing functors are the identity on the $K_0$ level. By Theorem
\ref{Thm:WC}, we have $[\Theta_i]=[s_{i*}]$. Note that $[\Theta_i]$ on $K_0(\A_\lambda^\theta(v)\operatorname{-mod})$
equals $s_i$, where $s_i$ stands for the action of image of
$\begin{pmatrix}0&1\\-1&0\end{pmatrix}\in \SL_2(\C)$ on $\bigoplus_v K_0(\A_\lambda^\theta(v)\operatorname{-mod})$.
Now, for $\alpha\in \Pi^\theta$, set
$e_\alpha= \sigma(e_i), f_\alpha=\sigma(f_i)$ if $\alpha=\sigma \alpha^i$.
By the definition of the functors $E_\alpha,F_\alpha$, we see that $[F_\alpha]=f_\alpha,
[E_\alpha]=e_\alpha$ on $\bigoplus_v K_0(\A^\theta_\lambda(v)\operatorname{-mod})$.
But we have an isomorphism $\bigoplus_v K_0(\A^\theta_\lambda(v)\operatorname{-mod})^*
\xrightarrow{\sim} \bigoplus_v K_0(\Coh_{\rho^{-1}(0)}(\M^\theta(v)))$ that intertwines the natural
map $$K_0(\A_\lambda^\theta(v)\operatorname{-mod}_{\rho^{-1}(0)})\rightarrow
K_0(\A_\lambda^\theta(v)\operatorname{-mod})^*$$ with the modified degeneration map
$$K_0(\A_\lambda^\theta(v)\operatorname{-mod}_{\rho^{-1}(0)})\rightarrow
 K_0(\Coh_{\rho^{-1}(0)}(\M^\theta(v))).$$ The isomorphism intertwines
$e_\alpha$ with $f_\alpha^*$ and $f_\alpha$ with $e_\alpha^*$. On the other
hand, the map $$\bigoplus_v K_0(\A_\lambda^\theta(v)\operatorname{-mod}_{\rho^{-1}(0)})
\rightarrow \bigoplus_v K_0(\A_\lambda^\theta(v)\operatorname{-mod})^*$$
intertwines $[E_\alpha]$ with $\epsilon_\alpha[F_\alpha]^*$ and $[F_\alpha]$  with $\epsilon_\alpha[E_\alpha]^*$,
where $\epsilon_\alpha$ acts by $(-1)^{\nu(h_\alpha)}$ on the $\nu$-weight space,
because of the adjointness properties of the functors $E_\alpha,F_\alpha$,
see (\ref{eq:EF_adj}) for the basic case of adjointness.
So it follows that the modified degeneration map $\bigoplus_v K_0(\A_\lambda^\theta(v)\operatorname{-mod}_{\rho^{-1}(0)})
\rightarrow \bigoplus_v K_0(\Coh_{\rho^{-1}(0)}(\M^\theta(v)))$ intertwines $[E_\alpha]$
with $e_\alpha$ and $[F_\alpha]$ with $f_\alpha$. This finishes the proof.
\end{proof}

\begin{proof}[Proof of Proposition \ref{Prop:a_CC_intertw}]
The equality $[\Theta_i]=[s_{i*}]$ from the proof of Proposition \ref{Prop:a_actions_intertwined} implies
$\CC(\Theta_i)=s_{i}$. Using this and (2) of Lemma \ref{Lem:W_funct_prop} we deduce
$\CC(E_\alpha)=e_\alpha, \CC(F_\alpha)=f_\alpha$ similarly to the proof of Proposition
\ref{Prop:a_actions_intertwined}.
\end{proof}

\section{Long wall-crossing and dimension of support}\label{S_long_WC}
In this section we prove Proposition \ref{Prop:long_shift}. A key ingredient is
a comparison of the long wall-crossing functor to the homological duality
functor. Then we mention some further properties of long wall-crossing bimodules
and finish the proof of (2) of Proposition \ref{Lem:ab_loc}. Throughout the section
we assume that $Q$ has finite and affine type. By Proposition \ref{Prop:der_MN},
this insures that $R\Gamma_\lambda^\theta$ is a derived equivalence for all
$\theta$ if and only if $\A_\lambda(v)$ has finite homological dimension.

\subsection{Homological duality}\label{SS:hom_dual}
Assume that $\A_\lambda(v)$ has finite homological dimension.
By homological duality functors we mean the functors \begin{align*} &D: D^b(\A_\lambda(v)\operatorname{-mod})\xrightarrow{\sim}
D^b(\A_{\lambda}(v)^{opp}\operatorname{-mod})^{opp},\\ &D^{-\theta}:D^b(\A_\lambda^{-\theta}(v)\operatorname{-mod})\xrightarrow{\sim}
D^b(\A_{\lambda}^{-\theta}(v)^{opp}\operatorname{-mod})^{opp}\end{align*} given by $$\operatorname{RHom}_{\A_\lambda(v)}(\bullet, \A_\lambda(v))[-N],
\operatorname{R}\mathcal{H}om_{\A_\lambda^{-\theta}(v)}(\bullet, \A_\lambda^{-\theta}(v))[-N],$$ where $N:=\frac{1}{2}\dim \M^\theta(v)$. Since $R\Gamma_\lambda^{-\theta}$ is a derived equivalence mapping $\A_\lambda^\theta(v)$ to $\A_\lambda(v)$, the following diagram is commutative.

\begin{picture}(120,30)
\put(2,2){$D^b(\A_\lambda(v)\operatorname{-mod})$}
\put(2,22){$D^b(\A_\lambda^{-\theta}(v)\operatorname{-mod})$}
\put(62,2){$D^b(\A_{\lambda}(v)^{opp}\operatorname{-mod})^{opp}$}
\put(62,22){$D^b(\A_{\lambda}^{-\theta}(v)^{opp}\operatorname{-mod})^{opp}$}
\put(33,3){\vector(1,0){28}}
\put(45,4){\tiny $D$}
\put(35,23){\vector(1,0){26}}
\put(45,24){\tiny $D^{-\theta}$}
\put(10,20){\vector(0,-1){13}}
\put(11,14){\tiny $R\Gamma_\lambda^{-\theta}$}
\put(70,20){\vector(0,-1){13}}
\put(71,14){\tiny $R\Gamma_{\lambda,opp}^{-\theta}$}
\end{picture}

Here $R\Gamma_{\lambda,opp}^{-\theta}$ stands for the derived global section functor for right modules.

\begin{Lem}\label{Lem:D_loc_shift}
The functor $D^{-\theta}$ gives a contravariant abelian equivalence between the categories of holonomic $\A^{-\theta}_\lambda(v)$- and $\A^{-\theta}_\lambda(v)^{opp}$-modules.
\end{Lem}
\begin{proof}
The claim boils down to checking that if $\mathcal{N}$ is a holonomic $\A^{-\theta}_\lambda(v)$-module, then
$\mathcal{E}xt^i(\mathcal{N}, \A^{-\theta}_\lambda(v))= 0$ whenever $i\neq N$. By the standard commutative algebra,
see, e.g., \cite[Proposition 18.4]{Eisenbud},
we see that $\mathcal{E}xt^i(\gr \mathcal{N}, \mathcal{O}_{\M^{-\theta}(v)})\neq 0$ implies $i\geqslant N$. Moreover,
if $i>N$, then the support of $\mathcal{E}xt^i(\gr \mathcal{N}, \mathcal{O}_{\M^{-\theta}(v)})$ has
dimension $<N$. The space $\mathcal{E}xt^i(\mathcal{N}, \A^{-\theta}_\lambda(v))$
has a natural filtration with
$\gr\mathcal{E}xt^i(\mathcal{N}, \A^{-\theta}_\lambda(v))$ being a subquotient of
$\mathcal{E}xt^i(\gr \mathcal{N}, \mathcal{O}_{\M^{-\theta}(v)})$.
Since the filtration is separated, we see that  $\mathcal{E}xt^i(\mathcal{N}, \A^{-\theta}_\lambda(v))=0$
for $i<N$ and $$\dim \Supp\mathcal{E}xt^i(\mathcal{N}, \A^{-\theta}_\lambda(v))<N$$ for $i>N$.
Since the support of any coherent $\A^{-\theta}_{\lambda}(v)^{opp}$-module is coisotropic, see Section  \ref{SSS_Supp_CC},
it cannot have dimension less than $N$ and  we are done.
\end{proof}

Now consider the functor $D$ for the categories of $\A_\lambda(v)$-modules.

\begin{Lem}\label{Lem:D_glob_shift}
Let $\mathcal{N}$ be a simple holonomic $\A_\lambda(v)$-module.  Then the following claims are true
\begin{enumerate}
\item
$H^i(D\mathcal{N})=0$  for $i<N-\dim\operatorname{Supp}\mathcal{N}$ or $i>N$.
\item $H^i(D\mathcal{N})$ is a nonzero module with support of dimension $\dim \Supp \mathcal{N}$
when $i=N-\dim \operatorname{Supp}\mathcal{N}$.
\end{enumerate}
\end{Lem}
\begin{proof}
The algebra $\C[\M(v)]$ is Cohen-Macaulay, see Corollary
\ref{Cor:prop_Mv}. Then \cite[Proposition 18.4]{Eisenbud} implies that,
for a finitely generated $\C[\M(v)]$-module $M$,  the minimal number $r$ such that
$\Ext^r(M,\C[\M(v)])\neq 0$ equals $\dim \M(v)-\dim \operatorname{Supp}M$.
Moreover, we have
\begin{align*}
&\dim \Supp \Ext^r(M,\C[\M(v)])=\dim \operatorname{Supp}M,\\
& \dim \Supp \Ext^i(M,\C[\M(v)])<\dim \operatorname{Supp}M \text{ for }i>r.
\end{align*}

The case $i<N-\dim\operatorname{Supp}\mathcal{N}$ is done similarly to the proof of Lemma
\ref{Lem:D_loc_shift} using the facts quoted in the previous
paragraph.  To deal with the case of $i>N$ we notice that
the homological dimension of $\A_\lambda(v)$ coincides with that of $\A_\lambda^{\theta}(v)$ because
$\Gamma_\lambda^\theta$ is an abelian equivalence. The homological dimension of $\A_\lambda^\theta(v)$
does not exceed that of $\operatorname{Coh}\M^\theta(v)$ that equals $2N=\dim\M^\theta(v)$. This completes
the $i>N$ case.

Let us prove (2). As in the proof
of Lemma \ref{Lem:D_loc_shift}, we see that
$\dim \Supp H^i(D\mathcal{N})<N-\dim \Supp \mathcal{N}$ for $i>N-\dim \Supp\mathcal{N}$. Since $D^2=\operatorname{id}$,
the inequality $H^i(D\mathcal{N})\neq 0$ for $i=N-\dim \Supp\mathcal{N}$ follows.
\end{proof}

Now we note that we have an isomorphism $\A_\lambda^{-\theta}(v)^{opp}= \A_{\lambda^*}^{-\theta}(v)$
(the equality of quantizations of $\M^{-\theta}(v)$).
Here $\lambda^*:=2\varrho(v)-\lambda$.
This follows, for example, from \cite[Proposition 5.4.4]{quant}. So in the above constructions, we can replace
$\A_\lambda^{-\theta}(v)^{opp}$ with $\A_{\lambda^*}^{-\theta}(v)$ and $\A_{\lambda}(v)^{opp}$
with $\A_{\lambda^*}(v)$.

\subsection{Proof of Proposition \ref{Prop:long_shift}}\label{SS:long_wc_proof}
Thanks to Proposition \ref{Prop:abelian_loc_BPW}, replacing $\lambda$ with $\lambda+k\theta$ for $k\gg 0$, we may assume that $(\lambda^*,-\theta)\in \mathfrak{AL}(v)$ (and still $(\lambda,\theta)\in \mathfrak{AL}(v)$). Now we have the
following commutative diagram.

\begin{picture}(140,30)
\put(52,2){$D^b(\A_\lambda(v))$}
\put(52,22){$D^b(\A_\lambda^{-\theta}(v))$}
\put(102,2){$D^b(\A_{\lambda^*}(v))^{opp}$}
\put(102,22){$D^b(\A_{\lambda^*}^{-\theta}(v))^{opp}$}
\put(73,3){\vector(1,0){28}}
\put(85,4){\tiny $D$}
\put(75,23){\vector(1,0){26}}
\put(85,25){\tiny $D^{-\theta}$}
\put(60,20){\vector(0,-1){14}}
\put(61,14){\tiny $R\Gamma_\lambda^{-\theta}$}
\put(114,20){\vector(0,-1){14}}
\put(115,14){\tiny $R\Gamma_{\lambda^*}^{-\theta}$}
\put(2,2){$D^b(\A_{\lambda^-}(v))$}
\put(2,22){$D^b(\A_{\lambda^-}^{-\theta}(v))$}
\put(50,3){\vector(-1,0){27}}
\put(32,5){\tiny $\WC_{\lambda\rightarrow \lambda^-}$}
\put(50,23){\vector(-1,0){26}}
\put(32,25){\tiny $\mathcal{T}_{\lambda,\lambda^--\lambda}$}
\put(13,20){\vector(0,-1){14}}
\put(14,14){\tiny $R\Gamma_{\lambda^-}^{-\theta}$}
\end{picture}

Here we write $D^b(\A_\lambda(v))$ for $D^b(\A_\lambda(v)\operatorname{-mod})$, etc.

The functor $R\Gamma_{\lambda^-}^{-\theta}\circ \mathcal{T}_{\lambda,\lambda^--\lambda}$ is an abelian equivalence
$\A_{\lambda}^{-\theta}(v)\operatorname{-mod}\xrightarrow{\sim} \A_{\lambda^-}(v)\operatorname{-mod}$.
Both functors $R\Gamma_{\lambda^*}^{-\theta}, D^{-\theta}$ are $t$-exact on $D^b_{hol}$, so
the functor $R\Gamma_{\lambda^*}^{-\theta}\circ D^{-\theta}$ intertwines the standard $t$-structures on
$D^b_{hol}(\A_{\lambda}^{-\theta}(v)), D^b_{hol}(\A_{\lambda^*}(v))$.   So we see that the
pull-backs of the $t$-structures on $D^b_{hol}(\A_{\lambda^*}(v))$
and on $D^b_{hol}(\A_{\lambda^-}^{-\theta}(v))$ to $D^b_{hol}(\A_{\lambda}(v))$ coincide
(with the push-forward of the $t$-structure on $D^b_{hol}(\A^{-\theta}_\lambda(v))$).

Let us prove (1). Thanks to Lemmas \ref{Lem:D_loc_shift},\ref{Lem:D_glob_shift}, the functor $D$
homologically shifts a simple $M$ by $N-\dim\operatorname{Supp}M$. Part (1) now follows from the coincidence
of the $t$-structures on $D^b_{hol}(\A_{\lambda}(v))$ established in the previous paragraph.

Let us prove part (2). The functor $\WC_{\lambda\rightarrow\lambda^-}$ restricts to a derived equivalence
$$D^b_{fin}(\A_\lambda(v)\operatorname{-mod})\xrightarrow{\sim} D^b_{fin}(\A_{\lambda^-}(v)\operatorname{-mod}).$$
By Lemma \ref{Lem:D_glob_shift}, for a finite dimensional module $M$, the only nonzero homology
of $\WC_{\lambda\rightarrow\lambda^-}M$ is $H_N$. We are done.

\begin{Rem}
Equip $\A_\lambda(v)\operatorname{-mod}_{hol}$ with a filtration by the dimension of support: let $\A_\lambda(v)\operatorname{-mod}_{hol}^{\leqslant i}$ consist of all modules whose dimension of support
does not exceed $i$. The functor $\WC_{\lambda\rightarrow \lambda^-}$ sends an object of
$\A_\lambda(v)\operatorname{-mod}_{hol}^{\leqslant i}$ to a complex whose homology are
in $\A_{\lambda^-}(v)\operatorname{-mod}_{hol}^{\leqslant i}$. The arguments of the proofs of
Lemma \ref{Lem:D_glob_shift} and Proposition \ref{Prop:long_shift} imply that the functor
$H_i(\WC_{\lambda\rightarrow \lambda^-}\bullet)$
gives rise to an equivalence $\A_\lambda(v)\operatorname{-mod}_{hol}^{\leqslant i}/\A_\lambda(v)\operatorname{-mod}_{hol}^{\leqslant i-1}
\xrightarrow{\sim} \A_{\lambda^-}(v)\operatorname{-mod}_{hol}^{\leqslant i}/\A_{\lambda^-}(v)\operatorname{-mod}_{hol}^{\leqslant i-1}$.
In particular, $\WC_{\lambda\rightarrow \lambda^-}$ is a perverse equivalence
$$D^b_{hol}(\A_\lambda(v)\operatorname{-mod})\rightarrow D^b_{hol}(\A_{\lambda^-}(v)\operatorname{-mod})$$
in the sense of Chuang and Rouquier.  See Section \ref{SS_aff_res} below for a precise definition of a
perverse equivalence in the case of derived categories
(the general case of triangulated categories is completely analogous, see, e.g.,
\cite{ABM}). We do not need this result in the rest of the paper
so we do not provide details.
\end{Rem}

\subsection{Further results}
We will need some further results on long wall-crossing functors. Let $\lambda,\lambda^-,\theta$ have the same meaning as above. We will write $\A^{(-\theta)}_{\lambda\rightarrow \lambda^-}(v)$ for
$\A_{\lambda,\lambda^--\lambda}(v)$.

\begin{Lem}\label{Lem:long_wc_simple}
 The long wall-crossing $\A_{\lambda^-}(v)$-$\A_\lambda(v)$-bimodule $\A^{(-\theta)}_{\lambda\rightarrow \lambda^-}(v)$ is simple.
\end{Lem}
\begin{proof}
The HC $\A^{-\theta}_{\lambda^-}(v)$-$\A^{-\theta}_\lambda(v)$
bimodule $\A_{\lambda\rightarrow \lambda^-}^{-\theta}(v)$ is simple because its  rank
equals $1$. The categories $\operatorname{HC}(\A_{\lambda^-}^{-\theta}(v)\text{-}\A_{\lambda}^{-\theta}(v))$
(see \cite[Section 6.1]{BPW} for the definition of this category)
and $\operatorname{HC}(\A_{\lambda^-}(v)\text{-}\A_{\lambda}(v))$ of Harish-Chandra bimodules are equivalent, see
\cite[Corollary 6.6]{BPW}.
\end{proof}

\begin{Rem}\label{Rem:long_WC_opp}
We can consider $\A_{\lambda\rightarrow \lambda^-}^{(-\theta)}$ as an $\A_\lambda^{opp}$-$\A_{\lambda^-}^{opp}$-bimodule.
It is straightforward to see that it is still a long wall-crossing bimodule.
\end{Rem}

\subsection{Corollaries}\label{SS_long_WC_corollaries}
Now we are ready to prove (2) of Proposition \ref{Lem:ab_loc}. To start with, let us prove a stronger version
of Proposition \ref{Prop:gen_simpl}.

\begin{Prop}\label{Prop:gen_simpl_strong}
For each indecomposable root $\alpha\leqslant v$, there is a finite subset $\Sigma_\alpha\subset \C$ such that the algebra
$\A_\lambda(v)$ is simple 
whenever $\langle\alpha,\lambda\rangle\not\in \Sigma_\alpha+\Z$ for all $\alpha\leqslant v$.
\end{Prop}
\begin{proof}
As in the proof of Proposition \ref{Prop:gen_simpl}, we will first show that, for each root $\alpha\leqslant v$, there is a finite subset $\Sigma_\alpha(v)$ such that the algebra $\A_\lambda(v)$ have no finite dimensional representations provided $\langle\lambda,\alpha\rangle\not\in \Sigma_\alpha(v)+\Z$ for all
$\alpha\leqslant v$.

{\it Step 1}. Let us construct $\Sigma_\alpha(v)$. Pick a Zariski generic point $p\in \ker\alpha$. The variety $\M_p(v)$
has a unique minimal symplectic leaf, compare with Step 2 of the proof of Proposition \ref{Prop:wall_non_essent}.
It corresponds to a semisimple representation of the form $r_0+r_1^{\oplus k}$, where $\dim r_1=\alpha$ and
$k$ is maximal such that $(v-k\alpha,1)$ is a root of the quiver $Q^w$. So we can form the slice algebras
$\hat{\A}^0_{\hat{r}(\lambda)}(\hat{v}), \hat{\A}_{\hat{r}(\lambda)}(\hat{v})$, compare to Step 2 of the proof of Proposition \ref{Prop:wall_non_essent}.  The set of $\langle \lambda,\alpha\rangle$ such that the translation bimodules $\hat{\A}^0_{\hat{r}(\lambda),1}(\hat{v}), \hat{\A}^0_{\hat{r}(\lambda)+1,-1}(\hat{v})$ are not mutually inverse Morita equivalences between $\hat{\A}_{\hat{r}(\lambda),1}(\hat{v})$ and $\A_{\hat{r}(\lambda)+1,-1}(\hat{v})$) is finite by Proposition \ref{Prop:transl_coinc}. We take this set for $\Sigma_\alpha(v)$.

{\it Step 2}. Let $\theta,\theta'$ be two stability conditions from chambers opposite with respect to
$\ker\alpha$. Similarly to the proof of Proposition \ref{Prop:wall_non_essent} we see that
$\WC_{\theta\rightarrow \theta'}$ is an abelian equivalence provided $\langle\lambda,\alpha\rangle
\not\in \Sigma_\alpha(v)+\Z$.

{\it Step 3}. Now suppose that $\langle\lambda,\alpha\rangle\not\in \Sigma_\alpha(v)$ for all indecomposable
$\alpha\leqslant v$. By Theorem \ref{Thm:wc_decomp_short}, the long wall-crossing functor $\WC_{\theta\rightarrow -\theta}$
decomposes as a composition of short wall-crossing functors and hence is an abelian equivalence. By Theorem \ref{Prop:long_shift}, the algebra $\A_\lambda(v)$ has no finite dimensional
representations if abelian localization holds for $(\lambda,\theta)$. In general, note that if $L$ is a finite dimensional representation of $\A_\lambda(v)$, then $L\Loc_\lambda^\theta(L)$ is a nonzero object
supported on $\rho^{-1}(0)$ (indeed, the functor $R\Gamma$ is a left inverse to $L\Loc$, both functors are considered between bounded from the right derived categories). It follows that for $n\gg 0$, the algebra $\A_{\lambda+n\theta}(v)$ has a finite dimensional representation. This gives  a contradiction that completes the proof of the claim that $\A_\lambda(v)$ has no finite dimensional representations provided $\langle\lambda,\alpha\rangle\not\in \Sigma_\alpha(v)+\Z$ for
all indecomposable $\alpha\leqslant v$.

{\it Step 4}. Similarly to the proof of Proposition \ref{Prop:gen_simpl}, if all  proper slice algebras
$\hat{\A}_{\hat{r}(\lambda)}(\hat{v})$ have no finite dimensional representations, then the
algebra $\A_\lambda(v)$ is simple. Now recall that $r^{-1}(\hat{\param}^{sing})\subset \param^{sing}$,
see \ref{SSS_class_slice}. Using this we get subsets $\Sigma_\alpha\in \C$ (for each indecomposable
root $\alpha$) such that the proper slice algebras $\hat{\A}_{\hat{r}(\lambda)}(\hat{v})$ have
no finite dimensional representations when $\langle \lambda,\alpha\rangle \not\in \Sigma_\alpha+\Z$.
So for such $\lambda$, the algebra $\A_\lambda(v)$ is simple.
%
\end{proof}

\begin{proof}[Proof of (2) of Proposition \ref{Lem:ab_loc}]
Let $\chi$ be inside of the chamber of $\theta$ and such that $$H^1(\M^\theta(v),\mathcal{O}(\chi))=
H^1(\M^{-\theta}(v), \mathcal{O}(-\chi))=0.$$ If $\langle \lambda,\alpha\rangle\not\in \Sigma_\alpha+\Z$
for any indecomposable root $\alpha$, then the algebras $\A_\lambda(v),\A_{\lambda+\chi}(v)$ are simple.
Consider the  bimodule homomorphisms
\begin{align*}
&\A^0_{\lambda+\chi,-\chi}(v)\otimes_{\A_{\lambda+\chi}(v)}\A^0_{\lambda,\chi}(v)\rightarrow \A_{\lambda}(v),\\
&\A^0_{\lambda,\chi}(v)\otimes_{\A_\lambda(v)}\A^0_{\lambda+\chi,-\chi}(v)\rightarrow \A_{\lambda+\chi}(v).
\end{align*}
It is enough to show that these maps are isomorphisms.
Note that the generic rank of $\gr\A^0_{\lambda,\chi}(v)$ on $\M(v)$ is $1$.
Indeed, we have $\A^0_{\lambda,\chi}(v)=\A_{\lambda,\chi}^{(\theta)}(v)$ by Proposition \ref{Prop:trans_reln}.
The generic rank of $\gr\A_{\lambda,\chi}^{(\theta)}(v)$ on $\M(v)^{reg}$ is $1$ by the construction.
For similar reasons, the generic rank of $\gr\A^0_{\lambda+\chi,-\chi}(v)$ equals $1$. So
the bimodule homomorphisms above become isomorphisms after microlocalizing to $\M(v)^{reg}$. Since the algebras
$\A_\lambda(v),\A_{\lambda+\chi}(v)$ are simple, we deduce that the bimodule homomorphisms are indeed
isomorphisms.
\end{proof}

\section{Finite short wall-crossing}\label{S_fin_WC}
In this section we investigate various questions related to categorification functors $E_\alpha,F_\alpha$,
the wall-crossing functor through $\ker\alpha$ and connections between them. In Section \ref{SS_category_C}
we study the category $\mathcal{C}$ introduced in the beginning of Section \ref{SS_upper_bound_outline}.
In particular, we show that every simple in $\mathcal{C}$ is {\it regular holonomic} in a suitable
sense. We use this to show that Proposition \ref{Prop:fin_dim_cryst} implies (II), while (II)
and (III) imply Proposition \ref{Prop:fin_dim_cryst}. In Section \ref{SS_fin_sing} we study
singular objects from \ref{SSS_short_wc_outline} and prove Proposition \ref{Prop:singular_equiv}.

\subsection{Category $\mathcal{C}$}\label{SS_category_C}
Recall that $\mathcal{C}\subset \A_\lambda^\theta(v)\operatorname{-mod}_{\rho^{-1}(0)}$
is the Serre subcategory spanned by the homology of the objects of the form
$\mathcal{F}L_0$, where $\mathcal{F}$ is some monomial in the functors $E_\alpha, F_\alpha$
and $L_0\in \A^\theta_\lambda(\sigma\bullet w)$ for $\sigma\in W(Q)$ such that $\sigma \omega$
is dominant for $\a$.

\subsubsection{Regular holonomic modules}\label{SSS_RH}
Let us define {\it regular holonomic} simples in $\A_\lambda^\theta(v)\operatorname{-mod}_{\rho^{-1}(0)}$.
An object in $\A_\lambda^\theta(v)\operatorname{-mod}$ is called {\it regular holonomic}
if it is obtained from a regular holonomic $(G,\lambda)$-equivariant $D(R)$-module by applying
$\pi^\theta_\lambda(v)$. Actually, we are not interested in all regular holonomic modules. Recall
the torus $T=(\C^\times)^{Q_1}\times (\C^\times)^{Q_0}$ from \ref{SSS_LMN} acting on $R$.
We will consider only weakly $T$-equivariant modules.

It is a standard fact that the category of regular holonomic weakly $T$-equivariant $D$-modules stays the same under changing the orientation of $R$ (partial Fourier transforms preserve weakly $T$-equivariant regular holonomic modules, \cite{Brylinski}) so the notion of a weakly $T$-equivariant regular holonomic $\A_\lambda^\theta(v)$-module is well-defined.

\begin{Lem}\label{Lem:RH}
Let $\lambda'=\sigma\bullet^v \lambda$.
Under the isomorphism $\A_{\lambda}^\theta(v)\cong \A_{\lambda'}^{\sigma \theta}(\sigma\bullet v)$,
a weakly $T$-equivariant regular holonomic module remains weakly $T$-equivariant regular holonomic.
\end{Lem}
\begin{proof}
The part concerning the $T$-action follows from the observation, see \ref{SSS_LMN_quant},
that $\sigma_*:\A_{\lambda}^\theta(v)\operatorname{-mod}\xrightarrow{\sim}
\A_{\lambda'}^{\sigma \theta}(\sigma\bullet v)\operatorname{-mod}$ is  $T$-equivariant.
It is enough to prove the claim that $s_{i*}M$ is  regular holonomic
for a simple reflection $s_i$ provided $M$ is regular holonomic.
Recall that in this case the isomorphism
$\A_{\lambda}^\theta(v)\cong \A_{\lambda'}^{\sigma\theta}(\sigma\bullet v)$
is induced  by the isomorphism
\begin{equation}\label{eq:DO_isom}
D_R\red_{\lambda_i}^{\theta_i}\GL(v_i)\cong D_R\red_{\lambda'_i}^{-\theta_i} \GL((s_i\bullet v)_i).\end{equation}
The former reduction is just $D^{\lambda_i}_{\operatorname{Gr}(v_i,\tilde{w}_i)}\otimes D_{\underline{R}}$ and so
there is an intrinsic notion of a regular holonomic module.

We claim  that a simple $D^{\lambda_i}_{\operatorname{Gr}(v_i,\tilde{w}_i)}\otimes D_{\underline{R}}$-module is regular holonomic if and only if it is obtained from a simple regular holonomic $D(R)$-module under the quotient functor
$$D(R)\operatorname{-mod}^{G,\lambda}\twoheadrightarrow
D^{\lambda_i}_{\operatorname{Gr}(v_i,\tilde{w}_i)}\otimes D_{\underline{R}}\operatorname{-mod}^{\underline{G},\underline{\lambda}}.$$

Indeed, let $L$ be a simple $(\GL(v_i),\lambda_i)$-equivariant $D$-module on $R$ whose support intersects $R^{\theta_i-ss}$. Then $L$ is regular holonomic if and only if the induced twisted $D$-module on the quotient $R\quo^{\theta_i}\GL(v_i)$ is regular holonomic. This follows from the classification of simple regular holonomic $D$-modules: those are precisely the intermediate extensions of regular holonomic local systems on smooth locally closed subvarieties, see \cite[Theorem 7.10.6, 7.12]{Borel}. Our claim in the beginning of the paragraph is proved.

Now, under the identifications of $D_R\red_{\lambda_i}^{\theta_i}\GL(v_i)\operatorname{-mod}$
and $D_R\red_{\lambda'_i}^{-\theta_i} \GL((s_i\bullet v)_i)\operatorname{-mod}$
with the category of  $D$-modules on $\operatorname{Gr}(v_i,\tilde{w}_i)\times \underline{R}$,
the equivalence induced by (\ref{eq:DO_isom}) becomes the identity, this follows from the
construction of an isomorphism. We deduce that the equivalence induced by  the isomorphism
$\A_{\lambda}^\theta(v)\cong \A_{\lambda'}^{\sigma\theta}(\sigma\bullet v)$
maps regular holonomic modules to regular holonomic ones.
\end{proof}

\begin{Cor}\label{Cor_RH}
The simples in $\mathcal{C}$ are weakly $T$-equivariant regular holonomic.
\end{Cor}
\begin{proof}
Let us show first that  Webster's functors $E_i,F_i$ preserve the category of direct sums of semisimple
weakly $T$-equivariant regular holonomic $\A_\lambda^\theta(v)$-modules with homological shifts.
The functors $E_i,F_i$ on $\bigoplus_{v_i} D^b(D^{\lambda_i}_{\operatorname{Gr}(v_i,\tilde{w}_i)}\otimes
D_{\underline{R}}\operatorname{-mod}^{\underline{G},\underline{\lambda}})$ have this property
by the construction. By the proof of Lemma \ref{Lem:RH}, a simple regular holonomic $\A_\lambda^\theta(v)$-module
is an image of such a module from  $D^{\lambda_i}_{\operatorname{Gr}(v_i,\tilde{w}_i)}\otimes
D_{\underline{R}}\operatorname{-mod}^{\underline{G},\underline{\lambda}}$ and our claim follows.

Lemma \ref{Lem:RH} and the previous paragraph  imply
that the  functors $E_\alpha,F_\alpha$ preserve semisimple complexes of weakly $T$-equivariant regular
holonomic modules when $ \alpha\in \Pi^\theta$.
Note that a unique indecomposable $\A_\lambda^\theta(\sigma\bullet w)$-module is weakly $T$-equivariant. So it is enough to check that it is regular holonomic. This is definitely true for $\sigma=1$
(the space $R$ is zero). For arbitrary $\sigma$, the claim again follows from Lemma \ref{Lem:RH}.
\end{proof}

\subsubsection{Crystal}
Consider the full subcategory $\mathcal{C}'
\subset D^b(\A_{\lambda}^\theta(v)\operatorname{-mod}_{\rho^{-1}(0)})$ consisting of all
objects $M$ such that $M\cong H_*(M)$ and $H_*(M)$ is a semisimple object of $\mathcal{C}$.
For $L\in \operatorname{Irr}(\mathcal{C})$, we write $d_\alpha(L)$
for the minimal dimension of an irreducible $\slf_2$-module in $U(\slf_2)[L]$
(where we consider the action corresponding to the operators $[E_\alpha],[F_\alpha]$)

\begin{Lem}\label{Lem:ss_cat_preserv}
The functors $E_\alpha,F_\alpha$ for all $\alpha\in \Pi^\theta$ preserve the subcategory
$\mathcal{C}'\subset D^b(\A_{\lambda}^\theta(v)\operatorname{-mod}_{\rho^{-1}(0)})$.
Furthermore, we have
\begin{equation}\label{eq:cat_action_simples}
\begin{split}
& F_\alpha L=\bigoplus_{i=0}^k \tilde{f}_\alpha L[m+2i]\oplus\bigoplus_{L', d(L')>d(L)}L'[?],\\
& E_\alpha L=\bigoplus_{i=0}^{\ell} \tilde{e}_\alpha L[n+2i]\oplus\bigoplus_{L'', d(L'')>d(L)}L''[?]
\end{split}
\end{equation}
\end{Lem}
Here $\tilde{f}_\alpha, \tilde{e}_\alpha$ are  maps  $\operatorname{Irr}(\mathcal{C})
\rightarrow \operatorname{Irr}(\mathcal{C})\sqcup \{0\}$ forming a crystal for $\slf_2$,
and $k,m,\ell,n$ are some numbers whose precise values are not important for us.

\begin{proof}
The claim that the functors $E_\alpha,F_\alpha$ preserve the category $\mathcal{C}'$
follows from the proof of Corollary \ref{Cor_RH}.
So we get a categorical $\slf_2$-action on the additive category $\mathcal{C}'$.
\cite[Theorem 5.8]{Rouquier_2Kac} applies to this action. It follows
that the basis $[L], L\in \operatorname{Irr}(\mathcal{C}),$
is a dual perfect basis for the $\slf_2$-action on $\bigoplus_v K_0(\A_\lambda^\theta(v)\operatorname{-mod}_{\rho^{-1}(0)})$ (meaning that the dual basis in $\bigoplus_v K_0(\A_\lambda^\theta(v)\operatorname{-mod}_{\rho^{-1}(0)})^*$
is perfect in the sense of Berenstein and Kazhdan, see \cite[Section 5]{BerKa}). This gives
rise to crystal operators $\tilde{e}_\alpha, \tilde{f}_\alpha$ on
$\bigsqcup_{v} \operatorname{Irr}(\A_\lambda^\theta(v)\operatorname{-mod}_{\rho^{-1}(0)})$
(\ref{eq:cat_action_simples}) follows.
\end{proof}


\begin{Cor}\label{Cor:C_CC}
We have $\CC(K_0(\mathcal{C}))=L_\omega^{\a}$.
\end{Cor}
\begin{proof}
It is clear from the construction of $\mathcal{C}$ and Proposition \ref{Prop:a_CC_intertw}
that $L_\omega^{\a}\subset \CC(K_0(\mathcal{C}))$. So let us prove the opposite inclusion.
Let $v$ be minimal such that $\CC(K_0(\mathcal{C}_v))\supsetneq L_\omega^{\a}[\nu]$.
Then $\nu$ is dominant (otherwise $\CC(K_0(\mathcal{C}_v))\subset \sum_\alpha \operatorname{im}F_\alpha$).
Pick $L\in \operatorname{Irr}(K_0(\mathcal{C}_v))$ with $\CC(L)\not\in L_\omega^\a[\nu]$
and $\alpha\in \Pi^\theta$ such that $d_\alpha(L)$ is maximal possible over all such $L$
and $\alpha$. Since $L\in \mathcal{C}$, we have $d_\alpha(L)>0$. So if $L'\in \operatorname{Irr}(\mathcal{C})$
satisfies $d_\alpha(L')>d_\alpha(L)$, then $\CC(L')\in L_\omega^{\a}$. Apply  (\ref{eq:cat_action_simples})
for the functor $F_\alpha$ and the simple $\tilde{e}_\alpha L$.  We see that, by Proposition
\ref{Prop:a_CC_intertw},  $\CC(F_\alpha (\tilde{e}_\alpha L))=f_\alpha \CC(\tilde{e}_\alpha L)$
lies in $L_\omega^{\a}$.  The characteristic cycles of the objects
$L'$ from (\ref{eq:cat_action_simples}) are also in $L_\omega^{\mathfrak{a}}$
by the choice of $\alpha,L$.  So (\ref{eq:cat_action_simples})
implies $\CC(L)\in L_\omega^{\a}$.
\end{proof}

So, indeed, Proposition \ref{Prop:fin_dim_cryst} implies (II) (and is equivalent to (II)
modulo (III)).

\subsection{Singular simples}\label{SS_fin_sing}
Recall that singular simples were defined in \ref{SSS_short_wc_outline}.
Let us start with an easy alternative characterization of a singular object.

\begin{Lem}\label{Lem:sing_altern_easy}
Let $\alpha\in \Pi^\theta$ and
$L\in \Irr(\A_\lambda^\theta(v)\operatorname{-mod}_{\rho^{-1}(0)})$.
Then the following are equivalent.
\begin{enumerate}
\item  $L$ is $\alpha$-singular.
\item $L\not\in \operatorname{im}\tilde{f}_\alpha$ if $\langle\nu,\alpha_i^\vee\rangle\geqslant 0$
or $L\not\in \operatorname{im}\tilde{e}_\alpha$ if $\langle\nu,\alpha_i^\vee\rangle\leqslant 0$.
\end{enumerate}
\end{Lem}
\begin{proof}
This follows from (\ref{eq:cat_action_simples}).
\end{proof}

\begin{proof}[Proof of Proposition \ref{Prop:singular_equiv}]
The proof is in several steps. Thanks to the construction of the functors $E_\alpha,F_\alpha$,
we may assume that $\alpha=\alpha^i$ and $\theta=\theta^+$.
Recall the quotient functor
$\underline{\pi}:\A^{\theta_i}_{\lambda_i}(v)\operatorname{-mod}\twoheadrightarrow
\A_\lambda^\theta(v)\operatorname{-mod}$ from \ref{SSS_quot_fun}. Let $\tilde{L}$ denote the simple
in  $\A^{\theta_i}_{\lambda_i}(v)\operatorname{-mod}$ with $\underline{\pi}(\tilde{L})=L$.
Note that $L$ is $\alpha$-singular if and only if $\tilde{L}$ is singular (for the Webster
functors $E_i,F_i$). If $L_1$ is such that, say, $\tilde{f}_\alpha L_1=L$
and $\tilde{L}_1$ is the simple in $\A^{\theta_i}_{\lambda_i}(v)\operatorname{-mod}$ with $\underline{\pi}(\tilde{L}_1)=L_1$, then $\tilde{f}_i\tilde{L}_1=\tilde{L}$ because $\underline{\pi}$ intertwines $F_i$ with $F_i$ and $E_i$ with $E_i$.

{\it Step 1}. Let us prove that (2) implies (1). Assume the contrary: $H_0(\WC_{\theta\rightarrow \theta'}L)\neq 0$ but $L$ is not $\alpha$-singular. Then $\tilde{L}$ is not singular.  On the other hand, $H_0(\WC^i_{\theta_i\rightarrow -\theta_i}\tilde{L})\neq 0$, this follows from \ref{SSS_WC_quotients}. Since $\tilde{L}$ is holonomic, a direct analog of Proposition \ref{Prop:long_shift} applies. Thanks to that, we see that the singular support of $\Gamma_{\lambda_i}^{\theta_i}(\tilde{L})$ intersects the open symplectic leaf in the affinization of $T^*\operatorname{Gr}(v_i,\tilde{w}_i)\times T^*\underline{R}$. Equivalently, the singular support of $\tilde{L}$ intersects $\mathbb{O}\times T^*\underline{R}$, where $\mathbb{O}$ is the open $\GL(\tilde{w}_i)$-orbit
in $T^*\operatorname{Gr}(v_i,\tilde{w}_i)$. We claim that this contradicts the condition that $\tilde{L}$
is not singular. Indeed, in the sake of being definite, assume $2v_i\leqslant \tilde{w}_i$ so that
$\tilde{L}\in \operatorname{im} \tilde{f}_i$. From the construction of the functor $F_i$,
the singular support of $\tilde{L}$ lies in the image of $Y\times T^*\underline{R}$ in
$T^*\operatorname{Gr}(v_i,\tilde{w}_i)$, where we write   $Y$ for the conormal bundle
to $\operatorname{Fl}(v_i-1,v_i;\tilde{w}_i)\subset \operatorname{Gr}(v_i-1,\tilde{w}_i)
\times \operatorname{Gr}(v_i,\tilde{w}_i)$. But the image of $Y$ does not intersect $\mathbb{O}$.
Contradiction. This finishes the proof of the implication (2)$\Rightarrow$(1).

{\it Step 2}. Let us prove that (1) implies (2). Here we will use Theorem \ref{Thm:WC} that says
that $s_{i*}\circ\WC_{\theta\rightarrow \theta'}=\Theta_i$. So (2) is equivalent to $H_0(\Theta_i L)\neq 0$. Below, to simplify the notation, we write $E,F,\Theta$ for $E_i,F_i,\Theta_i$.

In the proof we will assume that $2v_i\leqslant \tilde{w}_i$, the other case is similar.
Recall, see \ref{SSS_Rickard}, that $\Theta L$ is the iterated cone of
$$F^{(\ell)}L[-m]\rightarrow F^{(\ell+1)}EL[1-m]\rightarrow \ldots
\rightarrow F^{(\ell+m)}E^{(m)}L,$$
where $m=v_i,\ell=\tilde{w}_i-2v_i$. Then $\Theta^2 L$ is the cone of
the double complex with the term in the slot $(i-v_i,j-\tilde{w}_i+v_i)$ of the form $F^{(i)}E^{(\ell+i)}F^{(\ell+j)}E^{(j)}L[i+j-\tilde{w}_i]$. The terms with
$i>0$ do not contain $L$ in their homology because $L$ is singular. If $i=0, j>0$,
then we can commute $E^{(\ell)}$ and $F^{(\ell+j)}$ using the categorical
$\mathfrak{sl}_2$-relations, see, e.g., (iii) in  \cite[Section 2.2]{Kaetc}.
We get that these terms do not contain $L$ either.
For the same reason, the $(0,0)$ term splits into a direct sum of $L$ and some
object that does not contain $L$ in the homology. The summand $L$ will
contribute to $H_0$ of the iterated cone of $\Theta^2 L$. Now recall that
$L\mapsto \Theta L$ is right $t$-exact by Theorem \ref{Thm:WC}.
Since $H_0(\Theta^2 L)\neq 0$, we deduce that $H_0(\Theta L)\neq 0$.
This finishes the proof of the implication (1)$\Rightarrow$(2).

{\it Step 3}. Let us prove the claim about the singular simples in $H_*(\WC_{\theta\rightarrow \theta'}L)$:
only one occurs as a composition factor and it is a quotient of $H_0$.
This claim is equivalent to an analogous claim for $H_*(\WC^i_{\theta_i\rightarrow -\theta_i}\tilde{L})$.
By Step 1, the singular support of $\tilde{L}$ contains a point that is stable for $-\theta_i$.
Let $\pi_+,\pi_-$ be the quotient functors from $D_R\operatorname{-mod}^{G,\lambda}$ to the quotient categories
for the stability conditions $\theta_i,-\theta_i$. Then $\WC^i_{\theta_i\rightarrow -\theta_i}=\pi_- L\pi_+^!$, see
(\ref{eq:wc_formula}). Let $\widehat{L}$ be the simple in $D(R)\operatorname{-mod}^{G,\lambda}$
such that $\pi_+(\widehat{L})=\tilde{L}$. We see that $\pi_-(\widehat{L})$
occurs in $H_0(\WC_{\theta_i\rightarrow -\theta_i}(\tilde{L}))$ (as a quotient, in fact, because $\widehat{L}$
is a quotient of $\pi^!_+(\tilde{L})$). Clearly, the multiplicity is $1$. The other composition
factors of $H_j(\WC_{\theta_i\rightarrow -\theta_i}(\tilde{L}))$ are the images of simples in $\ker \pi_+$. So the
singular supports do not intersect $\mathbb{O}\times T^*\underline{R}$. Reversing the argument
of Step 1, we see that these simples are shifted by $\WC$ and hence by $\Theta$. By Step 2, they cannot be singular.
So we get  simples $\tilde{L}':=\pi_-(\widehat{L})\in \A_{\lambda_i'}^{-\theta_i}(v)\operatorname{-mod}$
and the corresponding simple $L'\in \A^{\theta'}_{\lambda'}(v)\operatorname{-mod}_{\rho^{-1}(0)}$.
Note that $\tilde{L}'$ singular by the argument above in this step. So $L'$ is singular.

{\it Step 4}. It remains to prove that the map $L\mapsto L'$ is a bijection between the sets of
$\alpha$-singular objects. By Step 3, $s_*\tilde{L}'$ is the only singular simple constituent
of $\bigoplus_i H_i(\Theta \tilde{L})$, it is a quotient of $H_0(\Theta \tilde{L})$. The proof
of (2)$\Rightarrow$(1) implies that $\tilde{L}$ is the only singular constituent of $\bigoplus_i H_i(\Theta (s_*\tilde{L}'))$. This gives rise to a map from the set of $\alpha$-singular simples in
$\A_{\lambda}^{s_i\theta}(v)\operatorname{-mod}_{\rho^{-1}(0)}$ to the set of singular simples in $\A_{\lambda}^{\theta}(v)\operatorname{-mod}_{\rho^{-1}(0)}$ that is the inverse to $L\mapsto L'$.
\end{proof}

\section{Affine short wall-crossing}\label{S_affine}
Here we consider the situation when $Q$ is an affine quiver, $v=n\delta, w=\epsilon_0$, and
$\A_\lambda(v)=eH_{\kappa,c}(n)e$. Let us note that $\mu$ is flat. In this section we study the wall-crossing functor through
the wall $\ker\delta$ proving in particular that the homological shifts of modules under this
functor are less than $n$ and that the functor is a perverse equivalence. The proof in
a more general situation is obtained in \cite[Section 3]{perv} (see also \cite[Section 6]{rouq_der}).

\subsection{Results}\label{SS_aff_res}
Let us introduce some conventions and notation.

Pick a generic stability condition $\theta$ in a classical chamber $C'$ that has
$\ker\delta$ as a wall. Let $C'$ denote the classical chamber sharing the wall $\ker\delta$ with $C$.

We consider a parameter $\lambda^\circ$ such that $(\lambda,\theta)\in \mathfrak{AL}(v)$
for any $\lambda\in \lambda^\circ+ (C\cap \Z^{Q_0})$, such a parameter exists by (2) of Proposition
\ref{Lem:ab_loc}.
The parameter $\lambda^\circ$ is represented in the form $(\kappa, c^\circ)$, where $\kappa=\langle \lambda^\circ,\delta\rangle$ so that $\A_{\lambda^\circ}(v)=e H_{\kappa,c^\circ}(n)e$. We view $c^\circ$ as an element of $\ker\delta$ (this includes some renormalization of the usual parameters for the SRA's).
Similarly, choose a parameter $\lambda'^\circ=(\kappa',c'^\circ)\in\lambda^\circ+\Z^{Q_0}$
such that $(\lambda',\theta')\in \AL(v)$ for any $\lambda'\in \lambda'^\circ+(C'\cap \Z^{Q_0})$.
Set $\chi:=\lambda'^\circ-\lambda^\circ$. Note that $\kappa'-\kappa=\langle \chi,\delta\rangle\in \Z$.


For $\lambda=(\kappa,c)$, we set $\A^c:=\A_\lambda(v), \A'^c:=\A_{\lambda+\chi}(v), \mathcal{B}^c:=\A^0_{\lambda,\chi}(v)$.
 We also consider the universal versions: we write $\param_0$ for $\ker\delta$   and consider the objects $\A^{\param_0}:=\A_{\lambda+\param_0}(v),\A'^{\,\param_0},\B^{\param_0}:=\A^0_{\lambda+\param_0,\chi}(v)$
 so that $\A^c,\A'^c,\B^c$ are the specializations of $\A^{\param_0},\A'^{\,\param_0},\B^{\param_0}$.

Let $m$ denote the denominator of $\kappa$ (we set $m=\infty$ if $\kappa$ is irrational)
if $\kappa\not\in \Z$. For $\kappa\in \Z$ we assume that $m=\infty$.

We will define chains of ideals $\{0\}=\J^{\param_0}_{q+1}\subsetneq \J^{\param_0}_q\subsetneq\ldots\subsetneq\J^{\param_0}_1\subsetneq \J^{\param_0}_0=\A^{\param_0}, \{0\}=\J'^{\,\param_0}_{q+1}\subsetneq \J'^{\,\param_0}_q\subsetneq\ldots\subsetneq\J'^{\,\param_0}_1\subsetneq \J'^{\,\param_0}_0=\A'^{\,\param_0}$, where $q=\lfloor n/m\rfloor$, and consider the corresponding  specializations $\J^c_i,\J'^c_i$.

One more piece of notation: $d_i:=(q+1-i)(m-1)$.

Here is our main technical result.

\begin{Thm}\label{Thm:affine_WC}
There is a principal open subset $\param^0_0\subset \param_0$ such that the HC bimodules $$\J^{\param_0}_i,\A^{\param_0}/\J_i^{\param_0}, \J_i'^{\,\param_0},\A'^{\,\param_0}/\J'^{\,\param_0},
\B^{\param_0}, \operatorname{Tor}^{\A^{\param_0}}_j(\B_{\param_0}, \A^{\param_0}/\J^{\param_0}_{i}),
\operatorname{Tor}^{\A'^{\param_0}}_j(\A'^{\param_0}/\J'^{\param_0}_{i},\B_{\param_0})$$
localized to $\param_0^0$ are free both as left and as right modules over $\C[\param^0_0]$
and moreover, for any $c\in \param_0^0$, the following claims are true:
\begin{enumerate}
\item $\J^c_i \J^c_j=\J^c_{\max(i,j)}, \J'^c_i \J'^c_j=\J'^c_{\max(i,j)}$.
\item For all $i,j$, we have $\J'^c_{i}\operatorname{Tor}^{\A^c}_j(\B^c, \A^c/\J^c_{i})=\operatorname{Tor}^{
\A'^c}_j(\A'^c/\J'^c_{i}, \B^c)\J^c_{i}=0$.
\item We have   $\operatorname{Tor}^{\A^c}_j(\B^c, \A^c/\J^c_{i})=0$
  for $j<d_i$.
\item We have $\J'^c_{i-1}\operatorname{Tor}^{\A^c}_j(\B^c,\A^c/\J^c_{i})=0$ for $j>d_i$.
\item Set $\B^c_i:=\operatorname{Tor}^{\A^c}_{d_i}(\B^c,\A^c/\J^c_{i})$.
Then $\J'^c_{i-1}\B^c_i=\B^c_i$.
\item The kernel and the cokernel of the natural homomorphism $$\B^c_{i}\otimes_{\A^c}
\operatorname{Hom}_{\A'^c}(\B^c_{i}, \A'^c/\J'^c_{i})\rightarrow \A'^c/\J'^c_{i}$$
are annihilated by $\J'^c_{i-1}$ on the left and on the right. Similarly,
the kernel and the cokernel of the natural homomorphism
$$\operatorname{Hom}_{\A^c}(\B^c_{i}, \A^c/\J_i^{c})\otimes_{\A'^c}\B^c_{i}
\rightarrow \A^c/\J^c_{i}$$ 
are annihilated on the left and on the right by $\J^c_{i-1}$.
\end{enumerate}
\end{Thm}

We remark that under the freeness condition we have imposed on $\param_0^0$, the bimodules with superscript
$c$ are the specializations of those with superscript $\param_0$ provided $c\in \param_0^0$. For example,
$\operatorname{Tor}^{\A^c}_j(\B^c, \A^c/\J^c_{i})=\operatorname{Tor}^{\A^{\param_0}}_j(\B^{\param_0}, \A^{\param_0}/\J^{\param_0}_{i})_c$. The existence of an open subset satisfying the freeness condition
follows from (2) of Corollary \ref{Cor:HC_supp}.

The scheme of the proof of Theorem \ref{Thm:affine_WC} is as follows. We first prove the theorem for the algebras $\bar{\A}_\kappa(n), \bar{\A}_{\kappa'}(n)$, where $m=n$, in Section \ref{SS_aff_typeA_easy}.
In this case we just have one proper ideal in either of these two algebras. Then, in Section \ref{SS_aff_ideals}, we construct the ideals $\J^{\param_0}_i,\J'^{\param_0}_i$ in general. After that we prove (2)-(6) of Theorem \ref{Thm:affine_WC}, first, for a Weil generic parameter $c$ and then for a Zariski generic parameter, Section \ref{SS_aff_proof1}.

Of course, $c^\circ+ (C_\theta\cap \ker\delta\cap \Z^{Q_0})$ intersects $\param_0^0$.
We remark that for $c\in [c^\circ+ (C_\theta\cap \ker\delta\cap \Z^{Q_0})]\cap \param_0^0$,
thanks to Lemma \ref{Lem:shift_coinc}, the functor $\B^c\otimes^L_{\A^c}\bullet$
is just $\WC_{\theta\rightarrow \theta'}$.

Theorem \ref{Thm:affine_WC} is used to prove that  $\WC_{\theta\rightarrow \theta'}:D^b(\A_{\lambda^\circ}(v)\operatorname{-mod})
\rightarrow D^b(\A_{\lambda'^\circ}(v)\operatorname{-mod})$ is a perverse equivalence.

Let us recall the general definition. Let $\mathcal{C},\mathcal{C}'$ be two abelian categories
equipped with filtrations $\{0\}= \mathcal{C}_{N+1}\subsetneq \mathcal{C}_N\subsetneq \mathcal{C}_{N-1}\subsetneq \ldots
\subsetneq \mathcal{C}_1\subsetneq\mathcal{C}_0=\mathcal{C}, \{0\}= \mathcal{C}'_{N+1}\subsetneq \mathcal{C}'_N\subsetneq\ldots\subsetneq \mathcal{C}'_1\subsetneq\mathcal{C}'_0=\mathcal{C}'$ by Serre subcategories.  Following Chuang and Rouquier, \cite[Section 2.6]{rouquier_ICM}, we say that a derived
equivalence $\varphi:D^b(\mathcal{C})\rightarrow D^b(\mathcal{C}')$ is {\it perverse} with respect to the filtrations
above if
\begin{itemize}
\item[(i)] $\varphi$ restricts to an equivalence $D^b_{\mathcal{C}_i}(\mathcal{C})\rightarrow D^b_{\mathcal{C}'_i}(\mathcal{C}')$,
where we write $D^b_{\mathcal{C}_i}(\mathcal{C})$ for the full subcategory of $D^b(\mathcal{C})$ consisting of all
complexes with homology in $\mathcal{C}_i$.
\item[(ii)] $H_j(\varphi M)=0$ for $M\in \mathcal{C}_i$ and $j<i$.
\item[(iii)] The functor $M\mapsto H_i(\varphi M)$ induces an equivalence $\mathcal{C}_i/\mathcal{C}_{i+1}\xrightarrow{\sim} \mathcal{C}'_i/\mathcal{C}'_{i+1}$. Moreover, $H_j(\varphi M)\in \mathcal{C}'_{i+1}$ for $j>i$
    and $M\in \mathcal{C}_{i}$.
\end{itemize}

We remark that, thanks to (iii), a perverse equivalence induces a natural bijection between the simple objects in $\mathcal{C}$ and $\mathcal{C}'$. We will write $S\mapsto S'$ for this bijection.

\begin{Thm}\label{Thm:perv}
Set $\mathcal{C}:=\A_{\lambda^\circ}(v)\operatorname{-mod}, \mathcal{C}':=\A_{\lambda'^\circ}(v)\operatorname{-mod}$.
Define $\mathcal{C}_i$ to be the subcategory of all modules in $\mathcal{C}$ annihilated by $\J_{q+1-\lfloor i/(m-1)\rfloor}$ (this is a Serre subcategory by (1) of Theorem \ref{Thm:affine_WC}) and $\mathcal{C}'_i\subset \mathcal{C}_i$ analogously. Then, perhaps after replacing $\lambda^\circ$ with $\lambda^\circ+\psi$ for
$\psi\in C\cap \ker\delta\cap\Z^{Q_0}$, the following holds.
\begin{enumerate}
\item $\WC_{\theta\rightarrow \theta'}$ is a  perverse equivalences with respect
to these filtrations.
\item The induced equivalence $\mathcal{C}_{j(m-1)}/\mathcal{C}_{j(m-1)+1}\rightarrow \mathcal{C}'_{j(m-1)}/\mathcal{C}'_{j(m-1)+1}$
is given by $\B^c_{q+1-j}\otimes_{\A^c}\bullet$. Moreover, for a simple $S\in \mathcal{C}_{j(m-1)}\setminus\mathcal{C}_{j(m-1)+1}$, the head of $\B^c_{q+1-j}\otimes_{\A^c}S$
coincides with $S'$.
\item The bijection $S\mapsto S'$ preserves the associated varieties of the annihilators.
\end{enumerate}
\end{Thm}

\begin{Rem}
A direct analog of Theorem \ref{Thm:perv} holds for all wall-crossing functors through faces
of classical chambers. This is proved in a subsequent paper, \cite{perv}, by the second named
author. In particular, the short wall-crossing functors through real walls studied in
Section \ref{S_fin_WC} and the bijection between singular objects $L\mapsto L'$
we considered from Proposition \ref{Prop:singular_equiv} is a restriction of the bijection
from  (3) of generalized Theorem \ref{Thm:perv}. Still, Proposition \ref{Prop:singular_equiv}
cannot be entirely replaced by a generalization of Theorem \ref{Thm:perv} as the former
relates the wall-crossing functor to the categorification functors $E_\alpha, F_\alpha$.
\end{Rem}

\subsection{HC bimodules for Symplectic reflection algebras}
Let $V,\Gamma,\mathcal{H}_c$ etc. have the same meaning as in \ref{SSS_SRA}. In this
section we recall some known facts about HC bimodules over the algebras $\mathcal{H}_c$
and $e\mathcal{H}_ce$. The most important cases are $\Gamma=\Gamma_n,\mathfrak{S}_n$.

Let us recall a description of the symplectic leaves in $V/\Gamma$. The leaves are parameterized by
conjugacy classes of stabilizers for the $\Gamma$-action on $V$. Namely, to a stabilizer $\Gamma'$
we assign the image of $\{v\in V| \Gamma_v=\Gamma'\}$ in $V/\Gamma$.

Below we will need a property of  restriction functors in the SRA setting. A similar property was obtained
in \cite[Proposition 3.7.2]{sraco} for a related, ``upgraded'', restriction functor.

Take a symplectic leaf $\mathcal{L}\subset V/\Gamma_n$ and consider the full subcategory  $\HC_{\overline{\mathcal{L}}}(\A_{\paramq}(v))\subset \HC(\A_{\paramq}(v))$ consisting of all HC bimodules $\M$ such that
$\VA(\M)\cap \M_0(v)\subset\overline{\mathcal{L}}$. Similarly, define the subcategory $\HC_{fin}(\hat{\A}_{\paramq}(\hat{v}))$.

\begin{Prop}\label{Prop:restr_adj}
For $x\in \mathcal{L}$, the functor $\bullet_{\dagger,x}: \HC_{\overline{\mathcal{L}}}(\A_{\paramq}(v))\rightarrow \HC_{fin}(\hat{\A}_{\paramq}(\hat{v}))$ admits a right adjoint
$\bullet^{\dagger,x}:\HC_{fin}(\hat{\A}_{\paramq}(\hat{v}))\rightarrow
\HC_{\overline{\mathcal{L}}}(\A_{\paramq}(v))$.
\end{Prop}
\begin{proof}
As in the proof of \cite[Proposition 3.7.2]{sraco}, we reduce the proof to showing the following claim: a Poisson $\C[\overline{\mathcal{L}}]$-submodule of $\C[\mathcal{L}]^{\wedge_x}$ that is finitely generated over $\C[\overline{\mathcal{L}}]$ and is weakly equivariant under the action of $\C^\times$ on $\mathcal{L}$ is contained in $\C[\mathcal{L}]\subset\C[\mathcal{L}]^{\wedge_x}$.
This is a special case of \cite[Lemma 3.9]{B_ineq}.
\end{proof}

Here is an application of the restriction functors obtained in \cite[Section 5]{sraco}. Consider the case when
$\Gamma=\mathfrak{S}_n$ and $V=\mathfrak{h}\oplus \mathfrak{h}^*$, where $\h=\C^{n-1}$
is the reflection representation of $\mathfrak{S}_n$. The resulting algebra $\mathcal{H}_{\kappa}(n)$
is known as the Rational Cherednik algebra of type A, in our previous notation $\mathcal{H}_{\kappa,\varnothing}(n)
=D(\C)\otimes \mathcal{H}_\kappa(n), e \mathcal{H}_\kappa(n)e=\bar{\A}_\kappa(n)$. One can describe all two-sided ideals in $\mathcal{H}_{\kappa}(n)$, see \cite[Section 5.8]{sraco}.

\begin{Prop}\label{Prop:RCA_typeA_id}
If $\kappa$ is irrational or $\kappa=\frac{r}{m}$, where $\operatorname{GCD}(r,m)=1$ and $m>n$, then the algebra $\mathcal{H}_{\kappa}(n)$
is simple. Otherwise, there are $q:=\lfloor n/m\rfloor$ proper ideals that form a chain: $\{0\}=\J_{q+1}\subsetneq
\J_q\subsetneq \J_{q-1}\subsetneq\ldots \subsetneq \J_1\subsetneq \J_0:=\mathcal{H}_{\kappa}(n)$. The associated variety
of $\mathcal{H}_{\kappa}/\J_i$ is the closure of the symplectic leaf associated to the parabolic subgroup $\mathfrak{S}^{q+1-i}_m\subset
\mathfrak{S}_n$. Moreover, we have $\J_i\J_j=\J_{\max(i,j)}$.
\end{Prop}

We will also need the following lemma  that is a consequence of Proposition \ref{Prop:RCA_typeA_id}
and Lemma \ref{Cor:assoc_var}. Consider a point $x$ in the
symplectic leaf of $(\mathfrak{h}\oplus \mathfrak{h}^*)/\mathfrak{S}_n$ corresponding to the parabolic subgroup
$\mathfrak{S}_m^q\subset \mathfrak{S}_n$, where $q=\lfloor n/m\rfloor$.

\begin{Lem}\label{Lem:RCA_id_techn}
The functor $\bullet_{\dagger,x}$ is faithful.
\end{Lem}

Let us use the notation from Proposition \ref{Prop:RCA_typeA_id}.
Set $\hat{\mathcal{H}}=H_{\kappa}(m)$ and let $\hat{\J}$ be the only proper two-sided ideal in
$\hat{\mathcal{H}}$. Let $\hat{\J}_i$ denote the two-sided ideal in $\hat{\mathcal{H}}^{\otimes q}$
that is obtained as the sum of all $q$-fold tensor products of $i$ copies of $\hat{\J}$
and $q-i$ copies of $\hat{\mathcal{H}}$ (in all possible orders).

\begin{Lem}\label{Lem:RCA_id_techn2}
Let $\J$ be a two-sided ideal in $\hat{\mathcal{H}}^{\otimes q}$. If
$\VA(\hat{\mathcal{H}}^{\otimes q}/\J)=\VA(\hat{\mathcal{H}}^{\otimes q}/\hat{\J}_i)$,
then $\J=\hat{\J}_i$.
\end{Lem}
\begin{proof}
Let $x$ be a generic point in an irreducible component  in $\VA(\hat{\mathcal{H}}^{\otimes q}/\hat{\J}_i)$.
Recall that these components are labelled by $i$-element subsets of $\{1,\ldots,q\}$.
Then $\J_{\dagger,x}$ is a proper ideal in $\hat{\mathcal{H}}^{\otimes i}$.
It follows that $\J$ lies in the kernel of the projection of $\hat{\mathcal{H}}^{\otimes q}$
to the product of $q-i$ copies of $\hat{\mathcal{H}}$ and $i$ copies of
$\hat{\mathcal{H}}/\hat{\J}_i$ (in all possible orders). But
by Step 3 in the proof of \cite[Theorem 5.8.1]{sraco}, the intersection of
these kernels is $\hat{\J}_i$. So $\J\subset\hat{\J}_i$. Moreover, by {\it loc.cit.},
the kernels are precisely the minimal prime ideals containing $\J$.
So $\hat{\J}_i$ is the radical of $\J$. By {\it loc.cit.}, $\hat{\J}_i^2=\hat{\J}_i$.
Therefore $\hat{\J}_i=\J$.
\end{proof}

\subsection{Type A: case $\kappa=\frac{r}{n}$}\label{SS_aff_typeA_easy}
Here we assume that $n=m, \kappa=\frac{r}{n}$ with $\operatorname{GCD}(r,n)=1$
and $\kappa\not\in (-1,0)$.  In this case we have just one proper ideal $\J\subset \mathcal{H}:=\mathcal{H}_\kappa(n)$ and the quotient
$\mathcal{H}/\mathcal{J}$ is finite dimensional, this is a special case of Proposition
\ref{Prop:RCA_typeA_id}. The algebra $\bar{\A}_\kappa(m)$ is Morita equivalent
to $\mathcal{H}$.

\begin{Lem}
Claims (2)-(6) of Theorem \ref{Thm:affine_WC} hold for the algebras $\bar{\A}_\kappa(m)$.
\end{Lem}
\begin{proof}
These claims amount to the following two claims (*) and (**), where
(*) implies those claims for $i=1$, while (**) implies them for $i=0$):
\begin{itemize}
\item[(*)] We have $\operatorname{Tor}^{\mathcal{H}}_j(\B, \mathcal{H}/\J)=
\operatorname{Tor}^{\mathcal{H}'}_j(\mathcal{H}'/\J',\B)=0$ if $j\neq n-1$.
Furthermore, $\operatorname{Tor}^{\mathcal{H}}_{n-1}(\B, \mathcal{H}/\J)=
\operatorname{Tor}^{\mathcal{H}'}_{n-1}(\mathcal{H}'/\J',\B)=\operatorname{Hom}_{\C}(L,L')$,
where we write $L$ (resp., $L'$) for the simple finite dimensional
$\mathcal{H}$-module (resp., $\mathcal{H}'$-module).
\item[(**)] The kernels and cokernels of the natural homomorphisms
$$\B\otimes_{\mathcal{H}} \operatorname{Hom}_{\mathcal{H'}}(\B,\mathcal{H}')\rightarrow \mathcal{H}',
\operatorname{Hom}_{\mathcal{H}}(\B,\mathcal{H})\otimes_{\mathcal{H}'}\B\rightarrow \mathcal{H}$$
are finite dimensional.
\end{itemize}

The Tor vanishing statement in (*) is a consequence of (1) of Proposition \ref{Prop:long_shift}
(by Remark \ref{Rem:long_WC_opp}, $\B\otimes^L_{\mathcal{H}'}\bullet$ is still the wall-crossing functor for the categories of right modules). The category of finite dimensional $\mathcal{H}$-modules
(resp., of finite dimensional $\mathcal{H}'$-modules) is a semisimple category
with a single indecomposable object $L$ (resp., $L'$). By
(2) of Proposition \ref{Prop:long_shift}, $\B\otimes^L_{\mathcal{H}}L=L'[n-1]$
and $L'\otimes^L_{\mathcal{H}'}\B=L[n-1]$. So the equality for the
$\operatorname{Tor}_{n-1}$'s follows from  the previous sentence and isomorphisms
$\mathcal{H}/\J=\Hom_{\C}(L,L), \mathcal{H}'/\J'=\Hom_{\C}(L',L')$.

Let us proceed to (**). Apply the functor $\bullet_{\dagger,x}$ to the homomorphisms
of interest, where $x\in \M_0(v)$ is generic. For the homomorphism
$$\B\otimes_{\mathcal{H}} \operatorname{Hom}_{\mathcal{H'}}(\B,\mathcal{H}')\rightarrow \mathcal{H}'$$
we get a natural homomorphism
$$\B_{\dagger,x}\otimes_{\mathcal{H}_{\dagger,x}} \operatorname{Hom}_{\mathcal{H'}_{\dagger,x}}
(\B_{\dagger,x},\mathcal{H}'_{\dagger,x})\rightarrow \mathcal{H}_{\dagger,x}.$$
But the algebras and bimodules involved are all just $\C$. So we see that the latter
homomorphism is an isomorphism. It follows that the kernel and the cokernel
of $\B\otimes_{\mathcal{H}} \operatorname{Hom}_{\mathcal{H'}}(\B,\mathcal{H}')\rightarrow \mathcal{H}'$
are killed by $\bullet_{\dagger,x}$. So they have proper associated varieties and hence are
finite dimensional.
\end{proof}

\subsection{Chain of ideals}\label{SS_aff_ideals}
For a partition $\mu=(\mu_1,\ldots,\mu_k)$ with $|\mu|\leqslant n$ set $\bar{\A}(\mu)=\bigotimes_{i=1}^k\bar{\A}_\kappa(\mu_i)$ and define $\bar{\A}'(\mu)$
similarly. Consider the restriction functors $\bullet_{\dagger,\mu}:\HC(\A^{\param_0})\rightarrow
\HC(\C[\param_0]\otimes \bar{\A}(\mu)), \HC(\A'^{\,\param_0}\operatorname{-}\A^{\param_0})\rightarrow \HC(\C[\param_0]\otimes \bar{\A}'(\mu)\operatorname{-}\C[\param_0]\otimes\bar{\A}(\mu))$ etc. Let $\bullet^{\dagger, \mu}$ denote the right adjoint functor (defined on bimodules that are finitely generated over
$\C[\param_0]$ by Proposition \ref{Prop:restr_adj}). Recall, Proposition \ref{Prop:RCA_typeA_id},
that the ideals in the algebra $\bar{\A}(qm)$ form a chain: $\bar{\A}(qm)=\bar{\J}_0(qm)\supsetneq \bar{\J}_1(qm)\supsetneq \ldots\supsetneq \bar{\J}_q(qm)\supsetneq \bar{\J}_{q+1}(qm)=\{0\}$.
We set $\bar{\J}_i(m^q):=\bar{\J}_i(qm)_{\dagger,m^q}$. These are precisely the ideals
appearing before Lemma \ref{Lem:RCA_id_techn2}.

We set $\J^{\param_0}_i$ to be the kernel of the natural map $$\A^{\param_0}\rightarrow (\C[\param_0]\otimes [\bar{\A}(m)/\bar{\J}_1(m)]^{\otimes q+1-i})^{\dagger, (m^{q+1-i})},$$
and define $\J'^{\,\param_0}_i$ similarly.

\begin{Rem}\label{Rem:Ji_defn}
Note that, by the definition of $\J^{\param_0}_i$ the following is true. If $\J\subset \A^{\param_0}$ is such that
$\J_{\dagger, (m^{q+1-i})}$ is in the kernel of $\C[\param_0]\otimes\bar{\A}(m^{q+1-i})\twoheadrightarrow
\C[\param_0]\otimes [\bar{\A}(m)/\bar{\J}_1(m)]^{\otimes q+1-i}$, then $\J\subset \J^{\param_0}_i$.
\end{Rem}

We are going to establish some properties of these ideals. First, let us describe properties that hold for all
parameters $c$.

\begin{Lem}\label{Lem:J_prop_easy}
The following is true.
\begin{itemize}
\item[(a)] $(\J^c_i)_{\dagger,(m^{q+1-i})}$ coincides with the maximal ideal of $\bar{\A}(m^{q+1-i})$.
\item[(b)] $\VA(\A^c/\J^c_i)$ coincides with $\overline{\mathcal{L}}_{(m^{q+1-i})}$, where the latter
is the closure of the symplectic leaf corresponding to the subgroup $\mathfrak{S}_m^{q+1-i}\subset \Gamma_n$.
\item[(c)] $(\J^c_i)_{\dagger, (m^q)}=\bar{\J}_i(m^q)$. Moreover,
$(\J^c_i)_{\dagger,(qm)}=\bar{\J}_i(qm)$.
\item[(d)] $\J^{\param_0}_q\subset \J^{\param_0}_{q-1}\subset \ldots\subset \J^{\param_0}_1$.
\end{itemize}
Similar claims hold for $\J'^c_i, \J'^{\,\param_0}_i$.
\end{Lem}
\begin{proof}
The ideal $(\J^c_i)_{\dagger,(m^{q+1-i})}\subset \bar{\A}(m^{q+1-i})$ is contained in the maximal ideal
as the latter is the kernel of $\bar{\A}(m^{q+1-i})\twoheadrightarrow (\bar{\A}(m)/\bar{\J}_1(m))^{\otimes (q+1-i)}$.
The inclusion $\VA(\A^c/\J^c_i)\subset\overline{\mathcal{L}}_{(m^{q+1-i})}$ follows from Proposition \ref{Prop:restr_adj}.
By Lemma \ref{Lem:dag_assoc}, $\VA(\bar{\A}^{\otimes q+1-i}/\J^c_{i,\dagger,(m^{q+1-i})})$ is a point.
The equality in (a) follows from Lemma \ref{Lem:RCA_id_techn2}.
In its turn, the equality in (a) implies the equality in (b).

(b) implies that  $\VA(\bar{\A}(m^q)/(\J^c_i)_{\dagger, (m^q)})=
\VA(\bar{\A}(m^q)/\bar{\J}_i(m^q))$. Lemma \ref{Lem:RCA_id_techn2}
yields $(\J^c_i)_{\dagger, (m^q)}=\bar{\J}_i(m^q)$.
The equality $(\J^c_i)_{\dagger,(qm)}=\bar{\J}_i(qm)$ is proved similarly
using Proposition \ref{Prop:RCA_typeA_id}.

Let us prove (d). We remark that $(\J^c_i)_{\dagger,(m^{q+2-i})}$ is a proper ideal because
its associated variety (computed using Lemma \ref{Lem:dag_assoc}) is proper.
Hence $(\J^c_i)_{\dagger,(m^{q+2-i})}$ is contained in the maximal ideal of $\bar{\A}(m^{q+2-i})$.  It follows that
$(\J^{\param_0}_i)_{\dagger, (m^{q+2-i})}$ lies in the kernel of the
epimorphism $\C[\param_0]\otimes\bar{\A}(m^{q+2-i})\twoheadrightarrow \C[\param_0]\otimes (\bar{\A}(m)/\bar{\J}_1(m))^{\otimes (q+2-i)}$.
The inclusion $\J^{\param_0}_i\subset \J^{\param_0}_{i-1}$ follows from Remark
\ref{Rem:Ji_defn}.
\end{proof}

Now let us analyze what happens when $c$ is Weil generic.

\begin{Lem}\label{Prop:gen_c}
Let $c$ be Weil generic. Then the following is true:
\begin{enumerate}
\item The functor $\bullet_{\dagger,(m^q)}$ is faithful.
\item The ideals  $\J^c_i, i=1,\ldots,q,$
exhaust all proper ideals in $\A^c$.
\item $\J^c_i \J^c_j=\J^c_{\max(i,j)}$.
\end{enumerate}
\end{Lem}
\begin{proof}
Let us show that for a Weil generic $c$, the algebra $\A^c$ has no finite dimensional representations.
Similarly to the proof of Proposition \ref{Prop:gen_simpl}, we see that otherwise there is a two-sided ideal ${\bf J}\subset\A^{\param_0}$ such that $\A^{\param_0}/{\bf J}$ is generically flat and finite over $\C[\param_0]$
and  $\Supp^r_{\paramq}(\A^{\param_0}/{\bf J})=\param_0$. So, by Proposition \ref{Prop:HC_support},
for the Poisson ideal $\operatorname{gr}{\bf J}\subset\C[\M_{\param_0}(v)]$
we have $\Supp_{\param}(\C[\M_{\param_0}(v)]/\operatorname{gr}{\bf J})=\param_0$. It follows
that, for every $p\in \param_0$, the variety $\M_p(v)$ contains a point that is a symplectic leaf.  We remark that
$\M_{\param_0}(v)=\M_{\param_0}(\delta)^n/\mathfrak{S}_n$ (the power is taken over $\param_0$),
this follows from the description of $\M_{\param_0}(v)$
as the generalized Calogero-Moser space, see \cite[Section 11]{EG}.
For $p$ generic, $\M_p(\delta)$ is smooth and symplectic and so the minimal dimension of a symplectic leaf in
$\M_p(v)$ is $2$. We arrive at a contradiction that shows that $\A^c$ has no finite dimensional
representations provided $c$ is Weil generic.

Now we are in position to prove (1). This boils down to checking that the associated variety of any HC $\A^c$-bimodule (or HC $\A'^c$-$\A^c$-bimodule, etc.) contains $\mathcal{L}_{(m^q)}$. First of all, let us show that the associated variety contains $\mathcal{L}_{n}$, the symplectic leaf corresponding to $\mathfrak{S}_n\subset \Gamma_n$. Indeed, the slice algebra $\bar{\A}_\lambda(\hat{v})$ for any leaf not containing $\mathcal{L}_n$ has a tensor
factor isomorphic to $e \mathcal{H}_{\kappa,c}(n')e$ with nonzero $n'\leqslant n$.
But that algebra has no finite dimensional irreducible representations by the first
paragraph of the proof, a contradiction. So we see that the associated variety is contained in $\mathcal{L}_n$.
Hence it is $\mathcal{L}_{\mu}$ for some partition $\mu$. The slice algebra $\bar{\A}_\lambda(\hat{v})$
is the product $\bigotimes_{i=1}^k e\mathcal{H}_\kappa(\mu_i)e$.

That the associated variety contains $\mathcal{L}_{(m^q)}$ follows from the fact that the algebra $\mathcal{H}_\kappa(n')$
has a finite dimensional irreducible representation if and only if $\kappa$ has denominator precisely $n'$.
The proof of faithfulness of $\bullet_{\dagger,(m^q)}$ is now complete.

Let us proceed to the proof of (2) and (3). By (b) of Lemma \ref{Lem:J_prop_easy}, $\VA(\A^c/\J^c_i)=
\overline{\mathcal{L}}_{(m^{q+1-i})}$. The functor $\bullet_{\dagger, (qm)}$ is faithful,
this follows from the argument in the previous paragraph.
So the map $\J\mapsto \J_{\dagger,(qm)}$ embeds the poset of two-sided ideals in
$\A^c$ into that for $\bar{\A}(qm)$. (2) follows from here and
Proposition \ref{Prop:RCA_typeA_id}. To check (3) note that
$\bullet_{\dagger,(qm)}$ respects the products of ideals
and again use Proposition \ref{Prop:RCA_typeA_id}.
\end{proof}

Now let us transfer some of the properties in Lemma \ref{Prop:gen_c} to the case when $c$ is only Zariski generic.

\begin{Lem}\label{Lem:ideals_easy}
We have $\J^c_{i}\J^c_{j}=\J^c_{\max(i,j)}$ for $c$ in some non-empty Zariski open subset of $\param_0$.
\end{Lem}
\begin{proof}
Consider the quotient $\J^{\param_0}_{\max(i,j)}/\J^{\param_0}_{i}\J^{\param_0}_{j}$.
By (3) of Lemma \ref{Prop:gen_c}, its specialization to a Weil generic $c$
is zero. It follows from (2) of Corollary \ref{Cor:HC_supp} that the specialization of this HC bimodule
to a Zariski generic  $c\in \param_0$ is zero. This implies our claim.
\end{proof}

\subsection{Proof of Theorem \ref{Thm:affine_WC}}\label{SS_aff_proof1}
We write $\bar{\B}$ for the wall-crossing $\bar{\A}_{\kappa'}(m)$-$\bar{\A}_\kappa(m)$-bimodule.

Without restrictions on $c$, we know that
\begin{equation}\label{eq:gen_to_A}\begin{split}
&\B^c_{\dagger,(m^{q+1-i})}=\bar{\B}^{\otimes q+1-i},\\
&\Tor^{\A^c}_j(\B, \A^c/\J^c_i)_{\dagger,(m^{q+1-i})}=
\operatorname{Tor}_j^{\bar{\A}(m^{q+1-i})}(\bar{\B}^{\otimes q+1-i},
(\bar{\A}(m)/\J(m))^{\otimes q+1-i}),\\
&\Tor^{\A'^c}_j(\A'^c/\J'^c_i,\B)_{\dagger,(m^{q+1-i})}=
\operatorname{Tor}_j^{\bar{\A}'(m^{q+1-i})}((\bar{\A}'(m)/\J'(m))^{\otimes q+1-i},
\bar{\B}^{\otimes q+1-i}).
\end{split}\end{equation}

The first equality is a special case of (\ref{eq:transl_restr}). The second and third
equalities follows from the first and Lemma \ref{Lem:tens_dag_intertw}.

\begin{proof}[Proof of Theorem \ref{Thm:affine_WC}]
First, we  assume that $c$ is Weil generic. By (b) of Lemma \ref{Lem:J_prop_easy} combined with
(2) of Lemma \ref{Prop:gen_c},
for any HC $\A'^c$-$\A^c$ bimodule $\mathcal{X}$
the following are equivalent
\begin{itemize}
\item $\mathcal{X}_{\dagger, (m^{q+1-i})}=0$,
\item $\mathcal{X}\J^c_{i-1}=0$,
\item $\J'^c_{i-1}\mathcal{X}=0$.
\end{itemize}
This, combined with (\ref{eq:gen_to_A}) and Section \ref{SS_aff_typeA_easy}, yields (4) and (6).
Further, the following conditions are equivalent as well:
\begin{itemize}
\item $\dim\mathcal{X}_{\dagger, (m^{q+1-i})}<\infty$,
\item $\mathcal{X}\J^c_{i}=0$,
\item $\J'^c_{i}\mathcal{X}=0$.
\end{itemize}
This yields (2).

Let us prove (3) and (5). Suppose that (3) is false. Pick the minimal $i$ such that there is $j<d_i$ with $\mathcal{X}:=\Tor^{\A^c}_j(\B^c, \A^c/\J^c_i)\neq 0$. Next, let $j$ be minimal for the given $i$.
Since $\mathcal{X}_{\dagger, (m^{q+1-i})}=0$, we see that
\begin{equation}\label{eq:bimod_vanish}\mathcal{X}\J^c_{i-1}=0.\end{equation} Consider the derived tensor  product
\begin{equation}\label{eq:der_tens_prod_bimod}\B^c\otimes^L_{\A^c}\A^c/\J^c_{i-1}=(\B^c\otimes^L_{\A^c}\A^c/\J^c_{i})\otimes^L_{\A^c/\J^c_i}
\A^c/\J^c_{i-1}.\end{equation}
By the choice of $j$, the $j$th homology of the right hand side
of (\ref{eq:der_tens_prod_bimod}) equals
$\mathcal{X}\otimes_{\A^c/\J^c_i}\A^c/\J^c_{i-1}=\mathcal{X}/\mathcal{X}\J^c_{i-1}$.
The latter equals $\mathcal{X}$ by (\ref{eq:bimod_vanish}). Since $j<d_{i-1}$
and the left hand side of (\ref{eq:der_tens_prod_bimod}) has non-vanishing $j$th
homology, we get a contradiction with our choice of $i$. This proves $\mathcal{X}=0$.

The equality $\Tor^{\A'^c}_j(\A'^c/\J'^c_i,\B)=0$
for $j<d_i$ is proved in the same way (using that $\B$ is a long wall-crossing bimodule
also when viewed as a $\A^{c,opp}$-$\A'^{c,opp}$-bimodule, Remark \ref{Rem:long_WC_opp}).
This completes the proof of (3).

Let us proceed to  (5) and prove $\B_i^c\J_{i-1}^c=\B_i^c$. Assume the converse, then
$\B_i^c\otimes \A^c/\J^c_{i-1}\neq 0$. Similarly to the proof of $\mathcal{X}=0$, this implies
that $\operatorname{Tor}_{d_i}^{\A^c}(\B^c,\A^c/\J^c_{i-1})\neq \{0\}$ that contradicts (3).
This completes the proof of Theorem \ref{Thm:affine_WC} for a Weil generic $c$.

Let us prove (2)-(6) for a Zariski generic $c$. We will do (2), the other claims are similar.
Consider the HC $\A'^{\param_0}$-$\A^{\param_0}$ bimodule $\J'^{\param_0}_{i}
\operatorname{Tor}^{\A^{\param_0}}_j(\B^{\param_0}, \A^{\param_0}/\J^{\param_0}_{i})$.
Its specialization to a Zariski generic parameter $c$ coincides with
$\J'^{c}_{i}\operatorname{Tor}^{\A^{c}}_j(\B^{c}, \A^{c}/\J^{c}_{i})$. So a Weil generic
specialization of this bimodule vanishes. Therefore the same is true for a Zariski generic
specialization, this is a consequence of (2) of Corollary \ref{Cor:HC_supp}.
\end{proof}

\subsection{Proof of Theorem \ref{Thm:perv}}\label{SS_perv}
Let us check (i) in the definition of a perverse equivalence. Recall that $\WC_{\theta\rightarrow \theta'}$ is $\B^c\otimes^L_{\A^c}\bullet$
and hence $\WC_{\theta\rightarrow \theta'}^{-1}$ is $R\operatorname{Hom}_{\A'^c}(\B^c,\bullet)$.
For example, let us prove $\WC_{\theta\rightarrow \theta'}^{-1}D^b_{\mathcal{C}'_i}(\mathcal{C}')\subset
D^b_{\mathcal{C}_i}(\mathcal{C})$. For $M'$ annihilated by  $\J_i'^c$, we have $R\operatorname{Hom}_{\A'^c}(\B^c,M')=R\operatorname{Hom}_{\A'^c/\J'^c_i}(\A'^c/\J'^c_i\otimes^L_{\A'^c}\B^c,M')$.
Now we use (2) of Theorem \ref{Thm:affine_WC} which says, in particular, that all homology of $\A'^c/\J'^c_i\otimes^L_{\A'^c}\B^c$ are annihilated by $\J^c_i$ on the right. This checks (i).

(ii) follows from (3) of Theorem \ref{Thm:affine_WC} and the observation that,
for $M$ annihilated by $\J^c_i$ we have $\B^c\otimes^L_{\A^c}M=
(\B^c\otimes^L_{\A^c}\A^c/\J^c_i)\otimes^L_{\A^c/\J^c_i}M$.

Let us prove (iii). By (6) of Theorem \ref{Thm:affine_WC}, the functor
$\B^c_i\otimes_{\A^c/\J^c_i}\bullet:\A^c/\J^c_i\operatorname{-mod}
\rightarrow \A'^c/\J'^c_i\operatorname{-mod}$ induces an equivalence
$\Cat_{q+1-i}/\Cat_{q+2-i}\xrightarrow{\sim}\Cat'_{q+1-i}/\Cat'_{q+2-i}$
(for example, a right inverse is given by tensoring with
$\Hom_{\A'^c}(\B^c_i, \A'^c/\J'^c)$). (6) also implies that,
for $M\in \Cat_{q+1-i}$, we have $\operatorname{Tor}^{\A^c/\J^c_i}_j(\B^c_i,
M)\in \Cat'_{q+2-i}$ for all $j>0$. Together with (4) of Theorem
\ref{Thm:affine_WC} this completes the proof of (iii). This finishes
the proof of (1) of Theorem \ref{Thm:perv} and also establishes
the first claim in (2).

To complete the proof of (2) we need to check that $\J'^c_{q-i}(\B^c_{q+1-i}\otimes_{\A^c}S)=
\B^c_{q+1-i}\otimes_{\A^c}S$. By (5) of Theorem \ref{Thm:affine_WC}, the natural homomorphism
$\J'^c_{q-i}\otimes_{\A'^c}\B^c_{q+1-i}\rightarrow \B^c_{q+1-i}$ is surjective. It follows that
the natural homomorphism
$\J'^c_{q-i}\otimes_{\A'^c}\B^c_{q+1-i}\otimes_{\A^c}S\rightarrow \B^c_{q+1-i}\otimes_{\A^c}S$
is surjective as well. This finishes the proof of (2) of Theorem \ref{Thm:perv}.

To show (3) -- that the associated varieties of the annihilators are preserved -- one can argue as follows.
Let $\I$ denote the annihilator of $S$. So $\B^c_{q+1-i}\otimes_{\A^c}S$ is a quotient of
$\B^c_{q+1-i}\otimes_{\A^c}\A^c/\I$, a HC bimodule annihilated by $\I$ on the right.
From Corollary \ref{Cor:assoc_var} one can now deduce that the associated variety of the annihilator
$\I'$ of $S'$ is  contained in  that of  $\I$.  On the other hand,  $S$ is a submodule of
$\Hom_{\A'^c}(\B^c_{q+1-i},S')=\Hom_{\A'^c}(\B^c_{q+1-i}/\I' \B^c_{q+1-i},S')$.
So the right annihilator of $\B^c_{q+1-i}/\I' \B^c_{q+1-i}$ is contained in
the annihilator of $\Hom_{\A'^c}(\B^c_{q+1-i},S')$. The latter is contained in  $\I$.
This shows that the associated variety of $\I'$  contains  that of $\I$
and completes the proof of (3).

Theorem \ref{Thm:perv} is now proved.

\section{Proof of counting result and some conjectures}\label{S_proof_compl}
\subsection{Extremal simples}\label{SS_fin_extr}
Let us start by proving Proposition \ref{Prop:extremal_bij}.

\begin{proof}[Proof of Proposition \ref{Prop:extremal_bij}]
We will need to consider the following three cases separately. Fix $\a$. Then the set of $\lambda$
with $\a^\lambda=\a$ looks as follows: we take the union of a countable discrete collection
of affine subspaces in $\paramq$ and remove another countable discrete collection of affine subspaces.
We say that $\lambda$ is {\it generic with respect to $\a$} if $\lambda$ is Weil generic
in the closure of a connected component.

The three cases we consider are as follows:
\begin{enumerate}
\item $\lambda\in \Q^{Q_0}$,
\item $\lambda$ is generic with respect to $\a$.
\item $\lambda$ is arbitrary with $\a^\lambda=\a$.
\end{enumerate}
We will also see that in (2) and (3) the endomorphisms $[E_\alpha],[F_\alpha]$
of $\bigoplus_v K_0(\A_\lambda^\theta(v)\operatorname{-mod}_{\rho^{-1}(0)})$ give rise
to an action of $\a$.

{\it Case 1}. Consider the case when $\lambda$ is rational. In this case,
by Proposition \ref{Prop:inj_K0_rat}, we see that \begin{equation}\label{eq:K0_inclusion}\bigoplus_v
K_0(\A_\lambda^\theta(v)\operatorname{-mod}_{\rho^{-1}(0)})\hookrightarrow
\bigoplus_v K_0(\Coh_{\rho^{-1}(0)}(\M^\theta(v))).\end{equation} By Proposition
\ref{Prop:a_actions_intertwined}, we see that the operators
$[E_\alpha],[F_\alpha]$ give rise to an $\a$-action on $\bigoplus_v
K_0(\A_\lambda^\theta(v)\operatorname{-mod}_{\rho^{-1}(0)})$ and (\ref{eq:K0_inclusion})
is equivariant. So we see that $K_0(\mathcal{C})=\bigoplus_\sigma U(\a)K_0(\A_\lambda^\theta(\sigma\bullet w)\operatorname{-mod})$, where the summation is taken over all $\sigma\in W(Q)$
such that $\sigma\omega$ is dominant for $\a$.

Recall, Proposition \ref{Prop:degen_indep_theta}, that  $[\WC_{\theta\rightarrow \theta'}]$ intertwines the embeddings
 $$K_0(\A_\lambda^\theta(v)\operatorname{-mod}_{\rho^{-1}(0)}),
K_0(\A_\lambda^{\theta'}(v)\operatorname{-mod}_{\rho^{-1}(0)})\hookrightarrow K_0(\Coh_{\rho^{-1}(0)}(\M^\theta(v))).$$ It follows that
$\WC_{\theta\rightarrow \theta'}$ maps $D^b_{\mathcal{C}}(\A_\lambda^\theta(v)\operatorname{-mod})$
to $D^b_{\mathcal{C}}(\A_\lambda^{\theta'}(v)\operatorname{-mod})$, where
the subscript $\mathcal{C}$ means that we consider the objects with homology
in $\mathcal{C}$.

Now consider the wall-crossing functor $\WC_{\theta\rightarrow \theta'}$ through $\ker\alpha$.
Let $L$ be extremal in $\A_\lambda^\theta(v)\operatorname{-mod}_{\rho^{-1}(0)}$.
As we have pointed out in the proof of Lemma \ref{Lem:extremal_sing}, $\nu$ is dominant for $\a$. So all constituents
of $H_*(\WC_{\theta\rightarrow \theta'}L)$ but $L'$ lie in the image of $\tilde{f}_\alpha$
and hence, by the minimality assumption on $v$, in  $\mathcal{C}$. We deduce that
$L'\not\in \mathcal{C}$. So $L'$ is extremal provided the minimality assumption on
$v$ holds for $\theta'$ as well. But if $v$ is not minimal for $\theta'$, then by
switching $\theta',\theta$ in the  argument above in this paragraph, we see that $v$
is not minimal for $\theta$ either.

{\it Case 2}. Now let $\Gamma$ be a connected component of the closure of
$\{\lambda| \a^{\lambda}=\a\}$. Let $\lambda_1\in\Q^{Q_0}$. Pick a Weil generic $\lambda\in \Gamma$.
The algebra $\A_\Gamma(v)$ and the sheaf $\A_\Gamma^\theta(v)$ are defined over $\mathbb{Q}$.
It follows that we have a specialization map $K_0(\A_{\lambda}(v)\operatorname{-mod}_{fin})\rightarrow
K_0(\A_{\lambda_1}(v)\operatorname{-mod}_{fin})$. The functors $E_\alpha,F_\alpha$ are also
defined over the rationals, so the specialization map intertwines $[E_\alpha],[F_\alpha]$.
The same is true for the wall-crossing functor $\WC_{\theta\rightarrow \theta'}$
and hence the specialization map intertwines $[\WC_{\theta\rightarrow \theta'}]$.

We claim that there is $\lambda_1\in \Q^{Q_0}$ with $\a^{\lambda_1}=\a$
such that the degeneration map
$K_0(\A_{\lambda}(v)\operatorname{-mod}_{fin})\rightarrow
K_0(\A_{\lambda_1}(v)\operatorname{-mod}_{fin})$ is an embedding.

Let $d$ be the maximal dimension of an irreducible finite dimensional $\A_{\lambda}(v)$-module.
Since $\lambda$ is Weil generic, $d$ is also the maximal dimension of a finite dimensional
irreducible for any other Weil generic parameter.
Let $A_\Gamma$ be the quotient of $\A_\Gamma(v)$ by the ideal generated by the elements
$$\sum_{\sigma\in S_{2d}}\operatorname{sgn}(\sigma)a_{\sigma(1)}\ldots a_{\sigma(2d)}.$$
The algebra $A_\Gamma$ is a finitely generated $\C[\Gamma]$-module. This is proved using
(\ref{eq:affine_decomp}) similarly to the proof of \cite[Theorem 7.2.1]{miura}, compare
to the proof of \cite[Lemma 5.1]{rouq_der}.

So the module of traces,
$A_\Gamma/[A_\Gamma,A_{\Gamma}]$ is finitely generated over $\C[\Gamma]$. Because of this there is a
Zariski open subset $\Gamma^0$ such that the specializations $A_{\lambda_2}$ with
$\lambda_2\in \Gamma^0$ have the same number of irreducible representations.
So we can take any $\lambda_1\in \Gamma^0\cap \Q^{Q_0}$. This shows that the
degeneration map $K_0(\A_{\lambda}(v)\operatorname{-mod}_{fin})\rightarrow
K_0(\A_{\lambda_1}(v)\operatorname{-mod}_{fin})$ is an inclusion that sends classes
of irreducibles to classes of irreducibles.

Since the degeneration map intertwines the operators $[E_\alpha],[F_\alpha]$, we see that $K_0(\mathcal{C}^\theta_{\lambda}(v))$ (the summand corresponding to the dimension $v$ in
the category $\mathcal{C}$ for $(\lambda,\theta)$)  gets mapped onto
$K_0(\mathcal{C}^\theta_{\lambda_1}(v))$ for any $v$ from a given fixed finite set.
It follows that, under the degeneration map the class of an extremal object goes
to the class of an extremal object. Since the degeneration map is compatible with
wall-crossing functors, we reduce the present case to Case 1. In particular, we get an $\a$-action
on $\bigoplus_v K_0(\A_{\lambda}(v)\operatorname{-mod}_{\rho^{-1}(0)})$.

{\it Case 3}. Now consider the general case. Let $\tilde{\lambda}$ denote the Weil generic
element in the connected component of $\overline{\{\lambda| \a^\lambda=\a\}}$ containing
$\lambda$. We can replace $\lambda$ with its integral shift and assume that $\lambda$
is Zariski generic in the closure of the connected component. We still have the injective degeneration maps $K_0(\A_{\tilde{\lambda}}(v)\operatorname{-mod}_{fin})
\rightarrow K_0(\A_\lambda(v)\operatorname{-mod}_{fin})$ intertwining the maps
$[E_\alpha],[F_\alpha]$ as well as the maps given by wall-crossing functors.
So again $K_0(\mathcal{C}^\theta_{\tilde{\lambda}}(v))$ maps bijectively onto
$K_0(\mathcal{C}^\theta_\lambda(v))$ for any $v$ from a given fixed finite set.
It follows that $\WC_{\theta\rightarrow \theta'}$ sends
$D^b_{\mathcal{C}}(\A_\lambda^{\theta}(v)\operatorname{-mod})$
to $D^b_{\mathcal{C}}(\A_{\lambda'}^{\theta'}(v)\operatorname{-mod})$
for any $v$ from a fixed finite set. Arguing as in Case 1, we see that
$L\mapsto L'$ sends extremal objects to extremal objects. We also see that
the operators $[E_\alpha],[F_\alpha]$ give an action of $\mathfrak{a}$ on
$K_0(\mathcal{C})$.
\end{proof}

\subsection{Absence of extremal simples}\label{SS_extrem_abs}
In this section we will use Proposition \ref{Prop:long_shift}, Theorem \ref{Thm:perv} and Proposition
\ref{Prop:extremal_bij} to complete the proof of (II) in the following two cases.
\begin{itemize}
\item[(a)] The quiver $Q$ is of finite type.
\item[(b)] $Q$ is an affine quiver, $v=n\delta, w=\epsilon_0$.
\end{itemize}

\begin{Lem}\label{Lem:fin_finish}
Let $\theta,\theta'$ be two stability conditions. Suppose that
\begin{itemize}
\item  $\theta,\theta'$ are not separated
by $\ker\alpha$, where $\alpha\leqslant v$ is an imaginary root with $\langle \alpha,\lambda\rangle\in \Z$.
\item If $\beta\leqslant v$ is an imaginary root  with $\langle\beta,\lambda\rangle\in \Z$, then $\langle \theta, \beta\rangle>0$.\end{itemize} Let $M$ be an extremal simple  $\A_\lambda^\theta(v)$-module. Then $H_0(\WC_{\theta\rightarrow \theta'}M)$ has a quotient that is an extremal simple $\A_\lambda^{\theta'}(v)$-module.
\end{Lem}
\begin{proof}
Let $\theta_1=\theta, \theta_2,\ldots,\theta_q=\theta'$ be stability conditions such that $\theta_i$ and $\theta_{i+1}$
are separated by $\ker\alpha_i$, where $\alpha_i$ is a real root with $\langle\alpha_i,\lambda\rangle\in \Z$
and $\alpha_i\leqslant v$. We assume that $q$ is minimal with this property. It follows from Proposition \ref{Prop:extremal_bij} that if $M_i$ is an extremal simple $\A_\lambda^{\theta_i}(v)$-module, then the head of $H_0(\WC_{\theta_i\rightarrow \theta_{i+1}}M_i)$ again contains an extremal simple, say $M_{i+1}$. We start with $M_1$ and produce the extremal simples $M_2,\ldots,M_q$.
By the construction,  $M_q$ is a quotient of $H_0(\WC_{\theta_1\rightarrow \theta_q}M)$.
\end{proof}

Now we are ready to prove (II) from the beginning of Section \ref{S_outline}.

\begin{proof}[Proof of (II)]
We need to prove that there are no extremal simples in $\A_\lambda^\theta(v)\operatorname{-mod}_{\rho^{-1}(0)}$
(in the affine case we
assume that $\langle \delta,\theta\rangle>0$). Assume the contrary.

Lemma \ref{Lem:fin_finish} together with Proposition \ref{Prop:long_shift} lead to a
contradiction in case (a).

Now let us deal with case (b) -- the SRA case. Pick an extremal simple $M\in \A_\lambda^\theta(v)\operatorname{-mod}_{\rho^{-1}(0)}$.
We can pick stability conditions $\theta_1,\ldots,\theta_q$ with the following
properties:
\begin{itemize}
\item[(a)] $\theta=\theta_j$ for some $j$.
\item[(b)] $-\theta_q$ and $\theta_1$ lie in chambers separated by $\ker\delta$
and $\langle\theta_1,\delta\rangle>0$.
\item[(c)] $\theta_i$ and $\theta_{i+1}$ are separated by a single wall defined by a real root.
\item[(d)] $q$ is minimal with this property.
\end{itemize}

Let $M_j:=M$ and find extremal simples $M_i\in \A^{\theta_i}_{\lambda_i}(v)\operatorname{-mod}, i=1,\ldots,q,$
(where $\lambda_i\in \lambda+\Z^{Q_0}$ is such that $(\lambda_i,\theta_i)\in \mathfrak{AL}(v)$)
such that $M_i$ and $M_{i+1}$ are in bijection produced by crossing the wall
between $\theta_i$ and $\theta_{i+1}$ (with $M=M_i$ and $M'=M_{i+1}$), see Proposition
\ref{Prop:extremal_bij}.

Using LMN isomorphisms, we can identify $\M^{\theta_i}(v)$ with $\M^{\tilde{\theta}_i}(n\delta)$
and $\A^{\theta_i}_{\lambda_i}(v)$ with $\A^{\tilde{\theta}_i}_{\tilde{\lambda}_i}(n\delta)$
for appropriate $n, \tilde{\theta}_i,\tilde{\lambda}_i$. Note that $\tilde{\theta}_i,\tilde{\theta}_{i+1}$
are still separated by a single wall, for $i=0$, this wall is $\ker\delta$,
and for $i>0$, this is the wall defined by a real root.
Moreover, the weight $\nu$ defined by $v$ is extremal if and only if $n=0$.
Let $M_0$ be the simple in $\A^{\tilde{\theta}_0}_{\tilde{\lambda}_0}(n\delta)\operatorname{-mod}
= \A_{\lambda_0}^{\theta_0}(v)\operatorname{-mod}$ corresponding to $M_1$
under the bijection in (3) of Theorem \ref{Thm:perv}.

Consider the complex $\WC_{\tilde{\theta}_0\rightarrow \tilde{\theta}_q}(M_0)$.
By Theorem \ref{Thm:wc_decomp_short},
$$\WC_{\tilde{\theta}_0\rightarrow \tilde{\theta}_q}(M_0)=\WC_{\tilde{\theta}_1\rightarrow
\tilde{\theta}_q}\circ \WC_{\tilde{\theta}_0\rightarrow \tilde{\theta}_1}(M_0).$$
By Proposition \ref{Prop:long_shift}, the left hand side has vanishing $H_k$ for $k<n$ because $\Gamma^{\tilde{\theta}_0}_{\tilde{\lambda_0}}(M_0)$ is finite dimensional. On the other hand,
by Theorem \ref{Thm:perv}, we have $H_j(\WC_{\tilde{\theta}_0\rightarrow \tilde{\theta}_1}(M_0))\twoheadrightarrow
M_1$ for some $j<n$ and $H_k(\WC_{\tilde{\theta}_0\rightarrow \tilde{\theta}_1}(M_0))=0$
for $k<j$. From Lemma \ref{Lem:fin_finish}, we deduce that
$H_j(\WC_{\tilde{\theta}_1\rightarrow
\tilde{\theta}_q}\circ \WC_{\tilde{\theta}_0\rightarrow \tilde{\theta}_1}(M_0))\twoheadrightarrow
M_q$. We arrive at a contradiction that completes the proof.
\end{proof}

\subsection{Injectivity of $\CC$}\label{SS_CC_inj}
In this section we prove (III): the map $\CC:\bigoplus_{v}K_0(\A_\lambda^\theta(v)\operatorname{-mod}_{\rho^{-1}(0)})
\rightarrow L_\omega$ is injective.

A key step is as follows.

\begin{Lem}\label{Lem:a_action}
The operators $[E_\alpha],[F_\alpha]$ on $\bigoplus_{v}K_0(\A_\lambda^\theta(v)\operatorname{-mod}_{\rho^{-1}(0)})$
give  an $\a$-action.
\end{Lem}
\begin{proof}
It was shown in the proof of Proposition \ref{Prop:extremal_bij} (Section \ref{SS_fin_extr})
that the restrictions of $[E_\alpha],[F_\alpha]$
to $K_0(\mathcal{C})$ define an action of $\a$. On the other hand (II) proved in the previous section
shows that $\mathcal{C}=\bigoplus_v \A_\lambda^\theta(v)\operatorname{-mod}_{\rho^{-1}(0)}$.
This finishes the proof. \end{proof}

\begin{proof}[Proof of (III)]
By Proposition \ref{Prop:a_CC_intertw}, the map $\CC: \bigoplus_{v}K_0(\A_\lambda^\theta(v)\operatorname{-mod}_{\rho^{-1}(0)})\rightarrow L_\omega$
is $\a$-linear. The image coincides with $L_\omega^{\a}$ by (I) and (II).
It follows from the construction of $\mathcal{C}$ and (\ref{eq:cat_action_simples})
that the $\a$-module
$K_0(\mathcal{C})$ is generated by $\bigoplus_{\sigma}  K_0(\A_\lambda^\theta(\sigma\bullet w)\operatorname{-mod}_{\rho^{-1}(0)})$, where the summation is taken
over $\sigma\in W(Q)$ such that $\sigma\omega$ is dominant for $\a$. 
Since  $\mathcal{C}=\bigoplus_{v}\A_\lambda^\theta(v)\operatorname{-mod}_{\rho^{-1}(0)}$, 
it follows that $\CC$ is injective.
\end{proof}

This finishes the proof of Theorem \ref{Thm:verymain}.

\begin{Rem}
Let us deduce the original conjecture of Etingof, \cite[Conjectures 6.3,6.8]{Etingof_affine},
from Conjecture \ref{Conj:main}.
According to results of \cite[Section 5]{GL}, we have a derived equivalence $D^b(\mathcal{H}_{\kappa,c}(n)\operatorname{-mod})
\xrightarrow{\sim} D^b(\A^\theta_\lambda(v)\operatorname{-mod})$ that restricts to
$D^b_{fin}(\mathcal{H}_{\kappa,c}(n)\operatorname{-mod})
\xrightarrow{\sim} D^b_{\rho^{-1}(0)}(\A^\theta_\lambda(v)\operatorname{-mod})$. So the number of finite dimensional irreducible $\mathcal{H}_{\kappa,c}(n)$-modules coincides with the number of irreducible $\A_\lambda^\theta(v)$-modules supported on $\rho^{-1}(0)$.
It is easy to see that the number given by Conjecture \ref{Conj:main} is the same as conjectured by Etingof.
\end{Rem}

\subsection{Conjectures on counting simples with arbitrary support}
In the remainder of the section we would like to discuss two counting problems that are more general
than the problem studied in this paper.

One can  pose a problem of counting $\A_{\lambda}(v)$-irreducibles with given (positive) dimension
of support. An obvious difficulty here is that the number of such modules is infinite. There are, at least, three
different approaches to the counting problem: to deal with a filtration by support on $K_0$,
to work in characteristic $p\gg 0$ or to restrict to a suitable category of modules in characteristic $0$.

\subsubsection{Category $\mathcal{O}$}
An easier special case is when there is a Hamiltonian $\C^\times$-action on $\M^\theta(v)$ with finitely many fixed
points. This action deforms to a Hamiltonian action on $\A_\lambda^\theta(v)$ and hence on $\A_\lambda(v)$. Let $h\in \A_\lambda(v)$
denote the corresponding hamiltonian so that $[h,\cdot]$ coincides with the derivation
of $\A$ induced by the $\C^\times$-action. The algebra $\A_\lambda(v)$ acquires an internal grading
by eigenspaces of $\operatorname{ad}h$, $\A_\lambda(v):=\bigoplus_{i\in \Z}\A_\lambda(v)^i$.
Then we can consider the category $\mathcal{O}_\lambda(v)$ of $\mathcal{A}_\lambda(v)$-modules consisting of all finitely generated modules
with locally nilpotent action of $\bigoplus_{i>0}\mathcal{A}_\lambda(v)^i$, compare with \cite{BGK,LOCat,GL,BLPW}.
The simples in this category are in one-to-one correspondence with the irreducible
modules over the algebra $\mathcal{A}_\lambda(v)^+:=
\mathcal{A}_\lambda(v)^0/(\bigoplus_{i>0}\A_\lambda(v)^{-i}\A_\lambda(v)^i)$. It is not difficult to see
that the algebra $\mathcal{A}_\lambda(v)^+$ is finite dimensional, compare to
\cite[Lemma 3.1.4]{GL}. Moreover, for $\lambda$ in some non-empty
Zariski open subset of $\paramq$, the algebra $\A_\lambda(v)^+$ is naturally identified with
$\C[\M^\theta(v)^{\C^\times}]$, see \cite[Section 5.1]{BLPW}. So we may assume that the irreducibles
in our category $\mathcal{O}$ are parameterized by $\M^\theta(v)^{\C^\times}$.
Also to every fixed point $p$ we can assign the corresponding Verma module, $\Delta_p:=\A_{\lambda}(v)\otimes_{\A_\lambda(v)^{\geqslant 0}}\C_p$.
Here $\C_p$ stands for the 1-dimensional $\A_\lambda(v)^+$-module corresponding to $p$, we view $\C_p$
as an  $\A_\lambda(v)^{\geqslant 0}$-module via the epimorphism $\A_\lambda(v)^{\geqslant 0}\twoheadrightarrow \A_\lambda(v)^+$. For $\lambda$ in some Zariski open subset the category $\mathcal{O}$ is highest weight
with standard objects $\Delta_p$, see \cite[Section 5.2]{BLPW}.
We identify $K_0(\mathcal{O}_\lambda(v))$ with $\C[\M^\theta(v)^{\C^\times}]$ by sending
the class $[\Delta_p]$ of $\Delta_p$ to the basis vector corresponding to $p$.
For example, suppose we consider $\M^\theta(n\delta)$ for a cyclic quiver $Q$ with $\ell$ vertices and $w=\epsilon_0$. Then we get the category $\mathcal{O}_{\kappa,c}(n)$ for  cyclotomic Rational Cherednik algebra $H_{\kappa,c}(\Gamma_n)$
with $\Gamma_n:=\mathfrak{S}_n\ltimes (\Z/\ell\Z)^n$
(at least for some Zariski open subset in $\paramq$; it was conjectured in \cite[Section 3]{GL} that the subset coincides
with the set of all spherical parameters). The Verma modules $\Delta_\tau$ in that category are indexed by the irreducible
representations $\tau$ of $\Gamma_n$ that are in a natural bijection with $\M^\theta(n\delta)^{\C^\times}$, as pointed out by Gordon in \cite[Section 5.1]{Gordon}, let us denote the fixed point corresponding to
$\tau$ by $p(\tau)$. It follows from results of \cite[Section 3]{GL} that $\Delta_\tau$ coincides with
$\Delta_{p(\tau)}$.

 One can ask the question to compute the number of the irreducibles in $\mathcal{O}_\lambda(v)$
with given dimension of support. In the cyclotomic Cherednik algebra case this problem was solved
by Shan and Vasserot in \cite{shanvasserot}. We will state  a conjecture in the case when $X=\M^\theta(v)$ and $Q$ is a cyclic quiver (in this case we do have a Hamiltonian $\C^\times$-action with finitely many fixed points).
We remark that different choices of  $\C^\times$ lead to different choices of the categories $\mathcal{O}$,
but our answer should not depend on the choice.
More precisely, there are derived equivalences relating categories $\mathcal{O}$ for different choices
of $\C^\times$, see \cite{CWR}, these equivalences can be seen to preserve the supports.

Set $\mathcal{O}_\lambda:=\bigoplus_v \mathcal{O}_\lambda(v)$. We also write $\mathcal{O}^w_\lambda$
if we want to indicate the dependence on $w$.

A description of the $\C^\times$-stable points in $\M^\theta(v)$ follows, for example, from the work of Nakajima,
\cite[Sections 3,7]{Nakajima_tensor}. Namely, consider a maximal torus $T\subset \prod_{k\in Q_0}\GL(w_k)$.
Then the  $T$-invariant points on $\M^\theta\{w\}:=\bigsqcup_v\M^\theta(v,w)$
are naturally identified with $$\prod_{k\in Q_0}\M^\theta\{\epsilon_k\}^{w_k}$$
The $\C^\times$-fixed locus in $\M^\theta\{w\}$ is then the union of the fixed points in
$\prod_{k\in Q_0}\M^\theta\{\epsilon_k\}^{w_k}$, in each dimension there are finitely many of those.
The fixed points in $\M^\theta(v,\epsilon_k)$ are indexed by $\ell$-multipartitions
of $n_v$ such that $n_v \delta\in W(Q)\nu$.

We want to state a conjecture on the filtration of $K_0(\mathcal{O}_\lambda)$ by the homological shifts
under the wall-crossing functor $\WC$ through the affine wall.

We again start with the case when $Q$ is a single loop.
Then $K_0(\mathcal{O}_\lambda)=\mathcal{F}^{\otimes r}$, where $r$ stands for  the framing
and $\mathcal{F}$ is the Fock space, i.e., the space with a basis indexed by partitions.
Consider the $r$ copies of the Heisenberg Lie algebras $\mathfrak{heis}^i$ with bases $b^i_j, j\in \mathbb{Z}\setminus \{0\}$ and one more copy of the Heisenberg, $\mathfrak{heis}_{\Delta}$, with basis $b_j, j\in \mathbb{Z}\setminus \{0\}$,
embedded into $\prod_{i=1}^r \mathfrak{heis}^i$ diagonally. Inside  $\mathcal{O}_\lambda(n)$ consider the Serre
subcategory $\operatorname{F}_j \mathcal{O}_\lambda(n)$ spanned by all simples with support of codimension
at least $j$ so that $\operatorname{F}_j\mathcal{O}_\lambda(n)$ is a decreasing filtration on $\mathcal{O}_\lambda(n)$.
We view each of the $r$ copies of $\mathcal{F}$ as a standard Fock space representation of the corresponding Heisenberg
algebra $\mathfrak{heis}^i$. So $K_0(\mathcal{O}_\lambda)$ becomes a $\prod_{i=1}^r \mathfrak{heis}^i$-module
and hence a $\mathfrak{heis}_\Delta$-module.

\begin{Conj}\label{Conj:O_Gies}
Let $m$ denote the denominator of $\lambda$ (equal to $+\infty$ if $\langle\lambda,\delta\rangle\not\in \mathbb{Q}$).
The subspace $K_0(\operatorname{F}_j \mathcal{O}_\lambda)\subset K_0(\mathcal{O}_\lambda)$ is the sum of the images
of the operators $b_{mj_1}\ldots b_{mj_k}$ with $j_1,\ldots,j_k\in \Z_{>0}$ and $(rm-1)(j_1+\ldots+j_k)\geqslant j$.
\end{Conj}

Now let us proceed to the general case: when $Q$ is a cyclic quiver with $\ell$ vertices and the framing $w$
is arbitrary. We again want to describe  the  filtration on $K_0(\mathcal{O}_\lambda^w)$ relative to
the affine wall.  The description will still be given in terms of some Heisenberg action on $K_0(\mathcal{O}_\lambda^w)$.

Let us specify that action. As we have seen above, $K_0(\mathcal{O}^{w}_\lambda)=\bigotimes_{k\in Q_0}K_0(\mathcal{O}^{\epsilon_k}_\lambda)^{\otimes w_k}$. The space $K_0(\mathcal{O}^{\epsilon_k}_\lambda)$ can be thought as an integrable highest weight module $\widetilde{L}_{\omega_k}$, where $\omega_k$ is the fundamental weight corresponding to $k$, for the Lie algebra $\hat{\gl}_{\ell}$ (so that $\widetilde{L}_{\omega_k}=L_{\omega_k}\otimes \mathcal{F}$),
compare to \cite[Section 6]{Etingof_affine}.
Inside $\hat{\gl}_{\ell}$ consider the  Heisenberg subalgebra corresponding to the center of
$\gl_\ell$. It has a basis $b_j$ with $j\in \Z$.

\begin{Conj}\label{Conj:O_shift_gen}
Let $m$ denote the denominator of $\langle\lambda,\delta\rangle$ (equal to $+\infty$ if $\lambda\not\in \mathbb{Q}$).
Consider the subcategory $\operatorname{F}^{aff}_j\mathcal{O}^w$ consisting of all modules $M$ with
$H_i(\WC M)=0$ for $i<j$, where $\WC$ stands for the short wall-crossing functor through the affine wall.
The subspace $K_0(\operatorname{F}^{aff}_j \mathcal{O}^w_\lambda)\subset K_0(\mathcal{O}^w_\lambda)$ is the sum of the images of the operators $b_{mj_1}\ldots b_{mj_k}$ with $j_1,\ldots,j_k\in \Z_{>0},$ and $(\overline{w}m-1)(j_1+\ldots+j_k)\geqslant j$, where $\overline{w}:=\sum_{k\in Q_0}w_k$.
\end{Conj}

Modulo Conjecture \ref{Conj:O_shift_gen},  one can state a conjecture
regarding the filtration by dimension of support.

\begin{Conj}\label{Conj:O_dim_gen}
The span in $K_0(\mathcal{O}_\lambda^w)$ of the classes of all modules with dimension of support $\leqslant i$ is the sum of $\a$-submodules generated by the singular vectors in $\operatorname{F}^{aff}_{s(\overline{w}m-1)}\mathcal{O}^w(v)$,
where $v$ and $i$ are subject to the following condition:
$$w\cdot v-(v,v)/2-s(\overline{w}m-1)\leqslant i.$$
\end{Conj}

We expect that Conjecture \ref{Conj:O_dim_gen} should be an easy corollary of Conjecture \ref{Conj:O_shift_gen}
and techniques developed in Sections \ref{SS_fin_extr},\ref{SS_extrem_abs}.
We remark that it is compatible with Conjecture \ref{Conj:main} and also with the main result of \cite{shanvasserot}.

\subsubsection{Filtration on $K_0$}\label{filtrK0}
One can also work with a filtration on $K_0(\A_\lambda(v)\operatorname{-mod})$ as in
\cite[Conjectures 6.1,6.7]{Etingof_affine}. As before, we assume that
$\A_\lambda(v)$ has finite homological dimension so that $K_0$ of the category of finitely generated
$\A_\lambda(v)$-modules is naturally identified with the split $K_0$ of the category of projective
$\A_\lambda(v)$-modules. Let $\operatorname{F}_j K_0(\A_{\lambda}(v)\operatorname{-mod})$ stand for
the subspace in $K_0(\A_\lambda(v)\operatorname{-mod})$ generated by all objects $\M$
such that $\operatorname{GK-}\dim(\A_{\lambda}(v)/\operatorname{Ann}\M)\leqslant \dim \M^\theta(v)-2j$. Then one can
state a conjecture similar to Conjecture \ref{Conj:O_dim_gen}.

\subsubsection{Characteristic $p$}
Yet another setting where one can state counting conjectures is in characteristic $p\gg 0$. We use
the notation of Section \ref{SS_pos_char}. In particular, $\mathbb{F}$ stands for an algebraically
closed field of characteristic $p$.
To simplify the statement we consider the case of a rational parameter $\lambda$
such that the algebra $\A_{\lambda}(v)_{\C}\otimes \A_\lambda(v)_{\C}^{opp}$ has finite homological dimension.

Set $K^0_p:= K_0({\mathcal A}_{\lambda}(v)_{\mathbb{F}}\operatorname{-mod})$  and
$K^0_\infty=K_0({\mathcal A}_{\lambda}(v)_{\C}\operatorname{-mod})$ (recall that we consider the $K_0$
groups over $\C$). We have the specialization map $\operatorname{Sp}: K^0_\infty \to K^0_p$ for every prime $p\gg 0$
and it is an isomorphism.

Consider the category $\A_\lambda(v)_{\mathbb{F}}\operatorname{-mod}_0$ of finitely generated modules with zero generalized $p$-character.
Set $K_0^p:=K_0(\A_\lambda(v)_{\mathbb{F}}\operatorname{-mod}_0)$. Similarly to Section \ref{SS_quiv_KH},
we have an identification $K_0^p\cong K_0(\operatorname{Coh}_{\rho^{-1}(0)}\M^\theta(v))$.

We have the Ext
pairing $\chi: K^0_p\times K_0^p\to {\mathbb{C}}$, compare to the proof of Proposition \ref{Prop:K0_nondeg_pairing}.
We have basically seen in Section \ref{SS_quiv_KH} that this pairing is non-degenerate.

\begin{Conj}\label{pol1var}
a) There exist polynomials in one variable  $D_i(t)\in \Q[t]$, $i=1,\dots , \dim K_0^p$,
such that for  $p\gg 0$ the dimensions of the irreducible modules
equal $D_i(p)$.

b) Let $\A_\lambda(v)\operatorname{-mod}_0^{\leq d}$ be the
Serre subcategory generated by irreducible objects $L_i$
such that the corresponding polynomial $D_i$ satisfies: $\deg(D_i)\leq d$.
Then the induced filtration on $K^0(\A_\lambda(v)_{\mathbb{F}}\operatorname{-mod}_0)$
is dual to the filtration on $K^0(\A_\lambda(v)_\C\operatorname{-mod})$ considered in  \ref{filtrK0} with respect to the pairing  $\chi$.
\end{Conj}

Let us speculate on a possible scheme of proof.
First, we need an analog of Proposition \ref{Prop:long_shift} that is not available yet,
we believe this is the most important thing missing.
Second, we need an analog of Webster's construction in positive characteristic. The latter
is not expected to be difficult. Theorem \ref{Thm:affine_WC} and an analog of Conjecture \ref{Conj:O_shift_gen}
should carry over to positive characteristic without significant modifications.
This should be sufficient to prove the counting conjecture.

\subsection{Infinite homological dimension}
In this subsection we will state a conjecture on the number of irreducible finite dimensional
$\A_\lambda(v)$-modules in the case when the homological dimension of $\A_\lambda(v)$ is
infinite. Similar in spirit conjectures can be stated for categories $\mathcal{O}$
(or their replacements) or in positive characteristic, but we are not going to
elaborate on that.

Consider the functor $R\Gamma_\lambda^\theta: D^b(\A_\lambda^\theta(v)\operatorname{-mod})
\rightarrow D^b(\A_\lambda(v)\operatorname{-mod})$. It should be a quotient functor,
at least, this is so in the SRA situation thanks to an equivalence $D^b(\A_\lambda^\theta(v)\operatorname{-mod})
\cong D^b(\mathcal{H}_{\kappa,c}(n)\operatorname{-mod})$ established in \cite[5.1]{GL} (see 5.1.6, in particular).
Under this equivalence the functor $R\Gamma^\theta_\lambda$ becomes the abelian quotient functor $\mathcal{M}\mapsto e\mathcal{M}$. According to Conjecture \ref{Conj:loc_main}, this quotient is proper if and only if $\lambda$
lies in the finite union of hyperplanes (to be called ``singular''),  the singular hyperplanes
can be (conjecturally) described explicitly when $Q$ is of finite or of affine type, see Section
\ref{SS_loc_conj}.

So $K_0(\A_\lambda(v)\operatorname{-mod}_{fin})$ becomes a quotient of $K_0(\A_\lambda^\theta(v)\operatorname{-mod}_{\rho^{-1}(0)})$.
Our goal is to provide a conjectural description of this quotient. Our conjecture will consist of two
parts. The first (easier) will deal with the case when $\lambda$ is a Zariski generic point of
a singular hyperplane. The second (much harder) will handle the general case.

Let us deal with the Grassmanian case first. So let $Q$ be a quiver with a single vertex and no arrows.
The singular locus is $\lambda=1-w,2-w,\ldots,-1$. Assume, for convenience, that $\theta>0$ and $2v\leqslant w$.
Identify $\A_\lambda^\theta(v)\operatorname{-mod}$ with $\A_0(v)\operatorname{-mod}$.
The ideals in the latter form a chain: $\{0\}=\J_{v+1}\subsetneq \J_v\subsetneq\ldots\subsetneq \J_1\subsetneq \J_0=\A_0(v)$.
The kernel of the functor $R\Gamma_\lambda^\theta$ can be shown to consist of all modules annihilated
by $\J_i$, where \begin{equation}\label{eq:for_i}i=v+1-\min(v,-\lambda, w+\lambda)\end{equation} (or, more precisely, the complexes with such
homology). On the level of the categorical $\sl_2$-action, those should be precisely the complexes
lying in the image of $F^{v+1-i}$.

Let us return to the general setting.

\begin{Conj}\label{Conj:inf_homol_dim}
\begin{enumerate}
\item Let $\alpha$ be a real root and $\lambda$ be a Zariski generic parameter on a singular hyperplane
$\langle\lambda,\alpha\rangle=s$. Then the complexified $K_0$ of the kernel of $D^b(\A_\lambda^\theta(v)\operatorname{-mod}_{\rho^{-1}(0)})\twoheadrightarrow
D^b(\A_\lambda(v)\operatorname{-mod}_{fin})$ coincides with the image of $f_\alpha^i$, where
$i$ is determined from $s$ and $\underline{v},\underline{w}$ as in (\ref{eq:for_i}).
\item Let $\langle \alpha_j,\cdot\rangle=s_j, j=1,\ldots,k,$ be all singular hyperplanes
with \underline{real} $\alpha_j$ that contain $\lambda$. Then $\ker K_0(D^b(\A_\lambda^\theta(v)\operatorname{-mod}_{\rho^{-1}(0)})\twoheadrightarrow
D^b(\A_\lambda(v)\operatorname{-mod}_{fin}))$ is spanned (as a  vector space)
by the sum of the images  of $f_{\alpha_j}^{i_j}, j=1,\ldots,k$, where the numbers $i_j$ are determined as in (1).
\end{enumerate}
\end{Conj}

We believe that one should not include the singular hyperplanes defined by imaginary roots.
The reason is that there are no finite dimensional irreducibles for a Weil generic $\lambda$
on a hyperplane defined by an imaginary root. When one deals with modules with higher dimensional
support on should modify the conjecture to account for imaginary roots. We are not going to
elaborate on  that.

\end{document}